%% file: main.tex
\documentclass[reqno, 12pt]{amsart}
\usepackage[letterpaper, margin=1in]{geometry}

\usepackage[utf8]{inputenc} 
\usepackage{amsmath}
\usepackage{amsthm} 
\usepackage{amssymb}
\usepackage{enumitem} 
\usepackage[dvipsnames]{xcolor}
\usepackage{graphicx}
\usepackage{tikz}
\usepackage[colorlinks=true]{hyperref}

\makeatletter
\renewcommand{\l@section}{\@tocline{1}{0pt}{1em}{}{}} % table of contents indentation
\renewcommand{\l@subsection}{\@tocline{2}{0pt}{2em}{}{}}
\makeatother 

\hypersetup{bookmarksdepth=3} % PDF bookmarks

\usepackage{biblatex}
\hypersetup{
	urlcolor = NavyBlue, % Colour for external hyperlinks
	linkcolor = NavyBlue, % Colour of internal links
	citecolor = Red, % Colour of citations
    linktocpage = true   % makes page numbers the clickable links in TOC
}

\newtheorem{theorem}{Theorem}[section]

\theoremstyle{definition}
\newtheorem{lemma}[theorem]{Lemma}

\newtheorem{definition}[theorem]{Definition}
\newtheorem{prop}[theorem]{Proposition}
\newtheorem{cor}[theorem]{Corollary}
\newtheorem{rmk}[theorem]{Remark}

\def\CC{\mathbb{C}}
\def\DD{\mathbb{D}}
\def\EE{\mathbb{E}}

\def\II{\mathbb{I}}

\def\PP{\mathbb{P}}
\def\QQ{\mathbb{Q}}
\def\RR{\mathbb{R}}
\def\SS{\mathbb{S}}

\def\ZZ{\mathbb{Z}}

\def\calE{\mathcal{E}}

\def\calG{\mathcal{G}}

\def\calJ{\mathcal{J}}

\def\calL{\mathcal{L}}
\def\calM{\mathcal{M}}
\def\calN{\mathcal{N}}
\def\calO{\mathcal{O}}

\def\calS{\mathcal{S}}

\def\calU{\mathcal{U}}

\newcommand\cA{\mathcal{A}}

\newcommand\cE{\mathcal{E}}

\newcommand\cG{\mathcal{G}}

\newcommand\cI{\mathcal{I}}
\newcommand\cJ{\mathcal{J}}

\newcommand\cL{\mathcal{L}}
\newcommand\cM{\mathcal{M}}
\newcommand\cN{\mathcal{N}}

\newcommand\cS{\mathcal{S}}

\newcommand\cZ{\mathcal{Z}}

\newcommand{\abs}[1]{\left\lvert#1\right\rvert}

\newcommand{\norm}[1]{\left\lVert#1\right\rVert}

\newcommand{\ang}[1]{\left\langle #1 \right\rangle}

\newcommand{\paren}[1]{\left( #1 \right)}
\newcommand{\sqb}[1]{\left[ #1 \right]}
\newcommand{\set}[1]{\left\{ #1 \right\}}
\newcommand{\setcond}[2]{\left\{ #1 \;\middle\vert\; #2 \right\}}

\newcommand{\wt}{\widetilde}

\newcommand{\pp}[2]{\frac{\partial #1}{\partial #2}}
\newcommand{\op}{\mathrm{op}}
\newcommand{\out}{\mathrm{out}}
\newcommand{\grad}{\nabla}
\newcommand{\diag}{\mathrm{diag}}
\newcommand{\Tr}{\mathrm{Tr}}
\newcommand{\pl}{\mathrm{pl}}
\newcommand{\nul}{\mathrm{null}}
\newcommand{\TV}{\mathrm{TV}}

\newcommand{\appropto}{\mathrel{\vcenter{
  \offinterlineskip\halign{\hfil$##$\cr
    \propto\cr\noalign{\kern2pt}\sim\cr\noalign{\kern-2pt}}}}}

\begin{document}

\title[Correlated Two-view Models in High Dimensions]{Optimal Spectral Algorithms for Correlated Two-view Models in High Dimensions}

\author{Hang Du}
\address{Department of Mathematics, Massachusetts Institute of Technology}
\email{hangdu@mit.edu}

\author{Henry Hu}
\address{Department of Mathematics, Massachusetts Institute of Technology}
\email{huhenry@mit.edu}

\author{Saba Lepsveridze}
\address{Department of Mathematics, Massachusetts Institute of Technology}
\email{sabal@mit.edu}

\date{}

\begin{abstract} 
We study high-dimensional inference in correlated two-view models, focusing on spectral methods for strong detection and weak recovery. We introduce a general framework, motivated by a TAP type heuristic from statistical physics, that provides a unified treatment of three canonical models: high-dimensional canonical correlation analysis, and the correlated spiked Wigner and Wishart models. Our main contribution is to construct explicit spectral algorithms in all three settings, that achieve strong detection and weak recovery down to the corresponding thresholds, where we prove matching information-theoretic lower bounds. Furthermore, our spectral procedures operate without knowledge of the model parameters, relying solely on the observed data. This demonstrates the optimality of spectral methods in these models and the broad statistical applicability of the framework.
\end{abstract}

\maketitle

\tableofcontents 

\hypersetup{linkcolor = NavyBlue} 

\input{introduction}

\input{TAP-heuristicMain}

\input{gordonMain}

\input{convergenceMain}

\input{outlierMain}

\input{lower-boundsMain}

\appendix

\input{TAP-heuristicsAppendix}

\input{gordonAppendix}

\input{largest-eigenvalue-bounds}

\input{convergence} 

\input{outlier} 

\input{lower-bounds}

\small
\printbibliography %%%%

\end{document}

%% file: introduction.tex
\section{Introduction}

A central theme in high-dimensional inference is to extract meaningful structure from high-dimensional (random) data. Prominent examples include principal component analysis for identifying low-rank signals \cite{Baik2005, Johnstone2001}, community detection in random graphs \cite{Mossel2015}, and sparse regression for variable selection \cite{Wainwright2009}. While classical formulations of these problems typically involve a single data source, increasing attention has been devoted to settings with multiple correlated views, where the goal is to detect and exploit the shared structure. Such settings arise naturally when multiple data sources are jointly observed and are governed by some common latent representation. For instance, this arises when two sets of measurements are collected on the same underlying objects, as in multi-omics studies \cite{Lock2013}. From a theoretical perspective, canonical correlation analysis (CCA) is a fundamental instance of this problem \cite{BykhovskayaGorin2023}, with closely related formulations arising in correlated spiked random matrix models \cite{Zhangsong25}.

The term high-dimensional refers to settings in which the dimension of the data is comparable to the sample size. In this context, statistical problems exhibit markedly different behavior from their classical low-dimensional counterparts. Standard estimators may become inconsistent, and sharp phase transitions emerge that separate regimes where inference is possible from those where it is impossible. Such phenomena have been extensively studied in a variety of high-dimensional single-view models, including spiked random matrix models, where precise information-theoretic thresholds for detection and recovery are well understood, and simple spectral methods are known to achieve these limits \cite{Baik2005, BenaychGeorgesNadakuditi2011, FeralPeche2007}. In contrast, for correlated two-view models such as high-dimensional canonical correlation analysis, a systematic understanding of the information-theoretic limits and the performance of simple and computationally efficient algorithms remains incomplete. 

More specifically, a gap remains between statistical optimality and algorithmic simplicity in correlated two-view models. Existing approaches that achieve information-theoretic limits often rely on detailed knowledge of model parameters or involve computationally intricate procedures. On the other hand, spectral methods are among the simplest and most widely used techniques in high-dimensional statistics, and are known to achieve optimal thresholds in a variety of classical single-view models. It is therefore natural to ask whether similarly simple spectral methods can attain information-theoretic limits in correlated two-view settings. While spectral approaches have been studied in such models \cite{BykhovskayaGorin2023, KeupZdeborova2025, mergny2025spectral}, existing results either do not achieve the information-theoretic thresholds, or rely on heuristic insights from statistical physics without corresponding rigorous guarantees. As a result, the optimality of spectral methods in these settings remains unclear in terms of rigorous guarantees.

In this work, we show that this gap can be resolved for a broad class of correlated two-view models. We introduce a general framework, motivated by a TAP-type heuristic from statistical physics, which leads to explicit spectral algorithms that do not require prior knowledge of model parameters. Applying this framework to high-dimensional canonical correlation analysis and to correlated spiked Wigner and Wishart models, we construct spectral methods that achieve strong detection and weak recovery down to the corresponding information-theoretic limits. While our approach is inspired by ideas from statistical physics, we establish these results through a fully rigorous analysis. These results demonstrate that, in these models, simple spectral methods suffice to attain optimal performance.

\subsection{Model and problem formulation} \label{sec:Model_Def}
We now introduce the class of two-view models studied in this paper, namely, high-dimensional canonical correlation analysis (CCA), the correlated spiked Wigner model (CSWig), and the correlated spiked Wishart model (CSWish). In all three cases, we observe a pair of high-dimensional random objects $(U,V)$ that exhibit correlations along low-dimensional structures, and the inference task is to recover this structure by leveraging the dependence between the two views.

To simplify notation, we adopt a unified set of conventions across these models. In each case, the correlation between the two views is governed by a parameter $\rho \in [0,1]$, and the model is specified by a probability distribution $\mathcal P_\rho$ over the tuple $(a,b,U,V)$ (or $(a,b,u,v,U,V)$), where $a,b$ (or $a,b,u,v$) are planted high-dimensional vectors, and $U,V$ are the two correlated views. We further denote by $\mathbb{P}_\rho$ for the marginal of $\mathcal P_\rho$ on the observations $(U,V)$. In addition to inference, we also consider the associated correlation detection problem, for which we introduce a null distribution $\QQ$ corresponding to the absence of correlation and low-dimensional structure. Throughout the paper, the notation $\mathcal P_\rho, \PP_\rho$ and $\QQ$ will be used generically, with the underlying model clear from context.

\subsubsection{Canonical correlation model} \label{subsec:CCA}
We fix two parameters $\tau_m,\tau_k>0$ and let $k,m,n$ be integers that go to infinity proportionally and satisfy
\[
\frac{n}{m}\to\tau_m\,,\qquad\frac{n}{k}\to\tau_k\,.
\]
Consider two hidden directions $a \in \SS^{m-1}$ and $b \in \SS^{k-1}$ that are drawn uniformly at random. Let $U \in \RR^{n \times m}$ and $V \in \RR^{n \times k}$, where the $n$ rows $(U_\mu, V_\mu),1\le \mu \le n$ are independent Gaussian pairs with marginal distributions
\[
\sqrt{m}\, U_\mu^\top \sim \calN(0, I_m), 
\qquad
\sqrt{k}\, V_\mu^\top \sim \calN(0, I_k),
\]
and are correlated only along the directions $a$ and $b$, in the sense that for any $u\in \mathbb{R}^m, v\in \mathbb{R}^k$, 
\[
\sqrt{mk}\, \EE[\langle u,U_\mu \rangle\langle v,V_\mu \rangle] = \rho\langle u,a\rangle\langle v,b\rangle .
\]
We denote by $\mathcal P_\rho$ the joint law of $(a,b,U,V)$ and by $\mathbb{P}_\rho$ its marginal on $(U,V)$ (both $\mathcal P_\rho,\mathbb{P}_\rho$ implicitly depend on $n,m,k$). The null model is defined as the law $\QQ=\mathbb{P}_0$ on $(U,V)$. 

\subsubsection{Correlated spiked Wigner model} \label{subsec:wigner}
Fix two constants $\alpha,\beta>0$. For each large integer $n$, consider symmetric matrices $U, V \in \RR^{n \times n}$ of the form
\[
U = \alpha\, aa^\top + Z_a, 
\qquad
V = \beta\, bb^\top + Z_b,
\]
where $\alpha, \beta > 0$ are signal strengths and $Z_a, Z_b \sim \mathrm{GOE}(1/n)$ are independent noise matrices. The vectors $a, b \in \RR^n$ are Gaussian with
\[
\sqrt{n}\, a \sim \calN(0, I_n), 
\qquad 
\sqrt{n}\, b \sim \calN(0, I_n), 
\qquad \quad n\, \EE[a b^\top] = \rho I_n.
\]
We denote by $\mathcal{P}_\rho$ the joint law of $(a,b,U,V)$ and let $\PP_\rho$ be its marginal on $(U,V)$ (both of them implicitly depend on $n,\alpha,\beta$). The null model $\QQ$ is the law of a pair of independent GOE matrices $(U,V)$ (corresponding to the case $a=b=0$). 

It is well known that when $\max\{\alpha,\beta\}>1$, one can achieve strong detection and weak recovery from a single view, by for instance looking at the top eigenvalue and top eigenvector of $U$ or $V$ \cite{BenaychGeorgesNadakuditi2011, FeralPeche2007}. Therefore, we focus on the case where $\alpha,\beta\le 1$ so that a single view is uninformative and the correlation between the two views becomes essential.

\subsubsection{Correlated spiked Wishart model} \label{subsec:wishart}
Fix $\alpha,\beta,\tau>0$ and let $n,m$ be integers that go to infinity proportionally such that $n/m\to\tau$. Consider matrices $U, V \in \RR^{n \times m}$ given by
\[
U = \sqrt{\frac{\alpha}{m}}\, u a^\top + Z_a, 
\qquad
V = \sqrt{\frac{\beta}{m}}\, v b^\top + Z_b,
\]
where $Z_a, Z_b$ have independent $\calN(0, m^{-1})$ entries, and $u, v \sim \calN(0, I_n)$ are independent latent factors. The vectors $a, b \in \RR^m$ satisfy
\[
\sqrt{m}\, a \sim \calN(0, I_m), 
\qquad 
\sqrt{m}\, b \sim \calN(0, I_m),
\qquad
m\, \EE[a b^\top] = \rho I_m.
\]
We denote by $\mathcal{P}_\rho$ the joint law of $(a,b,u,v,U,V)$ and by $\mathbb{P}_\rho$ its marginal on $(U,V)$ (both of them implicitly depend on $\alpha,\beta$ and $m,n$). The null model $\mathbb{Q}$ is a pair of independent $n \times m$ matrices $(U,V)$ with i.i.d. $\mathcal{N}(0,m^{-1})$ entries (corresponding to $a = b = 0$). 

Similarly as above, when $\max\{\tau\alpha^2,\tau\beta^2\}\ge 1$, it is well known that a single view suffices for strong detection and weak recovery, through the top eigenvalue and top eigenvector of $U^TU$ or $V^TV$ \cite{Baik2005}. Thus, we focus on the case where $\tau\alpha^2,\tau\beta^2\le 1$.

\subsubsection{Statistical questions for correlated two-view models}
For all three models, we study the problem of recovering nontrivial information about the hidden directions $a$ and $b$ from the observations $(U,V)$, as well as the closely related correlation detection problem: testing $\PP_\rho$ against $\QQ$. We formalize these tasks as follows.

\begin{itemize}
    \item A test with a statistic $T = T(U,V)$ and a threshold $t$ is said to achieve strong detection if as $n \to \infty$,
    \[
    \PP_\rho[ T < t] + \QQ[ T \ge t] = o(1).
    \]
    
    \item An estimator $(\hat{a},\hat{b})$, depending only on $(U,V)$ and satisfying $\|\hat{a}\|,\|\hat{b}\|\le 1$, is said to achieve high-probability weak recovery (respectively, positive-probability weak recovery), if there exists a constant $\delta > 0$ independent of $n$, such that as $n\to\infty$, 
\[
\mathcal P_\rho\big[ \langle \hat{a}, a \rangle^2 + \langle \hat{b}, b \rangle^2 \ge \delta \big] =1-o(1)\text{ (respectively, $\Omega(1)$)}\,.
\]
\end{itemize}

Our goal is to characterize the phase transitions for strong detection and weak recovery in terms of the correlation parameter $\rho$, and in particular to understand whether spectral algorithms can achieve these tasks down to the information-theoretic thresholds.

\subsection{Main results}
We now present the main results of this paper. In short, for each of the three models, we derive explicit spectral algorithms from the TAP heuristic (which we discuss in more detail in Section~\ref{sec:TAP_HeuristicsMain}), and prove that they achieve strong detection as well as weak recovery down to the corresponding information-theoretic threshold. We first present an informal statement that summarizes our results. For the precise theorem statement for each model, see Theorem~\ref{thm:CCA-outlierMainBody} (CCA), Theorem~\ref{thm:Spiked-outlierMainBody} (CSWig and CSWish).
\begin{theorem}[Informal]
 Let $(U, V)$ be sampled from one of the three above models (see Section~\ref{sec:Model_Def}) with correlation $\rho$. Then there exists a threshold $\kappa$, which may depend on the model parameters (the signal-to-noise ratios $\alpha, \beta$), such that:
 \begin{itemize}
\item If $\rho > \kappa$, there exists a spectral algorithm that achieves strong detection and high-probability weak recovery, without requiring prior knowledge of model parameters.
\item If $\rho < \kappa$, no algorithm achieves strong detection or positive-probability weak recovery, even with full knowledge of the model parameters. 
 \end{itemize}
\end{theorem}

Before stating our formal results, we introduce some conventions. Throughout the paper, we use the standard Landau notation $o(\cdot), O(\cdot), \Omega(\cdot), \Theta(\cdot)$. Precisely, for functions $f$ and $g$ of $n$, we write $f=o(g)$, $f=O(g)$, $f=\Omega(g)$, and $f=\Theta(g)$ to mean that $f/g\to 0$, $f \le Cg$, $f \ge c g$, and $c g \le f \le C g$, respectively, for some constants $c,C>0$ that may depend on some fixed parameters but not on $n$. Moreover, we say that a sequence of events $\{\mathcal E_n\}$ holds with high probability (w.h.p.) if $\PP(\mathcal E_n) = 1 - o(1)$ as $n$ goes to infinity. 

We now present our result for CCA.

\begin{theorem}[CCA Model]\label{thm:CCA-outlierMainBody}
 Let $(U, V)$ be sampled from the CCA model with correlation $\rho$ and planted directions $a, b$. Then there exists a threshold $\kappa = (\tau_m \tau_k)^{-1/4}$ such that:
\begin{itemize}
 \item If $\rho > \kappa$, we can achieve strong detection and weak recovery. Specifically, define
         \[W = \begin{pmatrix}-\kappa U^\intercal U & U^\intercal V \\ V^\intercal U & -\kappa V^\intercal V \end{pmatrix}.\]
        The top eigenvalue and eigenvectors of $W$ achieve the following guarantees:\begin{enumerate}[label=(\roman*)]
           \item The top eigenvalue of $W$ converges w.h.p. to an outlier eigenvalue $\lambda_\out$,
              \begin{equation*}
                \lambda_1(W) = \lambda_\out + o(1) \quad \text{ and } \quad    \lambda_2(W) \leq \lambda_* + o(1).
              \end{equation*}
            where $\lambda_* = (1-\kappa^2)/\kappa$, and $\lambda_\out > \lambda_*$ is the unique solution of $\det \calS(z) = 0$, for some explicit deterministic  $\cS(z)$, defined in \eqref{eq:SDefMain}. On the other hand, for $(U,V)$ sampled from the null model, w.h.p. $\lambda_1(W)\le \lambda_*+o(1)$, and thus $\lambda_1(W)$ achieves strong detection. 
           \item Furthermore, if $\hat w = (\hat a, \hat b)$ is a unit eigenvector corresponding to the largest eigenvalue of $W$, then w.h.p. the overlaps are bounded away from zero and satisfy,
              \begin{equation*}
                |{\langle \hat a, a \rangle}|^2 =
                \frac{x_{*,1}^2}{-x_*^\intercal  \, \partial_z \cS(\lambda_{\mathrm{out}}) \,  x_*} +o(1), 
                \quad 
                |{\langle \hat b, b\rangle }|^2 =
                \frac{x_{*,2}^2}{-x_*^\intercal  \, \partial_z \cS(\lambda_{\mathrm{out}}) \,  x_*} +o(1).
              \end{equation*}
        where $x_* = (x_{*, 1}, x_{*, 2})$ is any vector in $\ker \cS(\lambda_{\mathrm{out}})$. Thus, the top eigenvector of $W$ achieves high-probability weak recovery. 
        \end{enumerate} 
  \item If $\rho < \kappa$, there is no test that achieves strong detection, and there is no estimator that achieves positive-probability weak recovery.
  \end{itemize}
\end{theorem}

From Theorem~\ref{thm:CCA-outlierMainBody}, we identify $\kappa = (\tau_m \tau_k)^{-1/4}$ as the information-theoretic threshold for strong detection and weak recovery in the CCA model. To the best of our knowledge, this threshold has not been established previously. Moreover, we show that a simple and explicit spectral algorithm achieves both tasks down to this threshold. Notably, the associated spectral statistic $W$ does not require prior knowledge of model parameters, as $\tau_m$ and $\tau_k$ can be directly inferred from the data.

\begin{figure}[ht]
    \centering
    \includegraphics[width=0.48\linewidth]{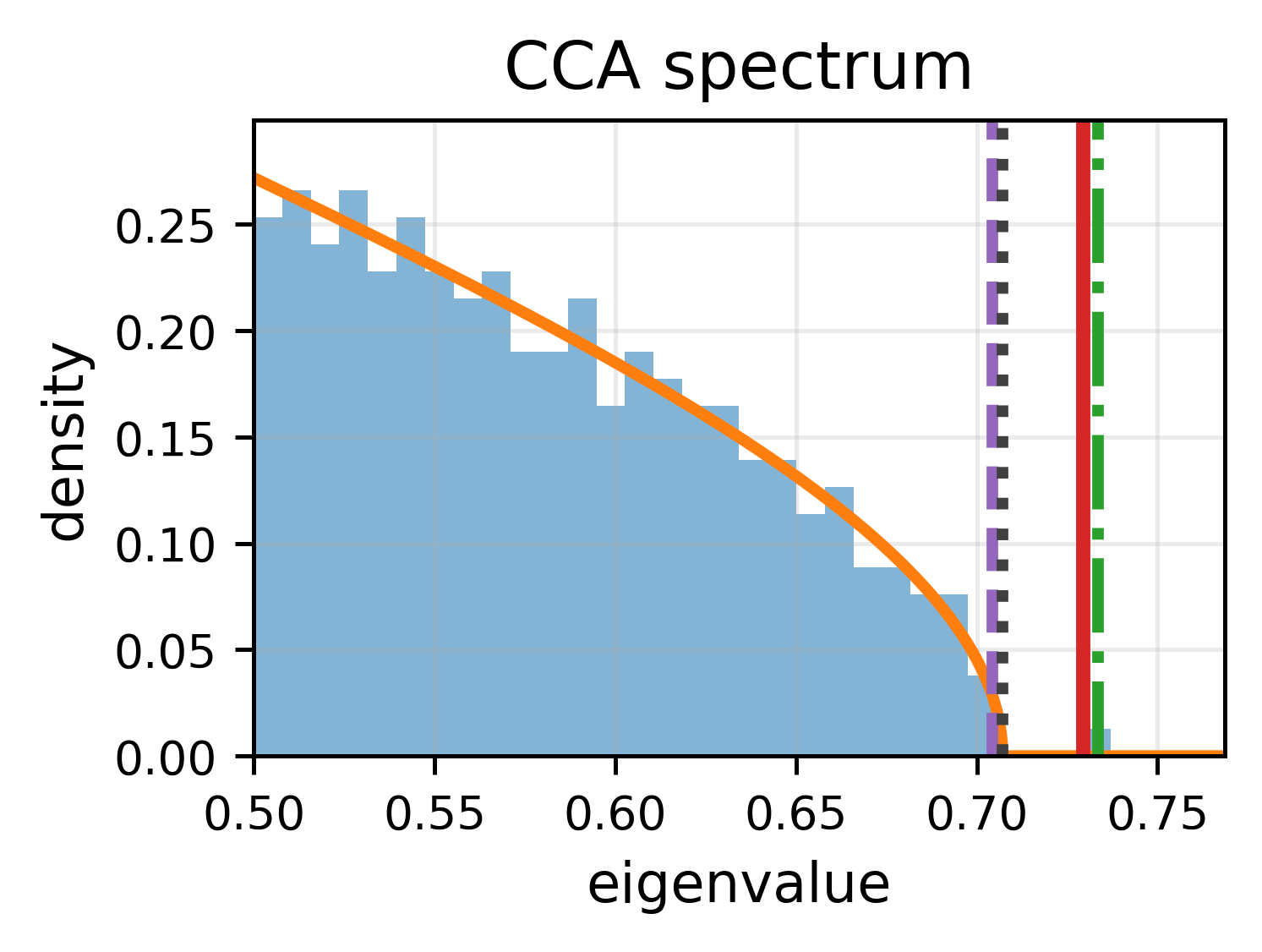}
    \hfill
    \includegraphics[width=0.48\linewidth]{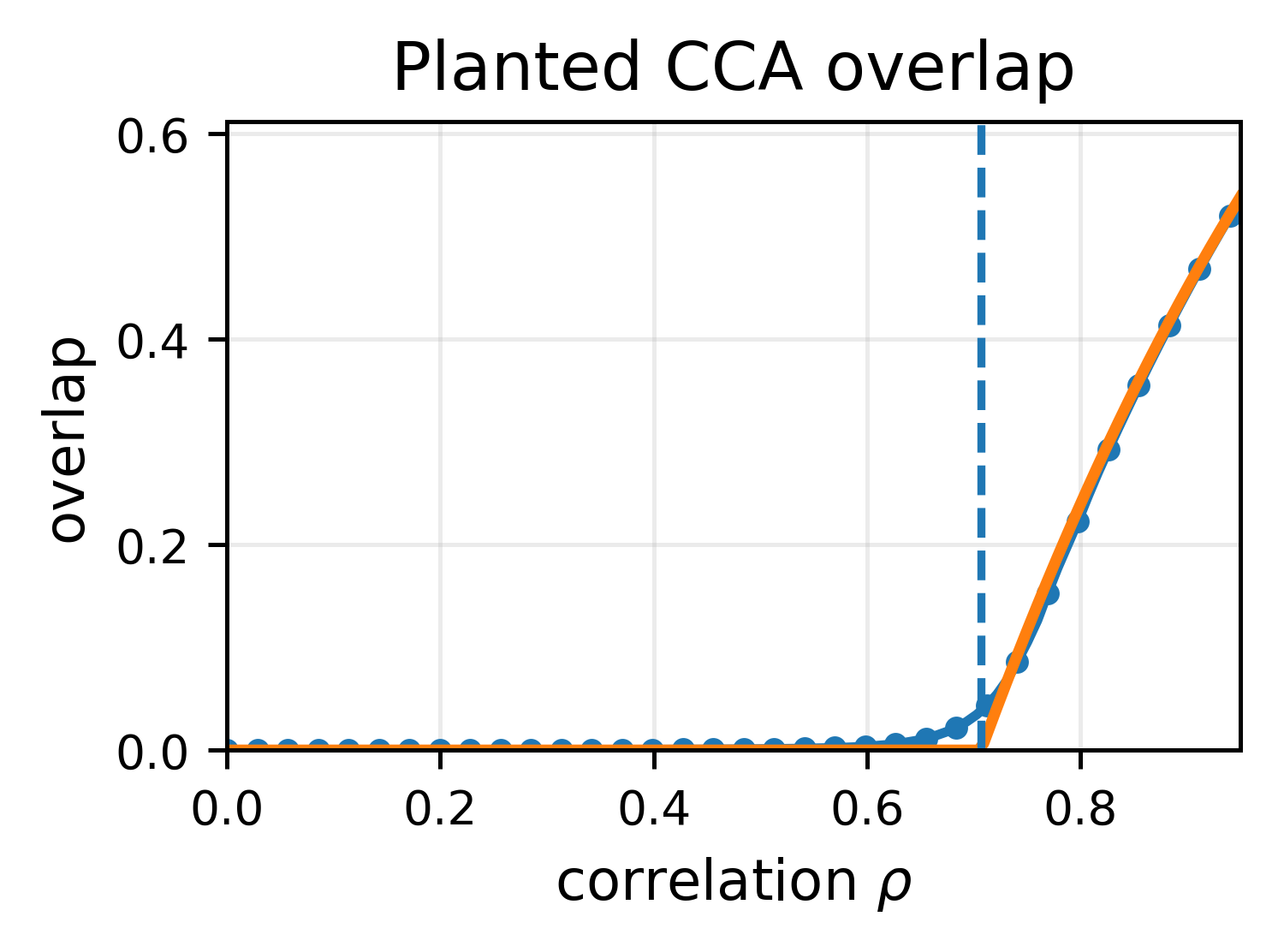}
    \caption{
    The panels show numerical simulations for the CCA model with parameters $n=10000$, $m=k=5000$, and threshold $\kappa = 0.707$. The left panel shows the spectrum of the CCA estimator $W$ with planted correlation $\rho = 0.85$. The right panel shows the overlap, $|\langle \hat w, (a,b)\rangle|^2$, averaged over 50 trials, as a function of $\rho$.} 

    \label{fig:CCA}
\end{figure}
In Figure~\ref{fig:CCA}, we provide a numerical simulation for the CCA model. The left panel shows the spectrum when the correlation ($\rho = 0.85$) is above the threshold ($\kappa = 0.707$). We see that the top eigenvalue ($\lambda_1 = 0.729$, solid red) matches well with the predicted location of the outlier ($\lambda_\out = 0.734$, dashed green) while the second eigenvalue ($\lambda_2 = 0.704$, dashed purple) matches closely with the predicted bulk edge ($\lambda_* = 0.707$, dotted black). The remaining spectrum matches the theoretical limiting spectrum of the null model (solid orange), which we establish in Section~\ref{sec:convergenceMain}. The right panel shows a numerical simulation of the overlap between the top eigenvector and the hidden directions at various correlation strengths. Each data point (blue dots) represents the average overlap, as a function of the correlation $\rho$, over $50$ trials. The predicted deterministic overlap is given as the solid orange line. Notice at the threshold $\kappa$ (dashed blue), the overlap between the top eigenvector with the hidden directions increases sharply with $\rho$.

We now turn to the correlated spiked models, where the model parameters $\alpha$ and $\beta$ are not directly observable from the data. We first present random matrix results that assume knowledge of $\alpha$ and $\beta$, and then show how to obtain algorithms that do not require this assumption.  In the context of the following theorems, $\tau = 1$ for the CSWig model.

\begin{theorem}\label{thm:Spiked-outlierMainBody} 
  Let $(U, V)$ be sampled from the CSWig model or CSWish model with correlation $\rho$, model parameters $0 < \alpha, \beta < \tau^{-1/2}$ and planted spikes $a, b$. Then there exists a threshold,
    \begin{equation}\label{eq:CSSpike-threshold}
    \kappa(\alpha,\beta,\tau) := \left(\frac{(1-\tau\alpha^2)(1-\tau\beta^2)}{\tau^2\alpha^2\beta^2}\right)^{1/4},
    \end{equation}
such that:
\begin{itemize}
 \item If $\rho > \kappa$, we can achieve strong detection and weak recovery. Specifically, define
       \begin{equation}\label{eq:W-DefWigWish}
       W(\alpha, \beta) := \begin{pmatrix}
            \bar U & \kappa\bar V\\
            \kappa \bar U & \bar V
            \end{pmatrix}
            \end{equation}
 where depending on the model,
   \[\begin{cases}
                \bar U_{\text{wig}} &:= \alpha U - \alpha^2 I_n, \\
                \bar V_{\text{wig}} & := \beta V - \beta^2 I_n,
            \end{cases} \qquad \begin{cases}
                \bar U_{\text{wish}} &:= \frac{\alpha}{1+\alpha} U^\intercal U - \alpha\tau I_m, \\
                \bar V_{\text{wish}} & := \frac{\beta}{1+\beta} V^\intercal V - \beta\tau I_m.
            \end{cases}\]
   The top eigenvalue and eigenvectors of $W$ achieve the following guarantees:
      \begin{enumerate}[label=(\roman*)]
          \item The top eigenvalue of $W$ converges w.h.p. to an outlier eigenvalue $\lambda_\out$,
            \[\lambda_1(W)=\lambda_{\out}+o(1) \qquad\text{and}\qquad \lambda_2(W)\le 1+o(1),\]
          where $\lambda_{\out}>1$ is the unique solution of $\det \cS(\lambda_{\out})=0$ for some explicit deterministic  $\cS(z)$, defined in \eqref{eq:wigner-Sigma-rho} (respectively  \eqref{eq:wishart-Sigma-rho}). On the other hand, for $(U, V)$ sampled from the null model, w.h.p. $\lambda_1(W) \le 1+o(1)$, and thus $\lambda_1(W)$ achieves strong detection.
          \item Furthermore if $\hat{w}=(\hat a,\hat b)$ is a unit eigenvector corresponding to the largest eigenvalue of $W$, then w.h.p. the overlaps are bounded away from zero and satisfy,
          \begin{equation*}
              |\langle \hat a,a\rangle|^2 \geq \Omega(1) \quad \text{ and } \quad  |\langle \hat b,b\rangle|^2 \geq \Omega(1).
          \end{equation*}
           
         Thus, the top eigenvector of $W$ achieves high-probability weak recovery.
      \end{enumerate}
  \item If $\rho < \kappa$, there is no test that achieves strong detection, and there is no estimator that achieves positive probability weak recovery. 
\end{itemize}
\end{theorem}

The thresholds \eqref{eq:CSSpike-threshold} for the two models
appearing in Theorem~\ref{thm:Spiked-outlierMainBody}
match the known results in \cite{Zhangsong25,YLS25}. Items (i) and (ii) in the above theorem show that in the regime where single-view is uninformative, there exists an explicit spectral algorithm, that achieves strong detection and weak recovery down to the information-theoretic threshold.

The following theorem shows there is a spectral algorithm which achieves both strong detection and weak recovery without prior knowledge of $\alpha$, $\beta$, or $\rho$, by searching over an appropriate set of candidate values $(\tilde{\alpha}, \tilde{\beta})$ and applying the estimator defined above. 

\begin{theorem}[Parameter-free detection and recovery]\label{thm:Spiked-parameter-free}
Let $(U,V)$ be sampled from the CSWig model or CSWish model with correlation $\rho$, model parameters $0<\alpha,\beta\le \tau^{-1/2}$ and planted spikes $a,b$. 
Fix $\varepsilon\in(0,1/4)$, let us define the admissible grid
\[
    \cZ_{\varepsilon}
    := \Bigl\{ (\tilde\alpha,\tilde\beta)\ :\ \tilde{\alpha}, \tilde{\beta}\in \varepsilon^9\mathbb Z \cap (0,\tau^{-1/2}),\ \kappa(\tilde\alpha,\tilde\beta, \tau)<1
    \Bigr\}.
\]
We define the test statistic, $ \Lambda_\varepsilon = \max_{\cZ_\varepsilon} \lambda_1(W(\tilde \alpha, \tilde \beta))$.
\begin{enumerate}[label=(\roman*)]
   \item Under the null model, $ \Lambda_\varepsilon\le 1+o(1)$ with high probability.
    \item Under the planted model, if $\alpha,\beta \in \tau^{-1/2}(2\varepsilon, 1]$ and $\rho > \kappa + \varepsilon$, then with high probability $\Lambda_\varepsilon > 1+\Omega(1)$ and the corresponding eigenvector achieves weak recovery:
    \[|\langle \hat a,a\rangle|^2\ge \Omega(1),
        \qquad
        |\langle \hat b,b\rangle|^2\ge \Omega(1).\]
\end{enumerate}
\end{theorem}

Theorem~\ref{thm:Spiked-parameter-free} suggests a practical algorithm for detection and recovery in the CSWig and CSWish models that does not require knowledge of $\alpha$ and $\beta$. Focusing on the CSWig model, for example, given matrices $U$ and $V$, one may first test whether $\max\{\alpha,\beta\}>1$ by examining the top eigenvalues of $U$ and $V$. In this regime, a single view suffices and classical algorithms apply. Otherwise, one can apply Theorem~\ref{thm:Spiked-parameter-free}: for a suitable $\varepsilon>0$, perform a search over $(\tilde{\alpha},\tilde{\beta})\in \mathcal Z_{\varepsilon}$\footnote{Indeed, if one further assumes that $\alpha,\beta\in (2\varepsilon,1-2\varepsilon)$, then it suffices to search over an $\epsilon^2$-grid, and similarly for the CSWish model. See Appendix~\ref{subsec:parameterFreeDetection} for more discussions on the choice of grid mesh size.} and compute $\Lambda_\varepsilon=\max_{\mathcal Z_{\varepsilon}}\lambda_1(\tilde{W}(\tilde{\alpha},\tilde{\beta}))$. If the correlation intensity is $\varepsilon$-above the threshold, $\Lambda_\varepsilon$ achieves strong detection\footnote{From a practical perspective, one may reject the null hypothesis if $\Lambda_\varepsilon\ge 1+\omega(n^{-2/3})$ from the Tracy-Widom heuristic, or more safely, reject if $\Lambda_\varepsilon\ge 1+\omega(n^{-1/2})$ from the bound implied by Gordon's inequality.}, while the associated eigenvector yields weak recovery. A similar procedure applies to the CSWish model. 

We make yet another interesting observation: for the edge case that $\max\{\alpha,\beta\}$ equals to $1$ (CSWig) or $\tau^{-1/2}$ (CSWish) and $\kappa=0$, while the matrix $W(\alpha,\beta)$ from Theorem~\ref{thm:Spiked-outlierMainBody} fails to achieve strong detection or weak recovery as $W(\alpha,\beta)$ becomes diagonal and does not capture correlations between the two views, the above algorithm still applies.

\begin{figure}[ht]
    \centering
    \includegraphics[width=0.48\linewidth]{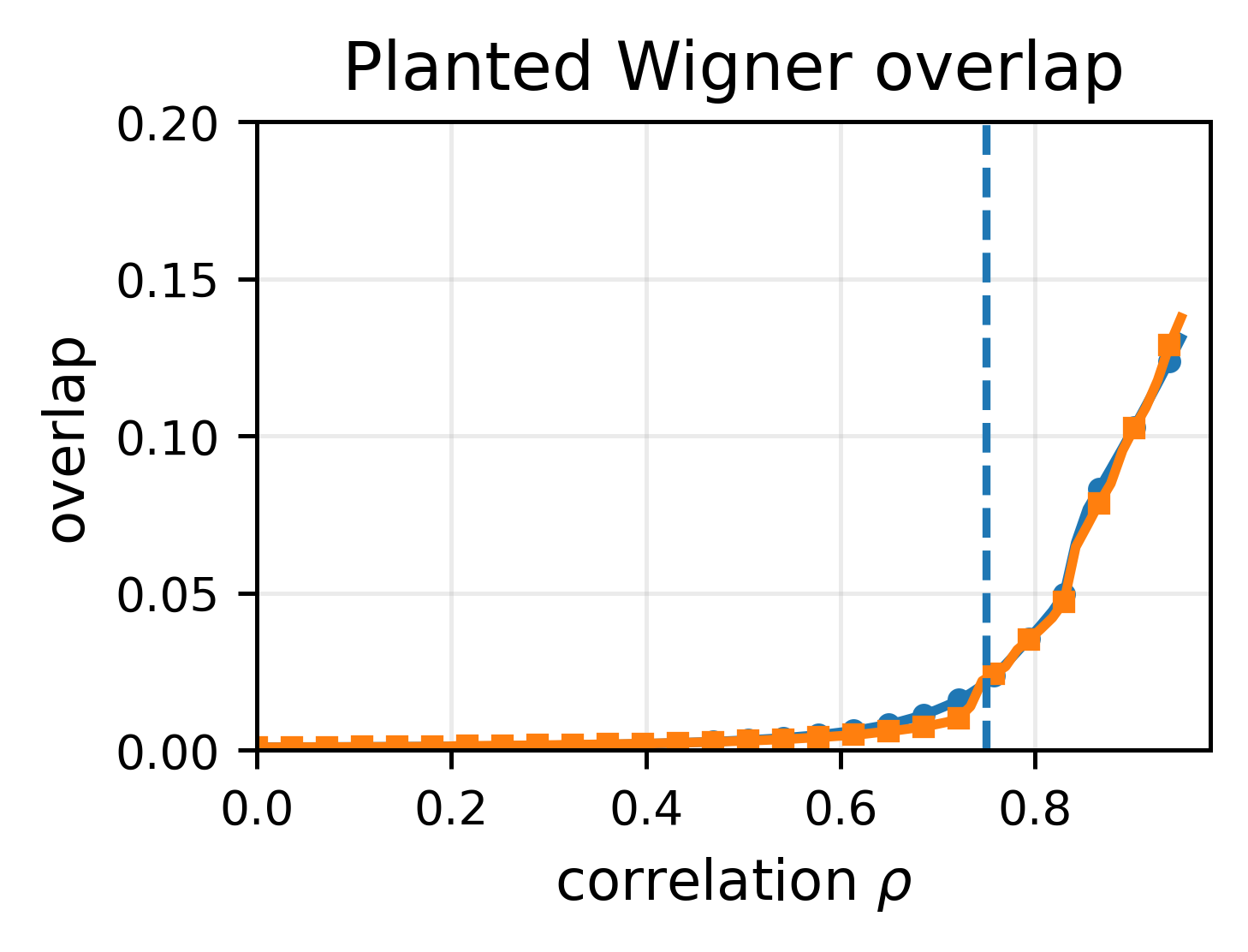}
    \hfill
    \includegraphics[width=0.48\linewidth]{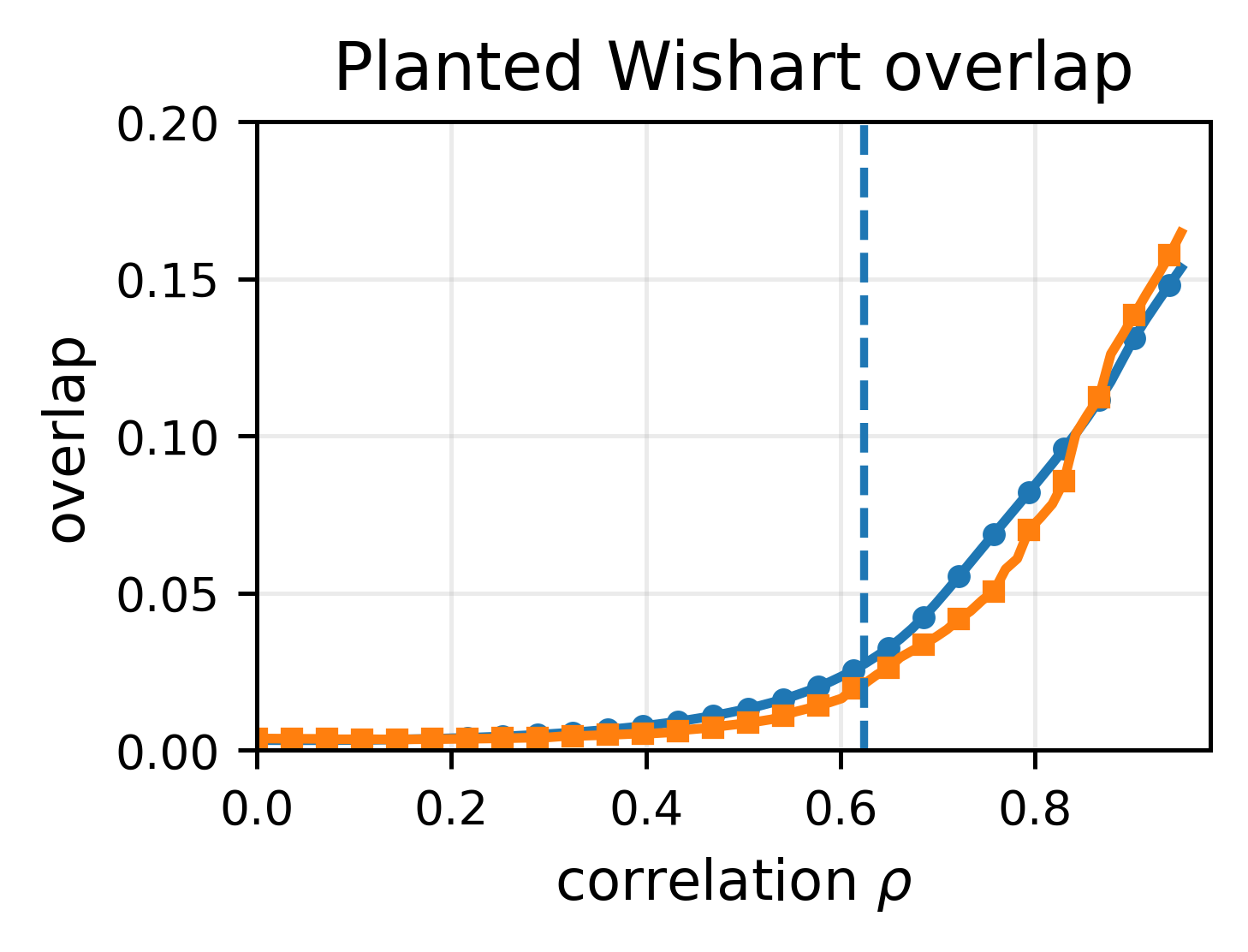}
    \caption{The two panels plot numerical simulations of the overlap, $|\langle \hat w, (a,b)\rangle|^2$, as a function of $\rho$. The left panel shows simulations for CSWig with parameters $n=5000$, signal strengths $\alpha=\beta=0.8$, and threshold $\kappa = 0.75$, averaged over 12 trials.  The right panel shows simulations for CSWish with parameters $n=10000$, $m=5000$, signal strengths $\alpha=\beta=0.6$, and threshold $\kappa = 0.624$, averaged over 24 trials.}
    \label{fig:wigner-wishart}
\end{figure}

In Figure~\ref{fig:wigner-wishart}, we present numerical simulations of our spectral algorithms. The blue dots show the overlap between the top eigenvector of $W=W(\alpha,\beta)$ and the hidden signals at various correlation levels $\rho$ for the CSWig and CSWish models. While the threshold effect is less pronounced than in the CCA model, we observe the same qualitative behavior: at the threshold $\kappa$ (dashed blue line), the overlap increases sharply as $\rho$ grows. We also implement the data-driven algorithm described above, performing a search over a $10\times 10$ grid of candidate values $(\tilde{\alpha}, \tilde{\beta})$. The resulting overlaps (orange squares) closely match those obtained from the top eigenvector of $W$.

Finally, taking advantage of Theorem~\ref{thm:Spiked-parameter-free}, we also prove that one can estimate the signal strengths from the data. 

\begin{theorem}\label{thm:recoverySignalStrengths-MainBody}
 For the CSWig and CSWish models, let $(U,V)$ be sampled from $\PP_\rho$ and $\kappa$ be the corresponding threshold. Assume for some $\varepsilon > 0$ it holds that $\rho > \kappa + \varepsilon$ and $\alpha,\beta \in \tau^{-1/2}(\varepsilon, 1]$. Then there exist estimators $\hat{\alpha},\hat{\beta}$ depending only on $U,V$ and $\varepsilon$ \footnote{The exact construction can be found in Theorem~\ref{thm:wigner-strength-estimation} (CSWig) and Theorem~\ref{thm:wishart-strength-estimation} (CSWish).} such that with high probability
   \[\hat{\alpha} = \alpha + o(1), \quad \hat{\beta} = \beta + o(1).\]
\end{theorem}

\subsection{Related work and connections to our approach}\label{subsec:prior_work}

High-dimensional statistical inference problems in single view settings have been extensively studied over the past two decades. A central class of examples is given by spiked random matrix models (a.k.a. low-rank estimation problems), including the spiked covariance and spiked Wigner models, in which a low-rank signal is planted in high-dimensional noise. These models are by now well understood: for a broad class of priors, such as Gaussian priors, simple and natural spectral methods are known to achieve the information-theoretic thresholds for detection and recovery \cite{Baik2005,BenaychGeorgesNadakuditi2011,FeralPeche2007,Johnstone2001} (these results are commonly known as BBP-type phase transitions). For more structured priors, such as sparse signals, statistical–computational gaps may arise, and this phenomenon has been precisely characterized in a variety of settings \cite{SchrammWein2022,SohnWein2025}. Generalizations of these models to higher-order tensors have also been extensively studied; see, e.g., \cite{LesieurEtAl2017,Li2025TensorPCA,RichardMontanari2014,WeinElAlaouiMoore2019} and the references therein.

In contrast, correlated two-view models remain far less understood, even for simple priors such as Gaussian or uniform distributions. A key difficulty in moving from single view to multi-view settings is the lack of an obvious spectral method: unlike in classical spiked models, there is a priori no canonical spectral statistic that directly captures the underlying signal. As a result, both the information-theoretic limits and the presence of statistical–computational gaps remain challenging to characterize in these models. While recent works have advanced our understanding \cite{BykhovskayaGorin2023, KeupZdeborova2025, Zhangsong25,  mergny2025spectral, SwainRidoutNemenman2025, TabanelliMergnyZdeborovaKrzakala2025, YLS25}, a systematic picture of their information-theoretic limits and optimal algorithms remains incomplete.

Specifically, for correlated spiked Wigner model, \cite{YLS25} establishes information-theoretic thresholds for both strong detection and weak recovery under a class of priors including Gaussian. On the algorithmic side, \cite{Zhangsong25} proposes a sophisticated cycle-counting algorithm, based on low-degree polynomial techniques, that achieves the limit (similar results are obtained for the correlated spiked Wishart model). These results indicate that no statistical–computational gap arises for detection and recovery in these models with Gaussian priors. However, the algorithms in \cite{Zhangsong25}, while taking polynomial-time, involve intricate polynomial computations and require precise knowledge of the model parameters $\alpha$ and $\beta$. \cite{mergny2025spectral} analyzes a simple parameter-free spectral algorithm arising from the partial least square (PLS) method. Despite its simplicity and broad applicability, the PLS method performs sub-optimally. Therefore, it remains unclear that whether spectral methods can achieve these tasks down to the threshold without requiring parameter knowledge.

For high-dimensional canonical correlation analysis, \cite{BykhovskayaGorin2023} analyzes a specific spectral algorithm, which can be interpreted as finding the most correlated directions in the column spaces of the two datasets. While this is arguably the most natural spectral approach for this model, it only succeeds when the correlation intensity is well above the threshold\footnote{Following notations in Section~\ref{subsec:CCA}, the algorithm in \cite{BykhovskayaGorin2023} achieves strong detection or weak recovery only when $\rho > ((\tau_m - 1)(\tau_k - 1))^{-1/4}$, whereas we show that the optimal threshold is $(\tau_m \tau_k)^{-1/4}$. Nevertheless, their method applies to a more general setting where the Gaussian covariance structures are arbitrary and unknown.}. Moreover, the information-theoretic limits for strong detection and weak recovery in this model are not yet fully characterized, and it remains open whether efficient algorithms can achieve these limits.

While natural spectral methods fail to achieve the desired limits, more sophisticated approaches based on insights from statistical physics suggest candidate statistics that may be optimal \cite{LesieurKrzakalaZdeborova2017, ZdeborovaKrzakala2016}. We now briefly review this perspective.

The starting point of the TAP approach is a heuristic connection between the feasibility of detection or recovery and the structure of the posterior distribution. In many high-dimensional inference problems, the thresholds for strong detection and weak recovery are conjectured to coincide with a structural phase transition of the posterior measure. In particular, under the so-called replica symmetric regime, the posterior mass is expected to concentrate around its barycenter, namely the posterior mean. Whether this barycenter significantly deviates from the origin then reflects the possibility or impossibility of weak recovery.

Motivated by this connection, one is led to study the posterior mean. In the models considered here, the posterior distribution can be expressed as a mean-field spin glass model. Insights from spin glass theory then suggest that the posterior mean approximately satisfies a system of TAP equations \cite{ThoulessAndersonPalmer1977}. These equations give rise to a scalar consistency equation for the relevant order parameters, whose solutions characterize the qualitative behavior of the model. In particular, the emergence of a nontrivial solution marks the onset of the phase transition and determines the information-theoretic threshold. 

While the TAP equations are typically nonlinear and difficult to analyze directly, they can be suitably linearized. This linearization yields a natural matrix, whose leading eigenvalue is expected to detect the onset of the phase transition, while the associated top eigenvector is expected to achieve nontrivial recovery of the underlying signal down to the threshold. We refer to \cite{MontanariSen2024} for a detailed discussion of this approach.

The TAP-based perspective has been successfully applied to a range of high-dimensional inference problems. In particular, a line of work \cite{LesieurKrzakalaZdeborova2015MMSE,LesieurKrzakalaZdeborova2015SparsePCA,LesieurKrzakalaZdeborova2017} combines TAP heuristics with approximate message passing algorithms to characterize phase transitions and optimal performance in low-rank estimation problems. More recently, related ideas have been extended to multi-view models, including a heuristic study of the correlated spiked Wishart model \cite{KeupZdeborova2025}. 
Additionally, near the completion of this manuscript, we became aware of concurrent work \cite{LuSenYangInPrep} that applies related techniques to correlated spiked Wigner models with multiple (general $L\ge 2$) views. Compared to these works, our approach introduces a conceptually simple but crucial modification: while TAP-based constructions typically depend on the correlation parameter $\rho$ (see \cite{KeupZdeborova2025,LuSenYangInPrep}), we replace $\rho$ by its critical value $\kappa$ in the spectral procedure. This removes the need for prior knowledge of $\rho$ and leads to significant technical simplifications.

More concretely, a key step in establishing the BBP-type phase transitions in Theorems~\ref{thm:CCA-outlierMainBody}, \ref{thm:Spiked-outlierMainBody} is to control the largest eigenvalue of the associated spectral matrix under the null model. A notable feature of our construction is that, after replacing $\rho$ by its critical value $\kappa$, the spectral edge of the matrix remains invariant as $\rho$ varies, while the leading eigenvalue undergoes a transition. This  enables a relatively simple analysis of the top eigenvalue in the null model via comparison techniques such as Gordon's inequality. In contrast, parameter-dependent constructions typically lead to spectral statistics whose edge varies with $\rho$, requiring more delicate random matrix analysis to establish the separation between the bulk and the outlier. 

Beyond technical simplification, the invariance of the spectral edge also plays a crucial role in enabling parameter-agnostic algorithms for the correlated spiked models. In particular, as discussed around Theorem~\ref{thm:Spiked-parameter-free}, one can perform a search over candidate values and identify informative points by detecting eigenvalues that significantly separate from the common spectral bulk edge. We believe that this technique may be of independent interest and applicable to a broader class of problems.

\vspace{0.5\baselineskip}
\noindent
\textbf{Concurrent work.} As mentioned above, during the completion of this manuscript, we became aware of concurrent work by Yang, Sen and Lu \cite{LuSenYangInPrep}. Their work studies the correlated Wigner model in the general $L$-view setting for fixed $L\ge 2$, and identifies the corresponding spectral threshold. At a high level, both works are inspired by the linearized TAP approach, though the resulting spectral statistics and analyses differ substantially.

In the two-view setting, the correlation structure is governed by a single scalar parameter, which allows us to replace the unknown correlation by a fixed threshold value and obtain parameter-free procedures. Such a reduction no longer appears possible in the general $L$-view setting, where the correlation structure is encoded by an $L\times L$ Gram matrix and characterizing the threshold intrinsically becomes a multi-dimensional problem. Accordingly, \cite{LuSenYangInPrep} studies a symmetrized matrix, arising directly from the linearized TAP equations, which depends explicitly on the underlying Gram matrix. This leads to rather different spectral behavior. In their setting, (approximately) the top eigenvalue remains fixed at $1$ and the phase transition occurs when the bulk edge moves below it. In contrast, our bulk edge is bounded by $1$ and the phase transition occurs when the top eigenvalue moves above it.

While our parameter-free approach in the two-view setting is more statistically flavored, their work provides an important step toward understanding the multi-view setting.

\subsection{Proof outline}

In the remainder of this paper, we focus on the main ideas underlying Theorem~\ref{thm:CCA-outlierMainBody} (CCA), and defer technical details as well as the proofs of Theorem~\ref{thm:Spiked-outlierMainBody} (CSWig and CSWish) to the appendix.

As discussed earlier, the spectral statistic $W$ is motivated by the TAP approach from spin glass theory. While this heuristic is not required for the proofs, it provides useful guidance for constructing the optimal spectral procedure. For completeness, we present in Section~\ref{sec:TAP_HeuristicsMain} a derivation of the TAP equations for CCA and explain how they lead to the matrix $W$.

The analysis of $W$ proceeds in three main steps. First, in Section~\ref{sec:gordonMain} we study its behavior under the null model $\QQ$, where no correlation is present, and establish tail bounds for the largest eigenvalue using Gordon's inequality. Then, in Section~\ref{sec:convergenceMain} we further analyze the spectral properties of $W$ under the null model by proving convergence of its Stieltjes transform, using tools from high-dimensional probability and random matrix theory. Finally, in Section~\ref{sec:outlierMain} we study the effect of the planted correlation structure by treating it as a perturbation of the null model, and establish the BBP-type phase transition for $W$ as stated in Theorem~\ref{thm:CCA-outlierMainBody}.

The matching lower bounds follow from standard techniques, including $\chi^2$-divergence methods (and their truncated variants) together with a structural analysis of the posterior distribution; see Section~\ref{sec:lower-boundsMain}.

\subsection{Extensions and open questions}

We briefly discuss potential extensions of our results and some future directions. We expect that many of our results, including the BBP-type phase transition and the impossibility results, extend to a broader class of priors and beyond Gaussian noise. Our analysis relies on Gaussianity primarily through controlling the top eigenvalue under the null model via Gordon's inequality. While this leads to technical simplifications, we do not believe it is essential, and expect that more general settings can be handled using sharper tools from random matrix theory.

It is also natural to extend our study from two-view to multi-view models (as considered in \cite{LuSenYangInPrep}). While much of our analysis should extend to this setting (at least conceptually), a key challenge is to obtain parameter-agnostic spectral algorithms. In the two-view case, this is achieved by replacing $\rho$ with its critical value $\kappa$, whereas in the multi-view setting the correlation strength is described by a Gram matrix, and there is no direct analogue of a single critical value $\kappa$. Designing data-driven algorithms that do not require prior knowledge of the model parameters in multi-view models remains an interesting open problem.

More broadly, while existing and concurrent works together with our results suggest that the TAP approach leads to optimal spectral algorithms across a range of high-dimensional statistical models, a unified theory is still lacking. It would be valuable to identify general conditions on the model and the priors under which linearizing the TAP equations yields optimal spectral statistics. We leave this to future work.

\vspace{0.5\baselineskip}
\noindent
\textbf{Acknowledgment}. We would like to warmly thank Nike Sun for many insightful comments and helpful suggestions on this manuscript. We are also grateful to Brice Huang for suggesting the Gordon approach to bound the largest eigenvalue in the null model. Finally, we thank Zhangsong Li, Yue Lu, Subhabrata Sen, and Xiaodong Yang for stimulating discussions.
H.~Du, H.~Hu, and S.~Lepsveridze are partially supported by the NSF–Simons Research Collaboration Grant (Award No.~2031883).

%% file: TAP-heuristicMain.tex
\section{TAP motivation for the CCA statistic}\label{sec:TAP_HeuristicsMain}
As discussed in Section~\ref{subsec:prior_work}, a key step in constructing the candidate spectral statistic is to derive the TAP equations and appropriately linearize them. The TAP equations can be viewed as a system approximately satisfied by the posterior mean. In this section, we provide a heuristic derivation of these equations in the CCA model following the classic physics literature \cite{Mezard2017, ThoulessAndersonPalmer1977}. The Wigner and Wishart cases can be treated similarly, and we defer the details to the appendix.

We now focus on the CCA model, following the notations in Section~\ref{subsec:CCA}. While the main results assume a prior $\operatorname{Unif}(\mathbb{S}^{m-1}\times \mathbb{S}^{k-1})$, for technical and presentational convenience we instead consider the prior $\operatorname{Unif}(\{\pm 1/\sqrt m\}^m\times \{\pm 1/\sqrt k\}^k)$. Compared to the spherical prior, although the resulting TAP equations (see \eqref{eq:TAP-equation-CCA} below) differ slightly, their linearization leads to the same matrix $W$ as in Theorem~\ref{thm:CCA-outlierMainBody}.

We begin with the derivation of the TAP equations. Let $\mathbb{P}(\cdot\mid U,V)$ denote the posterior distribution of $(\sqrt m a,\sqrt k b)\in \{\pm 1\}^{m+k}$ (under the cube prior) given the observations $(U,V)$; see \eqref{eq:CCA-posterior-TAP} below for its precise form. Fixing a realization of $(U,V)$, our goal is to understand the posterior mean, which amounts to analyzing the marginal means of $\mathbb{P}(\cdot\mid U,V)$ coordinate-wise. As we essentially only care the results near the threshold, we may assume $(U,V)$ is sampled from the null model (i.e., no correlation between $U,V$ is present). 

At a high level, the derivation proceeds as follows. We first view $\PP(\cdot\mid U,V)$ as a factor model and write down the corresponding belief propagation (BP) fixed-point equations. These equations, involving $\Theta((m+k)\times n)$ probability measures, can be interpreted as approximate consistency conditions for the marginals of $(\sqrt m\, a,\sqrt k\, b)$ under cavity versions of $\PP(\cdot \mid U,V)$ (where “cavity” refers to leaving one node out) and related auxiliary measures. Since the marginals of the cavity measures are expected to be close to those of the original posterior, analyzing the BP fixed-point equations provides insight into the posterior mean. For further details on the heuristics and derivation of BP equations, we refer to \cite[Chapter~9]{MezardMontanari2009}.

Next, although the BP equations are complicated and involve probability measures, they can be simplified in the large $n$ limit via Gaussian approximations, leading to a system involving $\Theta((m+k)\times n)$ scalar quantities. Finally, by applying suitable Taylor expansions and law-of-large-numbers approximations, this system reduces to equations for $\Theta(m+k+n)$ scalar quantities, with $m+k$ of them encoding the posterior means under $\PP(\cdot\mid U,V)$. This yields the final TAP equations. We refer to \cite{Mezard2017} for a more detailed derivation for a closely related model.

Precisely for the CCA model, a straightforward computation yields that $\PP(\cdot \mid U,V)$ is defined such that for any $a\in \{\pm1\}^m,b\in \{\pm 1\}^k$, 
\begin{equation}\label{eq:CCA-posterior-TAP}
  \PP(a,b\mid U,V)
  \propto
  \prod_{\mu\in [n]}
  \exp\set{-\frac{\rho^2 x_\mu(a)^2+\rho^2 y_\mu(b)^2-2\rho x_\mu(a)y_\mu(b)}{2(1-\rho^2)}}\,,
\end{equation}
where we define
\[
  x_\mu(a):=\sum_{i\in [m]} U_{\mu i}a_i,
  \qquad
  y_\mu(b):=\sum_{j\in [k]} V_{\mu j}b_j.
\]
Consider the following marginals for the entries of $a$ and $b$: 
\begin{itemize}
    \setlength\itemsep{3pt}
    \item For $i\in [m]$, $m^a_{i\to \mu}(\cdot)$ is the marginal of $a_i$ given all interactions with the $U^\mu$ term removed. 
    \item For $i\in [m]$, $\hat{m}^a_{\mu\to i}(\cdot)$ is the marginal of $a_i$ when only the interactions involving $U^\mu$ is kept. 
    \item We define $m^b_{i\to \mu}(\cdot)$ and $\hat{m}^b_{\mu\to i}(\cdot)$ for $i\in [k]$ similarly.
\end{itemize}
The BP fixed-point equations yield that approximately, for all $i\in [m],\mu\in [n]$ and $a_i\in \{\pm 1\}$
\begin{align*}
    m^a_{i\to \mu}(a_i) & \propto \prod_{\nu\ne \mu} \hat{m}_{\nu\to i}^a(a_i)\,, \\
    \hat{m}^a_{\mu\to i}(a_i) & \propto \sum_{\substack{a_j, b_j\\ \text{not $a_i$}}}\exp\set{ -\tfrac{\rho^2 x_\mu(a)^2 + \rho^2y_\mu(b)^2 - 2\rho x_\mu(a)y_\mu(b)}{2(1-\rho^2)} }
    \prod_{j\ne i} m^a_{j\to \mu}(a_j) \prod_{j} m^b_{j\to \mu}(b_j)\,,
\end{align*} 
and similarly expressions holds if we switch $a$ with $b$. 

 Because all the measures support on $\{\pm 1\}$, we may express 
 \begin{align*}
 &m_{i\to \mu}^a(a_i) \propto \exp\paren{h_{i\to \mu}^a a_i}, \qquad \hat{m}^a_{\mu\to i}(a_i) \propto \exp(\hat{h}_{\mu\to i}^a a_i)\,,
 \end{align*}
and likewise for $b$. So $\tanh(h^a_{i\to \mu})$ is the mean of $m^a_{i\to \mu}(a_i)$ and similarly $\tanh(h^b_{i\to \mu})$ is the mean of $m^b_{i\to \mu}(b_i)$. 

By the central limit theorem, under the product measure of $\prod m^a_{j\to \mu}\prod m^b_{j\to \mu}$, we may treat the terms inside the sums as independent Gaussians
  \[x_\mu(a) \sim\cN\paren{U_{\mu i} a_i + \sum_{j\in [m]\setminus\{i\}}U_{\mu j} \tanh(h^a_{j\to \mu}), \sum_{j\in [m]\setminus \{i\}}U_{\mu j}^2(1-\tanh^2(h^a_{j\to \mu}))},\]
and similarly for $y_\mu(b)$, where we used the fact that $a_j^2 = b_j^2 = 1$. The replica-symmetric heuristic allows us to substitute the variances by deterministic scaler quantities
\[
\sum_{j\in [m]\setminus\{i\}}U_{\mu j}^2(1-\tanh^2(h_{j\to \mu}^a))\approx \frac{1}{m}\sum_{j\in [m]}(1-\tanh^2(h_{i\to \mu}^a))\approx 1-q_m\,,
\]

and similarly the variance of $y_\mu(b)$ is approximately $1-q_k$, where
  \begin{equation}\label{eq:replica_symmetric}
  q_m \approx \frac{1}{m} \sum_{j\in [m]}\tanh^2(h^a_{j\to\mu})\,,\qquad 
  q_k\approx \frac{1}{k} \sum_{j\in [k]}\tanh^2(h^b_{j\to\mu})\,.
  \end{equation}
Under the Gaussian approximation, $\hat m^a_{\mu\to i}(a_i)$ can be written as the Gaussian expectation
\begin{align*}
  \hat m^a_{\mu\to i}(a_i)
  &\propto
  \EE_{\substack{x\sim \cN(X_{\mu\to i},\,1-q_m)\\ y\sim \cN(Y_\mu,\,1-q_k)}}
  \exp\set{-\frac{\rho^2(x+U_{\mu i}a_i)^2+\rho^2y^2-2\rho(x+U_{\mu i}a_i)y}{2(1-\rho^2)}}.
  \label{eq:CCA-factor-integral-a}
\end{align*}
where we define the quantities
  \[X_{\mu\to i} = \sum_{j \ne i} U_{\mu j} \tanh(h^a_{j\to\mu}), \qquad Y_\mu = \sum_{j \in [k]} V_{\mu j} \tanh(h^b_{j\to\mu}).\]
We may explicitly compute this to obtain that
\[h^a_{i\to\mu} = \frac{\rho}{\rho^2 q_mq_k - 1}\sum_{\nu\ne \mu} \Big(\rho q_k\sum_{j\ne i}U_{\nu j} \tanh(h^a_{j\to\nu}) - \sum_{j \in [k]} V_{\nu j} \tanh(h^b_{j\to\nu})\Big)U_{\nu i},\]
and we have similar expressions for $h_{i\to\mu}^b$. 

At this point, we may derive the order-parameter equations, i.e., a self-consistent equation system for the scale parameters $q_m,q_k$. Denote for simplicity
$
\gamma={\rho}/({\rho^2q_mq_k-1}).
$
From above we expect that for each $\mu\in [n]$, the empirical distribution of $\{h_{i\to \mu}^a\}_{i\in [m]}$ is given by a Gaussian with mean $0$ and variance
\[
\sigma_m^2=\gamma^2\sum_{\nu\neq \mu}\frac{1}{m}\Big(\rho^2q_k^2\cdot \frac{1}{m}\sum_{j\neq i}\tanh^2(h_{j\to \nu}^a)+\frac{1}{k}\sum_{j\in [k]}\tanh^2(h_{j\to\nu}^b)\Big)\,.
\]
Using the replica symmetric approximations, we see that
\[
\sigma_m^2\approx \bar{\sigma}_m^2:=\gamma^2\cdot \tau_m\cdot (\rho^2q_k^2q_m+q_k)\,.
\]
Similarly, the empirical distribution of $\{h_{i\to \mu}^b\}_{i\in [k]}$ is given by a Gaussian with mean $0$ and variance
\[
\sigma_k^2\approx \bar{\sigma}_k^2:=\gamma^2\cdot \tau_k\cdot (\rho^2q_m^2q_k+q_m)\,.
\]
Recall the definition of $q_m,q_k$ in \eqref{eq:replica_symmetric}, we conclude that
\[
q_m\approx \mathbb{E}_{h^a\sim \mathcal N(0,\bar{\sigma}_m^2)}\tanh^2(h^a)\,,\qquad q_k\approx \mathbb{E}_{h^b\sim \mathcal N(0,\bar{\sigma}_k^2)}\tanh^2(h^b)\,.
\]

These equations provide self-consistent conditions for parameters $q_m$ and $q_k$. While $q_m =q_k =0$ is always a solution to the above system (referred as the uninformative point), we now check that a non-trivial solution of $(q_m,q_k)\in (0,1)^2$ exists if and only if $\rho>(\tau_m\tau_k)^{-1/4} =: \kappa$. To see this, note that near the threshold where non-trivial solutions emerge, $q_m,q_k\approx 0$, and thus $\mathbb{E}\tanh^2(h^a)\sim \bar{\sigma}_m^2$ and $\mathbb{E}\tanh^2(h^b)\sim \bar{\sigma}_k^2$. Using this approximation, the equations become
\[q_m\sim \gamma^2\tau_mq_k(1+\rho^2q_mq_k) \quad \text{ and } \quad q_k\sim \gamma^2\tau_kq_m(1+\rho^2q_mq_k),\]
Multiplying them together and cancel out $q_mq_k$, we obtain $\gamma^4\tau_m\tau_k(1+\rho^2q_mq_k)^2\approx 1$. When $q_m,q_k\approx 0$, $\gamma^2\approx \rho^2$, and this yields the threshold $\rho^4\tau_m\tau_k=1$. This suggests that $\kappa$ should be the phase transition threshold. 

To further simplify the equations regarding $h_{i\to \mu}^a,h_{i\to\mu}^b$, we define for $i\in [m]$ and $j\in [k]$,
\[
h_i^a=\gamma\sum_{\nu}\Big(\rho q_k\sum_{j\ne i}U_{\nu j} \tanh(h^a_{j\to\nu}) - \sum_{j \in [k]} V_{\nu j} \tanh(h^b_{j\to\nu})\Big)U_{\nu i}\,, 
\]

and likewise for $h_j^b$, with $M_i^a=\tanh(h_i^a),M_{j}^b=\tanh(h_j^b)$. Moreover, we define for $\nu\in [n]$, 
\[
N_{\nu}^a=\rho q_k\sum_{i\in [m]}U_{\nu i}\tanh (h_{i\to\nu}^a)-\sum_{j\in [k]}V_{\nu j}\tanh (h_{j\to \nu}^b)\,,
\]

and likewise for $N^b_\nu$. We now proceed to derive approximate consistent conditions for the vectors $M^a,M^b,N^a,N^b$.  On the one hand, by definition, 
\[
h_i^a=\gamma \sum_{\nu}(N_\nu^a-\rho q_k U_{\nu i}\tanh(h_{i\to \nu}^a))U_{\nu i}\approx \gamma\sum_\nu N_\nu^a U_{\nu i}-\gamma\rho q_k\tau_m M_i^a\,,
\]
where we used the fact that 
\[
\sum_\nu \rho q_k U_{\nu i}^2\tanh(h_{i\to \nu}^a)\approx \frac{\rho q_k}{m}\sum_{\nu} \tanh(h_{i\to \nu}^a)\approx \rho q_k\tau_m M_i^a.
\]
Hence, 
\[
M_i^a\approx \tanh\Big(\gamma \sum_\nu N_\nu^a U_{\nu i}-\gamma\rho q_k\tau_m M_i^a\Big)\,,
\]
and similar expressions for $M_j^b$ hold. On the other hand, approximately we have $h_{i\to\nu}^a\approx h_i^a-\gamma N_\nu^aU_{\nu i}$ and $h_{i\to \nu}^b\approx h_i^b-\gamma N_\nu^bV_{\nu i}$. Hence,
\begin{align*}
    N_\nu^a&\approx \rho q_k\sum_{i\in [m]}U_{\nu i}\tanh(h_i^a-\gamma N_\nu^aU_{\nu i})-\sum_{i\in [k]}V_{\nu i}\tanh(h_i^b-\gamma N_\nu^b V_{\nu i})\\
    &\approx \rho q_k\sum_{i\in [m]}(U_{\nu i}M_i^a-(1-(M_i^a)^2)\gamma U_{\nu i}^2 N_\nu^a)-\sum_{i\in [k]}(V_{\nu i}M_i^b-(1-(M_i^b)^2)\gamma V_{\nu i}^2 N_\nu^b)\\
    &\approx \rho q_k\sum_{i\in [m]}U_{\nu i}M_i^a-\gamma \rho q_k(1-q_m)N_\nu^a-\sum_{i\in [k]}V_{\nu i}M_i^b+\gamma(1-q_k)N_\nu^b\,.
\end{align*}
Here in the second line, we perform Taylor expansion of $\tanh(\cdot)$ to the second order, and in the last line we use once again the replica symmetric approximations \eqref{eq:replica_symmetric}. Similarly, we get approximate equations for $N_\nu^b$. 

In conclusion, we get the TAP equations for the CCA model with respect to four vectors $M^a\in \mathbb{R}^m,M^b\in \mathbb{R}^k$ and $N^a,N^b\in \mathbb{R}^n$ as follows:
\begin{equation}\label{eq:TAP-equation-CCA}
\begin{cases}
    M^a=\tanh\left(\gamma U^\top N^a-\gamma \rho q_k\tau_m M^a\right)\,,\\
    M^b=\tanh\left(\gamma V^\top N^b-\gamma \rho q_m\tau_k M^b\right)\,,\\
    N^a=\rho q_kUM^a-VM^b-\gamma \rho q_k(1-q_m)N^a+\gamma (1-q_k)N^b\,,\\
    N^b=\rho q_m VM^b-UM^a-\gamma \rho q_m(1-q_k)N^b+\gamma (1-q_m)N^a\,.
\end{cases}
\end{equation}
 
We now linearize the TAP equation around the uninformative fixed point $q_m=q_k=0$. Since the prior is symmetric, this corresponds to $M^a,M^b,N^a,N^b\approx 0$. This suggests we may use $\tanh(x)\sim x$ for $x=o(1)$ to rewrite the first two equations in \eqref{eq:TAP-equation-CCA} as 
\[
\begin{cases}
    M^a=\gamma U^\top N^a-\gamma\rho q_k\tau_m M^a\,,\\
    M^b=\gamma V^\top N^b-\gamma\rho q_m\tau_kM^b\,.
\end{cases}
\]
Combining with the last two equations in \eqref{eq:TAP-equation-CCA}, we get self-consistent linear equations of $M^a,M^b,N^a,N^b$. We further plug in $q_m=q_k=0,\gamma=-\rho$, which yields
\[
\begin{cases}
    M^a=-\rho U^\top N^a\,,\\
    M^b=-\rho V^\top N^b\, \\
\end{cases} \quad
\begin{cases}
    N^a=-VM^b-\rho N^b\,,\\
    N^b=-UM^a-\rho N^a\,.
\end{cases}
\]
Solving the last two equations gives
\[
N^a=\frac{\rho UM^a-VM^b}{1-\rho^2}\,,\qquad N^b=\frac{-UM^a+\rho VM^b}{1-\rho^2}\,.
\]
Plugging into the first two equations, we obtain the equation 
\begin{equation*}
  W^\rho
  \begin{pmatrix}M^a\\ M^b\end{pmatrix}
  =
  \frac{1-\rho^2}{\rho}
  \begin{pmatrix}M^a\\ M^b\end{pmatrix},
  \qquad
  W^\rho:=
  \begin{pmatrix}
    -\rho U^\intercal U & U^\intercal V \\
    V^\intercal U & -\rho V^\intercal V
  \end{pmatrix}.
\end{equation*}
Thus the TAP heuristics predicts that the relevant estimator is the top eigenvector of $W^\rho$. At the threshold $\rho=\kappa=(\tau_m\tau_k)^{-1/4}$, the eigenvalue on the right-hand side becomes $\lambda_*=({1-\kappa^2})/{\kappa}$. Hence the threshold linearization is precisely
\begin{equation}\label{eq:CCA-threshold-W-eigenvalue}
  W
  \begin{pmatrix}H^a\\ H^b\end{pmatrix}
  =
  \lambda_*
  \begin{pmatrix}H^a\\ H^b\end{pmatrix},
  \qquad
  W=
  \begin{pmatrix}
    -\kappa U^\intercal U & U^\intercal V \\
    V^\intercal U & -\kappa V^\intercal V
  \end{pmatrix}.
\end{equation}
This is exactly the matrix estimator that appears throughout the rest of the paper.

%% file: gordonMain.tex
\section{Largest eigenvalue tail bounds}\label{sec:gordonMain}
In this section, we show that for $(U,V)$ sampled from the null CCA model (see Section~\ref{subsec:CCA}), the largest eigenvalue of $W$ is w.h.p. bounded by $\lambda_*+o(1)$. We will prove this by Gordon's inequality, which we briefly review below.

The classic Gordon's inequality \cite{Gordon85, TOH15, VRbook} allows one to compare an optimization problem over Gaussian processes with another (simpler) optimization problem that has similar covariance structure. In particular, because the top eigenvalue of a matrix with Gaussian entries can be written as an optimization problem over some Gaussian process, we can obtain good tail bounds for the largest eigenvalue by comparing to a suitable optimization problem. 

For our purpose, we will need a theorem similar to the one given in \cite{TOH15}, which we state below. The proof follows from Gordon's inequality and can be found in Appendix~\ref{sec:gordonAppendix}. Define
\begin{equation*}
    {G} = (G_1, \ldots, G_k), \quad {g}=(g_1,\ldots,g_k), \quad  {h}=(h_1,\ldots,h_k), 
\end{equation*}
where $G_j \in \RR^{n_j \times m_j}$, $g_j \in \RR^{n_j}$, and $h_j \in \RR^{m_j}$ for all $j \in [k]$. Let ${w}\in\RR^k$ and let
\[
\calS_{{x}} \subseteq \RR^{n_1} \times \cdots \times \RR^{n_k}, \qquad
\calS_{{y}} \subseteq \RR^{m_1} \times \cdots \times \RR^{m_k}, \qquad
\calS_{{z}} \subseteq \RR^d.
\]
be compact sets. Suppose $\psi : \calS_{{x}} \times \calS_{{y}} \times \calS_{{z}} \to \RR$ is a continuous function. We define two objects of optimization 
\begin{align*}
    \Phi({G}) &=
\min_{{z}\in \calS_{{z}}}
\min_{{x}\in \calS_{{x}}}
\max_{{y}\in \calS_{{y}}}
\Big\{
\sum_{j=1}^k x_j^\top G_j y_j + \psi({x},{y},{z})
\Big\}. \\
\phi({g},{h})
&=
\min_{{z}\in \calS_{{z}}}
\min_{{x}\in \calS_{{x}}}
\max_{{y}\in \calS_{{y}}}
\Big\{
\sum_{j=1}^k \|y_j\|_2\, g_j^\top x_j
+\sum_{j=1}^k \|x_j\|_2\, h_j^\top y_j
+\psi({x},{y},{z})
\Big\}.
\end{align*}
We treat $\Phi$ as the primary optimization objective and $\phi$ as an auxiliary optimization objective. The following theorem relates $\Phi$ and $\phi$.
\begin{theorem}\label{thm:gordon-usefulMain}
Suppose the entries of ${G}, {g}, {h}$ are i.i.d.\ $\calN(0,1)$. For all $t\in\RR$,
\[
\PP(\Phi({G})<t)
\le 2^k\,\PP(\phi({g},{h})\le t).
\]
\end{theorem}
We now apply this theorem to obtain tail estimates for the largest eigenvalue of $W$ under the CCA null model (see Section~\ref{subsec:CCA}). We will assume the threshold $\kappa = (\tau_m\tau_k)^{-1/4} < 1$ where $\tau_m = n/m+o(1)$ and $\tau_k = n/k+o(1)$. Recall $W$ is the block matrix of form
\begin{equation*}
	W = \begin{pmatrix}
		-\kappa U^\intercal U & U^\intercal V \\
		V^\intercal U & -\kappa V^\intercal V
	\end{pmatrix}.
\end{equation*}

\begin{theorem}\label{thm:CCAGordonEdge}
Let $W$ be a matrix as above. Then for any $\varepsilon \in (0,1)$
\begin{equation*}
    \PP\paren{\lambda_{\max}(W) \geq \tfrac{1-\kappa^2}{\kappa}  + \varepsilon} \leq e^{-\Omega(n \varepsilon^2)}.
\end{equation*}
\end{theorem}
\begin{proof}
    Since matrix $W$ is symmetric, the largest eigenvalue is obtained by the best unit test vector ${x} = (x_1,x_2)^\intercal \in \RR^{m} \times \RR^{k}$,
    \begin{align*}
	\lambda_{\max} &=\max_{\|x\|=1}x^TWx= \max_{\norm{x_1}^2 + \norm{x_2}^2 = 1} \set{-\kappa \norm{Ux_1}^2 -\kappa\norm{Vx_2}^2+2\ang{Ux_1, Vx_2}} \\
	&= \max_{\norm{x_1}^2 + \norm{x_2}^2 = 1} \set{-(1+\kappa) \norm{Ux_1}^2 -(1+\kappa)\norm{Vx_2}^2+\norm{Ux_1 +Vx_2}^2}.
\end{align*}
We linearize this optimization via 
\begin{gather*}
	-\norm{Ux_1}^2 = \min_{y_1'} \set{2\ang{y_1', Ux_1} + \norm{y_1'}^2}, \quad
	-\norm{Vx_2}^2 = \min_{y_2'} \set{2\ang{y_2', Vx_2} + \norm{y_2'}^2},  \\
	\norm{Ux_1+Vx_2}^2= \max_{z} \set{2\ang{z,Ux_1+Vx_2} -\norm{z}^2}.
\end{gather*}
Hence, we get that 
\begin{equation*}
	\lambda_{\max} = \max_{\substack{\norm{{x}} = 1}} \max_{z}  \,\,\,\min_{{y'}} \,\,  \Phi'( U, V ; \,{x}, {y'}, z),  
\end{equation*}
where the function  $\Phi'( U, V ; \,{x}, {y'}, z)$ is defined by
\begin{equation*}
   2\ang{(1+\kappa)y_1' + z, Ux_1} + 2\ang{(1+\kappa)y_2' + z, Vx_2}+(1+\kappa)\norm{y_1'}^2+(1+\kappa)\norm{y_2'}^2-\norm{z}^2.
\end{equation*}
Now observe that since $1+\kappa > 0$, for fixed $z$ it is equivalent to minimize with respect to $y_i = (1+\kappa)y_i' + z$ for $i \in \set{1,2}$. In other words,
\begin{equation}\label{eq:canonical-PO}
    \lambda_{\max} = \max_{\substack{\norm{{x}} = 1}} \, \max_{z}  \,\,\,\min_{{y}} \,\,  \Phi( U, V ; \,{x}, {y}, z), 
\end{equation}
where now $\Phi( U, V ; \,{x}, {y}, z)$ is defined by 
\begin{equation*}
   2\ang{y_1, Ux_1} + 2\ang{y_2, Vx_2}+ \tfrac{1}{1+\kappa}\paren{\norm{y_1-z}^2+\norm{y_2-z}^2-(1+\kappa)\norm{z}^2}.
\end{equation*}
To parallel this primary optimization problem, we define the auxiliary optimization objective by taking 
\begin{align}
\begin{split}\label{eq:ccaAuxProblem}
	 \phi({g}, {h} ; {x}, {y}, z) = 2\norm{y_1}g_1^\intercal x_1
 &+2 \norm{x_1}h_1^\intercal y_1 + 2\norm{y_2}g_2^\intercal x_2 + 2\norm{x_2}h_2^\intercal y_2 +  \\
 &+\tfrac{1}{1+\kappa}\paren{\norm{y_1-z}^2 + \norm{y_2-z}^2 - (1+\kappa)\norm{z}^2},
 \end{split}
 \end{align}
where ${g} = (g_1, g_2)^\intercal$ and ${h} = (h_1, h_2)^\intercal$ are Gaussian vectors 
\begin{equation*}
    g_1 \sim \calN(0, m^{-1}\II_m) \;g_2 \sim \calN(0, k^{-1}\II_k) \text{ and }h_1 \sim \calN(0, m^{-1}\II_n) \; h_2 \sim \calN(0, k^{-1}\II_n).
\end{equation*}
Then for any real number $\lambda \in \RR$, we have 
    \begin{equation*}
        \PP\paren{ \lambda_{\max}(W) > \lambda} \leq 4 \, \PP\paren{ \max_{\substack{\norm{{x}} = 1}} \, \max_{z}  \,\,\,\min_{{y}} \,\, \phi({g}, {h} ; {x}, {y}, z) \geq \lambda}
    \end{equation*}
This follows from applying Theorem~\ref{thm:gordon-usefulMain} to deduce that 
\begin{equation*}
\PP\Bigg(\max_{\substack{\norm{{x}} = 1 \\ \, \norm{z} \leq R_1 }} \min_{ \norm{y}  \leq R_2}  \Phi( U, V ; {x}, {y}, z) > \lambda\Bigg) \leq 4 \PP\Bigg(\max_{\substack{\norm{{x}} = 1 \\ \, \norm{z} \leq R_1}}   \min_{ \norm{y} \leq R_2}  \phi( {g}, {h} ; {x}, {y}, z) \geq  \lambda\Bigg).
\end{equation*}
Then take $R_2 \uparrow \infty$ and then $R_1 \uparrow \infty$ and apply continuity of probability measures. It suffices to solve the auxiliary optimization problem. By Lemma~\ref{lem:cca-AuxiliaryOptMain} below, we have 
\begin{equation*}
     \PP\paren{ \max_{\substack{\norm{{x}} = 1}} \, \max_{z}  \,\,\,\min_{{y}} \phi \leq \tfrac{1-\kappa^2}{\kappa} + O(\varepsilon)} = 1- \exp\set{-\Omega(n \varepsilon^2)},
\end{equation*}
and the desired result follows.
\end{proof}

\begin{lemma}\label{lem:cca-AuxiliaryOptMain}
 The auxiliary optimization problem over $\phi$, defined in \eqref{eq:ccaAuxProblem}, is bounded with probability $(1- \exp\set{-\Omega(n \varepsilon^2)})$ over the Gaussian vectors $g, h$,
   \[\max_{\substack{\norm{{x}} = 1}} \, \max_{z}  \,\,\,\min_{{y}} \phi({g}, {h} ; {x}, {y}, z) \leq \tfrac{1-\kappa^2}{\kappa} + O(\varepsilon),\]
 where $x = (x_1, x_2)^\intercal \in \RR^m\times \RR^k$, $y = (y_1, y_2)^\intercal \in \RR^n\times \RR^n$ and $z \in \RR^n$.
\end{lemma} 
 The proof of Lemma~\ref{lem:cca-AuxiliaryOptMain} is elementary,
 and it can be found in Appendix~\ref{subsec:CCAGordon-Appendix}.

%% file: convergenceMain.tex
\section{Analysis of the Stieltjes transform}\label{sec:convergenceMain}

In this section, we study in greater detail the spectral properties of the random matrix $W$ under the CCA null model (see Section~\ref{subsec:CCA}), focusing on its Stieltjes transform. Our analysis draws on a substantial body of work \cite{AjankiErdosKrugerQVE,AjankiErdosKrugerMDE,SilversteinChoi1995,BaiSilverstein1998,HachemLoubatonNajimVallet2013} that develops powerful tools for the study of random matrices. Let us recall a few important properties of the Stieltjes transform, which takes as input some probability measure and outputs a complex function. The Stieltjes transform, evaluated at $z\in \CC$, of the empirical measure on the eigenvalues of a matrix $A\in \RR^{n\times n}$ is given by
  \[\text{Stieltjes}_{A}(z) = \frac{1}{n}\Tr\paren{(A - zI_{n})^{-1}}.\]
The term $(A-zI_n)^{-1}$ is the resolvent of $A$ evaluated at $z$. We will study the resolvent of our spectral statistic $W$ and show that for any $z$ in the upper-half plane, $\text{Stieltjes}_W(z)$ converges to some explicit limit as $n\to\infty$. We then extend the convergence result to a region on the real line using estimates from Section~\ref{sec:gordonMain}. Importantly, the pointwise convergence of Stieltjes transforms implies weak convergence of the corresponding spectral measures \cite{BS10}. Lastly, we characterize the limiting transform, which will be crucial in the outlier analysis in Section~\ref{sec:outlierMain}.  

Throughout this section, we will assume the threshold $\kappa = (\tau_m\tau_k)^{-1/4} < 1$ where $\tau_m = n/m+o(1)$ and $\tau_k = n/k+o(1)$. Recall $W$ is the block matrix of form
\begin{equation*}
	W = \begin{pmatrix}
		-\kappa U^\intercal U & U^\intercal V \\
		V^\intercal U & -\kappa V^\intercal V
	\end{pmatrix}.
\end{equation*}
We define by $\lambda_* = (1-\kappa^2)/\kappa$ the bound on the top eigenvalue of $W$ obtained via Gordon's inequality in Section~\ref{sec:gordonMain}. We study the resolvent $(W- zI_{m+k})^{-1}$ through the linearization of $(W-zI_{m+k})$, given by
  \begin{equation}\label{eq:linMatrixCCAMain}
    \calL(z) := \begin{pmatrix}
            -z I_m & 0 & \sqrt{\kappa} U^\intercal  & 0 \\
            0 & -z I_k & 0 & \sqrt{\kappa} V^\intercal  \\
            \sqrt{\kappa} U & 0 & \alpha \kappa  I_n & \alpha I_n  \\ 
            0 & \sqrt{\kappa} V & \alpha  I_n & \alpha \kappa  I_n
              \end{pmatrix} \quad \text{ and } \quad \calG (z) = \calL(z)^{-1},
  \end{equation}
where $\alpha = \kappa/(\kappa^2 - 1)$. We observe that by Schur's complement, the top left $(m+k) \times (m+k)$ block $\calG (z)$ is precisely $(W - zI_{m+k})^{-1}$. Using tools from random matrix theory, we can deduce that $\cG(z)$ is close to the deterministic matrix $\calM (z)$ defined below. 

\begin{definition}\label{def:det_systemMain}
For each $z\in \CC_+$, let $r(z),s(z)\in\CC_+$ and a $2\times2$ complex matrix
$T(z)=(t_{ij}(z))_{i,j\in\{1,2\}}$ be the unique solution to the coupled system,
\begin{align}\label{eq:finite_systemMain}
\begin{split}
r(z) = \frac{1}{-z - \kappa \tau_m \,t_{11}(z)}, \quad
s(z) = \frac{1}{-z - \kappa \tau_k\,t_{22}(z)}, \\
\begin{pmatrix}
    t_{11}(z) & t_{12}(z) \\
    t_{21}(z) & t_{22}(z)
\end{pmatrix}\begin{pmatrix}
    \alpha \kappa - \kappa  r(z) & \alpha \\
    \alpha & \alpha \kappa  - \kappa s(z)
\end{pmatrix} = I_{2}
\end{split}
\end{align} 
We may extend $(r, s, T)$ analytically to $\DD := \CC \setminus (-\infty, \lambda_*]$ and define for all $z \in \DD$,
\begin{equation}\label{eq:MNdefMain}
\calM (z):=
\begin{pmatrix}
r(z) I_m & 0 & 0 & 0\\
0 & s(z) I_k & 0 & 0\\
0 & 0 & t_{11}(z) I_n & t_{12}(z) I_n\\
0 & 0 & t_{21}(z) I_n & t_{22}(z) I_n
\end{pmatrix}.
\end{equation}
\end{definition}

We remark that the existence and uniqueness of a solution to \eqref{eq:finite_systemMain} is non-trivial and is the content of Lemma~\ref{lem:convergence-to-finite-system}. The analytic continuation result is also nontrivial and is the content of Lemma~\ref{lem:extended-solution-of-finite-system}. We will need the following stronger version of `with high probability'.

\begin{definition}[Overwhelming probability]
    We say that a sequence of events $\{\mathcal E_n\}$ holds with overwhelming probability (w.o.p.) if for any $C>0$, $\PP(\mathcal E_n) \ge 1 - n^{-C}$ for all sufficiently large $n$. By a union bound, the union of polynomially many events that hold with overwhelming probability also holds with overwhelming probability.
\end{definition}
We are now ready to give the formal statement for the convergence of $\cG(z)$ to $\cM(z)$, given in terms of quadratic forms. Let $\CC_+=\{z\in \CC:\Im z>0\}$, and for $\eta>0$, 
\begin{equation}\label{eq:OmegaDefMain}
    \Omega_{\eta} = \{z \in \CC_+ : \Im z > \eta, \abs{z} < \eta^{-1}\}.
\end{equation}
Since the spectrum of $W$ is real, the resolvent $(W - zI_{m+k})^{-1}$ has bounded operator norm on $\Omega_\eta$, which allows us to control the error tightly. 

\begin{theorem}\label{thm:resolvConvCCAMain}
 Let $\calG$ be as in \eqref{eq:linMatrixCCAMain} and $\calM$ be as in Definition~\ref{def:det_systemMain}. Fix $\eta, \varepsilon > 0$. Then, for any unit vectors $x, y\in \RR^{m+k+2n}$, 
  \begin{equation*}
  \sup_{z\in \Omega_\eta } \abs{x^\intercal  \calG (z) y - x^\intercal  \calM (z) y} \le O(n^{-1/2 + \varepsilon})
  \end{equation*}
 with overwhelming probability, where the hidden constant depends only on $\tau_m, \tau_k, \eta, $ and $\varepsilon$.
\end{theorem}

We show Theorem~\ref{thm:resolvConvCCAMain} in two main steps. We first show a concentration result, $\calG (z) \approx \EE \calG (z)$. Then we show a deterministic stability $\EE\calG (z) \approx \calM(z)$ through Gaussian integration by parts. These steps are rather technical and details can be found in Appendix~\ref{subsec:CCA_ConvProof}. 

As a consequence of Theorem~\ref{thm:resolvConvCCAMain}, one can easily deduce that for any $z\in \CC_+$, the traces of the blocks of $\cG(z)$ also converge to those of $\cM(z)$. This implies that $\text{Stieltjes}_{W}(z)$ converges to $(\tau_k r(z) + \tau_m s(z))/(\tau_k + \tau_m)$ as $n\to\infty$, on $\CC_+.$ Moreover, combining Theorem~\ref{thm:resolvConvCCAMain} with the bound on $\lambda_{\max}(W)$ obtained in Theorem~\ref{thm:CCAGordonEdge}, we can control the resolvent on the real line and extend the Stieltjes transform convergence result to $z\in (\lambda_*,\infty)$. This is the context of the next theorem. 

\begin{theorem}\label{thm:real-stieltjes-transform-CCAMain}
Let $\cG(z) = (\cG^{(ij)}(z))_{i,j \in [4]}$ be written as a block matrix with the same shape as $\cL(z)$ as given in \eqref{eq:linMatrixCCAMain}. Let $\calM$ be as in Definition~\ref{def:det_systemMain}. Then, for any compact set $K \subset (\lambda_*, \infty)$ and $i,j\in [4]$, we have that 
\begin{equation*}
   \sup_{z \in K}  \frac{1}{n}\abs{\Tr \, \calG^{(ij)}(z)  - \Tr \, \calM^{(ij)}(z)} = o(1),
\end{equation*}
with overwhelming probability.
\end{theorem}

The conclusion of Theorem~\ref{thm:real-stieltjes-transform-CCAMain} is a direct consequence of Theorem~\ref{thm:resolvConvCCAMain} for $K \subset \CC_+$. Theorem~\ref{thm:real-stieltjes-transform-CCAMain} is extending this conclusion to the real line. Since $K \subset (\lambda_*, +\infty)$ contains no eigenvalues of $W$ with overwhelming probability, $\operatorname{Tr}\mathcal G^{(ij)}(\cdot)/n$ is smooth in the complex neighborhood of $K$, so we can extend transfer the result from $z + i\eta \to z$ by smoothness. The detailed proof of Theorem~\ref{thm:real-stieltjes-transform-CCAMain} can be found in the appendix. 

We end this section by a short lemma which tells us that indeed $r(z)$ and $s(z)$ are each the Stieltjes transform of probability measures supported on $(-\infty, \lambda_*]$. This implies their weighted average is again the Stieltjes transform of some probability measure. 

\begin{lemma}\label{lem:extended-solution-of-finite-systemMain}
Let $(r,s,T)$ be the unique solution of \eqref{eq:finite_systemMain} on $\CC_+$. Define $\DD$ as in Definition~\ref{def:det_systemMain}. Then there exists probability measures $\mu$ and $\nu$ supported on $(-\infty, \lambda_*]$ such that
 \[r(z) = \text{Stieltjes}_\mu(z), \quad s(z) = \text{Stieltjes}_\nu(z)\]
\end{lemma}

The proof of Lemma~\ref{lem:extended-solution-of-finite-systemMain} can be found in Appendix~\ref{subsec:CCA_ConvProof}. 
Lemma~\ref{lem:extended-solution-of-finite-systemMain} together with Theorem~\ref{thm:real-stieltjes-transform-CCAMain} imply that the empirical distribution of the eigenvalues of $W$ converges weakly to an explicit probability measure supported on $(-\infty, \lambda_*]$. The limiting measure can be recovered by taking the inverse Stieltjes transforms of $r(z)$ and $s(z)$ from Definition~\ref{def:det_systemMain}. For example, one checks that the right edge of the limiting measure indeed equals $\lambda_*$. For our purpose, however, we only need information about the Stieltjes transform (see Section~\ref{sec:outlierMain} below). 

%% file: outlierMain.tex
\section{Outlier analysis} \label{sec:outlierMain}
We can analyze the planted CCA model (see Section~\ref{subsec:CCA}) as a perturbation of the null model analyzed in Section~\ref{sec:convergenceMain}. Throughout this section, we will assume $\rho > \kappa$. Recall that $U\in \RR^{n\times m}$ and $V\in \RR^{n\times k}$ where $n/m \to \tau_m$ and $n/k\to \tau_k$. 

By the orthogonal invariance of the null model, we may assume without loss of generality that the planted directions $a\in \RR^m, b\in \RR^k$ are the standard basis vectors $e_1 \in \RR^m, e_1 \in \RR^k$. In this way, the planted directions show up as the first column of $U$ and $V$, which we may isolate and write,
 \[U = [u\ U_0], \qquad V=[v\ V_0],\]
where $u,v\in\RR^n$ are the correlated first columns and $U_0\in\RR^{n\times(m-1)}$, $V_0\in\RR^{n\times(k-1)}$ are independent. We may permute the rows and columns of $W$, without changing its spectrum, so that the blocks without $u, v$ are collected in the bottom right. More precisely, there is a permutation matrix $P$ so that
   \begin{equation}\label{eq:cca-permutedWMain}
       \bar{W} := P W P^\intercal =\begin{pmatrix}
                                C & D^\intercal\\
                                D & W_{0}
                                \end{pmatrix},
   \end{equation}
where $W_0\in \RR^{(m+k-2)\times (m+k-2)}$, $C \in \RR^{2 \times 2}$ and $D \in \RR^{(m+k-2) \times 2}$ are defined by
  \begin{equation*}
    W_{0}:=
        \begin{pmatrix}
        -\kappa U_0^\intercal U_0 & U_0^\intercal V_0\\
        V_0^\intercal U_0 & -\kappa V_0^\intercal V_0
        \end{pmatrix}, 
        \quad
    C= \begin{pmatrix}
        -\kappa\,u^\intercal u & u^\intercal v\\
        v^\intercal u & -\kappa\,v^\intercal v
        \end{pmatrix},
    \quad
    D= \begin{pmatrix}
        -\kappa U_0^\intercal u & U_0^\intercal v\\
        V_0^\intercal u & -\kappa V_0^\intercal v
        \end{pmatrix}.
  \end{equation*}
Note that limiting law and edge behavior of $W_{0}$ are the same as in the null model with the same ratios $\tau_m$ and $\tau_k$ as $W$. If $W$ has an outlier eigenvalue, it should appear to the right of the top eigenvalue of $W_0$. We can express this by relating the eigenvalues of $W$ with those of $W_0$ via the Schur complement identity,
 \begin{equation}\label{eq:cca-det-factorizationMain}
    \det({W}-zI_{m+k})= \det(\bar{W}-zI_{m+k}) = \det(W_{0}-zI_{m+k-2})\det \cS_n(z),
  \end{equation}
where for $z\notin \text{spec}(W_{0})$, we define,
  \begin{equation}\label{eq:cca-Schur-compMain}
    G_0(z):=(W_{0}-zI_{m+k-2})^{-1}, \qquad \cS_n(z):=C-zI_2-D^\intercal G_0(z)D.
  \end{equation}
In particular, every eigenvalue $\lambda$ of $W$ outside the spectrum of the null model $W_0$ must solve $\det \cS_n(\lambda) = 0$. Let us first identify the limiting object for $\cS_n$.
\begin{lemma}\label{lem:LimitingS_CCAMain}
 Fix a compact set $K\subset (\lambda_*, \infty)$. The matrix $\cS_n(z)$ defined in \eqref{eq:cca-Schur-compMain} converges to $\cS(z)$ w.o.p. in operator norm uniformly over $z\in K$, where
    \begin{equation}\label{eq:SDefMain}
     \cS (z):= \begin{pmatrix}
         r(z)^{-1} & -\rho\kappa^{-1}t_{12}(z)\\
         -\rho\kappa^{-1}t_{21}(z) & s(z)^{-1}
        \end{pmatrix}.
    \end{equation}
 The coefficients $r(z), s(z), t_{ij}(z)$ are the unique solution to \eqref{eq:finite_systemMain} from Definition~\ref{def:det_systemMain}.
\end{lemma}
The proof of Lemma~\ref{lem:LimitingS_CCAMain} essentially consists of using Schur's complement to express the resolvent $G_0(z)$ in terms of the linearization \eqref{eq:linMatrixCCAMain} and applying Theorem~\ref{thm:real-stieltjes-transform-CCAMain} to the resulting object to obtain limiting matrix $\cS(z)$ (details in Appendix~\ref{sec:outlierCCA-Appendix}). 
We are ready to prove the first part of Theorem~\ref{thm:CCA-outlierMainBody}.

\begin{theorem}[Theorem~\ref{thm:CCA-outlierMainBody}, (i)]\label{thm:CCA-outlier1iMain}
  Suppose $\rho > \kappa$. Then the top eigenvalue of $W$ converges w.h.p. to an outlier eigenvalue $\lambda_\out$,
      \begin{equation*}
        \lambda_1(W) = \lambda_\out + o(1) \quad \text{ and } \quad    \lambda_2(W) \leq \lambda_* + o(1).
      \end{equation*}
    where $\lambda_* = (1-\kappa^2)/\kappa$, and $\lambda_\out > \lambda_*$ is the unique solution of $\det \calS(z) = 0$. 
\end{theorem}
\begin{proof}
In light of Lemma~\ref{lem:LimitingS_CCAMain} and the previous discussion, it suffices to determine when does the matrix $\cS(z)$ has an eigenvalue of $0$. Taking a limit $z\downarrow \lambda_*$, one can first find that
    \begin{equation}\label{eq:detS_LimitMain}
        \lim_{z \to\lambda_*^+} \det \calS(z) = \lambda_*^2\Bigl(1-\frac{\rho^2}{\kappa^2}\Bigr).
    \end{equation}
This limit is proven in Lemma~\ref{lem:S-has-negative-determinant-at-edge}. Notably, when $\rho > \kappa$, we have $\det\cS(\lambda_*^+) < 0$. As $\cS(\lambda_*^+)$ is a symmetric $2\times 2$ matrix, this implies exactly one of its eigenvalues must be positive. On the other hand, Lemma~\ref{lem:extended-solution-of-finite-systemMain} tells us that $r(z)$ and $s(z)$ are Stieltjes transforms of probability measures, which implies that in a large $z$ limit
  \[r(z) = -\frac{1}{z} + O(z^{-2}), \quad s(z) = -\frac{1}{z} + O(z^{-2}), \quad t_{21}(z) = t_{12}(z)=-\frac1\kappa+O(z^{-1}),\]
where the last equality comes from substituting the asymptotic behavior of $r(z)$ and $s(z)$ into the identities in \eqref{eq:finite_systemMain}. Hence for sufficiently large $z$, 
  \[\cS(z) = -z I_2 + O(1)\]
has strictly negative eigenvalues. Furthermore, for each fixed $n$ and every $z>\lambda_{\max}(W_{0})$, using \eqref{eq:cca-Schur-compMain} and the formula for the derivative of an inverse of a matrix \eqref{eq:derivativeInverse}, we see that
\begin{equation}\label{eq:cca-SN-derivativeMain}
\pp{\cS_n(z)}{z}=-I_2-D^\intercal (W_{0}-zI_{m+k-2})^{-2}D\prec -I_2 \prec 0.
\end{equation}
Passing through the limit from Lemma~\ref{lem:LimitingS_CCAMain}, we conclude that $\partial_z\cS(z) \preceq -I_2$. By the analyticity of $r(z)$, $s(z)$, and $t_{ij}(z)$, the entries of $\cS(z)$ are continuous for $z\in(\lambda_*, \infty)$. So the eigenvalues of $\cS(z)$ decrease continuously with $z$ and exactly one must be equal to $0$ at some unique $\lambda_\out\in (\lambda_*, \infty)$. Hence, $\det\cS(\lambda_\out) = 0$ for some unique $\lambda_\out \in (\lambda_*,\infty)$. 

The uniqueness and existence of the outlier of $W$ follows by the noticing the top eigenvalue of $W$ converges to $\lambda_{\out}$, which follows from the convergence of $\cS_n\to \cS$ in operator norm as $n\to\infty$ (details in Appendix~\ref{sec:outlierCCA-Appendix}).  
\end{proof}
 
\begin{theorem}[Theorem~\ref{thm:CCA-outlierMainBody}, (ii)]\label{thm:CCA_eigVecOverlapMain}
Suppose $\rho > \kappa$.  If $\hat w = (\hat a, \hat b)$ is a unit eigenvector corresponding to the largest eigenvalue of $W$, then w.h.p. the overlaps are bounded away from zero and satisfy
              \begin{equation*}
                |{\langle \hat a, a \rangle}|^2 =
                \frac{x_{*,1}^2}{-x_*^\intercal  \, \partial_z \cS(\lambda_{\mathrm{out}}) \,  x_*} +o(1), 
                \quad 
                |{\langle \hat b, b\rangle }|^2 =
                \frac{x_{*,2}^2}{-x_*^\intercal  \, \partial_z \cS(\lambda_{\mathrm{out}}) \,  x_*} +o(1).
              \end{equation*}
        where $x_* = (x_{*, 1}, x_{*, 2})$ is any unit vector in $\ker \cS(\lambda_{\mathrm{out}})$.
\end{theorem}
\begin{proof}
From Theorem~\ref{thm:CCA-outlier1iMain}, we know w.h.p., the matrix $W$ has exactly one outlier eigenvalue as its top eigenvalue $\lambda_{\text{out}, n}$ which satisfies $\det\cS_n(\lambda_{\out,n}) = 0$. Let $x_n\in \ker \cS_n(\lambda_{\out,n})$ be a unit vector. Let us define
  \[\hat{w}' = P^\intercal \binom{x_n}{-G_{0}(\lambda_{\out,n})D x_n} \quad \text{ and } \quad \hat w =  \frac{\hat w'}{\norm{\hat w'}}.\]
where $P$ is the permutation matrix in \eqref{eq:cca-permutedWMain}. If $x\in \ker \cS_n(\lambda)\setminus\{0\}$, then using \eqref{eq:cca-Schur-compMain}, an eigenvector of  $\bar{W}$ with eigenvalue $\lambda$ is
\begin{equation}\label{eq:cca-eigvec-schurMain}
\bar{W}\begin{pmatrix}
    x \\
    -G_{0}(\lambda)D x
\end{pmatrix} = \lambda\begin{pmatrix}
    x \\
    -G_{0}(\lambda)D x
\end{pmatrix}.
\end{equation}
Then \eqref{eq:cca-eigvec-schurMain} tells us,
  \[W \hat{w}' = P^\intercal\bar{W}\binom{x_n}{-G_{0}(\lambda_{\out,n})D x_n} = \lambda_{\out,n}P^\intercal \binom{x_n}{-G_{0}(\lambda_{\out,n})D x_n} = \lambda_{\out,n} \hat{w}'.\]
Recalling \eqref{eq:cca-SN-derivativeMain} and that $G_0(z) = (W_0 - zI_{m+k-2})^{-1}$ is symmetric, we see that
  \begin{equation*}
    \norm{\hat{w}'}^2 = \norm{x_n}^2+x_n^\intercal D^\intercal G_{0}(\lambda_{\out,n})^2 Dx_n = -x_n^\intercal  \, \partial_z\cS_n(\lambda_{\out,n}) \, x_n.
  \end{equation*}
Hence the individual overlap of $\hat{w} = (\hat a, \hat b)$ with the planted directions $(a, b)\in \RR^m \times \RR^k$ are given by the first coordinate of $\hat{a}$ and $\hat{b}$,
\begin{align}
 \begin{split}\label{eq:cca-individual-overlap-finiteMain}
      |{\langle \hat a, a \rangle}|^2
      &= |\hat a_1|^2
      = \frac{x_{n,1}^2}{\norm{\hat{w}'}^2}
      = \frac{x_{n,1}^2}{-x_n^\intercal \partial_z \, \cS_n(\lambda_{\out,n}) \, x_n},
      \\
      |{\langle \hat b, b\rangle }|^2
      &=|\hat b_1|^2
      = \frac{x_{n,2}^2}{\norm{\hat{w}'}^2}=  \frac{x_{n,2}^2}{-x_n^\intercal \partial_z \, \cS_n(\lambda_{\out,n}) \, x_n}.
  \end{split}
\end{align}
Then we take the limit of the right hand side to conclude, carried out in Appendix~\ref{sec:outlierCCA-Appendix}. 
\end{proof}

%% file: lower-boundsMain.tex
\section{Lower bounds} \label{sec:lower-boundsMain}

In this section, we prove impossibility of strong detection and weak recovery  for the CCA model when the correlation intensity is below the threshold. 

Recall the probability distributions $\mathcal P_\rho,\PP_\rho,\QQ$ defined as in Section~\ref{subsec:CCA}. Theorems~\ref{thm:CCA-TV-boundMain} and \ref{thm:cca-no-weak-recoveryMain} below together show the second bullet point of Theorem~\ref{thm:CCA-outlierMainBody} and conclude the proof. For fixed $(a,b)$, it is convenient to define the conditional
likelihood ratio $\Lambda(a,b;U,V)$ by
\begin{equation*}
\prod_{i=1}^n
\frac{1}{\sqrt{1-\rho^2}}
\exp\set{
-\frac{
\rho^2 m\langle U_i,a\rangle^2
+
\rho^2 k\langle V_i,b\rangle^2
-
2\rho\sqrt{mk}\langle U_i,a\rangle\langle V_i,b\rangle
}{
2(1-\rho^2)}
},
\end{equation*}
where $U_i^\intercal $ and $V_i^\intercal$ are the rows of $U$ and $V$, respectively. Then, we can write
\[
L_\rho(U,V)
:=
\frac{\mathrm d\PP_\rho}{\mathrm d\QQ}(U,V)
=
\EE_{\pi}\bigl[\Lambda(a,b;U,V)\bigr],
\]
where $\pi$ is the product of the uniform laws on
$\SS^{m-1}$ and $\SS^{k-1}$.

\begin{theorem}[Impossibility of Strong Detection]\label{thm:CCA-TV-boundMain}
    Suppose $\rho < \kappa$, then we have that 
    \begin{equation*}
        \TV(\PP_\rho, \QQ) = 1 - \Omega(1).
    \end{equation*}
    where TV denotes the total variation distance.  
    Equivalently, below the threshold $\rho<\kappa$, strong detection is impossible. 
\end{theorem}

\begin{proof}
    Recall the standard TV to $\chi^2$-inequality (see \cite{ITBook}) given by 
    \begin{equation*}
        \TV(\PP_\rho, \QQ) \leq \max\set{\frac{1}{2}, \frac{\chi^2(\PP_\rho \| \QQ)}{1+ \chi^2(\PP_\rho \| \QQ)}}.
    \end{equation*}
    Therefore, it suffices to show $\chi^2(\PP_\rho \| \QQ) = O(1)$. In fact, it suffices to show $\EE_{\QQ} L_\rho^2 = O(1)$. Suppose $(a_j, b_j)$ for $j \in \set{1,2}$ are two independent copies of $(a,b)$. Then,
    \begin{align*}
        \EE_{\QQ} [L_\rho^2]=\frac{1}{(1-\rho^2)^n}\EE \exp \Big(-\sum_{i = 1}^n \sum_{j = 1}^2 \tfrac{\rho^2{m\langle U_i, a_j\rangle^2 + \rho^2k\langle{ V_i, b_j\rangle}^2} - 2\rho\sqrt{mk} \langle{ U_i, a_j}\rangle\langle{  V_i, b_j}\rangle}{2(1-\rho^2)}\Big).
    \end{align*}
    Let $\theta =\langle a_1, a_2\rangle$ and $\phi = \langle b_1, b_2\rangle$ be the random overlaps. Let us collate the independent vectors and take $g = (g_i)_{i\in [n]}$ where,
    \begin{equation*}
        g_i^\intercal = (
            \sqrt{m}\langle U_i, a_1 \rangle,  \sqrt{m}\langle U_i, a_2 \rangle,
            \sqrt{k}\langle V_i, b_1 \rangle,
            \sqrt{k}\langle V_i, b_2 \rangle ),
    \end{equation*}
    Then we can compactly rewrite the the objective as follows:
    \begin{equation*}
        \EE_{\QQ} L_\rho^2 =  \EE \set{\EE \sqb{ \frac{1}{1-\rho^2}\exp\set{-\frac{1}{2} \, g^\intercal R g } \Big\vert  \, \,\theta, \phi \, }^n }.
    \end{equation*}
    where $g \sim \calN(0 ,\Sigma)$ and $\Sigma = \Sigma(\theta,\phi)$ and $R$ are $4 \times 4$ matrices defined by
    \begin{equation*}
        \Sigma(\theta, \phi) = \begin{pmatrix}
            1\ & \theta\ & 0\ & 0 \\
             \theta\ & 1\ & 0\ & 0\\
             0\ & 0\ & 1\ & \phi \\
             0\ & 0\ & \phi\ & 1 \\
         \end{pmatrix} \quad \text{ and } \quad  R = \frac{1}{1-\rho^2}\begin{pmatrix}
            \ \rho^2 & \ 0 & -\rho &\ 0 \\
            \ 0 & \ \rho^2 & \ 0 & -\rho \\
            -\rho & \ 0 & \ \rho^2 & \ 0 \\
            \ 0 & -\rho & \ 0 & \ \rho^2 
            \end{pmatrix}.
    \end{equation*}
    Using Gaussian integration and the fact that $\det(I + \Sigma R) = (1-\theta \phi \rho^2)^2/(1-\rho^2)^2$, we can rewrite our objective as
    \begin{align*}
         \EE_{\QQ} L_\rho^2 &= (1-\rho^2)^{-n}\EE \set{ \det (I + \Sigma({\theta, \phi}) R)^{-n/2}} = \EE \sqb{(1-  \theta \phi \rho^2 )^{-n}}.
    \end{align*}
 It remains to estimate this last expectation. The overlaps have densities
    \begin{equation*}
        f_m(s)=c_m(1-s^2)^{(m-3)/2}, \qquad
        f_k(t)=c_k(1-t^2)^{(k-3)/2}.
    \end{equation*}
    where $c_m=O(\sqrt m)$ and $c_k=O(\sqrt k)$, so $c_m c_k=O(n)$. Define the saddle exponent
    \begin{equation}
        \Psi_n(s,t):=-\log(1-\rho^2st)+\frac{m-3}{2n}\log(1-s^2)+\frac{k-3}{2n}\log(1-t^2). \label{eq:cca-saddle-exponentMain}
    \end{equation}
    Then, we can rewrite the expectation as integral
    \begin{equation}
        \EE \sqb{(1-\rho^2\theta\phi)^{-n}}
        =c_m c_k\int_{-1}^{1}\int_{-1}^{1}\exp\{n\Psi_n(s,t)\}\,\mathrm ds\,\mathrm dt . \label{eq:cca-double-integralMain}
    \end{equation}
    We now locate the saddle. Let $A(x):=-\log({1-x^2})$. Since $\rho<\kappa=(\tau_m\tau_k)^{-1/4}$, we may choose $q<q_1<1$ such that $\rho^2 \max\{1, \sqrt{\tau_m \tau_k}\}<q$. For all $s,t\in(-1,1)$,
    \begin{align*}
        -\log(1-\rho^2st)
        &\le -\log(1-\rho^2|s||t|) \\
        &=\sum_{\ell\ge1}\frac{\rho^{2\ell}|s|^\ell |t|^\ell}{\ell}
        \le \sum_{\ell\ge1}\frac{q^\ell}{2\ell}\paren{\frac{s^{2\ell}}{\tau_m}+\frac{t^{2\ell}}{\tau_k}}
        \le \frac q2\paren{\frac{A(s)}{\tau_m}+\frac{A(t)}{\tau_k}}.
    \end{align*}
    Since $(m-3)/n\to1/\tau_m$ and $(k-3)/n\to1/\tau_k$, for all sufficiently large $n$ the previous display implies the global  bound
    \begin{equation}
        \Psi_n(s,t)
        \le
        -\frac{1-q_1}{2}\paren{\frac{A(s)}{\tau_m}+\frac{A(t)}{\tau_k}}
        \le -c(s^2+t^2), \label{eq:cca-saddle-boundMain}
    \end{equation}
    for some constant $c>0$ depending only on $\rho,\tau_m,\tau_k$. In particular, $\Psi_n$ has its unique maximizer at $(0,0)$. Combining \eqref{eq:cca-double-integralMain} and \eqref{eq:cca-saddle-boundMain},
    \begin{equation*}
        \EE \sqb{(1-\rho^2\theta\phi)^{-n}}
        \le Cn\int_{\RR^2}\exp\{-cn(s^2+t^2)\}\,\mathrm ds\,\mathrm dt
        =O(1).
    \end{equation*}
    This proves $\chi^2(\PP_\rho\|\QQ)=O(1)$ and concludes the proof.
\end{proof}

We now prove that weak recovery is impossible when $\rho <\kappa$. To this end, we will need the following estimate. 

\begin{lemma}\label{lem:cca-overlap-numeratorMain}
Suppose $\rho<\kappa$. For every fixed $\delta>0$, define
\[
I_\delta(U,V)
:=
\EE_{\pi^{\otimes 2}}
\Bigl[
\Lambda(a^1,b^1;U,V)
\Lambda(a^2,b^2;U,V)
\mathbf 1
\bigl\{
\langle a^1,a^2\rangle ^2+\langle b^1,b^2\rangle^2\ge\delta
\bigr\}
\Bigr].
\]
Then, we have that $\EE_{\QQ} I_\delta(U,V)\le \exp\set{-\Omega_\delta(n)}.$
\end{lemma}

\begin{proof}
Let $\theta:=\langle a^1,a^2\rangle$ and $\phi:=\langle b^1,b^2\rangle $. By the same Gaussian integration and determinant computation as in the proof of Theorem~\ref{thm:CCA-TV-boundMain}, conditionally on $(a^1,b^1,a^2,b^2)$,
\[
\EE_{\QQ}
\Bigl[
\Lambda(a^1,b^1;U,V)
\Lambda(a^2,b^2;U,V)
\Bigr]
=
(1-\rho^2\theta\phi)^{-n}.
\]
Using the overlap densities and the saddle exponent $\Psi_n$ from \eqref{eq:cca-saddle-exponentMain}, we therefore have
\[
\EE_{\QQ} I_\delta
=
 c_m c_k\int_{-1}^{1}\int_{-1}^{1}
\exp\{n\Psi_n(s,t)\}\mathbf 1\{s^2+t^2\ge\delta\}\,\mathrm ds\,\mathrm dt .
\]
The global saddle bound \eqref{eq:cca-saddle-boundMain} gives the following estimate whenever $s^2 + t^2 \geq \delta$
\[
\Psi_n(s,t)\le -c(s^2+t^2)\le -c\delta.
\]
where $c>0$ is the constant from \eqref{eq:cca-saddle-boundMain}. Since $c_m c_k=O(n)$, the conclusion follows.
\end{proof}

\begin{prop}\label{prop:cca-posterior-overlapMain}
Let $\rho<\kappa$ and $\mu_{U,V}$ denote the posterior law of
$(a,b)$ given $(U,V)$. Then, as $n\to\infty$, 
\begin{gather*}
\EE_{\PP_\rho}\Big[\mathbb{E}
_{\mu_{U,V}^{\otimes 2}}\big[
\langle a^1,a^2\rangle^2+\langle b^1,b^2\rangle^2
\big]\Big]=o(1)\,.
\end{gather*}
\end{prop}

\begin{proof}
Let $R_{12}:=\langle a^1,a^2\rangle^2+\langle b^1,b^2\rangle^2.$ By Bayes' formula, we have for any $\delta>0$, 
\[
\mu_{U,V}^{\otimes 2}(R_{12}\ge\delta)
=
\frac{I_\delta(U,V)}{L_\rho(U,V)^2}.
\]
Since $\PP_\rho$ has density $L_\rho$ with respect to $\QQ$, we have
\[
\EE_{\PP_\rho}
\big[\mu_{U,V}^{\otimes 2}(R_{12}\ge\delta)\big]
=
\EE_{\QQ}\Big[
\frac{I_\delta(U,V)}{L_\rho(U,V)}\Big].
\]
Let $c=c(\delta)>0$ be the exponent from Lemma~\ref{lem:cca-overlap-numeratorMain}, and we choose
$\gamma\in(0,c)$. Splitting according to the event
$\{L_\rho(U,V)\ge e^{-\gamma n}\}$ gives
\begin{align*}
\EE_{\QQ}
\Big[\frac{I_\delta(U,V)}{L_\rho(U,V)}\Big]
&\le
e^{\gamma n}\EE_{\QQ}[I_\delta(U,V)]
+
\EE_{\QQ}
\left[
\frac{I_\delta(U,V)}{L_\rho(U,V)}
\mathbf 1\{L_\rho(U,V)<e^{-\gamma n}\}
\right].
\end{align*}
Since $I_\delta\le L_\rho^2$, the second term is bounded by $e^{-\gamma n}$. Using Lemma~\ref{lem:cca-overlap-numeratorMain},
\[
\EE_{\PP_\rho}\big[
\mu_{U,V}^{\otimes 2}(R_{12}\ge\delta)\big]
\le
C e^{-(c-\gamma)n}+e^{-\gamma n}=o(1).
\]
Since $0\le R_{12}\le2$, we have for every $\delta>0$,
\[
\EE_{\PP_\rho}
\mu_{U,V}^{\otimes 2}[R_{12}]
\le
\delta
+
2\EE_{\PP_\rho}\big[
\mu_{U,V}^{\otimes 2}(R_{12}\ge\delta)\big].
\]
Taking $n\to\infty$ and then $\delta \downarrow0$ yields the desired result. 
\end{proof}

\begin{theorem}[Impossibility of Weak Recovery]\label{thm:cca-no-weak-recoveryMain}
Suppose $\rho<\kappa$. Let
$\hat a=\hat a(U,V)\in\RR^m$ and
$\hat b=\hat b(U,V)\in\RR^k$ be arbitrary estimators satisfying $\norm{\hat{a}} \leq 1$ and $\| \hat b\| \leq 1$  almost surely. Then, for every fixed $\delta>0$, we have
\[
\mathcal P_\rho\big(
\langle \hat a,a\rangle^2+\langle\hat b,b\rangle^2\ge\delta
\big)=o(1)\,.
\]
\end{theorem}

\begin{proof}
By Markov inequality, suffices to show that $\EE_{\mathcal P_\rho}
\big[
\langle \hat a,a\rangle^2+\langle\hat b,b\rangle^2
\big]=o(1)$. By the posterior identity, we have that 
\[
\EE_{\mathcal P_\rho}\big[\langle \hat a,a\rangle^2\big]
=
\EE_{\PP_\rho}\Big[
\EE_{\mu_{U,V}}
\bigl[
\langle \hat a,a\rangle^2
\bigr]\Big].
\]
For fixed $(U,V)$, define $M_a(U,V):=
\EE_{\mu_{U,V}}[aa^\intercal].$ Since $\norm{\hat a}\le1$, we have
\[
\EE_{\mu_{U,V}}
\bigl[
\langle \hat a,a\rangle^2
\bigr]
=
\hat a^\intercal M_a(U,V)\hat a
\le
\norm{M_a(U,V)}_{\mathrm F}.
\]
This together with Cauchy-Schwarz inequality implies that
\begin{align*}
    \EE_{\PP_\rho}\big[\langle \hat a,a\rangle^2\big]\le \EE\big[\|M_a(U,V)\|_{\operatorname{F}}\big]
&\le
\left(
\EE_{\PP_\rho}
\norm{M_a(U,V)}_{\mathrm F}^2
\right)^{1/2} = \left(\EE_{\PP_\rho} \EE_{\mu_{U,V}^{\otimes 2}}
\bigl[
\langle a^1,a^2\rangle^2
\bigr]\right)^{1/2}\,,
\end{align*}
which is $o(1)$ by Proposition~\ref{prop:cca-posterior-overlapMain}. Applying the same argument to $\hat b$ concludes the proof. 
\end{proof}

%% file: TAP-heuristicsAppendix.tex
\section{TAP heuristics for CSWig and CSWish}\label{sec:TAP_Heuristics}
In this section we provide TAP heuristics for the correlated spiked models CSWig and CSwish. As in Section~\ref{sec:TAP_HeuristicsMain}, we replace the
Gaussian spike priors used in the main theorems by Rademacher proxy priors.
The scaling is chosen so that the planted rank-one terms have the same
order as in the Gaussian model. The resulting linearized TAP equations are
the same as those used in the rigorous analysis.

\subsection{Correlated spiked Wigner model}\label{subsec:corSpikedWigner}
Let us first compute TAP equations for the single-view spiked Wigner model:
  \[U = \frac{\alpha}{n} aa^\intercal + Z,\]
where $a \in\{\pm 1\}^n$ and $Z\sim \text{GOE}(n)$ is sampled from the Gaussian Orthogonal Ensemble. That is, $Z_{ij}\sim \cN(0, 1/n)$ for $i\ne j$ and $Z_{ii}\sim \cN(0, 2/n)$.  We will let the prior on $a$ be the uniform distribution over $\{\pm 1\}^n$. Then the posterior distribution of the signal $a$ given $U$.
  \begin{equation}\label{eq:post-classicSpikedWigner}
    \PP(a  \mid U) \propto \prod_{i < j}\exp\paren{-\frac{1}{2}\paren{\sqrt{n}U_{ij} - \frac{\alpha}{\sqrt{n}} a_i a_j}^2}\propto \exp\paren{\alpha \sum_{i < j} U_{ij}a_ia_j}.
  \end{equation}
where we used the fact that $a_i^2 = 1$. The second step is to apply the BP heuristics. This amounts to writing down a system of marginals of the entries $a_i$, in the following sense. 
  \begin{itemize}
      \item $m_{i\to j}(a_i)$ is the marginal of $a_i$ when the interaction between $a_i$ and $a_j$ is removed. 
      \item $\hat{m}_{j\to i}(a_i)$ is the marginal of $a_i$ when only the interaction between $a_i$ and $a_j$ is kept. 
  \end{itemize}
Written in equations, we have for all $i\ne j \in [n]$,
  \begin{align}
    m_{i\to j}(a_i) & \propto \prod_{k \ne j} \hat{m}_{k\to i}(a_i),\label{eq:spikedWigner_mEquation} \\
    \hat{m}_{j\to i}(a_i) & \propto \sum_{a_j\in \{\pm 1\}}\exp\paren{\alpha U_{ij} a_i a_j} m_{j\to i}(a_j).
  \end{align}
Because $a_i\in \{\pm 1\}$, we may express 
 \[m_{i\to j}(a_i) \propto \exp\paren{h_{i\to j} a_i}, \qquad \hat{m}_{j\to i}(a_i) \propto \exp(\hat{h}_{j\to i} a_i),\]
and study how $h$ and $\hat{h}$ behave. As we expect $h_{i\to j}$ to be of constant order, we may approximate $\hat{m}$ to the first exponential order in $n$,
 \[\hat{m}_{j\to i}(a_i) \propto \cosh\paren{\alpha U_{ij} a_i + h_{j\to i}} \appropto \exp\paren{\alpha U_{ij} a_i \tanh(h_{j\to i})}.\]
Substituting this back, we obtain the BP equation for $h_{i\to j}$,
 \begin{equation}\label{eq:BP-classicSpikerWigner}
   h_{i\to j} = \sum_{k\ne j} \alpha U_{ik} \tanh(h_{k\to i}).
 \end{equation}
The third step is to add back in the cavity (the missing $j$-term) in \eqref{eq:BP-classicSpikerWigner} to obtain the TAP equation. We introduce the cavity-less $H_i$,
  \begin{equation}\label{eq:cavitylessHClassicWigner}
    H_i = \sum_{j\in [n]}\alpha U_{ij}\tanh(h_{j\to i}) =  \sum_{j \in [n]} \alpha X_{ij} \tanh\paren{H_j - \alpha U_{ij} \tanh(h_{i\to j})}.
  \end{equation}
Taylor expanding around each $H_j$ and keeping up to the $n^{-1}$ terms, and applying the law of large numbers we obtain the TAP equation.
  \begin{align*}
    H_i & = \sum_{j\in [n]} \alpha U_{ij} \tanh(H_j) - \alpha^2\sum_{j\in [n]} U_{ij}^2 (1-\tanh(H_j)^2)\tanh(H_i) \\
        & \approx \sum_{j\in [n]} \alpha U_{ij} \tanh(H_j) - \frac{\alpha^2}{n}\sum_{j\in [n]}(1-\tanh(H_j)^2)\tanh(H_i)
  \end{align*}

Finally we linearize these equations. As $\tanh(H_i)$ is describing the mean of $a_i$, because the prior for $a_i$ is symmetric, we expect at the uninformative point (i.e. as we approach the threshold for $\alpha$) the mean of $a_i$ should go to $0$ (i.e. $\tanh(H_i)$ vanishes). In particular, this means $\tanh(H_i)\sim H_i$ and $\tanh(H_i)^2 \approx 0$. Hence by approximation, we obtain a linear equation for the vector $H = (H_1,\ldots, H_n)$,
  \begin{equation}\label{eq:meanEq-classicSpikedWigner}
    (\alpha U - \alpha^2 I_n) H = H \quad \Leftrightarrow \quad U H = \paren{\alpha + \frac{1}{\alpha}} H\,.
  \end{equation}

The TAP heuristics for the two-view correlated spiked
Wigner model defined in Section~\ref{subsec:wigner} follows almost immediately from the above computation. The observed matrices can be written as
\begin{equation}\label{eq:corWigner-model}
  U = \frac{\alpha}{n} aa^\intercal + Z_a,
  \qquad
  V = \frac{\beta}{n} bb^\intercal + Z_b,
\end{equation}
where $Z_a,Z_b\sim \mathrm{GOE}(n)$ are independent and $(a,b)$ are
$\rho$-correlated Rademacher vectors. The posterior can be written as
\begin{equation}\label{eq:corWigner-posterior}
  \PP(a,b\mid U,V)
  \propto
  \exp\Biggl\{
    \alpha\sum_{i<j} U_{ij}a_i a_j
    +
    \beta\sum_{i<j} V_{ij}b_i b_j
    +
    \eta\sum_{i\in [n]} a_i b_i
  \Biggr\},
\end{equation}
where $\eta$ is defined such that $\tanh(\eta)=\rho$. 
We define the marginals of interest in the same way as before, with superscripts to differentiate between the marginals of $a$ and $b$. The main difference lies in the equation \eqref{eq:spikedWigner_mEquation} for $m^a_{i\to j}$ and $m^b_{i\to j}$ where there is an extra contribution from the correlation between $a$ and $b$. This gives rise to a simple modification of \eqref{eq:meanEq-classicSpikedWigner} for $a$. Due to the correlation term, the marginal of $a_i$ without the interaction with $a_j$ becomes,
  \[m^a_{i\to j}(a_i) \propto \hat{m}^a_{b\to i}(a_i)\prod_{k \ne j} \hat{m}^a_{k\to i}(a_i)\]
where $\hat{m}_{b\to i}(a_i)$ denotes the marginal of $a_i$ when only the interaction with $b_i$ is kept. In particular,
  \begin{equation}\label{eq:hatmSpikedWigner}
    \hat{m}^a_{b\to i}(a_i)\propto \sum_{b_i\in \{\pm 1\}} \exp\paren{\eta a_i b_i} m^b_{i\to b}(b_i), \qquad m^b_{i\to b}(b_i) = \prod_{j\in[n]} \hat{m}_{j\to i}^b(b_i).
  \end{equation}
As $a_i\in \{\pm 1\}$, we may parametrize 
  \[\hat{m}^a_{b\to i}(a_i)\propto \exp(\hat{h}^a_{b\to i}a_i), \qquad m^b_{i\to b}(b_i) \propto \exp(h^b_{i\to b} b_i).\]
Notice that $h^b_{i\to b}$ is precisely the cavity-less quantity for $b_i$ we introduced in \eqref{eq:cavitylessHClassicWigner} in the single-matrix case. We define $\hat{h}^b_{a\to i}$ and $h^a_{i\to a}$ in similar fashion. The cavity-less quantity in the correlated case includes the extra correlation term, 
  \[H^a_i = h^a_{i\to a} + \hat{h}^a_{b\to i}.\]
Taking the expectation of both sides of the relation \eqref{eq:hatmSpikedWigner} for $\hat{m}^a_{b\to i}$, we see that
  \[\tanh(\hat{h}^a_{b\to i}) = \frac{\cosh(\eta + h^b_{i\to b}) - \cosh(-\eta + h^b_{i\to b}))}{\cosh(\eta + h^b_{i\to b}) + \cosh(-\eta + h^b_{i\to b}))} = \rho \tanh(h^b_{i\to b}).\]
Or after linearization, $\hat{h}^a_{b\to i} \approx \rho h^b_{i\to b}$. The computation in for the single-view case tells us that after linearization,
  \[h^a_{i\to a} \approx [(\alpha U - \alpha^2 I_n) H^a]_i, \qquad h^b_{i\to b} \approx [(\beta V - \beta ^2 I_n) H^b]_i.\]
Hence, the corresponding linearized TAP-equation for $H^a$ would become:
  \[(\alpha U - \alpha^2 I_n) H^a + \rho(\beta V - \beta^2 I_n) H^b = H^a.\]
  A similar equation for $H^b$ also holds. 
Written in matrix form,
\begin{equation*}
  W^\rho
  \begin{pmatrix}H^a\\ H^b\end{pmatrix}
  =
  \begin{pmatrix}H^a\\ H^b\end{pmatrix},
  \qquad
  W^\rho :=
  \begin{pmatrix}
    \widetilde U & \rho\widetilde V \\
    \rho\widetilde U & \widetilde V
  \end{pmatrix}.
\end{equation*}
where $\tilde{U}$ and $\tilde{V}$ are
  \[\tilde{U} = \alpha U - \alpha^2 I_n, \qquad \tilde{V} = \beta V - \beta^2 I_n\]
Thus the TAP heuristics predicts that the relevant estimator is the top eigenvector of $W_\rho$. On the other hand, using order parameter equations one predicts for $\alpha,\beta\le 1$ the threshold is at
\begin{equation*}
  \kappa
  :=
  \paren{\frac{(1-\alpha^2)(1-\beta^2)}{\alpha^2\beta^2}}^{1/4}.
\end{equation*}
Replacing $\rho$ with the threshold $\kappa$, we obtain the threshold matrix that appears in Theorem~\ref{thm:Spiked-outlierMainBody},
\[
  W = W^\kappa =
  \begin{pmatrix}
    \widetilde U & \kappa\widetilde V \\
    \kappa\widetilde U & \widetilde V
  \end{pmatrix}.
\]

\subsection{Correlated spiked Wishart model}
We now repeat the same cavity calculation for spiked Wishart
model, defined in Section~\ref{subsec:wishart}. We once again work with the signals $a,b\in\{\pm1\}^m$ that are $\rho$-correlated Rademacher vectors. We observe matrices
\begin{equation*}
  U = \sqrt{\frac{\alpha}{m}}\,ua^\intercal + Z_a,
  \qquad
  V = \sqrt{\frac{\beta}{m}}\,vb^\intercal + Z_b,
\end{equation*}
where $\sqrt m\,u,\sqrt m\,v\sim\cN(0,I_n)$ are independent and
$Z_a,Z_b$ have i.i.d. $\cN(0,m^{-1})$ entries. We write
$U_\mu,V_\mu\in\RR^m$ for the $\mu$-th rows of $U,V$, and we set
\[
  \tau := \frac{n}{m}.
\]

If $\eta$ is defined by $\tanh(\eta)=\rho$, then after integrating out the Gaussian
latent vectors $u$ and $v$ row-by-row we obtain
\begin{equation}\label{eq:Wishart-posterior}
  \PP(a,b\mid U,V)
  \propto
  \exp\Biggl\{
    \frac{\alpha}{2(1+\alpha)}\sum_{\mu\in[n]}
      \langle U_\mu,a\rangle^2
    +
    \frac{\beta}{2(1+\beta)}\sum_{\mu\in[n]}
      \langle V_\mu,b\rangle^2
    + \eta\sum_{i\in[m]} a_i b_i
  \Biggr\}.
\end{equation}
Notice the correlation structure is identical to that of the correlated spiked Wigner in \eqref{eq:corWigner-posterior}. By repeating the same argument in Section~\ref{subsec:corSpikedWigner}, it suffices to understand the TAP-heuristics for the posterior Gibbs measure of a single spiked Wishart matrix
 \[\PP(a \mid U)\propto \exp\paren{ \frac{\alpha}{2(1+\alpha)}\sum_{\mu=1}^n
      \langle U_\mu,a\rangle^2}.\]
We note that this is the Gibbs measure of the Hopfield model with inverse temperature $\alpha/(1+\alpha)$ and TAP heuristics about its general behavior can be found in \cite{Mezard2017}. As in the Canonical Correlation model, let us define
 \begin{equation}\label{eq:WishartXDefinition}
  x_\mu(a):=\sum_{i\in [m]} U_{\mu i}a_i,
  \qquad
  X_{\mu\to i} = \sum_{j \ne i} U_{\mu j} \tanh(h_{j\to\mu})
\end{equation}
Consider the following marginals for the entries of $a$.
\begin{itemize}
    \item $m_{i\to \mu}(a_i)$ is the marginal of $a_i$ given all interactions with the $U^\mu$ term has been removed. 
    \item $\hat{m}_{\mu\to i}(a_i)$ is the marginal of $a_i$ when only the interactions involving $U^\mu$ is kept. 
\end{itemize}
Written in equations, we have for all $i\in [m]$ and $\mu\in [n]$,
\begin{align*}
    m_{i\to \mu}(a_i) & \propto \prod_{\nu\ne \mu} \hat{m}_{\nu\to i}(a_i) \\
    \hat{m}_{\mu\to i}(a_i) & \propto \sum_{a_j, j\ne i}\exp\paren{\frac{\alpha}{2(1+\alpha)} x_\mu(a)^2} \prod_{j\ne i} m_{j\to \mu}(a_j) 
\end{align*}
Because $a_i\in \{\pm 1\}$, we may express 
 \[m_{i\to \mu}(a_i) \propto \exp\paren{h_{i\to \mu} a_i}, \qquad \hat{m}_{\mu\to i}(a_i) \propto \exp(\hat{h}_{\mu\to i} a_i).\]
So that $\tanh(h_{i\to \mu})$ is the mean of $m_{i\to \mu}(a_i)$. By the central limit theorem, under the product measure of $\prod m_{j\to \mu}$, we may treat the terms inside the sum as independent Gaussians
  \[x_\mu(a) \sim\cN\paren{U_{\mu i}a_i + \sum_{j\ne i}U_{\mu j} \tanh(h_{j\to \mu}), \sum_{j\in [m]}U_{\mu j}^2(1-\tanh^2(h_{j\to \mu}))}\]
where we used the fact that $a_j^2 = 1$ and $U^\mu_j$ has variance $1/m$. The `replica-symmetric' heuristic allows us to substitute the variance by a deterministic quantity $(1-q)$ where
  \[q \approx \frac{1}{m} \sum_{j\in [m]}\tanh^2(h_{j\to\mu}).\]
Hence $\hat m_{\mu\to i}(a_i)$ can be written as the Gaussian expectation
\begin{align*}
  \hat m_{\mu\to i}(a_i)
  &\propto
  \EE_{x\sim \cN(X_{\mu\to i},\,1-q)}
  \exp\paren{\frac{\alpha}{2(1+\alpha)} (x + U_{\mu i} a_i)^2} \propto \exp\paren{\frac{\alpha X_{\mu\to i} U_{\mu i} a_i}{1+\alpha - \alpha(1-q)}},
\end{align*}
where $X_{\mu\to i}$ was defined in \eqref{eq:WishartXDefinition}. Hence we see that
  \[h_{i\to \mu} = \sum_{\nu \ne \mu} \frac{\alpha X_{\nu\to i} U_{\nu i}}{1+\alpha - \alpha(1-q)}.\]
We introduce the cavity-less quantities,
  \[H_i = \sum_{\mu\in [n]} \frac{\alpha X_{\mu\to i} U_{\mu i}}{1+\alpha - \alpha(1-q)}, \qquad X_\mu  = X_{\mu\to i} + U_{\mu i}\tanh(h_{i\to \mu}) =\sum_{j \in [m]} U_{\mu j} \tanh(h_{j\to\mu}).\]
Taylor expanding around $H_j$ up to terms of order $n^{-1}$, and applying the law of large numbers to the $U^2_{\mu j}$ term, we see that
  \[X_\mu \approx \sum_{j\in[m]} U_{\mu j} \tanh(H_j) - \frac{\alpha X_\mu}{1+\alpha - \alpha(1-q)}\frac{1}{m}\sum_{j\in[m]} (1-\tanh^2(H_j))\]
Since the prior is symmetric, the uninformative fixed point is at $h_{i\to\mu}=0$. This suggests that $\tanh(h_{i\to\mu})\approx h_{i\to\mu}$, $\tanh^2(h_{i\to \mu})\approx 0$, and $q \approx 0$. So linearizing (we do not care about terms of order lower than $n^{-1}$), we see that
  \[X_\mu \approx \sum_{j\in[m]} U_{\mu j} H_j - \alpha X_\mu.\]
Substituting this back into $H_i$,
  \[H_i\approx \sum_{\mu\in [n]} \frac{\alpha}{1+\alpha}\sum_{j\in[m]}U_{\mu i}U_{\mu j} H_j - \sum_{\mu\in[n]}\alpha U_{\mu i}^2 H_i.\]
Applying the law of large numbers to the $U_{\mu i}^2$ term and recalling $\tau = n/m$, we get the matrix equation for $H$,
  \[H = \paren{\frac{\alpha}{1+\alpha} U^\intercal U - \alpha \tau I_m}H.\]
Going back to our model, we find
\begin{equation*}
  \widetilde U := \frac{\alpha}{1+\alpha}U^\intercal U - \alpha\tau I_m,
  \qquad
  \widetilde V := \frac{\beta}{1+\beta}V^\intercal V - \beta\tau I_m.
\end{equation*}
Applying the same argument for the correlation term as in Section~\ref{subsec:corSpikedWigner}, the matrix equation becomes,
\begin{equation*}
  W^\rho
  \begin{pmatrix}H^a\\ H^b\end{pmatrix}
  =
  \begin{pmatrix}H^a\\ H^b\end{pmatrix},
  \qquad
W^\rho :=
  \begin{pmatrix}
    \widetilde U & \rho\widetilde V \\
    \rho\widetilde U & \widetilde V
  \end{pmatrix}.
\end{equation*}
Similarly as before, one predicts using the order parameter equations that for $\tau\alpha^2,\tau\beta^2\le 1$, the threshold is at 
\begin{equation*}
  \kappa
  :=
  \paren{\frac{(1-\tau\alpha^2)(1-\tau\beta^2)}{\tau^2\alpha^2\beta^2}}^{1/4},
\end{equation*}
Replacing $\rho$ by the threshold $\kappa$, 
we obtain the matrix $W$ that appears in Theorem~\ref{thm:Spiked-outlierMainBody}:
\[
  W = W^\kappa =
  \begin{pmatrix}
    \widetilde U & \kappa\widetilde V \\
    \kappa\widetilde U & \widetilde V
  \end{pmatrix}.
\]

%% file: gordonAppendix.tex
\section{Gordon's inequality and proof of Theorem~\ref{thm:gordon-usefulMain}}\label{sec:gordonAppendix}

We recall Gordon's inequality \cite{Gordon85, TOH15, VRbook} for centered Gaussian vectors.
\begin{theorem}[Gordon's inequality]\label{thm:gordon-og}
Suppose $\{X_{ij}\}$ and $\{Y_{ij}\}$ are two centered Gaussian processes indexed by finite sets $i \in \cI$ and $j \in \cJ$ such that for all $i,j,k$ and $l \neq i$,
\[
\EE X_{ij}^2 = \EE Y_{ij}^2, \qquad
\EE X_{ij}X_{ik} \ge \EE Y_{ij}Y_{ik}, \qquad
\EE X_{ij}X_{lk} \le \EE Y_{ij}Y_{lk}.
\]
Then for any choice of parameters $\lambda_{ij}\in \RR$,
\[
\PP\paren{\bigcap_{i\in \cI} \bigcup_{j\in \cJ} \{X_{ij} \ge \lambda_{ij}\}}
\le
\PP\paren{\bigcap_{i\in \cI} \bigcup_{j\in \cJ} \{Y_{ij} \ge \lambda_{ij}\}}.
\]
In particular, for any $\psi:\cI \times \cJ\to\RR$ and $t \in \RR$,
\[
\PP\paren{\min_{i\in \cI} \max_{j\in \cJ} \{X_{ij} + \psi(i,j)\} \ge t}
\le
\PP\paren{\min_{i\in \cI} \max_{j\in \cJ} \{Y_{ij} + \psi(i,j)\} \ge t}.
\]
\end{theorem}

Recall the definitions from Section~\ref{sec:gordonMain}:
\begin{equation*}
    {G} = (G_1, \ldots, G_k), \quad {g}=(g_1,\ldots,g_k), \quad  {h}=(h_1,\ldots,h_k), 
\end{equation*}
where $G_j \in \RR^{n_j \times m_j}$, $g_j \in \RR^{n_j}$, and $h_j \in \RR^{m_j}$ for all $j \in [k]$. Let ${w}\in\RR^k$ and let
\[
\calS_{{x}} \subseteq \RR^{n_1} \times \cdots \times \RR^{n_k}, \qquad
\calS_{{y}} \subseteq \RR^{m_1} \times \cdots \times \RR^{m_k}, \qquad
\calS_{{z}} \subseteq \RR^d.
\]
be compact sets. Suppose $\psi : \calS_{{x}} \times \calS_{{y}} \times \calS_{{z}} \to \RR$ is a continuous function. We define three objects of optimization 
\begin{align*}
    \Phi({G}) &=
\min_{{z}\in \calS_{{z}}}
\min_{{x}\in \calS_{{x}}}
\max_{{y}\in \calS_{{y}}}
\left\{
\sum_{j=1}^k x_j^\top G_j y_j + \psi({x},{y},{z})
\right\}. \\
\Psi({G}, {w})
&=
\min_{{z}\in \calS_{{z}}}
\min_{{x}\in \calS_{{x}}}
\max_{{y}\in \calS_{{y}}}
\left\{
\sum_{j=1}^k x_j^\top G_j y_j
+\sum_{j=1}^k w_j \|x_j\|_2 \|y_j\|_2
+\psi({x},{y},{z})
\right\}. \\
\phi({g},{h})
&=
\min_{{z}\in \calS_{{z}}}
\min_{{x}\in \calS_{{x}}}
\max_{{y}\in \calS_{{y}}}
\left\{
\sum_{j=1}^k \|y_j\|_2\, g_j^\top x_j
+\sum_{j=1}^k \|x_j\|_2\, h_j^\top y_j
+\psi({x},{y},{z})
\right\}.
\end{align*}

We recall the statement for Theorem~\ref{thm:gordon-usefulMain}. 
\begin{theorem}[Theorem~\ref{thm:gordon-usefulMain}]\label{thm:gordon-useful}
Suppose the entries of ${G}, {w}, {g}, {h}$ are i.i.d.\ $\calN(0,1)$. For all $t\in\RR$,
\[
\PP(\Phi({G})<t)\le 2^k\,\PP(\Psi({G},{w})<t)
\le 2^k\,\PP(\phi({g},{h})\le t).
\]
\end{theorem}

\begin{proof}[Proof of Theorem~\ref{thm:gordon-useful}, Theorem~\ref{thm:gordon-usefulMain}]
By a standard $\varepsilon$-net discretization argument \cite{TOH15}, it suffices to prove the claim when
$\calS_{{x}},\calS_{{y}},\calS_{{z}}$ are finite; the compact case follows by continuity of $\psi$.

For $i=({x},{z})\in \calS_{{x}}\times\calS_{{z}}$ and $j={y}\in\calS_{{y}}$, define centered Gaussian processes
\[
X_{ij}:=\sum_{\ell=1}^k \|y_\ell\|_2\, g_\ell^\top x_\ell+\sum_{\ell=1}^k \|x_\ell\|_2\, h_\ell^\top y_\ell,
\qquad
Y_{ij}:=\sum_{\ell=1}^k x_\ell^\top G_\ell y_\ell+\sum_{\ell=1}^k w_\ell\|x_\ell\|_2\|y_\ell\|_2,
\]
so that we precisely have
\[
\phi({g},{h})=\min_i\max_j\{X_{ij}+\psi({x},{y},{z})\},\qquad
\Psi({G},{w})=\min_i\max_j\{Y_{ij}+\psi({x},{y},{z})\}.
\]
A direct covariance calculation using independence across $\ell$ gives that for any choice of indices $i = ({x},{z}), j = {y}, i' = ({x'},{z'}), j' = {y'}$, we have 
\[
\EE Y_{ij}Y_{i'j'}-\EE X_{ij}X_{i'j'}
=\sum_{\ell=1}^k\big(\|x_\ell\|_2\|x'_\ell\|_2-x_\ell^\top x'_\ell\big)\big(\|y_\ell\|_2\|y'_\ell\|_2-y_\ell^\top y'_\ell\big)\ge 0,
\]
since each factor is nonnegative by Cauchy-Schwarz. Furthermore, this implies that for fixed $i$, we have $\EE X_{ij}X_{ij'}=\EE Y_{ij}Y_{ij'}$. Thus the hypotheses of Theorem~\ref{thm:gordon-og} apply with $(X_{ij})$ and $(Y_{ij})$. Setting $\lambda_{ij} = t - \psi(i, j)$ yields
\[
\PP\!\left(\phi({g},{h})\ge t\right)\le \PP\!(\Psi({G},{w})\ge t).
\]
Taking complements gives $\PP(\Psi({G},{w})<t)\le \PP(\phi({g},{h})\le t)$. For the second claim, let $\cE:=\{w_1<0,\dots,w_k<0\}$ so $\PP(\cE)=2^{-k}$. On $\cE$ we have pointwise
\[
\sum_{\ell=1}^k x_\ell^\top G_\ell y_\ell+\psi({x},{y},{z})
\;\ge\;
\sum_{\ell=1}^k x_\ell^\top G_\ell y_\ell+\sum_{\ell=1}^k w_\ell\|x_\ell\|_2\|y_\ell\|_2+\psi({x},{y},{z}),
\]
hence $\Phi({G})\ge \Psi({G},{w})$ on $\cE$ and therefore
$\PP(\Psi({G},{w})<t\mid\cE)\ge \PP(\Phi({G})<t)$ (since $\Phi({G})$ is independent of ${w}$).
Multiplying by $\PP(\cE)$ gives
\[
\PP(\Phi({G})<t)\le 2^k\,\PP(\Psi({G},{w})<t)
\le 2^k\,\PP(\phi({g},{h})\le t),
\]
as desired.
\end{proof}

Similarly, we can prove the operator norm version of this theorem.

\begin{theorem}\label{thm:gordon-useful-opnorm}
    Suppose entries of $\calG, g, h$ are i.i.d. $\calN(0,1)$. For all $t \in \RR$
    {\small
    \begin{equation*}
        \PP\paren{\max_{x,y} \set{\sum_{j = 1}^kx_j^\intercal G_j y_j} \geq t } \leq 2^k \PP\paren{\max_{x,y} \set{\sum_{j=1}^k \|y_j\|_2\, g_j^\top x_j + \|x_j\|_2\, h_j^\top y_j} \geq t },  
    \end{equation*}
    }
    where the maximum is taken over $(x,y) \in \calS_x \times \calS_y$.
\end{theorem}

%% file: largest-eigenvalue-bounds.tex
\section{Bounds on largest eigenvalues}

In this section, we will apply Theorem~\ref{thm:gordon-useful} to prove tight bounds on the largest eigenvalues of the null models (see Section~\ref{sec:Model_Def}) of the test matrices $W = W_\kappa$ for each model. 

\subsection{Canonical correlation model and proof of Lemma~\ref{lem:cca-AuxiliaryOptMain}}\label{subsec:CCAGordon-Appendix}
Recall in the CCA null model, $U \in \RR^{n \times m}$ and $V \in \RR^{n \times k}$ are two independent Gaussian matrices with i.i.d. entries  $U_{ij} \sim \calN(0, m^{-1})$ and $V_{ij} \sim \calN(0, k^{-1})$. Suppose
\begin{equation*}
    \frac{n}{m} \mapsto \tau_m  \text{ and } \frac{n}{ k} \mapsto \tau_k \text{ and } \kappa := (\tau_m\tau_k)^{-1/4}  < 1.
\end{equation*}
Furthermore, we denote by $W$ the block matrix of form
\begin{equation*}
	W = \begin{pmatrix}
		-\kappa U^\intercal U & U^\intercal V \\
		V^\intercal U & -\kappa V^\intercal V
	\end{pmatrix}.
\end{equation*}
We recall the main Theorem~\ref{thm:CCAGordonEdge} from Section~\ref{sec:gordonMain}.
\begin{theorem}[Theorem~\ref{thm:CCAGordonEdge}]\label{thm:CCAGordonEdgeAppendix}
Let $W$ be the null matrix for the CCA model from Theorem~\ref{thm:CCA-outlierMainBody}. Then for any $\varepsilon \in (0,1)$
\begin{equation*}
    \PP\paren{\lambda_{\max}(W) \geq \frac{1-\kappa^2}{\kappa}  + \varepsilon} \leq e^{-\Omega(n \varepsilon^2)}.
\end{equation*}
\end{theorem}
Recall the auxiliary optimization objective defined in \eqref{eq:ccaAuxProblem}:
\begin{align*}
\begin{split}
	 \phi({g}, {h} ; {x}, {y}, z) = 2\norm{y_1}g_1^\intercal x_1
 &+2 \norm{x_1}h_1^\intercal y_1 + 2\norm{y_2}g_2^\intercal x_2 + 2\norm{x_2}h_2^\intercal y_2 +  \\
 &+\frac{1}{1+\kappa}\paren{\norm{y_1-z}^2 + \norm{y_2-z}^2 - (1+\kappa)\norm{z}^2},
 \end{split}
 \end{align*}
where ${g} = (g_1, g_2)^\intercal$ and ${h} = (h_1, h_2)^\intercal$ are Gaussian vectors 
\begin{equation*}
    g_1 \sim \calN(0, m^{-1}\II_m) \;g_2 \sim \calN(0, k^{-1}\II_k) \text{ and }h_1 \sim \calN(0, m^{-1}\II_n) \; h_2 \sim \calN(0, k^{-1}\II_n).
\end{equation*}
Let us recall the lemma needed to complete the proof in Section~\ref{sec:gordonMain} of Theorem~\ref{thm:CCAGordonEdgeAppendix}.

\begin{lemma}[Lemma~\ref{lem:cca-AuxiliaryOptMain}]\label{lem:cca-AuxiliaryOpt}
 The auxiliary optimization problem over $\phi$ is bounded with probability $(1- \exp\set{-\Omega(n \varepsilon^2)})$ over the Gaussian vectors $g, h$,
   \[\max_{\substack{\norm{{x}} = 1}} \, \max_{z}  \,\,\,\min_{{y}} \phi({g}, {h} ; {x}, {y}, z) \leq \frac{1-\kappa^2}{\kappa} + O(\varepsilon),\]
 where $x = (x_1, x_2)^\intercal \in \RR^m\times \RR^k$, $y = (y_1, y_2)^\intercal \in \RR^n\times \RR^n$ and $z \in \RR^n$.
\end{lemma} 
\begin{proof}[Proof of Lemma~\ref{lem:cca-AuxiliaryOpt}, Lemma~\ref{lem:cca-AuxiliaryOptMain}]
First fix $x$ and $z$, and minimize over $y_1$ and $y_2$. This can be done explicitly by first fixing their norms, minimizing over their directions, and then minimizing over the norms. This yields 
\begin{align*}
    \min_{{y}}\phi({g}, {h} ; {x}, {y}, z) = \frac{1-\kappa}{1+\kappa}\norm{z}^2 &- \sum_ {i=1,2}(1+\kappa)\sqb{\norm{\norm{x_i}h_i - \frac{z}{1+\kappa}} - g_i^\intercal x_i}_+^2.
\end{align*}
Here $[x]_+$ denotes $\max\set{x,0}$. Observe that since $x \mapsto [x]_+$ is non-decreasing, maximizing over the direction of $x_1$ and $x_2$ results in vectors parallel to $g_1$ and $g_2$, respectively. moreover, we may reparametrize $z \mapsto z(1+\kappa) $  and see that 
\begin{equation*}
    \max_{\substack{\norm{{x}} = 1}} \, \max_{z}  \,\,\,\min_{{y}} \,\, \phi({g}, {h} ; {x}, {y}, z) = (1+\kappa) \max_{\norm{{x}} = 1} \max_{{z}} \phi'({g}, {h} ; {x}, z),
\end{equation*}
where  $\phi'({g}, {h} ; {x}, z) $ is defined by 
\begin{align*}
     \phi'({g}, {h} ; {x}, z) := (1-\kappa)\norm{z}^2 &- \sum_{i=1,2} \sqb{\norm{\norm{x_i} h_i - {z}} -  \norm{g_i}\norm{x_i}}_+^2.
\end{align*}
By Gaussian concentration for Lipschitz functions (Lemma \ref{lem:gaussianConLip}), 
    \[\PP\paren{\abs{\norm{h_1} - \sqrt{\tau_m}} \ge \varepsilon} \le 2 \exp\paren{- n \varepsilon^2/2},\]
and similar for $h_2, g_1, g_2$. By Berstein's inequality for subexponentials \cite{VRbook},
    \[\PP\paren{\abs{\ang{h_1,h_2}} \ge \varepsilon} \le 2\exp\paren{- \Omega( n \varepsilon^2)},\]
so with probability $1-\exp\set{-\Omega(n \varepsilon^2)}$, we have $\ang{h_1,h_2} = O(\varepsilon)$ and 
\begin{equation*}
    \norm{h_1} = \sqrt{\tau_m} + O(\varepsilon), \quad \norm{h_2} = \sqrt{\tau_k} + O(\varepsilon), \quad \norm{g_1} = 1 + O(\varepsilon), \quad  \norm{g_2} = 1 + O(\varepsilon). 
\end{equation*}
We may take, for instance, $\varepsilon = n^{-1/4}$ so that the error vanishes with $n$. In particular, under this event $\norm{z} = O(1)$ or else the objective value would shoot down to negative infinity as $\kappa > 0$. Hence, we can use the Lipschitzness of the expression to rewrite 
\begin{equation*}
    \phi' = (1-\kappa)\norm{z}^2 - \sqb{\norm{\norm{x_1}u\sqrt{\tau_m} - {z}} - \norm{x_1}}_+^2 - \sqb{\norm{\norm{x_2} v \sqrt{\tau_k} - {z}} - \norm{x_2}}_+^2 + f(\varepsilon),
\end{equation*}
where $u \perp v$ are orthogonal unit vectors and $f(\varepsilon)$ is the error
\begin{equation} \label{eq:gordonErrorTermCCA} 
 f(\varepsilon) = \paren{\norm{\norm{x_1}u\sqrt{\tau_m} - {z}} + \norm{\norm{x_2} v \sqrt{\tau_k} - {z}}}O(\varepsilon)+O(\varepsilon^2).
 \end{equation}
Note that this equality holds uniformly over $z$ in a ball of radius $O(1)$, in which the maximizer lies.  We may do an orthonormal transformation over to $z$ and $x$ separately, so that the optimization on the right hand side  becomes 
\begin{align*}
    \max_{z_1, z_2, z_3} \max_{\substack{x_1, x_2 \\ x_1^2 + x_2^2 = 1} } (1-\kappa)(z_1^2 + z_2^2 + z_3^2) &-  \sum_{i=1,2}\sqb{\sqrt{(x_i \sqrt{\tau_m} - z_2)^2 + z_1^2 + z_3^2} - x_i}^2_+.
\end{align*}
To simplify even further, we note the following inequalities:
\begin{gather*}
    \sqrt{(x_1 \sqrt{\tau_m} - z_2)^2 + z_1^2 + z_3^2} \geq  \sqrt{z_1^2 + z_3^2}, \quad \sqrt{(x_2 \sqrt{\tau_k} - z_1)^2 + z_2^2 + z_3^2} \geq \sqrt{z_2^2 + z_3^2},  \\
    z_1^2 + z_2^2 + z_3^2 \leq  z_1^2 + z_2^2 + 2z_3^2.
\end{gather*}
Plugging these inequalities in, we only increase the objective. Taking $r_1^2 = z_1^2 + z_3^2$
 and $r_2^2 = z_2^2 + z_3^2$, we conclude by the following claim: for any $r_1, r_2 \geq 0$ and $x_1, x_2 \geq 0$  with $x_1^2 + x_2^2 = 1$, we have
\begin{equation*}
    (1-\kappa)(r_1^2 + r_2^2) - \sqb{r_1 - x_1}_+^2 - \sqb{r_2 - x_2}_+^2 \leq \frac{1-\kappa}{\kappa}
\end{equation*}
Clearly, we can assume without loss of generality that $r_1 \geq x_1$ and $r_2 \geq x_2$. Hence, we can expand the left hand side and use $x_1^2 + x_2^2 = 1$ and Cauchy-Schwarz
\begin{equation*}
    -\kappa(r_1^2 + r_2^2)+2r_1x_1 + 2r_2x_2 - 1 \leq -\kappa(r_1^2 + r_2^2) + \kappa(r_1^2 + r_2^2) + \frac{x_1^2 + x_2^2}{\kappa} - 1  = \frac{1-\kappa}{\kappa}.
\end{equation*}
Together, we have just proved that with probability $1- \exp\set{-\Omega(n \varepsilon^2)}$, we have 
\begin{equation*}
     \max_{\substack{\norm{{x}} = 1}} \, \max_{z}  \,\,\,\min_{{y}} \phi \leq \frac{(1-\kappa)(1+\kappa)}{\kappa} +O(\varepsilon) = \frac{1-\kappa^2}{\kappa} + O(\varepsilon),
\end{equation*}
which is precisely what we wanted to prove.
\end{proof}

\subsection{Correlated spiked Wigner model}\label{subsec:WignerGordonEdge}

Throughout this subsection we work in the CSWig null model from Section~\ref{subsec:wigner}. More precisely, we assume that $U, V\sim \mathrm{GOE}(1/n)$ are independent and define 
\[
  W =
  \begin{pmatrix}
    \widetilde U & \kappa \widetilde V \\
    \kappa \widetilde U & \widetilde V
  \end{pmatrix},
  \qquad
  \widetilde U = \alpha U - \alpha^2 I_n,
  \qquad
  \widetilde V = \beta V - \beta^2 I_n,
\]
where
$\alpha,\beta\in(0,1)$ and $\kappa$ is defined by the following equation
\[
  \kappa
  :=
  \paren{\frac{(1-\alpha^2)(1-\beta^2)}{\alpha^2\beta^2}}^{1/4}
  \in (0,1).
\]

\begin{theorem}\label{thm:WignerGordonEdge}
For every $\varepsilon\in(0,1)$,
\[
  \PP\paren{\lambda_{\max}(W) \ge 1 + \varepsilon}
  \le e^{-\Omega(n\varepsilon^2)}.
\]
\end{theorem}

\begin{proof}
We break the proof into a sequence of reductions. We first symmetrize $W$. To this end define matrix $A \in \RR^{2 \times 2}$ and numbers $u, v \in \RR$ by
\[
  A :=
  \begin{pmatrix}
    1 & \kappa \\
    \kappa & 1
  \end{pmatrix},
  \qquad
  A^{1/2}
  =
  \begin{pmatrix}
    u & v \\
    v & u
  \end{pmatrix},
\]
so that $2u v = \kappa$ and $u^2 + v^2 = 1$. In this notation, we have that
\[
  W = (A \otimes I_n)
  \begin{pmatrix}
    \widetilde U & 0 \\
    0 & \widetilde V
  \end{pmatrix}.
\]
Now observe that $W$ is similar to the symmetric matrix $H$ defined by
\[
  H := (A^{1/2}\otimes I_n)
  \begin{pmatrix}
    \widetilde U & 0 \\
    0 & \widetilde V
  \end{pmatrix}
  (A^{1/2}\otimes I_n)
  = (A^{-1/2}\otimes I_n)
  W
  (A^{1/2}\otimes I_n).
\]
This means that $W$ and $H$ have the same real spectrum, and in particular the largest eigenvalues are equal $\lambda_{\max}(W) = \lambda_{\max}(H).$ Note that if  $r\in\SS^{2n-1}$ is any vector and $(x_1, x_2)^\intercal :=  (A^{1/2}\otimes I_n)r$ the image of $r$ under $A^{1/2} \otimes \II_n$ then  
\[
  1 = \norm{r}^2
  = \ang{(x_1,x_2), (A^{-1}\otimes I_n)(x_1,x_2)}
  = \frac{\norm{x_1}^2 + \norm{x_2}^2 - 2\kappa \ang{x_1,x_2}}{1-\kappa^2}.
\]
Thus the image of $\SS^{2n-1}$ under $A^{1/2}\otimes I_n$ is precisely the ellipsoid
\begin{equation}\label{eq:wigner-null-ellipsoid}
  \calE
  :=
  \setcond{(x_1,x_2)\in \RR^n\times\RR^n}{
    \norm{x_1}^2 + \norm{x_2}^2 - 2\kappa\ang{x_1,x_2} = 1-\kappa^2
  }.
\end{equation}
Consequently, we can rewrite
\begin{equation}\label{eq:wigner-null-rayleigh}
  \lambda_{\max}(W)
  =
  \max_{(x_1,x_2)\in\calE}
  \set{
    -\alpha^2\norm{x_1}^2 - \beta^2\norm{x_2}^2
    + \alpha x_1^\intercal U x_1
    + \beta x_2^\intercal V x_2
  }.
\end{equation}

Now without loss of generality, let $G_1,G_2\in\RR^{n\times n}$ be independent gaussian matrices with independent 
$\calN(0,1)$ entries. such that
\[
  U = \frac{G_1+G_1^\intercal}{\sqrt{2n}}, \quad \text{ and } \quad  V= \frac{G_2+G_2^\intercal}{\sqrt{2n}}.
\]
then, we can further rewrite the largest eigenvalue of $W$ as
\begin{equation}\label{eq:wigner-null-rayleigh-iid}
  \lambda_{\max}(W)
  \overset{d}{=}
  \max_{(x_1,x_2)\in\calE}
  \set{
    -\alpha^2\norm{x_1}^2 - \beta^2\norm{x_2}^2
    + \alpha\sqrt{\frac{2}{n}}\,x_1^\intercal G_1x_1
    + \beta\sqrt{\frac{2}{n}}\,x_2^\intercal G_2x_2
  }.
\end{equation}

\medskip

We now linearize the expression above. For any matrix $M$ and vector $u$,
\[
  2u^\intercal Mu = \norm{Mu+u}^2 - \norm{Mu}^2 - \norm{u}^2.
\]
We also use the identities
\[
  \frac{\gamma}{2}\norm{v}^2
  = \max_{z\in\RR^n}\set{\sqrt{2\gamma}\,z^\intercal v - \norm{z}^2},
  \qquad
  -\frac{\gamma}{2}\norm{v}^2
  = \min_{y'\in\RR^n}\set{\sqrt{2\gamma}\,{y'}^\intercal v + \norm{y'}^2}.
\]
Applying these identities separately to the two quadratic Gaussian terms in
\eqref{eq:wigner-null-rayleigh-iid}, and then changing variables from $y_i'$ and
$z_i$ to $y_i := y_i' + z_i$, yields the exact representation
\begin{equation}\label{eq:wigner-null-PO}
  \lambda_{\max}(W)
  \overset{d}{=}
  \max_{x\in\calE}
  \max_{z\in (\RR^n)^2}
  \min_{y\in (\RR^n)^2}
  \Phi(G_1,G_2;x,y,z),
\end{equation}
where $x=(x_1,x_2)$, $y=(y_1,y_2)$, $z=(z_1,z_2)$, and
\[
  \Phi(G_1,G_2;x,y,z)
  :=
  2\sqrt{\frac{\alpha}{n}}\,y_1^\intercal G_1x_1
  +
  2\sqrt{\frac{\beta}{n}}\,y_2^\intercal G_2x_2
  + \psi(x,y,z),
\]
with $\psi$ defined as follows
\begin{align*}
  \psi(x,y,z)
  := {}&
  \sqrt{2\alpha}\,z_1^\intercal x_1
  + \sqrt{2\beta}\,z_2^\intercal x_2
  + \norm{y_1-z_1}^2 + \norm{y_2-z_2}^2
  - \norm{z_1}^2 - \norm{z_2}^2 \\
  &
  - \paren{\alpha^2 + \frac{\alpha}{2}}\norm{x_1}^2
  - \paren{\beta^2 + \frac{\beta}{2}}\norm{x_2}^2.
\end{align*}

We now introduce the associated auxiliary optimization. Let
$g_1,g_2,h_1,h_2\in\RR^n$ be independent standard Gaussian vectors and define $ \phi(g,h;x,y,z)$ by
\[
  2\sqrt{\frac{\alpha}{n}}
  \paren{\norm{y_1}g_1^\intercal x_1 + \norm{x_1}h_1^\intercal y_1}
  +
  2\sqrt{\frac{\beta}{n}}
  \paren{\norm{y_2}g_2^\intercal x_2 + \norm{x_2}h_2^\intercal y_2}
  + \psi(x,y,z),
\]
and we will be referring this as 
\begin{equation}\label{eq:WignerAODefinition}
  \mathrm{AO}
  :=
  \max_{x\in\calE}
  \max_{z\in (\RR^n)^2}
  \min_{y\in (\RR^n)^2}
  \phi(g,h;x,y,z).
\end{equation}
We claim for every $t\in\RR$,
\begin{equation}\label{eq:wigner-null-cgmt}
  \PP\paren{\lambda_{\max}(W) \geq t}
  \le 4\PP\paren{\mathrm{AO} \ge t}.
\end{equation}

For $R_1,R_2>0$, let $\mathrm{PO}_{R_1,R_2}$ and $\mathrm{AO}_{R_1,R_2}$ denote the same
optimizations as in \eqref{eq:wigner-null-PO} and in the definition of $\mathrm{AO}$, but
with the additional constraints $\norm{z_i}\le R_1$ and $\norm{y_i}\le R_2$ for
$i=1,2$. Apply Theorem~\ref{thm:gordon-useful} to the negated objective, with the
Gaussian matrices rescaled by the deterministic factors
$2\sqrt{\alpha/n}$ and $2\sqrt{\beta/n}$. Since the proof of
Theorem~\ref{thm:gordon-useful} is a covariance comparison, these deterministic
prefactors are harmless. We obtain
\[
  \PP\paren{\mathrm{PO}_{R_1,R_2} > t}
  \le 4\PP\paren{\mathrm{AO}_{R_1,R_2} \ge t}.
\]
Now sending $R_2 \to \infty$ and then $R_1 \to \infty$ gives \eqref{eq:wigner-null-cgmt}. Together with Lemma~\ref{lem:Wigner-AuxiliaryOpt} below, we obtain
\[
  \PP\paren{\lambda_{\max}(W) > 1 + C\eta}
  \le 4\PP\paren{\mathrm{AO} \ge 1 + C\eta}
  \le e^{-\Omega(n\eta^2)}.
\]
Renaming $C\eta$ as $\varepsilon$ completes the proof.
\end{proof}

\begin{lemma}\label{lem:Wigner-AuxiliaryOpt}
   The auxiliary optimization defined in \eqref{eq:WignerAODefinition} is bounded with probability at least $1-e^{-\Omega(n \eta^2)}$,
\[
  \mathrm{AO} \le 1 + C\eta.
\]
  for some fixed constant $C > 0$.
\end{lemma}

Before proving Lemma~\ref{lem:Wigner-AuxiliaryOpt}, we prove the following two elementary inequalities: 

\begin{lemma}\label{lem:wigner-null-onedim}
For every $\ell\ge 0$ and every $d\in\RR$,
\[
  \sup_{r\ge 0}\set{\ell r - \sqb{r-d}_+^2}
  \le d\ell + \frac{\ell^2}{4}.
\]
\end{lemma}
\begin{proof}
If $r\le d$, then the left hand side is at most $d\ell$. If $r\ge d$, then
\[
  \ell r - \sqb{r-d}_+^2
  = d\ell + (r-d)\ell - (r-d)^2
  \le d\ell + \frac{\ell^2}{4}.
\]
Taking the supremum over $r\ge 0$ proves the claim.
\end{proof}

\begin{lemma}\label{lem:wigner-null-geometry}
Let $f_1,f_2\in\RR^n$ be orthonormal. Let $\cE$ be the ellipsoid defined in \eqref{eq:wigner-null-ellipsoid}. Then every
$(x_1,x_2)\in\calE$ satisfies
\begin{equation}\label{eq:wigner-null-geometry-main}
  \ang{f_1,x_1}^2 + \ang{f_2,x_2}^2 \le 1.
\end{equation}
\end{lemma}
\begin{proof}
Let us write 
\[
  x_1 = af_1 + bf_2 + u,
  \qquad
  x_2 = cf_1 + df_2 + v,
\]
with $u,v\perp \mathrm{span}\{f_1,f_2\}$. Since $(x_1,x_2)\in\calE$,
\begin{align*}
  1-\kappa^2
  &= \norm{x_1}^2 + \norm{x_2}^2 - 2\kappa\ang{x_1,x_2} \\
  &= \paren{a^2+c^2-2\kappa ac}
   + \paren{b^2+d^2-2\kappa bd}
   + \norm{u}^2 + \norm{v}^2 - 2\kappa\ang{u,v}.
\end{align*}
Now, note that we can rewrite
\begin{gather*}
	 a^2+c^2-2\kappa ac = (1-\kappa^2)a^2 + (c-\kappa a)^2, \\
	  b^2+d^2-2\kappa bd = (1-\kappa^2)d^2 + (b-\kappa d)^2,
\end{gather*}
and also notice that
\[
  \norm{u}^2 + \norm{v}^2 - 2\kappa\ang{u,v}
  = \norm{u-\kappa v}^2 + (1-\kappa^2)\norm{v}^2 \ge 0.
\]
Therefore, we conclude that
\[
  1-\kappa^2 \ge (1-\kappa^2)(a^2+d^2),
\]
which proves \eqref{eq:wigner-null-geometry-main}, since
$a=\ang{f_1,x_1}$ and $d=\ang{f_2,x_2}$.
\end{proof}

\begin{proof}[Proof of Lemma~\ref{lem:Wigner-AuxiliaryOpt}]
We first carefully analyze the auxilary optimization problem. We first minimize over $y$ for fixed $x$ and $z$. For 
$\gamma\in\{\alpha,\beta\}$, define
\[
  L_\gamma(x,y,z;g,h)
  :=
  2\sqrt{\frac{\gamma}{n}}
  \paren{\norm{y}g^\intercal x + \norm{x}h^\intercal y}
  + \norm{y-z}^2 - \norm{z}^2.
\]
If we first minimize over the direction of $y$, and then minimize over the magnitude of $\norm{y}$, we will obtain that
\begin{equation}\label{eq:wigner-null-ymin}
  \min_{y\in\RR^n} L_\gamma(x,y,z;g,h)
  =
  -\sqb{
    \norm{z-\sqrt{\frac{\gamma}{n}}\norm{x}h}
    - \sqrt{\frac{\gamma}{n}}\,g^\intercal x
  }_+^2.
\end{equation}
Using \eqref{eq:wigner-null-ymin}, if $\gamma_1 = \alpha$ and $\gamma _2 = \beta$, we get that 
\begin{equation}\label{eq:wigner-null-AO-after-y}
\begin{aligned}
  \mathrm{AO}
  =
  \max_{x\in\calE}
  \max_{z\in(\RR^n)^2}
  \sum_{i=1}^2
  &\Bigg\{
    \sqrt{2\gamma_i}\,z_i^\intercal x_i
    - \paren{\gamma_i^2+\frac{\gamma_i}{2}}\norm{x_i}^2 \\
        - &\sqb{
      \norm{z_i-\sqrt{\frac{\gamma_i}{n}}\norm{x_i}h_i}
      - \sqrt{\frac{\gamma_i}{n}}\,g_i^\intercal x_i
    }_+^2
  \Bigg\},
\end{aligned}
\end{equation}

To apply the Lemma~\ref{lem:wigner-null-onedim}, let us define
\[
  \ell := \sqrt{2\gamma}\,\norm{x},
  \qquad
  b := \sqrt{\frac{\gamma}{n}}\norm{x}h,
  \qquad
  d := \sqrt{\frac{\gamma}{n}}\,g^\intercal x.
\]
Indeed, writing $\tilde z:= z-b$ and using
$z^\intercal x = b^\intercal x + \tilde z^\intercal x \le b^\intercal x + \norm{\tilde z}\norm{x}$,
we obtain
\begin{align*}
  \sup_{z\in\RR^n}
  \set{
    \sqrt{2\gamma}\,z^\intercal x
    - \sqb{\norm{z-b}-d}_+^2
  } & \le
  \sqrt{2\gamma}\,b^\intercal x
  + \sup_{\norm{\tilde{z}}\ge 0}\set{\sqrt{2\gamma}\,\norm{x}\norm{\tilde z} - \sqb{\norm{\tilde z}-d}_+^2}\\
  &\le
  \sqrt{2\gamma}\,b^\intercal x
  + d\sqrt{2\gamma}\,\norm{x}
  + \frac{\gamma}{2}\norm{x}^2 \\
  &=  \sqrt{2}\gamma \frac{\norm{x}}{\sqrt{n}}(h + g)^\intercal x + \frac{\gamma}{2}\norm{x}^2.
\end{align*}
Now, substituting back and using
\eqref{eq:wigner-null-AO-after-y}, we arrive at the deterministic upper bound
\begin{equation}\label{eq:wigner-null-AO-zbound}
  \mathrm{AO}
  \le
  \max_{x\in\calE}
  \Bigg\{
    -\alpha^2\norm{x_1}^2 - \beta^2\norm{x_2}^2
    + \sqrt{2}\alpha\frac{\norm{x_1}}{\sqrt n}(g_1+h_1)^\intercal x_1
    + \sqrt{2}\beta\frac{\norm{x_2}}{\sqrt n}(g_2+h_2)^\intercal x_2
  \Bigg\}.
\end{equation}

\medskip
\noindent
Now, we analyze the right hand side. Let us set $a_i = g_i +h_i$. Then $a_1$ and $a_2$ are independent $\calN(0,2I_n)$ vectors. Fix small $\eta$. By standard Gaussian concentration arguments, with probability at least $1 - e^{-\Omega(\eta^2 n)}$, we have
\begin{equation}\label{eq:wigner-null-good-event}
    \abs{\frac{\norm{a_1}}{\sqrt n} - \sqrt{2}}\le \eta, \qquad 
     \abs{\frac{\norm{a_2}}{\sqrt n} - \sqrt{2}}\le \eta \qquad 
     \abs{\ang{e_1,e_2}}\le \eta,
\end{equation}
where $e_i = a_i / \norm{a_i}$ are unit vectors. Let $f_1 = e_1$ and $f_2$ is rescaled unit projection of $e_2$ against $e_1$. More precisely, 
\[
  f_2 := \frac{e_2 - \ang{e_1,e_2}e_1}{\sqrt{1-\ang{e_1,e_2}^2}}.
\]
Then $f_1$  and $f_2$ are orthonormal and $\norm{f_2-e_2}\le C\eta$ under the same event. 
Now fix $x=(x_1,x_2)\in\calE$ and write $r_i:=\norm{x_i}$. From
\eqref{eq:wigner-null-AO-zbound}
\begin{align*}
  -\alpha^2 r_1^2
  + \sqrt{2}\alpha\frac{r_1}{\sqrt n}a_1^\intercal x_1
  &=
  -\alpha^2 r_1^2
  + \sqrt{2}\alpha\frac{\norm{a_1}}{\sqrt n}r_1\ang{e_1,x_1} \\
  &\le
  -\alpha^2 r_1^2 + (2+O(\eta))\alpha r_1\abs{\ang{f_1,x_1}} \\
  &\le
  (1 + O(\eta))\ang{f_1,x_1}^2
\end{align*}
and similarly, we have that under the same event 
\[
  -\beta^2 r_2^2
  + \sqrt{2}\beta\frac{r_2}{\sqrt n}a_2^\intercal x_2
  \le (1 + O(\eta))\ang{f_2, x_2}^2.
\]
Hence, we deduce that under the same event
\begin{equation}\label{eq:wigner-null-AO-pregeom}
  \mathrm{AO}
  \le
  (1 +O(\eta))\max_{x\in\calE}
  \set{
    \ang{f_1,x_1}^2 + \ang{f_2,x_2}^2}
\end{equation}

The remaining input is to apply the deterministic lemma about the ellipsoid $\calE$. Combining \eqref{eq:wigner-null-AO-pregeom} with Lemma~\ref{lem:wigner-null-geometry}, we find that with probability at least $1-e^{-\Omega(n \eta^2)}$, we have that 
\[
  \mathrm{AO} \le 1 + C\eta. 
\]

\end{proof}

\subsection{Correlated spiked Wishart model}

Throughout this subsection, we work with the CSWish null model from Section~\ref{subsec:wishart}. More precisely,
$U,V\in\RR^{n\times m}$ are Gaussian matrices with independent
$\calN(0,m^{-1})$ entries, with aspect ratio $n/m\to\tau$. We further assume the regime
\[
  \alpha\sqrt\tau<1,
  \qquad
  \beta\sqrt\tau<1,
  \qquad
  \kappa^4=\frac{(1-\tau\alpha^2)(1-\tau\beta^2)}{\tau^2\alpha^2\beta^2}<1.
\]
Define matrices $\wt U$, $\wt V$ and $W$ as in Section~\ref{subsec:wishart}:
\[
  W:=
  \begin{pmatrix}
    \widetilde U & \kappa\widetilde V\\
    \kappa\widetilde U & \widetilde V
  \end{pmatrix}, \qquad \widetilde U:=\frac{\alpha}{1+\alpha}U^\intercal U-\alpha\tau I_m,
  \qquad
  \widetilde V:=\frac{\beta}{1+\beta}V^\intercal V-\beta\tau I_m.
\]

\begin{theorem}\label{thm:WishartGordonEdge}
For every fixed $\varepsilon\in(0,1)$,
\[
  \PP\paren{\lambda_{\max}(W)\ge 1+\varepsilon}
  \le e^{-\Omega(n\varepsilon^2)}.
\]
\end{theorem}

\begin{proof}
We follow the same pattern as in the Wigner case. We first symmetrize the block matrix,
rewrite the edge event as a primary optimization over an ellipsoid, compare it to an
auxiliary optimization via Theorem~\ref{thm:gordon-useful}, and then analyze the auxiliary
problem deterministically. Let
\begin{equation}\label{eq:ADef_WishartAO}
  A:=\begin{pmatrix}1&\kappa\\ \kappa&1\end{pmatrix},
  \qquad
  A^{1/2}=\begin{pmatrix} u&v\\ v&u\end{pmatrix}.
\end{equation}
so that $u^2+v^2=1$ and $2uv=\kappa$. Define the symmetric matrix
\[
  H:=(A^{1/2}\otimes I_m)
  \begin{pmatrix}
    \widetilde U & 0\\
    0 & \widetilde V
  \end{pmatrix}
  (A^{1/2}\otimes I_m).
\]
As in the proof of Theorem~\ref{thm:WignerGordonEdge}, we have that 
\[
  W=(A\otimes I_m)
  \begin{pmatrix}
    \widetilde U & 0\\
    0 & \widetilde V
  \end{pmatrix}
  =(A^{1/2}\otimes I_m)H(A^{-1/2}\otimes I_m),
\]
the matrices $W$ and $H$ are similar, and therefore $  \lambda_{\max}(W)=\lambda_{\max}(H).$ For notational convenience, let us introduce
\[
  \gamma_1:=\alpha,
  \qquad
  \gamma_2:=\beta,
  \qquad
  \theta_i:=\frac{\gamma_i}{1+\gamma_i}
  \qquad  G_1:=\sqrt m\,U^\intercal,  \qquad
  G_2:=\sqrt m\,V^\intercal.
\]
Then $G_1,G_2\in\RR^{m\times n}$ have independent standard Gaussian entries, and
\[
  \widetilde U=\frac{\theta_1}{m}G_1G_1^\intercal-\gamma_1\tau I_m,
  \qquad
  \widetilde V=\frac{\theta_2}{m}G_2G_2^\intercal-\gamma_2\tau I_m.
\]
Define auxilary matrix $ C:=\tau\gamma_1 (u,v)^{\intercal}(u,v)+\tau\gamma_2(v,u)^{\intercal}(v,u) \in \RR^{2 \times 2}$  and 
\[
  B:=(A^{1/2}\otimes I_m)
  \begin{pmatrix}
    \sqrt{\theta_1/m}\,G_1 & 0\\
    0 & \sqrt{\theta_2/m}\,G_2
  \end{pmatrix}.
\]
Now, a direct computation gives that for every $z > 0$ 
\[
  H=BB^\intercal-(C\otimes I_m) \quad \text{ and } \quad   zI_{2m}-H=(zI_2+C)\otimes I_m-BB^\intercal.
\]
By Schur complement, we conclude the initial form of the primary optimization
\begin{equation}\label{eq:eig-wishart-equivalence}
  \lambda_{\max}(W)<z
  \iff
  \lambda_{\max}(H)<z
  \iff
  \norm{\sqb{(zI_2+C)\otimes I_m}^{-1/2}B}_{\op}<1.
\end{equation}
Let us rewrite the argument of the latter expression as follows.
\[
  \sqb{(zI_2+C)\otimes I_m}^{-1/2}B
  =
  \sqb{(zI_2+C)^{-1/2}A^{1/2}\otimes I_m}
  \begin{pmatrix}
    \sqrt{\theta_1/m}\,G_1 & 0\\
    0 & \sqrt{\theta_2/m}\,G_2
  \end{pmatrix}.
\]
Observe that the image of $\SS^{2m-1}$ to under the map $x \mapsto A^{1/2}(zI_2 + C)^{-1/2} \otimes I_m$ is precisely the ellipsoid defined by 
\[
  \calE_z:=\setcond{x = (x_1,x_2)\in\RR^m\times\RR^m}{
    \ang{x ,(M_z\otimes I_m)x}=1},
\]
where, matrix $M_z$ satisfies the following identity. 
\begin{align}\label{eq:wishart-Mz-definition}
  M_z:=A^{-1/2}(zI_2+C)A^{-1/2}
  &=zA^{-1}+\tau\diag(\gamma_1,\gamma_2)\\
  &=
  \begin{pmatrix}
    \dfrac{z}{1-\kappa^2}+\tau\alpha & -\dfrac{z\kappa}{1-\kappa^2}\\[2mm]
    -\dfrac{z\kappa}{1-\kappa^2} & \dfrac{z}{1-\kappa^2}+\tau\beta
  \end{pmatrix},
\end{align}
then, \eqref{eq:eig-wishart-equivalence} can be written compactly as the following primary optimization problem
\begin{equation}\label{eq:wishart-lmax-via-PO}
  \lambda_{\max}(W)\ge z
  \iff
  \mathrm{PO}_z\ge 1,
\end{equation}
where $\mathrm{PO}_z$ is defined by the following standard form
\begin{equation}\label{eq:wishart-po-ellipsoid}
  \mathrm{PO}_z
  :=
  \max_{x\in\calE_z}
  \max_{\norm{y}=1}
  \Phi_z(G;x,y),
\end{equation}
where $x=(x_1,x_2)$, $y=(y_1,y_2)$, and
\[
  \Phi_z(G;x,y)
  :=
  \sqrt{\frac{\theta_1}{m}}\,x_1^\intercal G_1y_1
  +
  \sqrt{\frac{\theta_2}{m}}\,x_2^\intercal G_2y_2.
\]

We now introduce the corresponding auxiliary optimization. Let
$g_1,g_2\sim\calN(0,\II_m)$ and $h_1,h_2\sim\calN(0,\II_n)$ be independent Gaussian
vectors and define
\[
  \phi_z(g,h;x,y)
  :=
  \sqrt{\frac{\theta_1}{m}}
  \paren{\norm{y_1}\,g_1^\intercal x_1 + \norm{x_1}\,h_1^\intercal y_1}
  +
  \sqrt{\frac{\theta_2}{m}}
  \paren{\norm{y_2}\,g_2^\intercal x_2 + \norm{x_2}\,h_2^\intercal y_2},
\]
and the auxiliary optimization
\begin{equation}\label{eq:WishartAODefinition}
  \mathrm{AO}_z
  :=
  \max_{x\in\calE_z}
  \max_{\norm{y}=1}
  \phi_z(g,h;x,y).
\end{equation}
We claim that for every $z,t>0$,
\[
  \PP\paren{\mathrm{PO}_z>t}
  \le 4\PP\paren{\mathrm{AO}_z\ge t}.
\]
This is an immediate corollary of Theorem~\ref{thm:gordon-useful-opnorm}. The claim follows by Lemma~\ref{lem:Wishart-AuxiliaryOpt} below, 
\[
  \PP\paren{\lambda_{\max}(W) > 1 + \epsilon}
  \le 4\PP\paren{\mathrm{AO}_{1+\epsilon} \ge 1}
  \le e^{-\Omega(n\epsilon^2)}.
\]
\end{proof}

\begin{lemma}\label{lem:Wishart-AuxiliaryOpt}
  The auxiliary optimization problem defined in \eqref{eq:WishartAODefinition} is bounded with probability $1-\exp{\set{-\Omega(n\varepsilon^2)}}$,
    \[\mathrm{AO}_{1+\varepsilon} \leq 1- \Omega(\varepsilon).\] 

\end{lemma}

Before proving Lemma~\ref{lem:Wishart-AuxiliaryOpt}, we prove the following two elementary inequalities: 

\begin{lemma}\label{lem:wishart-basic-scalar}
Let $\gamma>0$, let $\theta=\gamma/(1+\gamma)$, and let $0\le p\le r$. Then
\begin{equation}\label{eq:wishart-basic-scalar}
  \theta\paren{p+\sqrt\tau\,r}^2
  \le
  p^2+\tau\gamma\,r^2.
\end{equation}
\end{lemma}

\begin{proof}
A direct expansion gives
\[
  (1+\gamma)\paren{p^2+\tau\gamma\,r^2}
  -
  \gamma\paren{p+\sqrt\tau\,r}^2
  =
  \paren{p-\gamma\sqrt\tau\,r}^2\ge 0.
\]
Dividing by $1+\gamma$ yields \eqref{eq:wishart-basic-scalar}.
\end{proof}

\begin{lemma}\label{lem:wishart-Ainv-geometry}
Let $e_1,e_2\in\RR^m$ be orthonormal, and let $x=(x_1,x_2)\in\RR^m\times\RR^m$.
Define overlaps $p_1:=\abs{\langle e_1,x_1\rangle}$ and $ p_2:=\abs{\langle e_2,x_2\rangle}.$ Let $A$ be the matrix defined in \eqref{eq:ADef_WishartAO}. Then
\begin{equation}\label{eq:wishart-Ainv-geometry}
  p_1^2+p_2^2 \le \ang{ x,(A^{-1}\otimes I_m)x}
\end{equation}
\end{lemma}

\begin{proof}
Write $x_1$ and $x_2$ in coordinates as
\[
  x_1=a e_1+b e_2+u,
  \qquad
  x_2=c e_1+d e_2+v,
\]
where $u,v\perp\mathrm{span}\set{e_1,e_2}$. Then $p_1^2=a^2$ and $p_2^2=d^2$, and
\begin{align*}
  (1-\kappa^2)&
 \ang{x,(A^{-1}\otimes I_m)x} = \\
  &\qquad=
  \norm{x_1}^2+\norm{x_2}^2-2\kappa\langle x_1,x_2\rangle \\
  &\qquad=
  \paren{a^2+c^2-2\kappa ac}
  +
  \paren{b^2+d^2-2\kappa bd}
  +
  \norm{u}^2+\norm{v}^2-2\kappa\langle u,v\rangle.
\end{align*}
Now as in the proof of Lemma~\ref{lem:wigner-null-geometry}, we have that 
\begin{gather*}
      a^2+c^2-2\kappa ac=(1-\kappa^2)a^2+(c-\kappa a)^2 \\
       b^2+d^2-2\kappa bd=(1-\kappa^2)d^2+(b-\kappa d)^2 \\
       \norm{u}^2+\norm{v}^2-2\kappa\langle u,v\rangle
  =
  \norm{u-\kappa v}^2+(1-\kappa^2)\norm{v}^2\ge 0.
\end{gather*}
Therefore, we can conclude that 
\[
  (1-\kappa^2)
 \ang{ x,(A^{-1}\otimes I_m)x}   \ge
  (1-\kappa^2)(a^2+d^2),
\]
which is exactly \eqref{eq:wishart-Ainv-geometry}.
\end{proof}

\begin{proof}[Proof of Lemma~\ref{lem:Wishart-AuxiliaryOpt}]
We now solve the auxiliary optimization. Fix $z = 1 + \varepsilon$. For a parameter $\delta\in(0,1)$ to be chosen later, consider the event
event that
\begin{equation}\label{eq:wishart-good-event}
  \Bigl|\frac{\norm{g_i}}{\sqrt m}-1\Bigr|\le \delta,
  \qquad
  \Bigl|\frac{\norm{h_i}}{\sqrt m}-\sqrt\tau\Bigr|\le \delta,
  \qquad
  \frac{\abs{\langle g_1,g_2\rangle}}{m}\le \delta,
  \qquad i=1,2.
\end{equation}
Since $m=\Theta(n)$, standard Gaussian norm concentration and concentration of
the inner product of two independent Gaussian vectors imply that this event happens with probability at least $1 - \exp \set{-\Omega(n\delta^2)}$. From now on we work on this event. Define
\[
  e_1:=\frac{g_1}{\norm{g_1}},
  \qquad
  e_2:=\frac{g_2-\langle g_2,e_1\rangle e_1}{\norm{g_2-\langle g_2,e_1\rangle e_1}}.
\]
Then $e_1$ and $ e_2$ are orthonormal, and in fact from the estimates above 
\[
  g_1=(1+O(\delta))\sqrt m\,e_1,
  \qquad
  g_2=(1+O(\delta))\sqrt m\,e_2+O(\delta)\sqrt m\,e_1.
\]
Hence for every $x=(x_1,x_2)\in\calE_z$, if we write
\[
  r_i:=\norm{x_i},
  \qquad
  p_i:=\abs{\langle e_i,x_i\rangle},
  \qquad i=1,2,
\]
then we can upper bound 
\begin{equation}\label{eq:wishart-gi-bound}
  \abs{g_i^\intercal x_i}
  \le \sqrt m\,(p_i+C\delta r_i),
  \qquad
  \norm{h_i}\le \sqrt m(\sqrt\tau+C\delta).
\end{equation}
For fixed $x$, maximizing first over the directions of $y_1,y_2$ and then over
their norms subject to $\norm{y_1}^2+\norm{y_2}^2=1$ gives
\[
  \max_{\norm{y}=1}\phi_z(g,h;x,y)
  =
  \sqrt{\sqb{f_1(x_1)}_+^2+\sqb{f_2(x_2)}_+^2}
  \le
  \sqrt{f_1(x_1)^2+f_2(x_2)^2},
\]
where functions $f_i$ are defined by 
\[
  f_i(x_i):=
  \sqrt{\frac{\theta_i}{m}}
  \paren{g_i^\intercal x_i+\norm{h_i}\,\norm{x_i}}.
\]
By \eqref{eq:wishart-gi-bound}, we therefore have
\[
  f_i(x_i)^2
  \le
  \theta_i\paren{p_i+(\sqrt\tau+C\delta)r_i}^2.
\]
Consequently, the auxiliary optimization is upper bounded by 
\begin{equation}\label{eq:wishart-AO-pre-det}
  \mathrm{AO}_z^2
  \le
  \max_{x\in\calE_z}
  \sum_{i=1}^2
  \theta_i\paren{p_i+(\sqrt\tau+C\delta)r_i}^2.
\end{equation}
Now define $ L_0(x):=\theta_1\paren{p_1+\sqrt\tau\,r_1}^2+\theta_2\paren{p_2+\sqrt\tau\,r_2}^2.$  Since $p_i\le r_i$ and $\delta\le 1$, for a larger constant $C$ we obtain an easy upper bound
\[
  \sum_{i=1}^2 \theta_i\paren{p_i+(\sqrt\tau+C\delta)r_i}^2
  \le L_0(x)+C\delta(r_1^2+r_2^2).
\]
Also, by \eqref{eq:wishart-Mz-definition}, the eigenvalues of $M_z$ is bounded away from $0$:
\[
  M_z\succeq zA^{-1}\succeq \frac{z}{1+\kappa}I_2 \succeq \frac{1}{1+\kappa}I_2.
\]
Since $x\in\calE_z$ satisfies $\langle x,(M_z\otimes I_m)x\rangle=1$, this implies that $\calE_z$ is bounded:
\begin{equation}\label{eq:wishart-ellipsoid-norm-bound}
  r_1^2+r_2^2\le 1+\kappa\le 2.
\end{equation}
Combining this with \eqref{eq:wishart-AO-pre-det}, we arrive at
\begin{equation}\label{eq:wishart-AO-det-reduction}
  \mathrm{AO}_z^2
  \le
  \max_{x\in\calE_z}L_0(x)+C\delta.
\end{equation}

By Lemma~\ref{lem:wishart-basic-scalar}, applied with $(\gamma,p,r)=(\alpha,p_1,r_1)$ and
$(\gamma,p,r)=(\beta,p_2,r_2)$,
\[
  L_0(x)
  \le
  p_1^2+p_2^2+\tau\alpha r_1^2+\tau\beta r_2^2.
\]
Then Lemma~\ref{lem:wishart-Ainv-geometry} and the fact that $M_1=A^{-1}+\tau\diag(\alpha,\beta) $ gives that 
\begin{equation}\label{eq:wishart-L0-bound}
  L_0(x)
  \le
  \ang{x,(A^{-1}\otimes I_m)x}
  +\tau\alpha\norm{x_1}^2+\tau\beta\norm{x_2}^2 = \ang{x, (M_1 \otimes I_m)  x},
\end{equation}
holds for all $x \in \RR^{2m}$. We now return to $x\in\calE_z$ with $z=1+\varepsilon$. Since
\[
  M_z-M_1=(z-1)A^{-1}=\varepsilon A^{-1},
\]
we have that $\ang{x,(M_1\otimes I_m)x}
  =
  1-\varepsilon
 \ang{x,(A^{-1}\otimes I_m)x}$. Write
\[
 x
  =
  \sqb{A^{1/2}(zI_2+C)^{-1/2}\otimes I_m}r
\]
for some $r\in\SS^{2m-1}$. Then, note that 
\[
 \ang{x,(A^{-1}\otimes I_m)x}
  =
  \langle r,\sqb{(zI_2+C)^{-1}\otimes I_m}r\rangle
  \ge
  \frac{1}{\lambda_{\max}(zI_2+C)}.
\]
Since $z\in[1,2]$ and $C \succeq 0$ is fixed, there exists a constant
$c_0>0$ such that
\begin{equation}\label{eq:wishart-margin}
 \ang{x,(M_1\otimes I_m)x}
  \le
  1-c_0\varepsilon
  \qquad\text{for all } x\in\calE_{1+\varepsilon}.
\end{equation}
Combining \eqref{eq:wishart-AO-det-reduction}, \eqref{eq:wishart-L0-bound}, and
\eqref{eq:wishart-margin}, we conclude that 
\[
  \mathrm{AO}_{1+\varepsilon}^2
  \le
  1-c_0\varepsilon+C\delta.
\]
If we choose $\delta/\varepsilon$ to be small constant, we conclude that $\mathrm{AO}_{1+\varepsilon} \leq 1- \Omega(\varepsilon)$. Since this happens with with $1-\exp{\set{-\Omega(n\varepsilon^2)}}$ probability, we are done.

\end{proof}

%% file: convergence.tex
\section{Analysis of the Stieltjes transform}

In this section, we will study in more depth the spectral properties of random matrices we are working with. In particular, we will analyze Stieltjes transform of each of our null models (see Section~\ref{sec:Model_Def}). More precisely, we obtain estimates for the rate of convergence in the region $\Omega_\eta $ defined in \eqref{eq:OmegaDef}. Let us recall the notion of `with overwhelming probability'.

\begin{definition}[Overwhelming probability]
    We say that a sequence of events $\{\mathcal E_n\}$ holds with overwhelming probability (w.o.p.) if for any $C>0$, $\PP(\mathcal E_n) \ge 1 - n^{-C}$ for all sufficiently large $n$. By a union bound, the union of polynomially many events that hold with overwhelming probability also holds with overwhelming probability.
\end{definition}

\subsection{Canonical correlation model and proof of Theorem~\ref{thm:resolvConvCCAMain}, Lemma~\ref{lem:extended-solution-of-finite-systemMain}}\label{subsec:CCA_ConvProof}

We briefly recall the definitions from Section~\ref{sec:convergenceMain}. Throughout this subsection, $U \in \RR^{n \times m}$ and $V \in \RR^{n \times k}$ are two independent Gaussian matrices with i.i.d. entries  $U_{ij} \sim \calN(0, m^{-1})$ and $V_{ij} \sim \calN(0, k^{-1})$. Suppose
\begin{equation*}
    \frac{n}{m} \to \tau_m \geq 1 \text{ and } \frac{n}{ k} \to \tau_k \geq 1 \text{ and } \kappa := (\tau_m\tau_k)^{-1/4}  < 1
\end{equation*}
We are interested in the block matrix $W$ given by
\begin{equation*}
	W = \begin{pmatrix}
		-\kappa U^\intercal U & U^\intercal V \\
		V^\intercal U & -\kappa V^\intercal V
	\end{pmatrix}.
\end{equation*}
We will define by $\lambda_* := (1-\kappa^2)/\kappa$. To study spectral properties of $W$, the relevant object is the resolvent $(W - z I_{m+k})^{-1}$. However, it turns out to be more convenient to work with linearized resolvent given by
  \begin{equation}\label{eq:linMatrixCCA}
    \calL(z) := \begin{pmatrix}
            -z I_m & 0 & \sqrt{\kappa} U^\intercal  & 0 \\
            0 & -z I_k & 0 & \sqrt{\kappa} V^\intercal  \\
            \sqrt{\kappa} U & 0 & \alpha \kappa  I_n & \alpha I_n  \\ 
            0 & \sqrt{\kappa} V & \alpha  I_n & \alpha \kappa  I_n
              \end{pmatrix} \quad \text{ and } \quad \calG (z) = \calL(z)^{-1}
  \end{equation}
where $\alpha = \kappa/(\kappa^2 - 1)$. We observe that by Schur's complement lemma, the top left $(m+k) \times (m+k)$ block $\calG (z)$ is precisely $(W - zI_{m+k})^{-1}$.
We will investigate the limiting form of $\calG $ and deduce that it is close to some deterministic $\calM (z)$ in the sense below. Let us first define 
\begin{equation}\label{eq:OmegaDef}
    \Omega_{\eta} = \{z \in \CC_+ : \Im z > \eta, \abs{z} < \eta^{-1}\} 
\end{equation}

\begin{definition}\label{def:det_system}
For each $z\in \CC_+$, define scalars $r(z),s(z)\in\CC_+$ and a $2\times2$ complex matrix
$T(z)=(t_{ij}(z))_{i,j\in\{1,2\}}$ as the unique solution to the coupled system
\begin{gather}\label{eq:finite_system}
\begin{aligned}
r(z) &= \frac{1}{-z - \kappa \tau_m \,t_{11}(z)},\\
s(z) &= \frac{1}{-z - \kappa \tau_k\,t_{22}(z)},\\
\end{aligned} 
\\
\begin{pmatrix}
    t_{11}(z) & t_{12}(z) \\
    t_{21}(z) & t_{22}(z)
\end{pmatrix}\begin{pmatrix}
    \alpha \kappa - \kappa  r(z) & \alpha \\
    \alpha & \alpha \kappa  - \kappa s(z)
\end{pmatrix} = I_{2}
\end{gather}
Extend $(r, s, T)$ analytically to $\DD := \CC \setminus (-\infty, \lambda_*]$ and set for all $z \in \DD$
\begin{equation}\label{eq:MNdef}
\calM (z):=
\begin{pmatrix}
r(z) I_m & 0 & 0 & 0\\
0 & s(z) I_k & 0 & 0\\
0 & 0 & t_{11}(z) I_n & t_{12}(z) I_n\\
0 & 0 & t_{21}(z) I_n & t_{22}(z) I_n
\end{pmatrix}.
\end{equation}
\end{definition}
\begin{rmk}
    The existence and uniqueness of the solution in Definition~\ref{def:det_system} is non-trivial and is the content of Lemma~\ref{lem:convergence-to-finite-system}. Analytic continuation result is also nontrivial and is the content of Lemma~\ref{lem:extended-solution-of-finite-system}.
\end{rmk}

\begin{theorem}[Theorem~\ref{thm:resolvConvCCAMain}]\label{thm:resolvConvCCA}
 Let $\calG$ be as in \eqref{eq:linMatrixCCA} and $\calM$ be as in Definition~\ref{def:det_system}. Fix $\eta, \varepsilon > 0$. Then, for any unit vectors $x, y\in \RR^{m+k+2n}$, 
  \begin{equation}\label{eq:resolvExpectation} \sup_{z\in \Omega_\eta } \abs{x^\intercal  \calG (z) y - x^\intercal  \calM (z) y} \le O(n^{-1/2 + \varepsilon})\end{equation}
 with overwhelming probability, where the hidden constant depends only on $\tau_m, \tau_k, \eta, $ and $\varepsilon$.
\end{theorem}

We show Theorem~\ref{thm:resolvConvCCA} in two main steps. We first show a concentration $\calG (z) \approx \EE \calG (z)$ and then we show a deterministic stability $\calG (z) \approx \calM(z)$. As a corollary of the theorem, we will obtain another more useful theorem below. First, write $\calG(z)$ in the block form
\begin{equation}
    \calG = \begin{pmatrix}
        \calG^{(11)} & \calG^{(12)} & \calG^{(13)} & \calG^{(14)} \\
        \calG^{(21)} & \calG^{(22)} & \calG^{(23)} & \calG^{(24)} \\ 
        \calG^{(31)} & \calG^{(32)} & \calG^{(33)} & \calG^{(34)} \\
        \calG^{(41)} & \calG^{(42)} & \calG^{(43)} & \calG^{(44)} \\
    \end{pmatrix}
\end{equation}
with $\calG^{(11)} \in \RR^{m \times m}$, $\calG^{(22)} \in \RR^{k \times k}$, $\calG^{(33)} \in \RR^{n \times n}$, and $\calG^{(44)} \in \RR^{n \times  n}$. This matches the block form of $\calL(z)$. 

\begin{theorem}[Theorem~\ref{thm:real-stieltjes-transform-CCAMain}]\label{thm:real-stieltjes-transform-CCA}
Let $\calG$ be as in \eqref{eq:linMatrixCCA} and $\calM$ and $\DD$ be as in Definition~\ref{def:det_system}. Then, for any compact set $K \subset (\lambda_*, \infty)$, we have that 
\begin{equation*}
   \sup_{z \in K}  \frac{1}{n}\abs{\Tr \, \calG^{(ij)}(z)  - \Tr \, \calM^{(ij)}(z)} = o(1),
\end{equation*}
with overwhelming probability.
\end{theorem}

\subsubsection{Concentration of linearized resolvent}

In this part, we adopt the notation introduced in the subsection. The main result is the following proposition and its corollaries.

\begin{prop}\label{prop:CCA-concentration}
   Let $\eta, \varepsilon > 0$ be arbitrary. Then for any unit $x, y\in \RR^{m+k+2n}$
  \begin{equation*} \sup_{z\in \Omega_\eta } \abs{x^\intercal  \calG (z) y - x^\intercal  \EE \calG (z) y} \le O(n^{-1/2 + \varepsilon})\end{equation*}
 with overwhelming probability, and the hidden constant depends only on $\tau_m$,  $\tau_k$, $\varepsilon$, and $\eta$. 
\end{prop}

To prove Proposition~\ref{prop:CCA-concentration}, we will show that for all $z\in \Omega_\eta $ the quadratic form $x^\intercal \calG (z)y$ is Lipschitz in the disorder and apply Gaussian concentration. Since we are working away from the real line, the spectrum of $W$ is bounded away from $\Omega_\eta $ so everything will be controlled. 

\begin{lemma}\label{lem:GNOpNormCCA}
Suppose $\eta > 0$. Then for the following hold 
\begin{equation*}
    \EE \sup_{z \in \Omega_\eta }\norm{\calG (z)}_{\op}^4 \leq O(\eta^{-4}) \quad \text{ and } \quad \sup_{z \in \Omega_\eta }\norm{\calG (z)}_{\op} \leq O(\eta^{-1}) \text{ w.o.p.}
\end{equation*}
where the hidden constant depends only on $\tau_m$ and $\tau_k$.
\end{lemma}
\begin{proof}
Let us rewrite $\calL(z)$ in the block form 
  \begin{equation*}
    \calL(z) = \begin{pmatrix}
             -z I_{m+k} & A^\intercal  \\
             A & B
              \end{pmatrix} \text{ with } 
    A := \begin{pmatrix}
            \sqrt{\kappa} U & 0 \\
            0 & \sqrt{\kappa} V
         \end{pmatrix},\ 
    B := \begin{pmatrix}
             \alpha \kappa  I_n &\alpha I_n  \\ 
            \alpha  I_n & \alpha \kappa I_n
         \end{pmatrix}
  \end{equation*}
Then by Schur's complement lemma, $\calG (z)$ is given by
  \begin{equation} \label{eq:SchurGnCCA}
    \calG (z) = \begin{pmatrix}
      (W - zI_{m+k})^{-1} & -(W - zI_{m+k})^{-1} A^\intercal  B^{-1} \\
      -B^{-1} A (W - zI_{m+k})^{-1} & B^{-1} + B^{-1}A(W - zI_{m+k})^{-1}A^\intercal B^{-1}
    \end{pmatrix}
  \end{equation}
Since $\kappa <1$, we have $\norm{B^{-1}}_{\op} = O(1)$. Note also that since $\Im z \geq \eta$, we also have
\begin{equation*}
    \norm{(W - zI_{m+k})^{-1}}_{\op} \leq \eta^{-1}.
\end{equation*}
Since the operator norm of a matrix is bounded by the sum of operator norms of its blocks, we see that 
\begin{equation*}
    \sup_{z \in \Omega_\eta }\norm{\calG (z)} \leq O(\eta^{-1})(1 + \norm{A}_{\op} + \norm{A}_{\op}^2).
\end{equation*}
To finish, observe that 
\begin{equation*}
    \norm{A}_{\op} \leq \sqrt{\kappa} \norm{U}_{\op} + \sqrt{\kappa} \norm{V}_{\op},
\end{equation*}
so with standard Gaussian matrix moment bounds  (see \cite{TaoBook} for example), the conclusion follows. 
\end{proof}

\begin{lemma}\label{lem:lipCCA}
  Fix $x, y\in \CC^{m+k+2n}$ with unit norm and $z \in \CC_+$. Then the function $(U,V) \mapsto x^\intercal \calG (z)y$ is Lipschitz. In particular,
    \[\norm{\grad_{U,V} (x^\intercal \calG (z) y)}^2 \le 4\kappa \norm{\calG (z)}_{\op}^4.\]
\end{lemma}
\begin{proof}
Let $F(z) = x^\intercal \calG (z)y$, and note that this is a function of disorder $U, V$. For readability, we can fix $z$ and drop it from the notation. We first compute the partial derivatives of $F$ with respect to the disorder. Recall that for any invertible matrix $A$ dependent on some variable $x$, the derivative of its inverse is
  \begin{equation}\label{eq:derivativeInverse} \frac{dA^{-1}}{dx} = -A^{-1}\frac{dA}{dx} A^{-1}. \end{equation}
 So to compute the partial derivatives of $\calG $, it suffices to understand the partial derivatives of the linearized matrix $\calL$, defined in \eqref{eq:linMatrixCCA}. Notice that $U$ and $V$ only show up in the off-diagonal blocks of $\calL$. In particular because $\calG $ is symmetric,
  \begin{align}\label{eq:partialDerU}
   \begin{split}
      \pp{F}{U_{ij}} & = -x^\intercal \calG \pp{\calL}{U_{ij}}\calG y = -\sqrt{\kappa} \paren{(\calG  x)_{m+k+i}(\calG  y)_{j} + (\calG  x)_{j} (\calG  y)_{m+k+i}}.
   \end{split}
  \end{align}
 By the same argument we obtain the formula for $\pp{F}{V_{ij}}$,
 \begin{align} \label{eq:partialDerV}
      \pp{F}{V_{ij}} & = -\sqrt{\kappa} \paren{(\calG  x)_{m+k+n+i}(\calG  y)_{m+j} + (\calG  x)_{m+j} (\calG  y)_{m+k+n+i}}.  
  \end{align}
  Using Cauchy-Schwarz inequality,
    \[\abs{\pp{F}{U_{ij}}}^2\le2\kappa \paren{\abs{(\calG  x)_{m+k+i}}^2\abs{(\calG  y)_{j}}^2 + \abs{(\calG  x)_{j}}^2\abs{(\calG  y)_{m+k+i}}^2}.\]\
  Repeating this for the partial derivatives in $V$ in \eqref{eq:partialDerV} and summing over indices:
    \[\norm{\grad_{U,V} F}^2 \le 4\kappa \norm{\calG x}^2\norm{\calG y}^2 \le 4\kappa \norm{\calG }_{\op}^4.\]
\end{proof}

Now we recall the Gaussian concentration inequality \cite{VRbook} for Lipschitz functions.

\begin{lemma}[Gaussian concentration for Lipschitz functions]\label{lem:gaussianConLip}
  Let $g \sim \cN(0, \sigma^2\II_d)$ be a Gaussian vector and $f:\RR^d \to \RR$ be an $L$-Lipschitz function. Then 
    \[\PP\paren{\abs{f(g) - \EE f(g)} \ge x} \le 2\exp\paren{-\frac{x^2}{2L^2\sigma^2}}.\]
\end{lemma}

Now we are now ready to prove Proposition~\ref{prop:CCA-concentration}. The idea is simply to apply Lemma~\ref{lem:gaussianConLip} together Lemma~\ref{lem:lipCCA} and Lemma~\ref{lem:GNOpNormCCA}.

\begin{proof}[Proof of Proposition~\ref{prop:CCA-concentration}]
 Throughout this proof, $C$ may change line to line and will denote a constant dependent only on $\tau_m$, $\tau_k$, $\varepsilon$, and $\eta$. We fix $x$ and $y$ and define $F(z) = x^\intercal \calG (z)y$. Note that $F(z)$ depends on $U$ and $V$. For clarity, we sometimes drop $z$ from the notation.
 
 From Lemma~\ref{lem:GNOpNormCCA}, for all $z \in \Omega_\eta $ with overwhelming probability $\norm{\calG (z)}_{\op} \leq O(1)$. We will denote this event by $\cE = \cE_z$. Recall that by Lemma~\ref{lem:lipCCA} the $F(z)$ is $O(\norm{\calG (z)}_{\op}^2)$ Lipschitz with respect to the disorder $(U,V)$. In particular, $F(z)$ is $C$ Lipschitz on $\cE$.

Note that since $F(z)$ is not necessarily Lipschitz on $\cE^c$, so we cannot immediately apply Lemma~\ref{lem:gaussianConLip}. To resolve this issue, we define $\tilde{F}(z)$ be McShane's Lipschitz extension of $F(z)|_{(U,V) \in\cE}$ to all of $\RR^{nm}\times \RR^{nk}$ by 
\begin{gather*}
    \Im \tilde{F}(z)\vert_{(U, V)} := \inf_{(U', V')\in \cE}\paren{\Im F(z)\vert_{(U', V')} + C\norm{(U, V) - (U', V')}_2} \\
    \Re \tilde{F}(z)\vert_{(U, V)} := \inf_{(U', V')\in \cE}\paren{\Re F(z)\vert_{(U', V')} + C\norm{(U, V) - (U', V')}_2}
\end{gather*}
 where we treat $(U,V)$ as a vectors in $\RR^{nm}\times \RR^{nk}$.  now $\tilde F(z)$ is $O(1)$ Lipschitz over entire disorder set and $\tilde F(z) = F(z)$ when $(U,V) \in \calE$.
 
 Recall that $U_{ij}$ and $V_{ij}$ are independent Gaussians with variances $\Theta(n^{-1})$. Fix $z\in \Omega_\eta $.  As $F(z) = \tilde{F}(z)$ when $(U, V)\in \cE$, we may apply Lemma~\ref{lem:GNOpNormCCA} and Lemma~\ref{lem:gaussianConLip} to both real and imaginary parts to see that
  \begin{align*}
    \PP\paren{\abs{F(z) - \EE \tilde{F}(z)} \ge t} & 
    \le \PP\paren{\calE^c} + \PP\paren{\abs{\tilde{F}(z) - \EE \tilde{F}(z)} \ge t} \\
    & \le \PP\paren{\calE^c} + \exp\set{- \Omega(nt^2)}
  \end{align*}
 Choosing $t = n^{-1/2+\varepsilon}$, we see that for every fixed $z\in \Omega_\eta $ with overwhelming probability, 
\begin{equation}\label{eq:FBoundCCA}
\abs{F(z) - \EE \tilde{F}(z)} \le O(n^{-1/2 + \varepsilon}).
\end{equation}
We now show that we can replace $\EE \tilde{F}(z)$ by $\EE F(z)$ at no cost in the error. Since $\tilde{F}(z)$ and $F(z)$ agree on $(U, V)\in \cE$, we can apply Cauchy-Schwarz
   \begin{align*}
     \abs{\EE F(z) - \EE \tilde{F}(z) } & \leq \EE |F(z) - \tilde F(z)| \II_{\calE^c} \\
        & \le \sqrt{ \PP\paren{\cE^c} \EE|F(z) - \tilde{F}(z)|^2}  \\
        & \le \sqrt{ 2\PP\paren{\cE^c} (\EE\abs{F(z)}^2  + \EE |{\tilde F}(z)|^2)}
   \end{align*}
In particular, it suffices to show that the second moment of $F(z)$ and $\tilde{F}(z)$ grows at most polynomially in $n$. As $\tilde{F}(z)$ is Lipschitz and $(0, 0)\in \cE$, its second moment is bounded by
  \[\EE |\tilde{F}(z)\vert_{(U, V)}|^2 \le 2|\tilde{F}(z)\vert_{(0,0)}|^2 + 2C\EE \norm{(U, V)}^2 \le O(n).\]
  On the other hand, the second moment of $F(z) = x^\intercal  \calG (z) y$ is bounded by the second moment of $\norm{\calG (z)}_{\op}$, which is polynomially bounded by Lemma~\ref{lem:GNOpNormCCA}. Combining all bounds together, we see that for each fixed $z \in \Omega_\eta $, we have 
\begin{equation*}
   \abs{F(z) - \EE F(z)} \leq O(n^{-1/2 + \varepsilon}),
\end{equation*}
with overwhelming probability.
It now remains to show that we can obtain the same conclusion simultaneously for all $z \in \Omega_\eta $. We will do this by the net argument.

First, we show that $F(z)$ is Lipschitz in $z$ on $\Omega_\eta $. Applying \eqref{eq:derivativeInverse}, we get that with overwhelming probability, for all $z \in \Omega_\eta $

\begin{equation}\label{eq:zLipCCAArgument}
    \abs{{\pp{F}{z}(z)}} = \abs{x^\intercal  \calG (z) (I_m\oplus I_k \oplus 0_{2n}) \calG (z) y}\le \norm{\calG (z)}_{\op}^2 =O(1) .
\end{equation}
In fact from this bound, it also follows that $z \mapsto \EE F(z)$ is $O(1)$ Lipschitz on $\Omega_\eta $. Let $\calS$ be a $1/n$-net of $\Omega_\eta $ with $\abs{\calS} = O(n^2)$. Taking a union bound over $z\in \calS$ of \eqref{eq:FBoundCCA}, we see with overwhelming probability, for all $z\in \calS$,
    \[\abs{F(z) - \EE F(z)} \le O(n^{-1/2+\varepsilon}).\]
Since $z \mapsto F(z) - \EE F(z)$ is $O(1)$ Lipschitz on $\Omega_\eta $ also with overwhelming probability, we conclude that for all $z \in \Omega_\eta $
\begin{equation*}
     \abs{F(z) - \EE F(z)} \le O(n^{-1/2 + \varepsilon}),
\end{equation*}
which proves the claim.  
\end{proof}

\subsubsection{Derivations of approximate equations}
In this section, we will describe $\EE \calG (z)$. In particular, we will prove the following Proposition~\ref{prop:approxSelfEqCCA}, which shows that $\EE \calG (z)$ has a form that is very similar to $\calM (z)$ described in Definition~\ref{def:det_system}.

\begin{prop}\label{prop:approxSelfEqCCA}
  For all $z \in \CC_+$, the expectation $\EE \calG (z)$ has the following block form
    \begin{equation}\label{eq:exptGn}
     \EE \calG (z) = \begin{pmatrix} 
                  r_n(z) I_m & 0 & 0 & 0 \\
                  0 & s_n(z) I_k & 0 & 0 \\
                  0 & 0 & t_{11, n}(z) I_n & t_{12, n}(z) I_n \\
                  0 & 0 & t_{21, n}(z) I_n & t_{22, n}(z) I_n
                \end{pmatrix},
    \end{equation}
where $ r_n(z), s_n(z), t_{ij, n}(z) \in \CC$. In fact, for any fixed $\eta, \varepsilon > 0$, whenever $z \in \Omega_\eta $, these coefficients satisfy the following system of equations
\begin{gather}\label{eq:finite_systemApprox}
\begin{aligned}
r_n(z) &= \frac{1 + O(n^{-1 + \varepsilon})}{-z - \kappa \tau_m \,t_{11, n}(z)},\\
s_n(z) &= \frac{1 + O(n^{-1 + \varepsilon})}{-z - \kappa \tau_k\, t_{22, n}(z)},\\
\end{aligned} 
\\
\begin{pmatrix}
    t_{11, n}(z) & t_{12, n}(z) \\
    t_{21, n}(z) & t_{22, n}(z)
\end{pmatrix}\begin{pmatrix}
    \alpha \kappa - \kappa  r_n(z) & \alpha \\
    \alpha & \alpha \kappa  - \kappa  s_n(z)
\end{pmatrix} = I_{2} + O(n^{-1 + \varepsilon}) \cdot 11^\intercal
\end{gather}

\end{prop}

Let us introduce the following notation for this section: express $\calG (z) = \calL(z)^{-1}$ in blocks of the same size as in \eqref{eq:linMatrixCCA}. Let $\calG^{(ij)}$ denote the block in the $i$-th row and $j$-th column. We use the same notation for any matrix of the same shape as $\calG (z)$.

\begin{lemma}\label{lem:exptGNCCA}
  For all $z \in \CC_+$, the expectation $\EE \calG (z)$ has the following block form
    \begin{equation}
     \EE \calG (z) = \begin{pmatrix} 
                  r_n(z) I_m & 0 & 0 & 0 \\
                  0 & s_n(z) I_k & 0 & 0 \\
                  0 & 0 & t_{11, n}(z) I_n & t_{12, n}(z) I_n \\
                  0 & 0 & t_{21, n}(z) I_n & t_{22, n}(z) I_n
                \end{pmatrix},
    \end{equation}
where $r_n(z), s_n(z), t_{ij, n}(z) \in \CC$..
\end{lemma}
\begin{proof}
    Observe that $\calL(z)$ is invariant in distribution under block-wise orthogonal transformations. That is, $O^\intercal  \calL(z)O \overset{d}{=} \calL(z)$ for $O = O_m \oplus O_k \oplus O_n \oplus O_n$ the direct sum of orthogonal matrices. Since $\calG (z) = \calL^{-1}(z)$, it is also invariant in distribution under $O$. This means that $\EE \calG (z)$ also has this property. Since only deterministic matrices that are invariant under orthogonal transformations are multiples of identity, we conclude the form \eqref{eq:exptGn}. 
\end{proof}

We now set out to derive a system of equations which describe $r_n(z), s_n(z), t_{ij, n}(z)$. The underlying idea here is to expand out the equation 
\begin{equation}\label{eq:LNGNIdentity}
    \calL(z)  \calG (z) = I_{m+k+2n}
\end{equation}
in terms of each individual block and apply Gaussian integration by parts. Recalling $\calL(z)$ as defined in \eqref{eq:linMatrixCCA}, the terms which contribute to the top left block $I_m$ of \eqref{eq:LNGNIdentity}:
  \[-z \calG^{(11)} + \sqrt{\kappa} U^\intercal  \calG^{(31)} = I_m.\]
Taking trace and expectations of both sides, we obtain an equation for $r_n$,
  \begin{equation} \label{eq:utildeEq} -z r_n(z) + \sqrt{\kappa}\frac{1}{m}\EE \Tr(U^\intercal  \calG^{(31)}) = 1.\end{equation}
We will look at the terms of the form $\frac{1}{m}\EE \Tr(U^\intercal  \calG^{(31)})$ more closely using Gaussian integration by parts, which we recall here for convenience.
\begin{lemma}[Gaussian integration by parts]\label{lem:GaussianIBP}
  Let $g \sim \cN(0, \sigma^2)$. Let $f: \RR \to \RR$ be a differentiable function so that $\EE f'(g)$ exists.  Then
    \[\EE g f(g) = \sigma^2 \EE f'(g)\]
\end{lemma}

\begin{lemma}\label{lem:coeffIdentitiesCCA}
  The following identities hold:
    \begin{align*}
      \EE \Tr(U^\intercal \calG^{(31)}) & = -\frac{\sqrt{\kappa}}{m} \EE \Big[\Tr (\calG^{(33)})\Tr (\calG^{(11)}) + \Tr (\calG^{(13)}\calG^{(31)})\Big] \\
       \EE \Tr(V^\intercal \calG^{(42)}) &=  -\frac{\sqrt{\kappa}}{k} \EE \Big[\Tr (\calG^{(44)})\Tr (\calG^{(22)}) + \Tr (\calG^{(24)}\calG^{(42)})\Big]\\
      \EE \Tr(U \calG^{(14)}) &= -\frac{\sqrt{\kappa}}{m} \EE \Big[\Tr(\calG^{(34)}) \Tr(\calG^{(11)})  + \Tr(\calG^{(13)}\calG^{(41)})\Big]\\
      \EE \Tr(V \calG^{(23)}) &= -\frac{\sqrt{\kappa}}{k} \EE \Big[\Tr(\calG^{(43)}) \Tr(\calG^{(22)})  + \Tr(\calG^{(24)}\calG^{(32)})\Big]
    \end{align*}
\end{lemma}
\begin{proof}
 We can first express the trace as 
    \[\Tr(U^\intercal  \calG^{(31)}) = \sum_{ij} U_{ij}\calG^{(31)}_{ij}.\]
 So it suffices to compute the expectations $\EE[U_{ij}\calG^{(31)}_{ij}]$. Treating $\calG^{(31)}$ as a function of the Gaussian variable $U_{ij}$, we can apply Gaussian integration by parts Lemma~\ref{lem:GaussianIBP} and \eqref{eq:derivativeInverse},
   \[\EE[U_{ij}\calG^{(31)}_{ij}] = \frac{1}{m}\EE\paren{\pp{\calG^{(31)}_{ij}}{U_{ij}}} = -\frac{1}{m} \EE \paren{\calG \pp{\calL}{U_{ij}}\calG }^{(31)}_{ij}.\]
 Recalling that $U_{ij}$ only shows up in $\calL$ in the $(31)$ and $(13)$ blocks, we see that,
   \[\paren{\calG \pp{\calL}{U_{ij}}\calG }^{(31)}_{ij} = \sqrt{\kappa}\paren{\calG^{(33)}_{ii} \calG^{(11)}_{jj} + (\calG^{(31)}_{ij})^2}.\]
 Putting everything together, we see that
   \[\EE \Tr(U^\intercal \calG^{(31)}) = -\frac{\sqrt{\kappa}}{m} \sum_{ij} \EE \paren{G_{ii}^{(33)} G_{jj}^{(11)} + (G_{ij}^{(31)})^2}.\]
 Writing this in terms of the trace we obtain the first formula. Each other formula follows from here.
\end{proof}

Notice that the equations in Lemma~\ref{lem:coeffIdentitiesCCA} are composed of a term containing a product of traces and one containing a trace of products. In particular, by Proposition~\ref{prop:CCA-concentration} we expect that the trace of products should be small while the product of traces should be close to their expectations.

\begin{lemma}\label{lem:hard-traces-into-easy-ones}
Fix $\eta, \varepsilon> 0$, then for all $z \in \Omega_{\eta}$, 
    \begin{align}
     \EE \Tr(U^\intercal \calG^{(31)}) & = -
      \sqrt{\kappa} n  \, t_{11, n}(z) \, r_n(z) + O(n^{\varepsilon}) \label{eq:UTG31CCA}\\
       \EE \Tr(V^\intercal \calG^{(42)}) &=  - \sqrt{\kappa} n  \, t_{22, n}(z)\,   s_n(z) + \, O(n^\varepsilon) \label{eq:VTG42CCA}\\
      \EE \Tr(U \calG^{(14)}) &= -\sqrt{\kappa} n \,  t_{12, n}(z) \, r_{n}(z) + O(n^\varepsilon) \label{eq:UG14CCA}\\
      \EE \Tr(V \calG^{(23)}) &= - \sqrt{\kappa} n \,  t_{21, n}(z)  \, s_n(z) + O(n^\varepsilon) \label{eq:VG23CCA}
    \end{align}
\end{lemma}

\begin{proof}
We prove only the first identity; the remaining three follow by the same argument. For this proof, set
\[
X(z):= n^{-1} \Tr \, \calG^{(33)}(z),
\qquad
Y(z):= m^{-1} \Tr \, \calG^{(11)}(z).
\]
Then Lemma~\ref{lem:coeffIdentitiesCCA} gives
\begin{equation}\label{eq:UY-start}
\EE \Tr(U^\top \calG^{(31)})
=
-\sqrt{\kappa}\,n\,\EE\!\big[X(z)Y(z)\big]
-\frac{\sqrt{\kappa}}{m}\,\EE \Tr\!\big(\calG^{(13)} \calG^{(31)}\big).
\end{equation}
We first dispose of the trace-of-products term. Since \(\calG (z)\) is symmetric, we have $\calG^{(13)}=(\calG^{(31)})^\intercal$ and hence
\[
\big|\Tr(\calG^{(13)}\calG^{(31)})\big|
=
\Big|\sum_{i, j} (\calG^{(31)}_{ij})^2\Big|
\le
\|\calG^{(31)}\|_F^2
\le
m \|\calG (z)\|_{\op}^2.
\]
By Lemma~\ref{lem:GNOpNormCCA}, we have $\EE \norm{\calG (z)}_\op^2 \leq O(1)$, so this error term is of the lower order. Next we factorize the product of traces. By Lemma~\ref{lem:exptGNCCA},
\[
\EE X(z)= t_{11,n}(z),
\qquad
\EE Y(z)= r_n(z).
\]
So it remains to show that the covariance of \(X(z)\) and \(Y(z)\) is negligible. We claim that
\begin{equation}\label{eq:var-small}
\EE\big|X(z)-\EE X(z)\big|^2 = O(n^{-1 + 2\varepsilon}),
\qquad
\EE\big|Y(z)-\EE Y(z)\big|^2 = O(n^{-1 + 2\varepsilon}).
\end{equation}
To prove this claim, note that by Proposition~\ref{prop:CCA-concentration}, if apply anisotropic concentration to each basis vector and take a union bound, we get that with overwhelming probability
\begin{equation*}
    |X(z) - \EE X(z)| \leq O(n^{-1/2 + \varepsilon}).
\end{equation*}
Moreover, since $|X(z)| \leq \norm{\calG (z)}_{\op}$ and $\EE \norm{\calG (z)}_\op^4 $ is $O(1)$ from Lemma~\ref{lem:GNOpNormCCA}, we see that indeed $\EE\big|X(z)-\EE X(z)\big|^2 = O(n^{-1 + 2\varepsilon})$, and the same argument works for $Y(z)$.
Applying Cauchy-Schwarz implies
\[
\abs{\EE(X - \EE X) (Y - \EE Y)}\leq 
\sqrt{\EE \abs{X -\EE X}^2 \EE \abs{Y - \EE Y}^2} \leq O(n^{-1 + 2\varepsilon}).
\]
Hence, we see that 
\[
\EE[X(z)Y(z)]
=
 t_{11,n}(z)\, r_n(z)+O(n^{-1 + 2\varepsilon}).
\]
Combining this and above together, we conclude the first identity. The other three identities are proved in similar fashion.
\end{proof}

\begin{proof}[Proof of Proposition~\ref{prop:approxSelfEqCCA}]
  We already know from Lemma~\ref{lem:exptGNCCA} the form of $\EE \calG (z)$. It now suffices to identify the equations. Starting from $\calL(z)\calG (z) = I_{m+k+2n}$ as in \eqref{eq:LNGNIdentity}, and looking at the terms contributing to the top-left $I_m$ block we get \eqref{eq:utildeEq},
    \[-z r_n(z) + \frac{\sqrt{\kappa}}{m}\EE \Tr(U^\intercal  \calG^{(31)}) = 1\]
  Plugging in \eqref{eq:UTG31CCA} and collecting the terms containing $r_n(z)$, we obtain the first equation.
    \[r_n(z) = \frac{1 + O(n^{-1 + \varepsilon})}{-z - \kappa \tau_m t_{11, n}(z)}\]
 Repeating this for the second diagonal block of $I_k$ gives us the second equation, for $s_n(z)$. In similar fashion, we can obtain the matrix equation for $t_{ij}$. For example, the terms which contribute to the third diagonal block $I_n$ gives the equation
   \[\frac{\sqrt{\kappa}}{n}\Tr(U\calG^{(13)}) + \frac{\alpha \kappa }{n}\Tr(\calG^{(33)}) + \frac{\alpha}{ n} \Tr(\calG^{(43)}) = 1\]
 Recalling that $G$ is symmetric, and plugging in \eqref{eq:UTG31CCA} we see that
   \[-\kappa t_{11, n}(z) r_n(z) + \alpha \kappa \, t_{11, n}(z) + \alpha \, t_{21, n}(z) = 1 + O(n^{-1 + \varepsilon})\]
 Using the fact that $t_{12, n} = t_{21, n}$ and repeating this for the remaining blocks we obtain the final equation.
\end{proof}

\subsubsection{Stability of expectation}
We are now ready to show that indeed $\calM(z)$ defined in \eqref{eq:MNdef} is close to $\EE \calG (z)$, which would finish the proof of Theorem~\ref{thm:resolvConvCCA}. We will do this by showing the self-consistent equations for the coefficients \eqref{eq:finite_system} are stable under small perturbations, such as in \eqref{eq:finite_systemApprox}.

\begin{prop}\label{prop:detStabilityCCA}
Let $\eta, \varepsilon > 0$, then the following holds. 
\begin{enumerate}
\item For each \(z\in\CC_+\), the system \eqref{eq:finite_system} has a unique solution
  \[ r(z),s(z)\in\CC_+\quad  \text{ and } \qquad T(z)=\begin{pmatrix}t_{11}(z)&t_{12}(z)\\t_{21}(z)&t_{22}(z)\end{pmatrix}. \]
\item Let \( (r_n(z), s_n(z), T_n(z))\) be the coefficients from Proposition~\ref{prop:approxSelfEqCCA}. Then
  \[\sup_{z\in \Omega_\eta }
\Big( | r_n(z)-r(z)|+ | s_n(z)-s(z)|+ \| T_n(z)-T(z)\|_{\op} \Big) \le O(n^{-1+\varepsilon}).\]
\end{enumerate}
Combined with Proposition~\ref{prop:CCA-concentration}, this implies Theorem~\ref{thm:resolvConvCCA} and Theorem~\ref{thm:resolvConvCCAMain}.
\end{prop}

We will first prove the existence, uniqueness and convergence result in Lemma~\ref{lem:convergence-to-finite-system}, then we will make the convergence quantitative. 

\begin{lemma}\label{lem:convergence-to-finite-system}
    For each $z \in \CC_+$, the system \eqref{eq:finite_system} has a unique solution 
    \begin{equation*}
        r(z), s(z) \in \CC_+ \quad \and \quad T(z)=\begin{pmatrix}t_{11}(z)&t_{12}(z)\\t_{21}(z)&t_{22}(z)\end{pmatrix}
    \end{equation*}
    Moreover, the functions if $(r_n(z), s_n(z), T_n(z))$ are coefficients from  Proposition~\ref{prop:approxSelfEqCCA}, then 
    \begin{equation*}
        r_n \to r,  \quad s_n \to s, \text{ and } T_n \to T,
    \end{equation*}
    locally uniformly on $\CC_+$.
\end{lemma}

\begin{proof}[Proof of uniqueness in Lemma~\ref{lem:convergence-to-finite-system}]
For uniqueness, we proceed via standard contraction argument (see \cite{BS10} for example). Fix $z \in \CC_+$ and assume that \((r(z),s(z),T(z))\) and \((\hat r(z),\hat s(z),\hat T(z))\) are two solutions of \eqref{eq:finite_system}. For notational convenience, we drop the dependence on $z$ in this proof only. In particular, we define 
\begin{equation*}
   r = r(z) \quad s = s(z) \quad T = T(z) \quad B = T(z)^{-1},
\end{equation*}
and likewise $(\hat r, \hat s, \hat T, \hat B)$. We first record the usual imaginary-part identities. Since
  \[T-T^*=T^*(B^*-B)T  \text{ and }  B^*-B= 2i\kappa \begin{pmatrix} \Im r & 0\\ 0 & \Im s \end{pmatrix},\]
we obtain that 
\begin{equation}\label{eq:ImT-identity}
  \Im T = \kappa T^* \begin{pmatrix} \Im r & 0\\ 0 & \Im s \end{pmatrix} T.
\end{equation}
In particular, looking at the diagonal entries, we see
\begin{equation}\label{eq:Imq11} 
  \Im t_{11} = \kappa\bigl(|t_{11}|^2\Im r+|t_{21}|^2\Im s\bigr), \qquad 
  \Im t_{22} = \kappa\bigl(|t_{12}|^2\Im r+|t_{22}|^2\Im s\bigr).
\end{equation}
Taking imaginary parts in \eqref{eq:finite_system} gives
\begin{equation} \label{eq:Imu-identity}
  \Im r = |r|^2\bigl(\Im z+\kappa\tau_m\Im t_{11}\bigr), \qquad 
  \Im s = |s|^2\bigl(\Im z+\kappa\tau_k\Im t_{22}\bigr).
\end{equation}
Now compare the two solutions. From \eqref{eq:finite_system},
\begin{equation} \label{eq:u-diff}
  r - \hat r = \kappa\tau_m r\hat r (t_{11}-\hat t_{11}), \qquad 
  s - \hat s = \kappa\tau_k s\hat s (t_{22}-\hat t_{22}). 
\end{equation}
Also, since \(T=B^{-1}\) and \(\hat T= \hat B^{-1}\), we get
  \[T-\hat T = T\bigl(\hat B - B\bigr)\hat T = \kappa T \begin{pmatrix} r-\hat r & 0\\ 0 & s-\hat s \end{pmatrix} \hat T. \]
Hence, looking at the diagonal elements
\begin{align}
  t_{11}-\hat t_{11} &= \kappa\Bigl( t_{11}\hat t_{11}(r-\hat r)+t_{12}\hat t_{21}(s-\hat s) \Bigr), \label{eq:q11-diff} \\
  t_{22}-\hat t_{22} &= \kappa\Bigl(t_{21}\hat t_{12}(r-\hat r)+t_{22}\hat t_{22}(s-\hat s) \Bigr).\label{eq:q22-diff}
\end{align}
Substituting \eqref{eq:q11-diff} and \eqref{eq:q22-diff} into
\eqref{eq:u-diff}, and taking absolute values, we get component-wise $d \leq Ad$, where 
\[d :=  \begin{pmatrix} 
          |r-\hat r| \\ |s-\hat s| 
        \end{pmatrix} 
  \quad \text{ and } \quad 
  A:= \kappa^2 
  \begin{pmatrix} 
    \tau_m |r||\hat r||t_{11}||\hat t_{11}| & \tau_m |r||\hat r||t_{12}||\hat t_{21}| 
    \\[0.3em] 
    \tau_k |s||\hat s||t_{21}||\hat t_{12}| & \tau_k |s||\hat s||t_{22}||\hat t_{22}|
  \end{pmatrix}.\]
Now consider a vector given by
\[x:= \begin{pmatrix}
\sqrt{\Im r\,\Im \hat r}\\
\sqrt{\Im s\,\Im \hat s}
\end{pmatrix}.\]
By Cauchy-Schwarz, \eqref{eq:Imq11}, and \eqref{eq:Imu-identity}, we see that 
\begin{align*}
(Ax)_1 & \le \kappa^2\tau_m |r||\hat r| \sqrt{\bigl(|t_{11}|^2\Im r+|t_{12}|^2\Im s\bigr)
      \bigl(|\hat t_{11}|^2\Im\hat r+|\hat t_{21}|^2\Im\hat s\bigr)} \\
       & = |r||\hat r|\kappa\tau_m\sqrt{\Im t_{11}\,\Im\hat t_{11}} \\
       & = \sqrt{\bigl(\Im r-|r|^2\Im z\bigr)\bigl(\Im\hat r-|\hat r|^2\Im z\bigr)} < \sqrt{\Im r\,\Im\hat r} = x_1.
\end{align*}
Analogously \((Ax)_2<x_2\). Thus \(Ax<x\) component-wise. Since $A$ has non-negative entries and $x$ has strictly positive entries, we see that $\rho(A)< 1$, which contradicts $Ad \geq d$. 

Alternatively, note that $A^px \to 0$ as $p \to \infty$. Since entries of $x$ are strictly positive, we see that $A^p \to 0$ as $p \to\infty$. However, this means that 
\begin{equation*}
   0 \leq  d \leq A^p d \to 0 \text{ as } p \to \infty.
\end{equation*}
This implies that $d = 0$, and hence $r = \hat r$ and $s = \hat s$, which proves the uniqueness.
\end{proof}

\smallskip

\begin{proof}[Proof of existence in Lemma~\ref{lem:convergence-to-finite-system} via convergence]
We will now show that for all \(z\in\CC_+\), the system \eqref{eq:finite_system} has at least one solution $(r(z), s(z), T(z))$ with $r(z), s(z) \in \CC_+$. Moreover, if $\eta > 0$, then on the family $\bigl( r_n(z), s_n(z), T_n(z)\bigr)$ converges uniformly on $\bar \Omega_{\eta}$ and the limit solves \eqref{eq:finite_system}. Fix \(z\in\CC_+\) and recall that by Lemma~\ref{lem:exptGNCCA},
\begin{equation*}
     r_n(z)=\frac{1}{m} \, \EE\, \Tr \, \calG^{(11)}(z) \text{ and } s_n(z)=\frac{1}{k} \EE\,\Tr \, \calG^{(22)}(z).
\end{equation*}
Since the upper-left \((m+k)\times(m+k)\) block of \(\calG (z)\) is
\((W-zI)^{-1}\), both \( r_n(z)\) and \( s_n(z)\) are Stieltjes transforms
of probability measures. In particular,
  \[ r_n(z), s_n(z)\in\CC_+ \quad \text{ and } \quad 
    | r_n(z)|,\,| s_n(z)|\le (\Im z)^{-1}.\]
Also, by Lemma~\ref{lem:GNOpNormCCA} and triangle inequality, for each $\eta > 0$, 
  \[\sup_{n\ge 1}\sup_{z\in \bar \Omega_{\eta}}\| T_n(z)\|_{\op} \leq \sup_{n\ge 1}\sup_{z\in  \bar \Omega_{\eta}}\| \EE \calG (z)\|_{\op} <O_\eta(1).\]
Thus the families \( r_n(z), s_n(z), T_n(z)\) are locally bounded and analytic on \(\CC_+\), so by Montel's theorem (See \cite{SteinBook} for example) every subsequence admits a further subsequence converging locally uniformly on \(\CC_+\):
  \[ r_{n_j}\to r_*, \qquad  s_{n_j}\to s_*, \qquad  T_{n_j}\to T_*.\]

Now pass to the limit in the approximate self-consistent equations from
Proposition~\ref{prop:approxSelfEqCCA}. Since the errors are \(O(n^{-1 + \varepsilon})\), we obtain that for every $z \in \CC_+$
  \[r_*(z)=\frac{1}{-z-\kappa\tau_m t_{*,11}(z)},
    \quad \text{ and } \quad 
    s_*(z)=\frac{1}{-z-\kappa\tau_k t_{*,22}(z)},\]
and
  \[T_*(z)
    \begin{pmatrix}
    \alpha \kappa -\kappa r_*(z) & \alpha\\
    \alpha  & \alpha \kappa -\kappa s_*(z)
    \end{pmatrix} = I_2.\]
Thus \((r_*(z),s_*(z),T_*(z))\) solves \eqref{eq:finite_system}. Let us now check $r_*(z), s_*(z) \in \CC_+$. Note that since
\(\Im  r_n(z)>  0\) and \(\Im  s_n(z)>0\) for all \(n\), the locally uniform limit satisfies
\begin{equation*}
    \Im r_*(z)\ge0 \quad \text{ and } \quad \Im s_*(z)\ge0
\end{equation*}
Since $z\in \CC_+$, \eqref{eq:Imq11} combined with \eqref{eq:Imu-identity} implies that $\Im r_*(z) > 0$ and $\Im s_*(z) > 0$. 

To finish, note that since all subsequences have further subsequence that converge to the unique solution of \eqref{def:det_system}, the sequence itself converges to that limit locally uniformly. This proves the claim.
\end{proof}

\begin{proof}[Proof of Proposition~\ref{prop:detStabilityCCA}]
For \(r,s\in\CC_+\), define a matrix valued functions
  \[ B(r,s):= \begin{pmatrix}
                \alpha \kappa -\kappa r & \alpha\\
                \alpha  & \alpha \kappa -\kappa s
              \end{pmatrix} \quad \text{ and } \quad T(r,s):=B(r,s)^{-1} =:\bigl(t_{ij}(r,s)\bigr)_{i,j=1}^2, \]
whenever \(B(r,s)\) is invertible. Then \eqref{eq:finite_system} is equivalent to
\begin{equation}\label{eq:fixed-point-T}
r(z)=\frac1{-z-\kappa\tau_m t_{11}(r(z),s(z))}, \qquad s(z)=\frac1{-z-\kappa\tau_k t_{22}(r(z),s(z))}.
\end{equation}
For any $z \in \CC_+$ define a map $F_z : \CC_+^2 \mapsto \CC_+^2$ by 
  \[F_z(r,s):=
    \begin{pmatrix}
     \dfrac{1}{-z-\kappa\tau_m t_{11}(r,s)} \\
     \dfrac{1}{-z-\kappa\tau_k t_{22}(r,s)}
    \end{pmatrix}.\]
Then the solution of \eqref{eq:finite_system} satisfies $F_z(r(z),s(z)) = (r(z),s(z))^\intercal$. Since \(T(r,s)=B(r,s)^{-1}\), applying \eqref{eq:derivativeInverse}, we get that
  \begin{gather*}
    \partial_r T(r,s)=\kappa\,T(r,s)E_{11}T(r,s), \qquad
    \partial_s T(r,s)=\kappa\,T(r,s)E_{22}T(r,s),
  \end{gather*}
where $E_{ij}$ denotes the matrix with only non-zero entry $1$ in the $i$th row and $j$th column. Hence, since $t_{12}(z) = t_{21}(z)$ we see that the Jacobian matrix is 
  \[D F_z(r(z),s(z))= \kappa^2
    \begin{pmatrix}
    \tau_m r(z)^2 t_{11}(z)^2 & \tau_m r(z)^2 t_{12}(z)^2\\
    \tau_k s(z)^2 t_{21}(z)^2 & \tau_k s(z)^2 t_{22}(z)^2
    \end{pmatrix}.\]
Taking absolute values,  we get
  \[|DF_z(r(z),s(z))| := \kappa^2
    \begin{pmatrix}
    \tau_m |r(z)|^2 |t_{11}(z)|^2 & \tau_m |r(z)|^2 |t_{12}(z)|^2\\
    \tau_k |s(z)|^2 |t_{21}(z)|^2 & \tau_k |s(z)|^2 |t_{22}(z)|^2
    \end{pmatrix}.\]
Now, at the exact solution $(r(z),s(z)),$ we can apply equations \eqref{eq:Imq11} and \eqref{eq:Imu-identity}, to see
  \[|DF_z(r(z),s(z))|
    \begin{pmatrix}\Im r(z)\\ \Im s(z)\end{pmatrix} =
    \begin{pmatrix}
    \Im r(z)-|r(z)|^2\Im z\\
    \Im s(z)-|s(z)|^2\Im z
    \end{pmatrix}
    < \begin{pmatrix}\Im r(z)\\ \Im s(z)\end{pmatrix}. \]
Hence, this implies that the spectral radius of $DF_z$ at $(r(z),s(z))$ is bounded
\begin{equation*}
    \rho(DF_z(r(z),s(z)))\le \rho(|DF_z(r(z),s(z))|) < 1.
\end{equation*}
Consequently, if  $H_z(r,s) = (r,s)^\intercal - F_z(r,s)$ then $DH_z(r(z),s(z))$ is invertible for all $z \in \CC_+$. In fact, by continuity of $DH_z(r(z), s(z))$ and compactness of $\bar\Omega_{\eta}$, it follows that
  \[\sup_{z\in \bar \Omega_\eta } \|DH_z(r(z),s(z))^{-1}\|_{\op} \le O_\eta(1).\]
By the complex inverse function theorem, for every $\eta > 0$, there exists $\delta >0$ such that if $z \in \bar{\Omega}_\eta$ and $\|(r',s')-(r(z),s(z))\| \le \delta$, then 
\begin{equation}\label{eq:local-inverse-bound}
    \|(r',s')-(r(z),s(z))\| \le O_\eta(\norm {H_z(r' ,s')}).
\end{equation}
We will use this to tightly control errors. Let us now rewrite the approximate equations from Proposition~\ref{prop:approxSelfEqCCA} as follows
\begin{gather}
 r_n(z)
=
\frac{1+\epsilon_{1,n}(z)}{-z-\kappa\tau_m  t_{11,n}(z)},
\label{eq:approx-u-resid}
\\
 s_n(z)
=
\frac{1+\epsilon_{2,n}(z)}{-z-\kappa\tau_k  t_{22,n}(z)},
\label{eq:approx-v-resid}
\\
 T_n(z) B_n(z)
=
I_2+E_n(z),
\label{eq:approx-T-resid}
\end{gather}
where, $B_n(z):=B( r_n(z), s_n(z))$ and the error terms satisfy
\[
\sup_{z\in \Omega_\eta }
\bigl(|\epsilon_{1,n}(z)|+|\epsilon_{2,n}(z)|+\|E_n(z)\|_{\op}\bigr)
\le C n^{-1 + \varepsilon}.
\]
For \(n\) large enough, \(\|E_n(z)\|_{\op}\le 1/2\) uniformly on \(\Omega_\eta \), so
\(I_2+E_n(z)\) is invertible and \eqref{eq:approx-T-resid} implies that \(B_n(z)\) is invertible. Define $\hat T_n(z):=T(r_n(z),s_n(z)) = B_n(z)^{-1}$ and note that 
\[
\hat T_n(z)=(I_2+E_n(z))^{-1} T_n(z),
\]
and since \(\norm{ T_n(z)}_\op \leq O(1)\) uniformly on \(\Omega_\eta \), we get that for all large $n$,  
\begin{equation}\label{eq:Tsharp-close}
\sup_{z\in \Omega_\eta }
\|\hat T_n(z)- T_n(z)\|_{\op}
\le O(n^{-1 + \varepsilon }).
\end{equation}
Next, since \(| r_n(z)|\le \eta^{-1}\), from \eqref{eq:approx-u-resid} we get that for all large $n$
\[
\left|
 r_n(z)-\frac1{-z-\kappa\tau_m t_{11,n}(z)}
\right|
= \abs{\frac{{ r_n(z) \epsilon_{1,n} (z)}}{1 + \epsilon_{1,n}(z)}}
\le O(n^{-1+\varepsilon}),
\]
uniformly on \(\Omega_\eta \). Now we want to replace $t_{11,n}(z)$ by $\hat t_{11,n}(z) = t_{11}(r_n(z), s_n(z))$. To see why we can do this, since $|r_n(z)| \leq \eta^{-1} = O(1)$, 
\begin{equation*}
    |-z - \kappa \tau_m t_{11,n}(z)| \geq \Omega(1).
\end{equation*}
Using \eqref{eq:Tsharp-close} and the fact that the $x \mapsto 1/x$ map is $O(1)$-Lipschitz on $\Omega_\eta $, we get that 
\[
\left|
 r_n(z)-\frac1{-z-\kappa\tau_m t_{11}\bigl( r_n(z), s_n(z)\bigr)}
\right|
\le O(n^{-1 + \varepsilon}).
\]
The same argument gives
\[
\left|
 s_n(z)-\frac1{-z-\kappa\tau_k t_{22}\bigl( r_n(z), s_n(z)\bigr)}
\right|
\le O(n^{-1 + \varepsilon}).
\]
In other words, we get that
\begin{equation}\label{eq:Phi-residual}
\sup_{z\in \bar \Omega_\eta }
\|H_z( r_n(z), s_n(z))\|
\le O(n^{-1 + \varepsilon}).
\end{equation}
In the argument above, we showed that \(( r_n(z), s_n(z))\to (r(z),s(z))\) locally uniformly on $\CC_+$. Hence, for all $\eta, \delta > 0$ and sufficiently large $n$, 
\[
\sup_{z\in \bar \Omega_{\eta}}
\|( r_n(z), s_n(z))-(r(z),s(z))\|<\delta,
\]
so \eqref{eq:local-inverse-bound} applies at every \(z\in \bar \Omega_{\eta}\). Combining this with
\eqref{eq:Phi-residual}, we obtain
\[
\sup_{z\in \bar \Omega_{\eta}}
\Big(
| r_n(z)-r(z)|+| s_n(z)-s(z)|
\Big)
\le O(n^{-1 + \varepsilon}).
\]
Finally, we bound $\| T_n(z) - T(z)\|_{\op}$ as follows.
\[
 T_n(z)-T(z)
=
\sqb{ T_n(z)-\hat T_n{(z)}}
+
\sqb{T( r_n(z), s_n(z))-T(r(z),s(z))}.
\]
The first term is \(O(n^{-1 + \varepsilon})\) uniformly by \eqref{eq:Tsharp-close}; the second is
\(O(n^{-1 + \varepsilon})\) uniformly on \(\bar \Omega_{\eta}\), since \(T(r,s)=B(r,s)^{-1}\) is analytic in a neighborhood of the compact set
\(\{(r(z),s(z)):z\in \bar \Omega_{\eta}\}\). Therefore, we conclude that
\[
\sup_{z\in \Omega_\eta }
\| T_n(z)-T(z)\|_{\op}
\le O(n^{-1+ \varepsilon})
\] 
This proves the proposition.
\end{proof}

\begin{lemma}[Includes Lemma~\ref{lem:extended-solution-of-finite-systemMain}]\label{lem:extended-solution-of-finite-system}
Let $(r,s,T)$ be the unique solution of \eqref{eq:finite_system} on $\CC_+$. Define
\[
\lambda_*:=\frac{1-\kappa^2}{\kappa},
\qquad
\DD:=\CC\setminus(-\infty,\lambda_*].
\]
Then $r$, $s$, and each entry of $T$ extend analytically to $\DD$ and still satisfy $\eqref{eq:finite_system}$ on whole $\DD$. 
Moreover, there exists probability measures $\mu$ and $\nu$ supported on $(-\infty, \lambda_*]$ such that
\begin{equation*}
    r(z) = \int_\RR \frac{\mu(dx)}{x - z} \quad \text{ and } \quad s(z) = \int_\RR \frac{\nu(dx)}{x - z}.
\end{equation*}
\end{lemma}

\begin{proof}[Proof of Lemma~\ref{lem:extended-solution-of-finite-system}, Lemma~\ref{lem:extended-solution-of-finite-systemMain}]
Recall that both $r_n$ and $s_n$ are Stieltjes transforms of probability measures. That is, there are probability measures $\mu_n,\nu_n$ on $\RR$ such that
\[
r_n(z)=\int_{\RR}\frac{1}{x-z}\,\mu_n(dx),
\qquad
s_n(z)=\int_{\RR}\frac{1}{x-z}\,\nu_n(dx),
\qquad z\in\CC_+.
\]
By Lemma~\ref{lem:convergence-to-finite-system}, we have $r_n \to r$ and $s_n \to s$ locally uniformly on $\CC_+$. We would like to apply to Stieltjes continuity theorem. To do so, recall that functions $r$ and $s$ are analytic on $\CC_+$ and  $\Im r(z), \Im s(z) >0$ by definition, so we need to check the normalization condition.
\begin{equation*}
    \lim_{\eta \to \infty} i \eta s(i\eta) = \lim_{\eta \to \infty} i \eta r(i\eta) = -1.
\end{equation*}
To this end, let us set
\[
B(r,s):=
\begin{pmatrix}
\alpha\kappa-\kappa r & \alpha\\
\alpha & \alpha\kappa-\kappa s
\end{pmatrix},
\qquad
d(r,s):=1-\kappa(r+s)-(1-\kappa^2)rs.
\]
Using $\alpha=\kappa/(\kappa^2-1)$, a direct computation gives
\[
B(r,s)^{-1}
=
\frac{1}{\kappa\, d(r,s)}
\begin{pmatrix}
\kappa+(1-\kappa^2)s & -1\\
-1 & \kappa+(1-\kappa^2)r
\end{pmatrix}.
\]
Hence, on $\CC_+$, the system \eqref{eq:finite_system} is equivalent to
\begin{gather}\label{eq:poly-r-s}
(1+zr)\,d(r,s)+\tau_m \bigl(\kappa+(1-\kappa^2)s\bigr) r=0, \\
(1+zs)\,d(r,s)+\tau_k \bigl(\kappa+(1-\kappa^2)r\bigr)s=0.
\end{gather}
Recall that $|r(i\eta)| \leq O(\eta^{-1})$ and $|s(i\eta)| \leq O(\eta^{-1})$ are inherited from $r_n$ and $s_n$, so $d(r(i\eta), s(i\eta)) \to 1$ as $\eta \to \infty$, but then from \eqref{eq:poly-r-s}, we see that 
\begin{equation*}
    1+ i\eta r(i\eta) \to 0 \quad \text{ and } \quad  1+ i\eta s(i\eta) \to 0.
\end{equation*}
Hence by \cite{GeronimoHill2003} and the Stieltjes continuity theorem, there exist probability measures $\mu,\nu$ such that $\mu_n\Rightarrow \mu$,
$\nu_n\Rightarrow \nu$, and
\[
r(z)=\int_{\RR}\frac{1}{x-z}\,\mu(dx),
\qquad
s(z)=\int_{\RR}\frac{1}{x-z}\,\nu(dx),
\qquad z\in\CC_+.
\]

We now show that $ \mathrm{supp}\, \mu, \mathrm{supp} \, \nu \subset (-\infty,\lambda_*]$. Fix $\varepsilon>0$. Since $\mu_n$ and $\nu_n$ are expectations of random spectral measures supported on $\sigma(W)$, their mass above
$\lambda_*+\varepsilon$ is bounded by the probability that $W$ has an
eigenvalue there:
\[
\mu_n([\lambda_*+\varepsilon,\infty))
\le \PP\!\bigl(\lambda_{\max}(W)\ge \lambda_*+\varepsilon\bigr) \to 0
\]
by Theorem~\ref{thm:CCAGordonEdgeAppendix} and similarly $\nu_n([\lambda_*+\varepsilon,\infty)) \to 0.$ Passing to the weak limit and recalling that $\varepsilon$ was arbitrary, it follows that 
\[
 \mathrm{supp}\, \mu, \mathrm{supp} \, \nu \subset (-\infty,\lambda_*].
\]
Therefore the Stieltjes transforms above extend analytically to $\DD$. We continue to denote these analytic extensions by $r$ and $s$. 

Now let us go back to \eqref{eq:poly-r-s}. Both left-hand sides are analytic on $\DD$ so \eqref{eq:poly-r-s} hold on all of $\DD$. We claim that $d(r(z),s(z))\neq 0$ for all $z\in \DD.$ Indeed, suppose that $d(r(z_0),s(z_0))=0$ for some $z_0\in \DD$. Then
\eqref{eq:poly-r-s}  implies that
\[
r(z_0)\bigl(\kappa+(1-\kappa^2)s(z_0)\bigr)=0,
\qquad
s(z_0)\bigl(\kappa+(1-\kappa^2)r(z_0)\bigr)=0.
\]
If $r(z_0)=0$, then the second identity gives $s(z_0)=0$, and hence
$d(r(z_0),s(z_0))=1$, a contradiction. Thus
\[
\kappa+(1-\kappa^2)s(z_0)=0 \quad  \text{ and } \quad \kappa+(1-\kappa^2)r(z_0)=0.
\]
Therefore, we deduce that $r(z_0)=s(z_0)=-\kappa / (1-\kappa^2).$ However, then
\[
d(r(z_0),s(z_0))
=
1-\kappa\Bigl(-\frac{2\kappa}{1-\kappa^2}\Bigr)
-(1-\kappa^2)\frac{\kappa^2}{(1-\kappa^2)^2}
=
\frac{1}{1-\kappa^2}\neq 0,
\]
again a contradiction. This proves the claim. We may therefore define, for $z\in \DD$,
\[
\tilde T(z)
:=
\frac{1}{\kappa\, d(r(z),s(z))}
\begin{pmatrix}
\kappa+(1-\kappa^2)s(z) & -1\\
-1 & \kappa+(1-\kappa^2)r(z)
\end{pmatrix}.
\]
Since $d(r(z),s(z))\neq 0$ on $\DD$, the matrix $\tilde T$ is analytic on
$\DD$. On $\CC_+$ it agrees with the original $T$, so $\tilde T$ is the
analytic continuation of $T$. Moreover, by construction,
\[
\tilde T(z)
\begin{pmatrix}
\alpha\kappa-\kappa r(z) & \alpha\\
\alpha & \alpha\kappa-\kappa s(z)
\end{pmatrix}
=I_2,
\qquad z\in D.
\]
Also, using
\[
\kappa \tilde t_{11}(z)
=
\frac{\kappa+(1-\kappa^2)s(z)}{d(r(z),s(z))},
\qquad
\kappa \tilde t_{22}(z)
=
\frac{\kappa+(1-\kappa^2)r(z)}{d(r(z),s(z))},
\]
the identities \eqref{eq:poly-r-s}  become
\[
1+zr+\kappa\tau_m r\,\tilde t_{11}=0,
\qquad
1+zs+\kappa\tau_k s\,\tilde t_{22}=0.
\]
Since $d(r(z),s(z))\neq 0$, these identities imply in that
$r(z)\neq 0$ and $s(z)\neq 0$, hence
\[
r(z)=\frac{1}{-z-\kappa\tau_m \tilde t_{11}(z)},
\qquad
s(z)=\frac{1}{-z-\kappa\tau_k \tilde t_{22}(z)}.
\]
Thus $(r,s,\tilde T)$ solves \eqref{eq:finite_system} on $\DD$.
Renaming $\tilde T$ as $T$ completes the proof.
\end{proof}

Finally, we give the proof of Theorem~\ref{thm:real-stieltjes-transform-CCA} using all the ingredients that we have now. 

\begin{proof}[Proof of Theorem~\ref{thm:real-stieltjes-transform-CCA}, Theorem~\ref{thm:real-stieltjes-transform-CCAMain}]
    Fix a compact set $K \subset (\lambda_*, \infty)$ and fix small $\delta < d(\lambda_*, K)$. We also fix $\eta < \delta /4$, but we will eventually take this to $0$.
    
    As a first, step note that if we apply the deterministic standard basis vectors to Theorem~\ref{thm:resolvConvCCA} with parameter $\eta / 2$ and take union bound, we get that 
    \begin{equation*}
        \sup_{z \in K} \frac{1}{n}\abs{\Tr \,\calG^{(ij)}(z+ i \eta) - \Tr\, \calM^{(ij)}(z+i\eta)} \leq C_\eta n^{-1/4},
    \end{equation*}
    with probability of failure at most $n^{-p}$, whenever $n > n_0(\eta, p)$ for any $p > 0$. Our goal is to remove $i \eta$ from both of the arguments. On one hand, note that 
    \begin{align*}
        \|\calG(z + i\eta) - \calG(z)\|_{\op} &= \|\calG(z + i\eta)\sqb{\calL(z) - \calL(z + i\eta)}\calG(z)\|_\op \\
        & \leq \norm{\calG(z)}_\op \norm{\calG(z + i\eta)}_\op \eta
    \end{align*}
    By Theorem~\ref{thm:CCAGordonEdgeAppendix}, we know that $\lambda_{\max}(W) \leq \lambda_* + \delta/4$ with overwhelming probability. By the proof of Lemma~\ref{lem:GNOpNormCCA}, we get that with overwhelming probability 
    \begin{equation*}
        \norm{\calG(z)}_\op \leq O(1) \text{ and }\norm{\calG(z+i\eta)}_\op \leq O(1).
    \end{equation*}
    In particular, this means that $\|\calG(z + i\eta) - \calG(z)\|_{\op} \leq O(\eta)$, which in turn implies that 
    \begin{equation*}
        \frac{1}{n}\abs{\Tr \,\calG^{(ij)}(z+ i \eta) -  \Tr \,\calG^{(ij)}(z)} \leq O(\eta). 
    \end{equation*}
    On the other hand, recall that by Lemma~\ref{lem:extended-solution-of-finite-system}, entries of $\calM^{(ij)}$ are analytic on the open neighborhood of $K$. This means that we have 
    \begin{equation*}
        \frac{1}{n}\abs{\Tr\, \calM^{(ij)}(z+i\eta) - \Tr\, \calM^{(ij)}(z)} \leq O(\eta).
    \end{equation*}
    Hence, we see that by triangle inequality, for any $p >0$ with probability $1 - n^{-p}$ we have
    \begin{equation*}
        \sup_{z \in K}\frac{1}{n}\abs{\Tr\, \calG^{(ij)}(z) - \Tr\, \calM^{(ij)}(z)} \leq O(\eta) + C_\eta n^{-1/4},
    \end{equation*}
    whenever $n > n_0(\eta,p)$. That means that for any fixed $\varepsilon > 0$, we can take $\eta$ small enough so that the difference of normalized traces is bounded by $\varepsilon$ with overwhelming probability. This concludes the proof.
\end{proof}

\subsection{Correlated spiked Wigner model}\label{subsec:Wigner-convergence}

Throughout this subsection we work under the null model from Section~\ref{subsec:wigner}. That is, $U, V\sim \mathrm{GOE}(1/n)$ are independent. We keep the notation and write
\begin{equation}\label{eq:kappaDefWignerSectionSomething}
\kappa^4 = \frac{(1-\alpha^2)_+(1-\beta^2)_+}{\alpha^2\beta^2},
\qquad
\widetilde U := \alpha U -\alpha^2 I_n,
\qquad
\widetilde V := \beta V -\beta^2 I_n.
\end{equation}
The corresponding null spectral matrix is
\[
W :=
\begin{pmatrix}
\widetilde U & \kappa \widetilde V \\
\kappa \widetilde U & \widetilde V
\end{pmatrix}.
\]
For convenience, let us introduce vectors and matrices 
\begin{equation}\label{eq:defPQWigner}
A := \begin{pmatrix}1&\kappa\\ \kappa&1\end{pmatrix},
\qquad
p:=A^{1/2}e_1,
\qquad
q:=A^{1/2}e_2,
\qquad
P:=pp^\intercal,
\qquad
Q:=qq^\intercal.
\end{equation}
Note that $\norm{p} = \norm{q} = 1$. Then $W$ is similar to the symmetric matrix
\begin{align}\label{eq:wigner-symmetrized-null}
H 
&:=(A^{1/2}\otimes I_n)
\begin{pmatrix}
\widetilde U&0\\0&\widetilde V
\end{pmatrix}
(A^{1/2}\otimes I_n) \\
&=
-\alpha^2 P\otimes I_n
-\beta^2 Q\otimes I_n
+\alpha P\otimes U
+\beta Q\otimes V.
\end{align}
Hence $W$ and $H$ have the same eigenvalues. In the rest of this subsection we therefore study the resolvent 
\[
\calG(z):=(H-zI_{2n})^{-1}
\qquad \text{ for } z\in\CC_+,
\]
and we continue to use the spectral domain $\Omega_\eta$ from \eqref{eq:OmegaDef}.

\begin{definition}\label{def:det_system_Wigner}
For each $z\in\CC_+$, define scalars $r(z),s(z)\in \CC_+$ and a matrix
$T(z)\in \CC^{2\times 2}$ as the unique solution to the coupled system
\begin{gather}
\begin{aligned}
r(z)&=p^\intercal T(z)p,\\
s(z)&=q^\intercal T(z)q,
\end{aligned}
\label{eq:wigner-finite-system-scalars}
\\
T(z)\Bigl(-zI_2-\alpha^2(1+r(z))P-\beta^2(1+s(z))Q\Bigr)=I_2.
\label{eq:wigner-finite-system}
\end{gather}
Let $m(z):=\frac12\Tr T(z)$. Then $m$ is the Stieltjes transform of a compactly supported
probability measure $\mu$ on $\RR$. We write
\[
\lambda_*:=1,
\qquad
\DD:=\CC\setminus \mathrm{supp}(\mu),
\qquad
\calM(z):=T(z)\otimes I_n,
\qquad z\in \DD.
\]
\end{definition}

\begin{theorem}\label{thm:resolvConvWigner}
Let $\calG$ be the resolvent of $H$ and let $\calM$ be as in
Definition~\ref{def:det_system_Wigner}. Fix $\eta,\varepsilon>0$. Then, for any unit vectors
$x,y\in \RR^{2n}$,
\begin{equation}\label{eq:wigner-resolv-conv}
\sup_{z\in\Omega_\eta}
\abs{x^\intercal\calG(z)y-x^\intercal\calM(z)y}
\le O(n^{-1/2+\varepsilon})
\end{equation}
with overwhelming probability, where the hidden constant depends only on $\eta$. 
\end{theorem}

Write $\calG(z)$ in $n\times n$ blocks as follows
\[
\calG(z)=
\begin{pmatrix}
\calG^{(11)}(z) & \calG^{(12)}(z)\\
\calG^{(21)}(z) & \calG^{(22)}(z)
\end{pmatrix},
\qquad
\calG^{(ij)}(z)\in \RR^{n\times n}.
\]
This matches the block decomposition of $H$ and of $\calM(z)=T(z)\otimes I_n$.

\begin{theorem}\label{thm:real-stieltjes-transform-Wigner}
Let $\calG$ be as above and let $\calM$ and $\lambda_*$ be as in
Definition~\ref{def:det_system_Wigner}. Then, for any compact set $K\subset (\lambda_*,\infty)$, we have
\[
\sup_{z\in K}
\frac1n\abs{\Tr\,\calG^{(ij)}(z)-\Tr\,\calM^{(ij)}(z)}=o(1)
\]
with overwhelming probability.
\end{theorem}

\subsubsection{Concentration of the resolvent}

As a first step, we will prove the concentration of anisotropic resolvent forms on $\Omega_\eta$. 
\begin{prop}\label{prop:Wigner-concentration}
Let $\eta,\varepsilon>0$. Then for any unit vectors $x,y\in\RR^{2n}$,
\[
\sup_{z\in\Omega_\eta}
\abs{x^\intercal\calG(z)y-x^\intercal\EE\calG(z)y}
\le O(n^{-1/2+\varepsilon})
\]
with overwhelming probability, where the hidden constant depends only on $\alpha,\beta$ and $\eta$.
\end{prop}

We begin with the deterministic resolvent bounds that hold throughout the spectral domain.

\begin{lemma}\label{lem:GNOpNormWigner}
For every $\eta>0$, we have that 
\[
\sup_{z\in\Omega_\eta}\norm{\calG(z)}_{\op}\le \eta^{-1},
\qquad
\sup_{z\in\Omega_\eta}\norm{\pp{\calG(z)}{z}}_{\op}\le \eta^{-2}.
\]
\end{lemma}
\begin{proof}
Since $H$ is real symmetric, the spectral theorem gives
\[
\norm{\calG(z)}_{\op}=\norm{(H-zI_{2n})^{-1}}_{\op}\le (\Im z)^{-1}
\]
for every $z\in\CC_+$. On $\Omega_\eta$ we have $\Im z\ge \eta$, proving the first bound. The second follows from \eqref{eq:derivativeInverse}, since
\[
\pp{}{z}\calG(z)=\calG(z)^2,
\qquad
\norm{\calG(z)^2}_{\op}\le \norm{\calG(z)}_{\op}^2\le \eta^{-2}.
\]
\end{proof}

Next we show that the quadratic form of the resolvent is Lipschitz in the Gaussian disorder.

\begin{lemma}\label{lem:lipWigner}
Fix $x,y\in\CC^{2n}$ with $\norm{x}=\norm{y}=1$ and $z\in\CC_+$. Then the gradient of the smooth map $(U,V) \mapsto x^\intercal \calG(z)y$ satisfies 
\[
\norm{\grad_{U,V}F}^2\le 4(\alpha^2+\beta^2)\norm{\calG(z)}_{\op}^4.
\]
In particular by Lemma~\ref{lem:GNOpNormWigner}, for $z\in\Omega_\eta$, the map $(U,V)\mapsto x^\intercal\calG(z)y$ is globally $O_\eta(1)$-Lipschitz.
\end{lemma}
\begin{proof}
Fix $z\in\CC_+$ and write $\xi:=\calG(z)x$ and $\zeta:=\calG(z)y$. Partition these vectors according to the $2\times2$ block structure of $H$. That is we write $\xi=(\xi^{(1)},\xi^{(2)}) \in \CC^n \times \CC^n$ and similarly $\zeta=(\zeta^{(1)},\zeta^{(2)}) \in \CC^n \times \CC^n$. Define
\[
a:=(p^\intercal\otimes I_n)\xi=p_1\xi^{(1)}+p_2\xi^{(2)},
\qquad
b:=(p^\intercal\otimes I_n)\zeta=p_1\zeta^{(1)}+p_2\zeta^{(2)},
\]
For $1\le i<j\le n$, let $S_{ij}:=e_ie_j^\intercal+e_je_i^\intercal$, and for $1\le i\le n$, let $S_{ii}:=e_ie_i^\intercal$. Since
\[
H=-\alpha^2P\otimes I_n-\beta^2Q\otimes I_n+\alpha P\otimes U+\beta Q\otimes V,
\]
differentiating with respect to $U_{ij}$ for $i \leq j$  gives $\partial H/\partial U_{ij} = \alpha P \otimes S_{ij}$. Therefore, by \eqref{eq:derivativeInverse}, we can write 
\[
\pp{F}{U_{ij}}
=-x^\intercal\calG(z)\pp{H}{U_{ij}}\calG(z)y
=-\alpha\,a^\intercal S_{ij}b.
\]
Therefore, we can write 
\[
\pp{F}{U_{ij}}=-\alpha(a_i b_j+a_j b_i) \quad \text{ for } i < j \quad \text{ and }\quad  \pp{F}{U_{ii}}=-\alpha a_i b_i \quad \text{ for } i = j.
\]
Squaring and summing up over the indices $i$ and $j$, we get that
\begin{equation*}
\sum_{1\le i\le j\le n}\abs{\pp{F}{U_{ij}}}^2
\le \alpha^2\norm{ab^\intercal+ba^\intercal}_F^2
\le 4\alpha^2\norm{a}^2\norm{b}^2,
\end{equation*}
Since $\norm{p}=\norm{q}=1$, by Cauchy-Schwarz $\norm{a}\le \norm{\xi}, \norm{b} \le \norm{\zeta}$. Hence, if we perform the same computation for $\partial F/ \partial V$ we conclude that 
\[
\norm{\grad_{U,V}F}^2
\le 4(\alpha^2+\beta^2)\norm{\xi}^2\norm{\zeta}^2
\le 4(\alpha^2+\beta^2)\norm{\calG(z)}_{\op}^4,
\]
as claimed.
\end{proof}

We can now prove the concentration proposition.

\begin{proof}[Proof of Proposition~\ref{prop:Wigner-concentration}]
Fix unit vectors $x,y\in\RR^{2n}$ and define
\[
F(z):=x^\intercal\calG(z)y,
\qquad z\in\Omega_\eta.
\]
Throughout the proof, $C$ denotes a positive constant depending only on $\alpha,\beta$ and $\eta$ and it may change from line to line. By Lemma~\ref{lem:lipWigner}, $\norm{\grad_{U,V}F(z)}
\le C\norm{\calG(z)}_{\op}^2$. On the other hand, by Lemma~\ref{lem:GNOpNormWigner}, $\norm{\calG(z)}_{\op}\le \eta^{-1}$ for every $z\in\Omega_\eta$, deterministically. Therefore, both the real and imaginary parts of $F(z)$ are globally $C$-Lipschitz. Applying Lemma~\ref{lem:gaussianConLip} to $\Re F(z)$ and $\Im F(z)$ separately and combining the two tail bounds, we obtain that for every fixed $z\in\Omega_\eta$,
\[
\PP\paren{\abs{F(z)-\EE F(z)}\ge t}
\le 4\exp\paren{-c n t^2}
\]
for some constant $C>0$. Taking $t=n^{-1/2+\varepsilon}$, we conclude that
\begin{equation}\label{eq:FBoundWigner}
\abs{F(z)-\EE F(z)}\le O(n^{-1/2+\varepsilon})
\end{equation}
with overwhelming probability, for every fixed $z\in\Omega_\eta$. It remains to make the estimate uniform in $z$. By Lemma~\ref{lem:GNOpNormWigner} and \eqref{eq:derivativeInverse},
\[
\pp{}{z}F(z)=x^\intercal\calG(z)^2y,
\qquad
\sup_{z\in\Omega_\eta}\abs{\pp{F(z)}{z}}\le \eta^{-2}.
\]
The same bound holds for $\pp{}{z}\EE F(z)$ by taking expectations. Let $\cS$ be an $1/n$-net of $\Omega_\eta$ with cardinality $|\cS|=O_\eta(n^2)$. Applying \eqref{eq:FBoundWigner} to each point of $\cS$ and taking a union bound, we obtain that with overwhelming probability,
\[
\sup_{z\in\cS}\abs{F(z)-\EE F(z)}\le O(n^{-1/2+\varepsilon}).
\]
If $z\in\Omega_\eta$, choose $w\in\cS$ with $|z-w|\le n^{-1}$. Then
\[
\abs{F(z)-F(w)}+\abs{\EE F(z)-\EE F(w)}\le 2\eta^{-2}n^{-1}=o(n^{-1/2+\varepsilon}).
\]
Hence the same bound holds uniformly over $z\in\Omega_\eta$, proving the proposition.
\end{proof}

\subsubsection{Derivations of approximate equations}

In this subsection we identify the structure of $\EE\calG(z)$ and derive the finite size system satisfied by its coefficients. 

\begin{prop}\label{prop:approxSelfEqWigner}
For every $z\in\CC_+$, the expectation $\EE\calG(z)$ has the form
\begin{equation}\label{eq:WignerExpectedResolventForm}
\EE\calG(z)=T_n(z)\otimes I_n,
\end{equation}
where $T_n(z)\in\CC^{2\times 2}$ is symmetric matrix. Define $r_n(z):=p^\intercal T_n(z)p$ and $s_n(z):=q^\intercal T_n(z)q$. Then for every fixed $\eta,\varepsilon>0$, uniformly for $z\in\Omega_\eta$,
\begin{equation}\label{eq:WignerApproxSelfEq}
\Bigl(-zI_2-\alpha^2P-\beta^2Q-\alpha^2 r_n(z)P-\beta^2 s_n(z)Q\Bigr)T_n(z)
=I_2+O(n^{-1+\varepsilon})\cdot 11^\intercal.
\end{equation}
The hidden constants depend only on $\alpha,\beta,\eta$, and $\varepsilon$.
\end{prop}

\begin{lemma}\label{lem:exptGNWigner}
For every $z\in\CC_+$, the expectation $\EE\calG(z)$ has the form \eqref{eq:WignerExpectedResolventForm}, where $T_n(z)$ is symmetric.
\end{lemma}
\begin{proof}
Let $O\in O(n)$ and set $\calO:= I_2\otimes O$. Since $U$ and $V$ are independent GOE matrices, we have $(O^\intercal U O,\,O^\intercal V O)\overset{d}{=}(U,V).$ Recalling \eqref{eq:wigner-symmetrized-null}, it follows that distribution of $H$ is invariant under $\calO$, that is $\calO^\intercal H \calO\overset{d}{=}H.$ Taking inverse, we see
\[
\calO^\intercal\calG(z)\calO\overset{d}{=}\calG(z).
\]
Therefore $\calO \EE \calG(z) \calO = \EE \calG(z)$. Since $\EE \calG(z)$ is deterministic, we must have that 
\[
\EE\calG(z)=T_n(z)\otimes I_n
\]
for some $T_n(z)\in\CC^{2\times 2}$. Finally, since $H$ is symmetric, so is $\calG(z)$, and therefore $T_n(z)$ is symmetric as well.
\end{proof}

We next derive equations for the coefficients of $T_n(z)$ from the block identity
\begin{equation}\label{eq:WignerHGidentity}
(H-zI_{2n})\calG(z)=I_{2n}.
\end{equation}
For indices $b\in\{1,2\}$ and standard basis vector $e_b \in \CC^2$ we set 
\begin{align*}
A_p(z)&:=(p^\intercal\otimes I_n)\calG(z)(p\otimes I_n),
&
B_{pb}(z)&:=(p^\intercal\otimes I_n)\calG(z)(e_b\otimes I_n),\\
A_q(z)&:=(q^\intercal\otimes I_n)\calG(z)(q\otimes I_n),
&
B_{qb}(z)&:=(q^\intercal\otimes I_n)\calG(z)(e_b\otimes I_n).
\end{align*}
Then $A_p(z)$ and $A_q(z)$ are symmetric, because $\calG(z)^\intercal=\calG(z)$.

\begin{lemma}\label{lem:coeffIdentitiesWigner}
For every $b\in\{1,2\}$ and $z\in\CC_+$, we have 
\begin{align}
\frac1n\EE\Tr\bigl(U B_{pb}(z)\bigr)
&=-\alpha\,\EE\Bigl[\frac1n\Tr A_p(z)\cdot \frac1n\Tr B_{pb}(z)\Bigr]
-\frac{\alpha}{n^2}\EE\Tr\bigl(A_p(z)B_{pb}(z)\bigr),
\label{eq:WignerCoeffIdentityU}
\\
\frac1n\EE\Tr\bigl(V B_{qb}(z)\bigr)
&=-\beta\,\EE\Bigl[\frac1n\Tr A_q(z)\cdot \frac1n\Tr B_{qb}(z)\Bigr]
-\frac{\beta}{n^2}\EE\Tr\bigl(A_q(z)B_{qb}(z)\bigr).
\label{eq:WignerCoeffIdentityV}
\end{align}
\end{lemma}
\begin{proof}
We prove only \eqref{eq:WignerCoeffIdentityU} as the proof of \eqref{eq:WignerCoeffIdentityV} is identical. For convenience, let $g_{ij} \sim \calN(0,1)$ be independent standard Gaussians for $1 \leq i \leq j \leq n$. Write  
\[
U=\frac1{\sqrt n}\sum_{1\le i<j\le n} g_{ij}S_{ij}
+\sqrt{\frac2n}\sum_{i=1}^n g_{ii}S_{ii},
\]
where $S_{ij}:=e_ie_j^\intercal+e_je_i^\intercal$ for $i<j$ and  $S_{ii}:=e_ie_i^\intercal$. Therefore
\[
\frac1n\Tr\bigl(U B_{pb}\bigr)
=\frac1{n^{3/2}}\sum_{1\le i<j\le n} g_{ij}\Tr\bigl(S_{ij}B_{pb}\bigr)
+\frac{\sqrt2}{n^{3/2}}\sum_{i=1}^n g_{ii}\Tr\bigl(S_{ii}B_{pb}\bigr).
\]
Applying Gaussian integration by parts Lemma~\ref{lem:GaussianIBP} to each $g_{ij}$ and using \eqref{eq:derivativeInverse}, we get
\begin{align}
\frac1n\EE\Tr\bigl(U B_{pb}\bigr)
&=
\frac1{n^{3/2}}\sum_{1\le i<j\le n}\EE\,\pp{}{g_{ij}}\Tr\bigl(S_{ij}B_{pb}\bigr)
+\frac{\sqrt2}{n^{3/2}}\sum_{i=1}^n\EE\,\pp{}{g_{ii}}\Tr\bigl(S_{ii}B_{pb}\bigr).
\label{eq:WignerCoeffIdentityU-start}
\end{align}
Now note that we can write 
\[
\pp{H}{g_{ij}}=\frac{\alpha}{\sqrt n}P\otimes S_{ij}
\quad \text{ for } i<j \quad \text{ and } \quad 
\pp{H}{g_{ii}}=\alpha\sqrt{\frac2n}\,P\otimes S_{ii},
\]
so we can conveniently write the derivative   
\[
\pp{B_{pb}}{g_{ij}}
=-(p^\intercal\otimes I_n)\calG\pp{H}{g_{ij}}\calG(e_b\otimes I_n)
=-\frac{\alpha}{\sqrt n}A_p S_{ij}B_{pb}
\quad \text{ for } i<j,
\]
and similarly for diagonal terms 
\[
\pp{B_{pb}}{g_{ii}}=-\alpha\sqrt{\frac2n}\,A_pS_{ii}B_{pb}.
\]
Substituting this into \eqref{eq:WignerCoeffIdentityU-start}, we obtain
\[
\frac1n\EE\Tr\bigl(U B_{pb}\bigr)
=-\frac{\alpha}{n^2}\EE\Bigg[
\sum_{1\le i<j\le n}\Tr\bigl(S_{ij}A_pS_{ij}B_{pb}\bigr)
+2\sum_{i=1}^n\Tr\bigl(S_{ii}A_pS_{ii}B_{pb}\bigr)
\Bigg].
\]
Since $A_p$ is symmetric, a direct coordinate expansion gives
\[
\sum_{1\le i<j\le n}\Tr\bigl(S_{ij}A_pS_{ij}B_{pb}\bigr)
+2\sum_{i=1}^n\Tr\bigl(S_{ii}A_pS_{ii}B_{pb}\bigr)
=
\Tr A_p\,\Tr B_{pb}+\Tr\bigl(A_pB_{pb}\bigr).
\]
This is exactly \eqref{eq:WignerCoeffIdentityU}.
\end{proof}

As before, the identities in Lemma~\ref{lem:coeffIdentitiesWigner} split into a product of traces term and a trace of products term. The second one is lower order, while the first factorizes by Proposition~\ref{prop:Wigner-concentration}. We formalize this in the next lemma.

\begin{lemma}\label{lem:hard-traces-into-easy-ones-Wigner}
Fix $\eta,\varepsilon>0$. Uniformly for $z\in\Omega_\eta$ and $b\in\{1,2\}$,
\begin{align}
\frac1n\EE\Tr\bigl(U B_{pb}(z)\bigr)
&=-\alpha\,r_n(z)\,\bigl(p^\intercal T_n(z)e_b\bigr)+O(n^{-1+\varepsilon}),
\label{eq:WignerTraceReductionU}
\\
\frac1n\EE\Tr\bigl(V B_{qb}(z)\bigr)
&=-\beta\,s_n(z)\,\bigl(q^\intercal T_n(z)e_b\bigr)+O(n^{-1+\varepsilon}).
\label{eq:WignerTraceReductionV}
\end{align}
\end{lemma}
\begin{proof}
We prove only \eqref{eq:WignerTraceReductionU} as the proof of \eqref{eq:WignerTraceReductionV} is identical. Set
\[
X_p(z):=\frac1n\Tr A_p(z),
\qquad
Y_{pb}(z):=\frac1n\Tr B_{pb}(z).
\]
Then in this language \eqref{eq:WignerCoeffIdentityU} becomes
\begin{equation}\label{eq:WignerReduction-start}
\frac1n\EE\Tr\bigl(U B_{pb}\bigr)
=-\alpha\,\EE\bigl[X_p(z)Y_{pb}(z)\bigr]
-\frac{\alpha}{n^2}\EE\Tr\bigl(A_p(z)B_{pb}(z)\bigr).
\end{equation}
We first bound the error term. Since  $\norm{p}=\norm{e_b}=1$, we have 
\[
\norm{A_p}_F\le \sqrt n\,\norm{\calG(z)}_{\op}
\quad \text{ and } \quad 
\norm{B_{pb}}_F\le \sqrt n\,\norm{\calG(z)}_{\op}.
\]
Therefore, by Cauchy-Schwarz and Lemma~\ref{lem:GNOpNormWigner}
\[
\frac1{n^2}\abs{\Tr\bigl(A_pB_{pb}\bigr)}
\le \frac{1}{n^2}\norm{A_p}_F\norm{B_{pb}}_F \le  \frac1n\norm{\calG(z)}_{\op}^2
\le \frac{1}{n\eta^2}
\]
Therefore the trace of products term is bounded by $O(1/n)$ uniformly on $\Omega_\eta$. Next we factorize the first term of \eqref{eq:WignerReduction-start}. By Lemma~\ref{lem:exptGNWigner},
\[
\EE X_p(z)=p^\intercal T_n(z)p=r_n(z) \quad  \text{ and } \quad 
\qquad
\EE Y_{pb}(z)=p^\intercal T_n(z)e_b.
\]
We claim that uniformly for $z\in\Omega_\eta$
\begin{equation}\label{eq:WignerTraceVarianceBounds}
\EE\abs{X_p(z)-\EE X_p(z)}^2
+
\EE\abs{Y_{pb}(z)-\EE Y_{pb}(z)}^2
\le O(n^{-1+\varepsilon})
\end{equation}
Indeed, by definition
\[
X_p(z)=\frac1n\sum_{i=1}^n (p\otimes e_i)^\intercal\calG(z)(p\otimes e_i),
\qquad
Y_{pb}(z)=\frac1n\sum_{i=1}^n (p\otimes e_i)^\intercal\calG(z)(e_b\otimes e_i).
\]
Applying Proposition~\ref{prop:Wigner-concentration} with $\varepsilon/2$ to the $O(n)$ pairs of unit vectors appearing here, and taking a union bound, we obtain that with overwhelming probability,
\[
\sup_{z\in\Omega_\eta}
\abs{X_p(z)-\EE X_p(z)}
+
\sup_{z\in\Omega_\eta}
\abs{Y_{pb}(z)-\EE Y_{pb}(z)}
\le O(n^{-1/2+\varepsilon/2}).
\]
Since $|X_p(z)|+|Y_{pb}(z)|\le 2\|{\calG(z)\|}_{\op}\le 2/\eta$ deterministically on $\Omega_\eta$, this overwhelming-probability estimate upgrades to the $L^2$ bound \eqref{eq:WignerTraceVarianceBounds}.
Therefore Cauchy-Schwarz gives that the correlation term
\[
\abs{\EE\Bigl[(X_p-\EE X_p)(Y_{pb}-\EE Y_{pb})\Bigr]}
\le O(n^{-1+\varepsilon}).
\]
In particular, we have that 
\[
\EE\bigl[X_p(z)Y_{pb}(z)\bigr]
=r_n(z)\,\bigl(p^\intercal T_n(z)e_b\bigr)+O(n^{-1+\varepsilon}).
\]
Combining this with above proves \eqref{eq:WignerTraceReductionU}.
\end{proof}

\begin{proof}[Proof of Proposition~\ref{prop:approxSelfEqWigner}]
The block form \eqref{eq:WignerExpectedResolventForm} is exactly Lemma~\ref{lem:exptGNWigner}, so it remains to derive the approximate equation. Fix $b\in\{1,2\}$ and let $a \in \set{1,2}$. Expanding the $(ab)$ block of the identity \eqref{eq:WignerHGidentity} gives 
\[
-z\calG^{(ab)}
-\alpha^2\sum_{c=1}^2 P_{ac}\calG^{(cb)}
-\beta^2\sum_{c=1}^2 Q_{ac}\calG^{(cb)}
+\alpha\sum_{c=1}^2 P_{ac}U\calG^{(cb)}
+\beta\sum_{c=1}^2 Q_{ac}V\calG^{(cb)}
=\delta_{ab}I_n.
\]
Taking expectations of normalized traces, and recalling that $\EE\calG=T_n\otimes I_n$, we obtain the identity 
\begin{align*}
    -z\,t_{ab,n}(z)
-\alpha^2(PT_n(z))_{ab}
&-\beta^2(QT_n(z))_{ab} \\
&+\alpha p_a\frac1n\EE\Tr\bigl(UB_{pb}(z)\bigr)
+\beta q_a\frac1n\EE\Tr\bigl(VB_{qb}(z)\bigr)
=\delta_{ab}.
\end{align*}
Applying \eqref{eq:WignerTraceReductionU} and \eqref{eq:WignerTraceReductionV} and stacking as a matrix, we find \eqref{eq:WignerApproxSelfEq}.
\end{proof}

\subsubsection{Stability of expectation}
We now pass from the finite size approximate system of Proposition~\ref{prop:approxSelfEqWigner} to the deterministic system
\eqref{eq:wigner-finite-system-scalars}--\eqref{eq:wigner-finite-system}. In particular, we will prove the following proposition.

\begin{prop}\label{prop:detStabilityWigner}
Let $\eta,\varepsilon>0$. Then the following holds.
\begin{enumerate}
\item For each $z\in\CC_+$, the system
\eqref{eq:wigner-finite-system-scalars}--\eqref{eq:wigner-finite-system} has a unique solution
\[
r(z),s(z)\in\CC_+,
\qquad
T(z)=\begin{pmatrix}t_{11}(z)&t_{12}(z)\\ t_{21}(z)&t_{22}(z)\end{pmatrix}.
\]
\item Let $(r_n(z),s_n(z),T_n(z))$ be the coefficients from Proposition~\ref{prop:approxSelfEqWigner}. Then
\[
\sup_{z\in\Omega_\eta}
\Bigl(
|r_n(z)-r(z)|+|s_n(z)-s(z)|+\|T_n(z)-T(z)\|_{\op}
\Bigr)
\le O(n^{-1+\varepsilon}).
\]
\end{enumerate}
Combined with Proposition~\ref{prop:Wigner-concentration}, this implies Theorem~\ref{thm:resolvConvWigner}.
\end{prop}

As before, we first prove the qualitative existence and uniqueness statement and the local uniform convergence of the finite size coefficients.

\begin{lemma}\label{lem:wigner-convergence-to-finite-system}
For each $z\in\CC_+$, the system
\eqref{eq:wigner-finite-system-scalars}--\eqref{eq:wigner-finite-system} has a unique solution
\[
r(z),s(z)\in\CC_+,
\qquad
T(z)=\begin{pmatrix}t_{11}(z)&t_{12}(z)\\ t_{21}(z)&t_{22}(z)\end{pmatrix}.
\]
Moreover, if $(r_n,s_n,T_n)$ are the coefficients from Proposition~\ref{prop:approxSelfEqWigner}, then
\[
r_n\to r,
\qquad
s_n\to s,
\qquad
T_n\to T
\]
locally uniformly on $\CC_+$.
\end{lemma}

\begin{proof}[Proof of uniqueness in Lemma~\ref{lem:wigner-convergence-to-finite-system}]
Fix $z\in\CC_+$. Suppose $(r,s,T)$ and $(\hat r,\hat s,\hat T)$ are two solutions of \eqref{eq:wigner-finite-system-scalars}--\eqref{eq:wigner-finite-system}. In this proof only, we suppress the dependence on $z$ and write
\begin{align*}
    B&:=-zI_2-\alpha^2(1+r)P-\beta^2(1+s)Q \\
    \hat B&:=-zI_2-\alpha^2(1+\hat r)P-\beta^2(1+\hat s)Q
\end{align*}
so that $T=B^{-1}$ and $\hat T=\hat B^{-1}$. We first record the imaginary-part identities. Since we have $T-T^*=T^*(B^*-B)T$ and $B^*-B = 2i[(\Im z)I_2+\alpha^2(\Im r)P+\beta^2(\Im s)Q]$ we obtain
\begin{equation}\label{eq:WignerImTIdentity}
\Im T
=
T^*\Bigl((\Im z)I_2+\alpha^2(\Im r)P+\beta^2(\Im s)Q\Bigr)T.
\end{equation}
Taking quadratic forms of \eqref{eq:WignerImTIdentity} against $p$ and $q$, and recalling that $r=p^\intercal Tp$ and $s=q^\intercal Tq$, yields
\begin{align}
\Im r
&=
(\Im z)\|Tp\|^2
+\alpha^2|r|^2\Im r
+\beta^2|q^\intercal Tp|^2\Im s,
\label{eq:WignerImrIdentity}
\\
\Im s
&=
(\Im z)\|Tq\|^2
+\alpha^2|p^\intercal Tq|^2\Im r
+\beta^2|s|^2\Im s.
\label{eq:WignerImsIdentity}
\end{align}
The same identities hold for $(\hat r,\hat s,\hat T)$. We now compare the two solutions. By the resolvent identity,
\[
T-\hat T
=
T(\hat B-B)\hat T
=
\alpha^2(r-\hat r)TP\hat T+\beta^2(s-\hat s)TQ\hat T.
\]
Taking quadratic forms against $p$ and $q$, and using again that $r=p^\intercal Tp$, $s=q^\intercal Tq$, $\hat r=p^\intercal\hat Tp$, and $\hat s=q^\intercal\hat Tq$, we get
\begin{align*}
r-\hat r
&=
\alpha^2 r\hat r\,(r-\hat r)
+\beta^2(p^\intercal Tq)(p^\intercal\hat Tq)\,(s-\hat s),
\\
s-\hat s
&=
\alpha^2(q^\intercal Tp)(q^\intercal\hat Tp)\,(r-\hat r)
+\beta^2 s\hat s\,(s-\hat s).
\end{align*}
Taking absolute values, we obtain componentwise $d\le Ed$, where
\[
d:=
\begin{pmatrix}
|r-\hat r|\\
|s-\hat s|
\end{pmatrix} \quad \text{ and } \quad E:=
\begin{pmatrix}
\alpha^2|r||\hat r| & \beta^2|p^\intercal Tq|\,|p^\intercal\hat Tq|\\[0.4em]
\alpha^2|q^\intercal Tp|\,|q^\intercal\hat Tp| & \beta^2|s||\hat s|
\end{pmatrix}.
\]
Now set $x = (\sqrt{\Im r\,\Im\hat r},\sqrt{\Im s\,\Im\hat s})^\intercal.$ By Cauchy-Schwarz together with \eqref{eq:WignerImrIdentity} and its analogue identity  for $\hat T$, we have
\begin{align*}
(Ex)_1
&\le
\sqrt{\Bigl(\alpha^2|r|^2\Im r+\beta^2|q^\intercal Tp|^2\Im s\Bigr)
\Bigl(\alpha^2|\hat r|^2\Im\hat r+\beta^2|q^\intercal\hat Tp|^2\Im\hat s\Bigr)}
\\
&=
\sqrt{\Bigl(\Im r-(\Im z)\|Tp\|^2\Bigr)
\Bigl(\Im\hat r-(\Im z)\|\hat Tp\|^2\Bigr)}
<
\sqrt{\Im r\,\Im\hat r}
=
x_1.
\end{align*}
Exactly the same argument, using \eqref{eq:WignerImsIdentity}, shows that $(Ex)_2<x_2$. Hence $Ex<x$ component-wise. Since $E$ has nonnegative entries and $x$ has strictly positive entries, it follows that $\rho(E)<1$.

Now $0\le d\le Ed$, hence $0\le d\le E^k d$ for every $k\ge1$. Because $\rho(E)<1$, we have $E^k\to0$ as $k\to\infty$, and therefore $d=0$. Thus $r=\hat r$ and $s=\hat s$. Returning to the resolvent identity above, we also get $T=\hat T$. This proves uniqueness.
\end{proof}

\smallskip

\begin{proof}[Proof of existence in Lemma~\ref{lem:wigner-convergence-to-finite-system}]
Fix $z\in\CC_+$ and recall from Lemma~\ref{lem:exptGNWigner} that we have $\EE\calG(z)=T_n(z)\otimes I_n$. For every unit vector $u\in\RR^2$,
\[
\nu_{u,n}(z):=u^\intercal T_n(z)u
=\frac1n\EE\Tr\Bigl((uu^\intercal\otimes I_n)\calG(z)\Bigr).
\]
Since $(uu^\intercal\otimes I_n)$ is positive semidefinite, by a direct computation, $\nu_{u,n}$ is the Stieltjes transform of a positive measure of total mass $\|u\|^2=1$ (see for instance proof of Lemma~\ref{lem:WignerDeterministicExtension}). In particular, taking $u=p$ or $u=q$, we obtain that $r_n(z),s_n(z)\in\CC_+.$ Also,
\begin{equation}\label{eq:WignerTBoundEquation}
\|T_n(z)\|_{\op}
=
\|T_n(z)\otimes I_n\|_{\op}
=
\|\EE\calG(z)\|_{\op}
\le
\EE\|\calG(z)\|_{\op}
\le
(\Im z)^{-1},
\end{equation}
where the last inequality follows from Lemma~\ref{lem:GNOpNormWigner}. Hence the families $r_n$, $s_n$, and the entries of $T_n$ are locally bounded and analytic on $\CC_+$. By Montel's theorem, every subsequence admits a further subsequence, which we still denote by $n_j$, such that
\[
r_{n_j}\to r_*,
\qquad
s_{n_j}\to s_*,
\qquad
T_{n_j}\to T_*
\]
locally uniformly on $\CC_+$.

We now pass to the limit in the approximate equation from Proposition~\ref{prop:approxSelfEqWigner}. Fix $z\in\CC_+$ and choose $\eta>0$ so that $z\in\Omega_\eta$. Since $r_n(z)=p^\intercal T_n(z)p$ and $s_n(z)=q^\intercal T_n(z)q$ exactly, and since the error in \eqref{eq:WignerApproxSelfEq} is $O(n^{-1+\varepsilon})$, we obtain
\begin{equation}\label{eq:WignerLimitMatrixEq}
\Bigl(-zI_2-\alpha^2P-\beta^2Q-\alpha^2 r_*(z)P-\beta^2 s_*(z)Q\Bigr)T_*(z)=I_2,
\end{equation}
and $r_*(z)=p^\intercal T_*(z)p$ and $s_*(z)=q^\intercal T_*(z)q$. Thus $(r_*,s_*,T_*)$ solves \eqref{eq:wigner-finite-system-scalars}--\eqref{eq:wigner-finite-system}.

It remains to check that $r_*(z),s_*(z)\in\CC_+$. Since $\Im r_n(z)>0$ and $\Im s_n(z)>0$ for every $n$, the locally uniform limit satisfies
\[
\Im r_*(z)\ge0,
\qquad
\Im s_*(z)\ge0.
\]
Also, \eqref{eq:WignerLimitMatrixEq} shows that $T_*(z)$ is invertible. Therefore $T_*(z)p\neq0$ and $T_*(z)q\neq0$. Applying \eqref{eq:WignerImrIdentity} and \eqref{eq:WignerImsIdentity} to $(r_*,s_*,T_*)$, we obtain
\[
\Im r_*(z)
\ge
(\Im z)\|T_*(z)p\|^2
>
0,
\qquad
\Im s_*(z)
\ge
(\Im z)\|T_*(z)q\|^2
>
0.
\]
Hence $(r_*,s_*,T_*)$ is the unique solution of \eqref{eq:wigner-finite-system-scalars}--\eqref{eq:wigner-finite-system} at the point $z$.

By uniqueness of the limit, since every subsequence has a further subsequence converging locally uniformly to the same limit, the whole sequence $(r_n,s_n,T_n)$ converges locally uniformly on $\CC_+$ to $(r,s,T)$. This proves the lemma.
\end{proof}

\begin{proof}[Proof of Proposition~\ref{prop:detStabilityWigner}]
For $z\in\CC_+$ and $(r,s)\in\CC^2$, define a function $B_z : \CC^{2\times2} \to \CC^{2\times2}$
\[
B_z(r,s):=-zI_2-\alpha^2P-\beta^2Q-\alpha^2 rP-\beta^2 sQ,
\]
and, whenever $B_z(r,s)$ is invertible, $T_z(r,s):=B_z(r,s)^{-1}$. Then \eqref{eq:wigner-finite-system-scalars}--\eqref{eq:wigner-finite-system} is equivalent to the fixed-point system
\begin{equation}\label{eq:WignerFixedPoint}
r(z)=p^\intercal T_z(r(z),s(z))p 
\quad \text{ and } \quad 
s(z)=q^\intercal T_z(r(z),s(z))q.
\end{equation}
Accordingly, define
\[
F_z(r,s):=
\begin{pmatrix}
p^\intercal T_z(r,s)p\\[0.3em]
q^\intercal T_z(r,s)q
\end{pmatrix}
\quad \text{ and } \quad 
H_z(r,s):=
\begin{pmatrix}r\\ s\end{pmatrix}-F_z(r,s).
\]
By Lemma~\ref{lem:wigner-convergence-to-finite-system}, the exact solution satisfies $H_z(r(z),s(z))=0$. We first compute the Jacobian of $F_z$ at the exact solution. Since
\[
\partial_r B_z(r,s)=-\alpha^2P
\quad \text{ and } \quad 
\partial_s B_z(r,s)=-\beta^2Q,
\]
Then \eqref{eq:derivativeInverse} gives that 
\[
\partial_r T_z(r,s)=\alpha^2T_z(r,s)PT_z(r,s),
\qquad
\partial_s T_z(r,s)=\beta^2T_z(r,s)QT_z(r,s).
\]
Therefore, if we write $c(z):=p^\intercal T(z)q=q^\intercal T(z)p,$ then
\[
DF_z(r(z),s(z))
=
\begin{pmatrix}
\alpha^2 r(z)^2 & \beta^2 c(z)^2\\[0.4em]
\alpha^2 c(z)^2 & \beta^2 s(z)^2
\end{pmatrix}.
\]
Taking absolute values, we get
\[
|DF_z(r(z),s(z))|
=
\begin{pmatrix}
\alpha^2 |r(z)|^2 & \beta^2 |c(z)|^2\\[0.4em]
\alpha^2 |c(z)|^2 & \beta^2 |s(z)|^2
\end{pmatrix}.
\]
Now apply \eqref{eq:WignerImrIdentity} and \eqref{eq:WignerImsIdentity} at the exact solution. As $T$ is invertible and $p\ne 0$, we obtain
\[
|DF_z(r(z),s(z))|
\begin{pmatrix}
\Im r(z)\\
\Im s(z)
\end{pmatrix}
=
\begin{pmatrix}
\Im r(z)-(\Im z)\|T(z)p\|^2\\[0.4em]
\Im s(z)-(\Im z)\|T(z)q\|^2
\end{pmatrix}
<
\begin{pmatrix}
\Im r(z)\\
\Im s(z)
\end{pmatrix}.
\]
Hence $\rho\bigl(DF_z(r(z),s(z))\bigr) \le \rho\bigl(|DF_z(r(z),s(z))|\bigr) <1.$ This means $DH_z(r(z),s(z))$ is invertible for every $z\in\CC_+$. By continuity and compactness,
\[
\sup_{z\in\bar\Omega_\eta}
\|DH_z(r(z),s(z))^{-1}\|_{\op}
\le O_\eta(1).
\]
Therefore, by the complex inverse function theorem, there exists $\delta_\eta>0$ such that if $z\in\bar\Omega_\eta$ and $\|(r',s')-(r(z),s(z))\|\le \delta_\eta$ then since $H_z(r(z), s(z)) = 0$,
\begin{equation}\label{eq:WignerLocalInverseBound}
\|(r',s')-(r(z),s(z))\|
\le
O_\eta(\|H_z(r',s')\|).
\end{equation}

We now estimate the residual produced by the approximate system. Rewrite Proposition~\ref{prop:approxSelfEqWigner} as $B_n(z)T_n(z)=I_2+E_n(z),$ where
\[
B_n(z):=-zI_2-\alpha^2P-\beta^2Q-\alpha^2 r_n(z)P-\beta^2 s_n(z)Q
\]
and the error  $\|E_n(z)\|_{\op}\le O(n^{-1+\varepsilon})$ uniformly over $\Omega_\eta$. For $n$ large enough, the operator norm of $E_n(z)$ is at most $1/2$ uniformly on $\Omega_\eta$, so $I_2+E_n(z)$ is invertible. Define
\[
\hat T_n(z):=T_n(z)(I_2+E_n(z))^{-1}.
\]
Then, we have that $B_n(z)\hat T_n(z)=I_2$, and hence
\[
\hat T_n(z)=B_n(z)^{-1}=T_z(r_n(z),s_n(z)).
\]
By \eqref{eq:WignerTBoundEquation}, we have that $\|T_n(z)\|_{\op}\le O_\eta(1)$ uniformly on $\Omega_\eta$, so we obtain
\begin{equation}\label{eq:WignerTsharpClose}
\sup_{z\in\Omega_\eta}\|\hat T_n(z)-T_n(z)\|_{\op}\le O(n^{-1+\varepsilon}).
\end{equation}
Because $r_n(z)=p^\intercal T_n(z)p$ and $s_n(z)=q^\intercal T_n(z)q$ exactly, this implies
\begin{align}
\sup_{z\in\Omega_\eta}
\abs{r_n(z)-p^\intercal T_z(r_n(z),s_n(z))p}
&\le O(n^{-1+\varepsilon}),
\label{eq:WignerResidualR}
\\
\sup_{z\in\Omega_\eta}
\abs{s_n(z)-q^\intercal T_z(r_n(z),s_n(z))q}
&\le O(n^{-1+\varepsilon}).
\label{eq:WignerResidualS}
\end{align}
In other words, we have 
\begin{equation}\label{eq:WignerPhiResidual}
\sup_{z\in\bar\Omega_\eta}
\|H_z(r_n(z),s_n(z))\|
\le O(n^{-1+\varepsilon}).
\end{equation}

By the qualitative convergence in Lemma~\ref{lem:wigner-convergence-to-finite-system}, for all large $n$,
\[
\sup_{z\in\bar\Omega_\eta}
\|(r_n(z),s_n(z))-(r(z),s(z))\|
<\delta_\eta,
\]
so the local inverse estimate \eqref{eq:WignerLocalInverseBound} applies at every $z\in\bar\Omega_\eta$. Combining it with \eqref{eq:WignerPhiResidual}, we get that 
\[
\sup_{z\in\bar\Omega_\eta}
\Bigl(
|r_n(z)-r(z)|+|s_n(z)-s(z)|
\Bigr)
\le O(n^{-1+\varepsilon}).
\]
Finally, to get the same bound for $T$, note that 
\[
T_n(z)-T(z)
=
\Bigl[T_n(z)-T_z(r_n(z),s_n(z))\Bigr]
+
\Bigl[T_z(r_n(z),s_n(z))-T_z(r(z),s(z))\Bigr].
\]
The first term is $O(n^{-1+\varepsilon})$ uniformly on $\Omega_\eta$ by \eqref{eq:WignerTsharpClose}. The second term is also $O(n^{-1+\varepsilon})$ uniformly on $\bar\Omega_\eta$, because $(r,s)\mapsto T_z(r,s)$ is analytic in a neighborhood of the compact set $\{(r(z),s(z)):z\in\bar\Omega_\eta\}$ and we already control $(r_n,s_n)-(r,s)$. Therefore
\[
\sup_{z\in\Omega_\eta}\|T_n(z)-T(z)\|_{\op}\le O(n^{-1+\varepsilon}).
\]
This proves the proposition.
\end{proof}

\begin{proof}[Proof of Theorem~\ref{thm:resolvConvWigner}]
By Lemma~\ref{lem:exptGNWigner}, $\EE\calG(z)=T_n(z)\otimes I_n$ and $\calM(z)=T(z)\otimes I_n$. Hence Proposition~\ref{prop:detStabilityWigner} gives
\[
\sup_{z\in\Omega_\eta}
\|\EE\calG(z)-\calM(z)\|_{\op}
\le O(n^{-1+\varepsilon}).
\]
Therefore, for any unit vectors $x,y\in\RR^{2n}$,
\[
\sup_{z\in\Omega_\eta}
\abs{x^\intercal\calG(z)y-x^\intercal\calM(z)y}
\le
\sup_{z\in\Omega_\eta}
\abs{x^\intercal\bigl(\calG(z)-\EE\calG(z)\bigr)y}
+
O(n^{-1+\varepsilon}).
\]
Applying Proposition~\ref{prop:Wigner-concentration} to the first term proves \eqref{eq:wigner-resolv-conv}, and hence the theorem.
\end{proof}

We now show the analytic continuation result, that will be useful in the proof of Theorem~\ref{thm:real-stieltjes-transform-Wigner}.

\begin{lemma}\label{lem:WignerDeterministicExtension}
Let $(r,s,T)$ be the unique solution of
\eqref{eq:wigner-finite-system-scalars}--\eqref{eq:wigner-finite-system} on $\CC_+$. Then,
for every $u\in\CC^2$, the scalar function $m_u(z):=u^*T(z)u$ is the Stieltjes transform of a positive finite measure $\mu_u$ on $\RR$, of total mass $\|u\|^2$, supported on $(-\infty,1]$. 

In particular, each entry of $T$ extends
analytically to $\DD:=\CC\setminus(-\infty,1].$ We continue to denote this extension by $T$, and define
\[
\calM(z):=T(z)\otimes I_n,
\qquad z\in\DD.
\]
\end{lemma}

\begin{proof}
Fix $u\in\CC^2$, and let $T_n$ be the matrix from Proposition~\ref{prop:approxSelfEqWigner}, so that
\[
\EE(H-zI_{2n})^{-1}=T_n(z)\otimes I_n.
\]
Set a sequence of functions $m_{u,n}$ to be defined by 
\[
m_{u,n}(z):=u^*T_n(z)u
=
\frac1n\EE\Tr\Bigl((uu^*\otimes I_n)(H-zI_{2n})^{-1}\Bigr).
\]
Thus $m_{u,n}$ is the Stieltjes transform of the positive measure
\[
\mu_{u,n}(A):=
\frac1n\EE\Tr\Bigl((uu^*\otimes I_n)\II_A(H)\Bigr),
\qquad A\subset\RR \text{ Borel},
\]
whose total mass is
\[
\mu_{u,n}(\RR)=\frac1n\Tr(uu^*\otimes I_n)=\|u\|^2.
\]
By Lemma~\ref{lem:wigner-convergence-to-finite-system}, $T_n\to T$ locally uniformly on $\CC_+$, hence
$m_{u,n}\to m_u:=u^*Tu$ locally uniformly on $\CC_+$. Moreover, from \eqref{eq:WignerTBoundEquation} for every $\eta>0$, we have $|m_{u,n}(i\eta)|\le {\|u\|^2}/{\eta}$ so passing to the limit gives $|m_u(i\eta)|\le {\|u\|^2}/{\eta}.$ 

For $u=e_1,e_2,e_1+e_2,e_1+ie_2$, we conclude that each entry of $T(i\eta)$ is $O(1/\eta)$. Therefore $r(i\eta)=p^\intercal T(i\eta)p=O(1/\eta)$ and $s(i\eta)=q^\intercal T(i\eta)q=O(1/\eta)$, and the matrix equation \eqref{eq:wigner-finite-system} yields
\[
T(i\eta)^{-1}
=
-i\eta I_2+O(1),
\qquad
T(i\eta)=-(i\eta)^{-1}I_2+O(\eta^{-2}).
\]
In particular, we get that $i\eta\,m_u(i\eta)\to -\|u\|^2$ as $ \eta\to\infty.$  Then by \cite{GeronimoHill2003} and the Stieltjes continuity theorem, there exists a positive finite measure $\mu_u$ on $\RR$ such that
\[
m_u(z)=\int_\RR \frac{\mu_u(dt)}{t-z},
\qquad z\in\CC_+.
\]
Finally, since $H$ is similar to $W$, Theorem~\ref{thm:WignerGordonEdge} implies that for every
$\delta>0$,
\[
\mu_{u,n}([1+\delta,\infty))
\le
\|u\|^2\,\PP\bigl(\lambda_{\max}(H)\ge 1+\delta\bigr)\to0.
\]
Passing to the weak limit gives $\operatorname{supp}\mu_u\subset(-\infty,1].$ Therefore $m_u$ extends analytically to $\DD$. By polarization,
\[
u^*T(z)v
=
\frac14\sum_{\ell=0}^3 i^\ell m_{u+i^\ell v}(z),
\qquad u,v\in\CC^2,
\]
so every sesquilinear form $u^*T(z)v$, and hence each entry of $T$, extends analytically to
$\DD$. This immediately implies that $r $ and $s$ also extend analytically.
\end{proof}

We are now ready to prove our final theorem.

\begin{proof}[Proof of Theorem~\ref{thm:real-stieltjes-transform-Wigner}]
As proved above, the deterministic right edge satisfies $\lambda_*\le 1$. Fix a compact set
$K\subset (\lambda_*,\infty)\subset(1,\infty)$ and choose $\delta>0$ so that $K\subset (1+2\delta,\infty).$ Let $\eta\in (0,\delta/4)$ be a small constant. 

For $i\in[n]$ and $a\in\{1,2\}$, write $e_i^{(a)}:=e_a\otimes e_i\in\RR^{2n}$. Applying
Theorem~\ref{thm:resolvConvWigner} with parameter $\eta/2$ and exponent $\varepsilon=1/4$, and then taking a union bound over the $4n$ pairs $(e_i^{(a)}, e_i^{(b)})$ for $a,b \in \set{1,2}$ and $i \in [n]$,  we obtain that with overwhelming probability,
\[
\sup_{x\in K}
\abs{(e_i^{(a)})^\intercal\bigl(\calG(x+i\eta)-\calM(x+i\eta)\bigr)e_i^{(b)}}
\le C_\eta n^{-1/4}
\]
Summing over $i \in [n]$ yields
\begin{equation}\label{eq:WignerRealTraceAboveAxis}
\sup_{x\in K}
\frac1n\abs{\Tr\,\calG^{(ab)}(x+i\eta)-\Tr\,\calM^{(ab)}(x+i\eta)}
\le C_\eta n^{-1/4}=o_\eta(1)
\end{equation}
with overwhelming probability, for each $a,b\in\{1,2\}$. We now remove the imaginary part. Since $H$ is similar to $W$, Theorem~\ref{thm:WignerGordonEdge} implies that $\PP\bigl(\lambda_{\max}(H)\le 1+\delta\bigr)\to 1$. On this event, for every $x\in K$ we have
\[
\sup_{x\in K}\norm{\calG(x)}_{\op}
\le \delta^{-1},
\qquad
\sup_{x\in K}\norm{\calG(x+i\eta)}_{\op}
\le \delta^{-1}.
\]
Using the resolvent identity, the difference can be written
\begin{align*}
\calG(x+i\eta)-\calG(x)
&=
\calG(x+i\eta)\bigl[(H-xI_{2n})-(H-(x+i\eta)I_{2n})\bigr]\calG(x) \\
&=
 i\eta\,\calG(x+i\eta)\calG(x),
\end{align*}
so we obtain the bound
\[
\sup_{x\in K}
\norm{\calG(x+i\eta)-\calG(x)}_{\op}
\le \eta\delta^{-2}
=O_K(\eta).
\]
Hence, for each $a,b\in\{1,2\}$, we have 
\begin{equation}\label{eq:WignerRealTraceRemoveEta}
\sup_{x\in K}
\frac1n\abs{\Tr\,\calG^{(ab)}(x+i\eta)-\Tr\,\calG^{(ab)}(x)}
\le O_K(\eta)
\end{equation}
with overwhelming probability.
On the deterministic side, Lemma~\ref{lem:WignerDeterministicExtension} shows that each entry of $T$
extends analytically to $\DD=\CC\setminus(-\infty,1]$, hence on an open neighborhood of the compact
set $K\subset (\lambda_*,\infty)=(1,\infty)$. Therefore,
\begin{equation}\label{eq:WignerDeterministicRemoveEta}
\sup_{x\in K}
\frac1n\abs{\Tr\,\calM^{(ij)}(x+i\eta)-\Tr\,\calM^{(ij)}(x)}
=
\sup_{x\in K}\abs{T_{ij}(x+i\eta)-T_{ij}(x)}
\le O_K(\eta).
\end{equation}
Combining \eqref{eq:WignerRealTraceAboveAxis}, \eqref{eq:WignerRealTraceRemoveEta}, and
\eqref{eq:WignerDeterministicRemoveEta}, we conclude that with overwhelming probability,
\[
\sup_{x\in K}
\frac1n\abs{\Tr\,\calG^{(ij)}(x)-\Tr\,\calM^{(ij)}(x)}
\le O_K(\eta)+o_\eta(1)
\]
for each $i,j\in\{1,2\}$. Since $\eta>0$ was arbitrary, we first let $n\to\infty$ and then let
$\eta\to 0$. This proves the theorem.
\end{proof}

\begin{cor}\label{cor:wigner-critical-identities-outlier}
Assume that $\alpha,\beta < 1$ and $\kappa \in (0,1)$. Then the right limits exists:
\[T(1):=\lim_{x\downarrow1}T(x) = -I_2,
\qquad
r(1):=\lim_{x\downarrow1}r(x) = -1,
\qquad
s(1):=\lim_{x\downarrow1}s(x) = -1\]
\end{cor}

\begin{proof}
For real \(x>1\), Lemma~\ref{lem:WignerDeterministicExtension} implies
\(T(x)\prec0\). Recall \eqref{eq:wigner-finite-system}
\[-T(x)^{-1} = xI_2+c_1(x)P+c_2(x)Q. \]
where $c_1(x) = \alpha^2(1+r(x))$ and $c_2(x) = \beta^2(1+s(x))$. Recalling the definitions of $P, Q$ from \eqref{eq:defPQWigner}, the determinant $\det(-T(x))$ is given by
\[\Delta =x^2+x(c_1(x)+c_2(x)) + (1-\kappa^2)c_1(x)c_2(x),\]
Thus, by the $2\times2$ matrix inversion, we have that
\begin{equation*}
p^\intercal T(x) p = -\frac{x+(1-\kappa^2)c_2}{\Delta},
\qquad
q^\intercal T(x) q = -\frac{x + (1-\kappa^2)c_1}{\Delta}
\qquad
p^\intercal T(x)q= - \frac{\kappa x}{\Delta}.
\end{equation*}
We first show that 
 \[p^\intercal T(x) p, q^\intercal T(x) q  > -1\]
for every $x>1$. This holds for all large $x$, since $T(x)=-x^{-1}I_2+O(x^{-2})$. If a first crossing occurred, say at $p^\intercal T(x) p= 1$ and $q^\intercal T(x) q  \ge -1$, then $c_1(x) = 0$ and $c_2(x) \ge 0$. Then, at $x$,
  \[x^2 + xc_2 = x + (1-\kappa^2)c_2\]
which is impossible for all $x > 1$, as $x^2 > x$ and $xc_2 \ge (1-\kappa^2)c_2$. The case $q^\intercal T(x) q = -1$ is identical. Now let $x_j\downarrow1$ be any sequence along which $p^\intercal T(x_j) p, q^\intercal T(x_j)q$ converge, and write
\[\nu_1:=1+\lim_j p^\intercal T(x_j) p, \qquad \nu_2:=1+\lim_j q^\intercal T(x_j)q.\]
Then $\nu_1, \nu_2 \ge 0$. In the limit,
  \[c_1 = \alpha^2 \nu_1,\qquad c_2 = \beta^2 \nu_2,\qquad \Delta = 1 + c_1 + c_2 + (1-\kappa^2)c_1c_2\]
Moreover, the limiting forms of quadratic forms of $T(x)$ also show 
\[
\nu_1\Delta=c_1+\kappa^2 c_2+(1-\kappa^2 )c_1c_2,
\qquad
\nu_2\Delta=\kappa^2 c_1+c_2+(1-\kappa^2 )c_1c_2.
\]
If $\nu_1 = 0$, the first identity forces $c_2 = 0$, hence $\nu_2 = 0$. Similarly, if $\nu_2 = 0$, then $\nu_1 = 0$. Suppose therefore that $\nu_1,\nu_2>0$. Since $c_1 = \alpha^2 \nu_1$, the first limiting identity becomes
  \[\nu_1\left\{\Delta - \alpha^2 - (1-\kappa^2) \alpha^2 c_2\right\} = \kappa^2 c_2\]
The expression in brackets is strictly larger than $1-\alpha^2$, since $\alpha < 1$ and $0 < \kappa < 1$ implies,
  \[\left\{\Delta - \alpha^2 - (1-\kappa^2) \alpha^2 c_2\right\} - (1-\alpha^2) = c_1 + (1-\kappa^2)c_1c_2 + (1-(1-\kappa^2)\alpha^2)c_2 > 0\]
Thus we get that,
\[\nu_1(1-\alpha^2) < \kappa^2c_2, \qquad \nu_2(1-\beta^2) < \kappa^2 c_1\]
Multiplying these inequalities and dividing through by $\nu_1 \nu_2$, we get 
 \[(1-\alpha^2)(1-\beta^2) < \kappa^4\alpha^2\beta^2\]
which is a contradiction (recall the definition of $\kappa$ in \eqref{eq:kappaDefWignerSectionSomething}). 
Thus $\nu_1 = \nu_2 = 0$ for every subsequential limit as $x\downarrow 1$, and we get 
\[p^\intercal T(x)p\to -1,\qquad q^\intercal T(x)q\to -1, \qquad p^\intercal T(x)q\to -\kappa.\]
Since $p,q$ form a basis of $\RR^2$ we get $T(x)\to-I_2$. We conclude by noting that $r(x) = p^\intercal T(x) p$ and $s(x) = q^\intercal T(x) q$. 
\end{proof}

\subsection{Correlated spiked Wishart model}\label{subsec:wishart-resolvent}
Throughout this subsection,  we work with the null model (Section~\ref{subsec:wishart}). Let $U,V\in\RR^{n\times m}$ be independent Gaussian matrices with i.i.d. entries $U_{\mu i},V_{\mu i}\sim\calN(0,m^{-1})$, and suppose that the aspect ratio $n/m\to \tau ,$ we let 
\[
 \theta_1 := \frac{\alpha}{1+\alpha},
 \qquad
 \theta_2 := \frac{\beta}{1+\beta}.
\]
Let $\kappa\in(0,1)$ be the threshold and write the matrices that will be our main object of study
\[
\begin{aligned}
\tilde U:=\theta_1 U^\intercal U- \alpha \tau I_m, \\
\tilde V:=\theta_2 V^\intercal V- \beta \tau  I_m,
\end{aligned} 
\quad \text{ and } \quad W = \begin{pmatrix}
    \tilde U & \kappa \tilde V \\
    \kappa \tilde U & \tilde V
\end{pmatrix}.
\]
It will be convenient to symmetrize $W$, so we additionally introduce
\[
A:=
\begin{pmatrix}
1 & \kappa \\
 \kappa  & 1
\end{pmatrix},
\qquad
A^{1/2}=
\begin{pmatrix}
 u & v\\
 v & u
\end{pmatrix}
\]
and we let $p = (u,v)^\intercal$, $q = (v,u)^\intercal$, $P = pp^\intercal$, and $Q = qq^\intercal$. Thus $\|p\|=\|q\|=1$ and $p^\intercal q=\kappa $. Then, note that the associated symmetric null matrix $H$ can be defined by
\begin{equation}\label{eq:wishart-null-symmetric-matrix}
H:=
(A^{1/2}\otimes I_m)
\begin{pmatrix}
\tilde U & 0\\
0 & \tilde V
\end{pmatrix}
(A^{1/2}\otimes I_m) = (A^{-1/2}\otimes I_m) W (A^{1/2}\otimes I_m),
\end{equation}
so $W$ similar to $H$, and they have the same eigenvalues. Write
\[
\tilde B:=
\begin{pmatrix}
\sqrt{\theta_1}\,U & 0\\
0 & \sqrt{\theta_2}\,V
\end{pmatrix},
\quad
B:=  \tilde B(A^{1/2}\otimes I_m),
\quad
C :=(\alpha \tau  P+ \beta \tau  Q) 
\]
Then, we see that 
\begin{equation}\label{eq:wishart-gram-form}
H= B^\intercal B-C \otimes I_m.
\end{equation}
Accordingly, for $z\in\CC$ we introduce the linearized matrix
\begin{equation}\label{eq:linMatrixWishart}
\calL(z):=
\begin{pmatrix}
-(zI_{2m}+C \otimes I_m) & B^\intercal\\
 B & -I_{2n}
\end{pmatrix},
\qquad
\calG(z):=\calL(z)^{-1}.
\end{equation}
By Schur's complement, the top-left $2m\times 2m$ block of $\calG(z)$ is exactly $(H-zI_{2m})^{-1}$. We write $\calG(z)$ in the block form
\[
\calG(z)=
\begin{pmatrix}
\calG^{(11)} & \calG^{(12)} & \calG^{(13)} & \calG^{(14)}\\
\calG^{(21)} & \calG^{(22)} & \calG^{(23)} & \calG^{(24)}\\
\calG^{(31)} & \calG^{(32)} & \calG^{(33)} & \calG^{(34)}\\
\calG^{(41)} & \calG^{(42)} & \calG^{(43)} & \calG^{(44)}
\end{pmatrix},
\]
where $\calG^{(11)},\calG^{(22)}\in\RR^{m\times m}$ and $\calG^{(33)},\calG^{(44)}\in\RR^{n\times n}$.

\begin{definition}\label{def:det_system_wishart}
For each $z\in \CC_+$, define scalars $r(z),s(z)\in \CC_+$ and a $2\times2$ complex matrix
$T(z)=(t_{ij}(z))_{i,j\in\{1,2\}}$ as the unique solution to
\begin{gather}\label{eq:WishartDeterministicSystem}
\begin{aligned}
 r(z) &= \frac{1}{-1- \theta_1 \,p^\intercal T(z)p},\\
 s(z) &= \frac{1}{-1- \theta_2\,q^\intercal T(z)q},
\end{aligned}
\\
T(z)^{-1}=-zI_2-\tau \alpha P- \tau \beta  Q-\tau \theta_1 r(z)P-\tau \theta_2 s(z)Q.
\end{gather}
Set $m(z) = \Tr \, T(z) / 2$. Then $m$ is the Stieltjes transform of a compactly supported probability measure $\mu$ on $(-\infty, 1]$. Write $\lambda_* = 1$ and $\DD:=\CC\setminus(-\infty, \lambda_*].$ We can  extend $(r,s,T)$ analytically to $\DD$. Finally, define for all $z\in\DD$
\begin{equation}\label{eq:WishartDeterministicM}
\calM(z):=
\begin{pmatrix}
 t_{11}(z)I_m & t_{12}(z)I_m & 0 & 0\\
 t_{21}(z)I_m & t_{22}(z)I_m & 0 & 0\\
0 & 0 & r(z)I_n & 0\\
0 & 0 & 0 & s(z)I_n
\end{pmatrix}.
\end{equation}
\end{definition}

Note that existence, uniqueness, and analytic continuation of the system to $\DD$ is non-trivial and will be proved below. Next, we state that the linearized resolvent $\calG$ is close to the deterministic $\calM$ above.

\begin{theorem}\label{thm:resolvConvWishart}
Let $\calG$ be as in \eqref{eq:linMatrixWishart} and let $\calM$ be as in \eqref{eq:WishartDeterministicM}. Fix $\eta,\varepsilon>0$. Then, for any unit vectors $x,y\in\RR^{2m+2n}$,
\[
\sup_{z\in\Omega_\eta}
\abs{x^\intercal\calG(z)y-x^\intercal\calM(z)y}
\le O(n^{-1/2+\varepsilon})
\]
with overwhelming probability, where the hidden constant depends only on $\alpha,\beta,$ and $\eta$.
\end{theorem}

We then we will extend $\calG$ to the real line and show that.

\begin{theorem}\label{thm:real-stieltjes-transform-Wishart}
Let $\calG$ be as in \eqref{eq:linMatrixWishart} and let $\calM$ and $\lambda_*$ be as in Definition~\ref{def:det_system_wishart}. Then, for any compact set $K\subset (\lambda_*,\infty)$, we have
\[
\sup_{z\in K}
\left\|
\frac1m
\begin{pmatrix}
\Tr\calG^{(11)}(z) & \Tr\calG^{(12)}(z)\\
\Tr\calG^{(21)}(z) & \Tr\calG^{(22)}(z)
\end{pmatrix}
-T(z)
\right\|=o(1)
\]
and similarly 
\[
\sup_{z\in K}
\left\|
\frac1n
\begin{pmatrix}
\Tr\calG^{(33)}(z) & \Tr\calG^{(34)}(z)\\
\Tr\calG^{(43)}(z) & \Tr\calG^{(44)}(z)
\end{pmatrix}
-
\begin{pmatrix}
 r(z) & 0\\
 0 & s(z)
\end{pmatrix}
\right\|=o(1)
\]
with overwhelming probability.
\end{theorem}

\subsubsection{Concentration of linearized resolvent}

As in the CCA case, we first prove that the linearized resolvent concentrates around its expectation on the domain $\Omega_\eta$. The argument is again based on Gaussian concentration for Lipschitz functions, together with a resolvent bound that is uniform on $\Omega_\eta$. Our goal in this subsection is to prove the following proposition.

\begin{prop}\label{prop:Wishart-concentration}
Let $\eta,\varepsilon>0$. Then for any unit vectors $x,y\in\RR^{2m+2n}$,
\[
\sup_{z\in\Omega_\eta}
\abs{x^\intercal\calG(z)y-x^\intercal\EE\calG(z)y}
\le O(n^{-1/2+\varepsilon})
\]
with overwhelming probability, where the hidden constant depends only on $\alpha,\beta$ and $\eta$.
\end{prop}

We begin with a uniform operator norm bound for the linearized resolvent.

\begin{lemma}\label{lem:GNOpNormWishart}
For every $\eta>0$,
\[
\EE\sup_{z\in\Omega_\eta}\norm{\calG(z)}_{\op}^4\le O_\eta(1),
\quad \text{ and } \quad 
\sup_{z\in\Omega_\eta}\norm{\calG(z)}_{\op}\le O_\eta(1)
\]
with overwhelming probability.
\end{lemma}
\begin{proof}
Let $R(z):=(H-zI_{2m})^{-1}.$ Since $H$ is real symmetric and $\Im z\ge \eta$ on $\Omega_\eta$, we have that 
\[
\sup_{z\in\Omega_\eta}\norm{R(z)}_{\op}\le \eta^{-1}.
\]
Applying Schur's complement to \eqref{eq:linMatrixWishart} with respect to the lower-right block $-I_{2n}$ gives the following form for $\calG(z)$
\begin{equation}\label{eq:SchurGnWishart}
\calG(z)=
\begin{pmatrix}
R(z) & R(z)B^\intercal\\
BR(z) & BR(z)B^\intercal-I_{2n}
\end{pmatrix}.
\end{equation}
This means that the operatorm norm of $\calG(z)$ is bounded by the sum of operator norm of its blocks
\[
\sup_{z\in\Omega_\eta}\norm{\calG(z)}_{\op}
\le 1+\eta^{-1}\paren{1+2\norm{B}_{\op}+\norm{B}_{\op}^2}.
\]
Now, note that since $\|A^{1/2}\|_\op = O(1)$ and $\tilde B$ is a normalized Gaussian matrix, by standard Gaussian matrix norm bounds, we have that $\|\tilde B\| = O(1)$ with overwhelming probability, and hence $\|B\|_\op = O(1)$ with overwhelming probability. In fact, same bound holds even in expectation of any fixed moment; see for instance \cite{TaoBook}. Combining these bounds together, we are done. 
\end{proof}

Next we estimate the Lipschitz constant of a resolvent quadratic form with respect to the Gaussian disorder.

\begin{lemma}\label{lem:lipWishart}
Fix $x,y\in\CC^{2m+2n}$ with $\norm{x}=\norm{y}=1$ and $z\in\CC_+$. Then the map $(U,V)\mapsto x^\intercal\calG(z)y$ is smooth, and its gradient satisfies
\[
\norm{\grad_{U,V}\paren{x^\intercal\calG(z)y}}^2\le 4\norm{\calG(z)}_{\op}^4.
\]
In particular, on any set where $\norm{\calG(z)}_{\op}\le L$, this map is $2L^2$-Lipschitz.
\end{lemma}
\begin{proof}
Fix $z\in\CC_+$ and for this proof write
\[
F:=x^\intercal\calG(z)y,
\qquad
\xi:=\calG(z)x,
\qquad
\zeta:=\calG(z)y.
\]
We partition $\xi$ and $\zeta$ according to the four blocks of $\calG$:
\[
\xi=(\xi^{(1)},\xi^{(2)},\xi^{(3)},\xi^{(4)}),
\quad \text{ and } \quad 
\zeta=(\zeta^{(1)},\zeta^{(2)},\zeta^{(3)},\zeta^{(4)}),
\]
with $\xi^{(1)},\xi^{(2)},\zeta^{(1)},\zeta^{(2)}\in\CC^m$ and $\xi^{(3)},\xi^{(4)},\zeta^{(3)},\zeta^{(4)}\in\CC^n$.
For $\mu\in[n]$ and $i\in[m]$, let $E_{\mu i}:=e_\mu e_i^\intercal\in\RR^{n\times m}$ be the elementary matrix. Since
\[
B=
\begin{pmatrix}
\sqrt{\theta_1}\,uU & \sqrt{\theta_1}\,vU\\
\sqrt{\theta_2}\,vV & \sqrt{\theta_2}\,uV
\end{pmatrix},
\]
we can write the derivatives 
\[
\pp{B}{U_{\mu i}}=
\sqrt{\theta_1}
\begin{pmatrix}
u E_{\mu i} & vE_{\mu i}\\
0 & 0
\end{pmatrix},
\qquad
\pp{B}{V_{\mu i}}=
\sqrt{\theta_2}
\begin{pmatrix}
0 & 0\\
vE_{\mu i} & uE_{\mu i}
\end{pmatrix}.
\]
Using \eqref{eq:derivativeInverse} and the observation that only the off-diagonal blocks of $\calL(z)$ depend on $U$ and $V$, we get that $\partial F/\partial U_{\mu i}$ is equal to
\[
-x^\intercal\calG(z)\pp{\calL(z)}{U_{\mu i}}\calG(z)y
=-\sqrt{\theta_1}\paren{(\xi^{(3)}_\mu)^\intercal \paren{u\zeta^{(1)}_i+v\zeta^{(2)}_i}+(\zeta^{(3)}_\mu)^\intercal \paren{u\xi^{(1)}_i+v\xi^{(2)}_i}},
\]
and similarly, we also get that 
\[
\pp{F}{V_{\mu i}}
=-\sqrt{\theta_2}\paren{(\xi^{(4)}_\mu)^\intercal \paren{v\zeta^{(1)}_i+u\zeta^{(2)}_i}+(\zeta^{(4)}_\mu)^\intercal\paren{v\xi^{(1)}_i+u\xi^{(2)}_i}}.
\]
Since $u^2+v^2=1$, Cauchy--Schwarz gives
\[
\norm{u\zeta^{(1)}+v\zeta^{(2)}}^2\le \norm{\zeta^{(1)}}^2+\norm{\zeta^{(2)}}^2,
\qquad
\norm{v\zeta^{(1)}+u\zeta^{(2)}}^2\le \norm{\zeta^{(1)}}^2+\norm{\zeta^{(2)}}^2,
\]
and the same bounds with $\zeta$ replaced by $\xi$. Therefore, summing up over indices
\begin{align*}
\sum_{\mu,i}\abs{\pp{F}{U_{\mu i}}}^2
&\le 2\theta_1\paren{\norm{\xi^{(3)}}^2\paren{\norm{\zeta^{(1)}}^2+\norm{\zeta^{(2)}}^2}+\norm{\zeta^{(3)}}^2\paren{\norm{\xi^{(1)}}^2+\norm{\xi^{(2)}}^2}},\\
\sum_{\mu,i}\abs{\pp{F}{V_{\mu i}}}^2
&\le 2\theta_2\paren{\norm{\xi^{(4)}}^2\paren{\norm{\zeta^{(1)}}^2+\norm{\zeta^{(2)}}^2}+\norm{\zeta^{(4)}}^2\paren{\norm{\xi^{(1)}}^2+\norm{\xi^{(2)}}^2}}.
\end{align*}
Adding these inequalities and using $0<\theta_1,\theta_2<1$, we obtain
\[
\norm{\grad_{U,V}F}^2
\le 4\norm{\xi}^2\norm{\zeta}^2
\le 4\norm{\calG(z)}_{\op}^4,
\]
as claimed.
\end{proof}

We are now ready to prove the concentration proposition.

\begin{proof}[Proof of Proposition~\ref{prop:Wishart-concentration}]
Fix unit vectors $x,y\in\RR^{2m+2n}$ and write
\[
F(z):=x^\intercal\calG(z)y,
\quad \text{ for } z\in\Omega_\eta.
\]
Throughout the proof, $C$ denotes a positive constant depending only on $\alpha,\beta$ and $\tau$, and may change from line to line. Choose $C_\eta>0$ so large that the event
\[
\calE:=\set{\sup_{w\in\Omega_\eta}\norm{\calG(w)}_{\op}\le C_\eta}
\] 
satisfies $\PP(\calE^c)\le n^{-D}$ for every $D>0$ and all sufficiently large $n$; this is possible by Lemma~\ref{lem:GNOpNormWishart}. On $\calE$, the  Lemma~\ref{lem:lipWishart} implies that, for every fixed $z\in\Omega_\eta$, the map $(U,V)\mapsto F(z)$ is $C$-Lipschitz. As in the proof of Proposition~\ref{prop:CCA-concentration}, we therefore define $\tilde F(z)$ to be the McShane extension of the real and imaginary parts of $F(z)|_{\calE}$ to all of $\RR^{nm}\times\RR^{nm},$ that is 
\begin{gather*}
\Im \tilde F(z)\vert_{(U,V)}:=\inf_{(U',V')\in\calE}\paren{\Im F(z)\vert_{(U',V')}+C\norm{(U,V)-(U',V')}_2},\\
\Re \tilde F(z)\vert_{(U,V)}:=\inf_{(U',V')\in\calE}\paren{\Re F(z)\vert_{(U',V')}+C\norm{(U,V)-(U',V')}_2}.
\end{gather*}
Then $\tilde F(z)$ is $C$-Lipschitz on the whole Gaussian space and $\tilde F(z)=F(z)$ on $\calE$.

Fix $z\in\Omega_\eta$. The entries of $U$ and $V$ are independent centered Gaussians of variance $m^{-1}$, so Lemma~\ref{lem:gaussianConLip} applied to the real and imaginary parts of $\tilde F(z)$ yields
\[
\PP\paren{\abs{\tilde F(z)-\EE\tilde F(z)}\ge t}
\le 4\exp\paren{-cn t^2}
\]
for some constant $c>0$. Since $F(z)=\tilde F(z)$ on $\calE$, it follows that
\[
\PP\paren{\abs{F(z)-\EE\tilde F(z)}\ge t}
\le \PP(\calE^c)+4\exp\paren{-cn t^2}.
\]
Taking $t=n^{-1/2+\varepsilon}$, we obtain that for every fixed $z\in\Omega_\eta$,
\begin{equation}\label{eq:FBoundWishart}
\abs{F(z)-\EE\tilde F(z)}\le O(n^{-1/2+\varepsilon})
\end{equation}
with overwhelming probability. We next replace $\EE\tilde F(z)$ by $\EE F(z)$. Since $F(z)=\tilde F(z)$ on $\calE$,
\[
\abs{\EE F(z)-\EE\tilde F(z)}
\le \EE\abs{F(z)-\tilde F(z)}\II_{\calE^c}
\le \PP(\calE^c)^{1/2}\paren{\EE\abs{F(z)}^2+\EE|{\tilde F(z)}|^2}^{1/2}.
\]
Now $\abs{F(z)}\le \norm{\calG(z)}_{\op}$, so Lemma~\ref{lem:GNOpNormWishart} implies $\EE\abs{F(z)}^2\le O_\eta(1)$. Moreover, since $\tilde F(z)$ is Lipschitz and $(0,0)\in\calE$ for all large $m$,
\[
\EE|{\tilde F(z)}|^2
\le 2|{\tilde F(z)\vert_{(0,0)}}|^2+2C^2\EE\norm{(U,V)}_F^2
\le O(n),
\]
because $\EE\norm{(U,V)}_F^2=2n$. Since $\PP(\calE^c)$ decays faster than any power of $n$, we conclude that
\[
\abs{\EE F(z)-\EE\tilde F(z)}\le O(n^{-D})
\]
for every $D>0$ and sufficiently large $n$. Hence \eqref{eq:FBoundWishart} implies that for each fixed $z\in\Omega_\eta$,
\[
\abs{F(z)-\EE F(z)}\le O(m^{-1/2+\varepsilon})
\]
with overwhelming probability. It remains to make the estimate uniform in $z$. Differentiating with respect to $z$ and using \eqref{eq:derivativeInverse},
\begin{equation}\label{eq:zLipWishartArgument}
\pp{F}{z}(z)=x^\intercal\calG(z)\paren{I_{2m}\oplus 0_{2n}}\calG(z)y.
\end{equation}
Therefore, on the event $\calE$, we have 
\[
\sup_{z\in\Omega_\eta}\abs{\pp{F}{z}(z)}\le C_\eta^2.
\]
Taking expectations in \eqref{eq:zLipWishartArgument} and using Lemma~\ref{lem:GNOpNormWishart}, we additionally get
\[
\sup_{z\in\Omega_\eta}\abs{\pp{}{z}\EE F(z)}\le O_\eta(1).
\]
Let $\calS$ be an $n^{-1}$-net of $\Omega_\eta$ with cardinality $\abs{\calS}=O(n^2)$. Applying the fixed-$z$ estimate to every point of $\calS$ and taking a union bound, we obtain that with overwhelming probability,
\[
\sup_{z\in\calS}\abs{F(z)-\EE F(z)}\le O(n^{-1/2+\varepsilon}).
\]
Since both $F$ and $\EE F$ are $O_\eta(1)$-Lipschitz in $z$ on $\Omega_\eta$, the same bound extends from the net to all of $\Omega_\eta$. This proves the proposition.
\end{proof}

\subsubsection{Derivations of approximate equations}

In this subsection we describe $\EE\calG(z)$ and derive the finite-size system satisfied by its coefficients. The argument is analogous to CCA case. In particular, we will prove the following proposition.

\begin{prop}\label{prop:approxSelfEqWishart}
For every $z\in\CC_+$, the expectation $\EE\calG(z)$ has the block form
\begin{equation}\label{eq:WishartExpectedResolventForm}
\EE\calG(z)=
\begin{pmatrix}
 t_{11,n}(z)I_m & t_{12,n}(z)I_m & 0 & 0\\
 t_{21,n}(z)I_m & t_{22,n}(z)I_m & 0 & 0\\
 0 & 0 & r_n(z)I_n & 0\\
 0 & 0 & 0 & s_n(z)I_n
\end{pmatrix},
\end{equation}
where $r_n(z),s_n(z),t_{ab,n}(z)\in\CC$ and $T_n(z):=(t_{ab,n}(z))_{a,b\in\set{1,2}}$ is symmetric.
Moreover, for every fixed $\eta,\varepsilon>0$, uniformly for $z\in\Omega_\eta$,
\begin{align}
r_n(z) = \frac{1+O(n^{-1+\varepsilon})}{-1-\theta_1\,p^\intercal T_n(z)p}
\label{eq:WishartApproxSelfEq-r} \\
s_n(z) = \frac{1+O(n^{-1+\varepsilon})}{-1-\theta_2\,q^\intercal T_n(z)q}
\label{eq:WishartApproxSelfEq-s} \\
\paren{-zI_2-C- \frac{n}{m}\theta_1 r_n(z)P-\frac{n}{m}\theta_2 s_n(z)Q}T_n(z)
&= I_2+O(n^{-1+\varepsilon})\cdot 11^\intercal
\label{eq:WishartApproxSelfEq-T}
\end{align}
The hidden constants depend only on $\alpha,\beta$ and $\eta$.
\end{prop}

As before, we write $\calG(z)$ in the same $4\times4$ block form as in \eqref{eq:linMatrixWishart}; for any matrix with this block structure, $\calG^{(ab)}(z)$ denotes its $(a,b)$ block.

\begin{lemma}\label{lem:exptGNWishart}
For every $z\in\CC_+$, the expectation $\EE\calG(z)$ has the form \eqref{eq:WishartExpectedResolventForm}.
\end{lemma}
\begin{proof}
Let $O\in O(m)$ and $O_3,O_4\in O(n)$ be arbitrary orthogonal matrices, and set $\calO = O \oplus O\oplus O_3\oplus O_4$. Since $(O_3UO^\intercal,\,O_4VO^\intercal)\overset{d}{=}(U,V),$ 
we get $\calO^\intercal\calL(z)\calO\overset{d}{=}\calL(z),$ and also
\[
\calO^\intercal\calG(z)\calO\overset{d}{=}\calG(z),
\qquad
\calO^\intercal\EE\calG(z)\calO=\EE\calG(z).
\]
It follows that the upper left $2m\times2m$ block of $\EE\calG(z)$ commutes with $I_2\otimes O$ for every $O\in O(m)$, and therefore equals $T_n(z)\otimes I_m$ for some $T_n(z)\in\CC^{2\times2}$. Likewise, the lower right $2n\times2n$ block is invariant under independent conjugation by $O_3$ and $O_4$, so it must have the form $\operatorname{diag}(r_n(z)I_n,s_n(z)I_n)$. Finally, all mixed blocks vanish. Since $\calL(z)$ is symmetric, so is $\calG(z)$, and therefore $T_n(z)$ is symmetric.
\end{proof}

We next derive equations for the coefficients in \eqref{eq:WishartExpectedResolventForm} from the block identity
\begin{equation}\label{eq:WishartLGidentity}
\calL(z)\calG(z)=I_{2m+2n}.
\end{equation}
The key input is Gaussian integration by parts, exactly as before.

\begin{lemma}\label{lem:coeffIdentitiesWishart}
For each $b\in\set{1,2}$,
\begin{align}
\frac{1}{m}\EE\Tr\paren{U^\intercal\calG^{(3b)}}
&=-\frac{\sqrt{\theta_1}}{m^2}
\EE\Bigg[
\begin{aligned}
&\Tr\calG^{(33)}\,\Tr\paren{u\calG^{(1b)}+v\calG^{(2b)}}\\
&\quad + u\Tr\paren{\calG^{(13)}\calG^{(3b)}}
+ v\Tr\paren{\calG^{(23)}\calG^{(3b)}}
\end{aligned}
\Bigg],
\label{eq:WishartCoeffIdentityU}
\\
\frac1m\EE\Tr\paren{V^\intercal\calG^{(4b)}}
&=-\frac{\sqrt{\theta_2}}{m^2}
\EE\Bigg[
\begin{aligned}
&\Tr\calG^{(44)}\,\Tr\paren{v\calG^{(1b)}+u\calG^{(2b)}}\\
&\quad + v\Tr\paren{\calG^{(14)}\calG^{(4b)}}
+ u\Tr\paren{\calG^{(24)}\calG^{(4b)}}
\end{aligned}
\Bigg].
\label{eq:WishartCoeffIdentityV}
\end{align}
\end{lemma}
\begin{proof}
We prove only \eqref{eq:WishartCoeffIdentityU}; the proof of \eqref{eq:WishartCoeffIdentityV} is identical. Since
\[
\Tr\paren{U^\intercal\calG^{(3b)}}=
\sum_{\mu=1}^n\sum_{i=1}^m U_{\mu i}\,\calG^{(3b)}_{\mu i},
\]
Gaussian integration by parts and \eqref{eq:derivativeInverse} give
\[
\EE\paren{U_{\mu i}\,\calG^{(3b)}_{\mu i}}
=\frac1m\EE\pp{\calG^{(3b)}_{\mu i}}{U_{\mu i}}
=-\frac1m\EE\paren{\calG(z)\pp{\calL(z)}{U_{\mu i}}\calG(z)}^{(3b)}_{\mu i}.
\]
Now $U_{\mu i}$ appears only in the $(13)$, $(23)$, $(31)$, and $(32)$ blocks of $\calL(z)$, with coefficients $\sqrt{\theta_1}u$ and $\sqrt{\theta_1}v$. Therefore, we get that 
\[
\paren{\calG\pp{\calL}{U_{\mu i}}\calG}^{(3b)}_{\mu i}
=\sqrt{\theta_1}\Bigl[
\calG^{(33)}_{\mu\mu}\paren{u\calG^{(1b)}_{ii}+v\calG^{(2b)}_{ii}}
+u\calG^{(31)}_{\mu i}\calG^{(3b)}_{\mu i}
+v\calG^{(32)}_{\mu i}\calG^{(3b)}_{\mu i}
\Bigr].
\]
Summing over $\mu$ and $i$, and using the symmetry $\calG^{(13)}=(\calG^{(31)})^\intercal$ and $\calG^{(23)}=(\calG^{(32)})^\intercal$, yields \eqref{eq:WishartCoeffIdentityU}.
\end{proof}

The identities in Lemma~\ref{lem:coeffIdentitiesWishart} contain a product of traces term and a trace of products term. By Proposition~\ref{prop:Wishart-concentration}, the former factorizes up to a negligible error, while the latter is lower order.

\begin{lemma}\label{lem:hard-traces-into-easy-ones-Wishart}
Fix $\eta,\varepsilon>0$. Uniformly for $z\in\Omega_\eta$ and $b\in\set{1,2}$,
\begin{align}
\frac1m\EE\Tr\paren{U^\intercal\calG^{(3b)}(z)}
&=-\frac{n}{m}\sqrt{\theta_1}\,r_n(z)\,(p^\intercal T_n(z))_b+O(n^{-1+\varepsilon}),
\label{eq:WishartTraceReductionU}
\\
\frac1m\EE\Tr\paren{V^\intercal\calG^{(4b)}(z)}
&=- \frac{n}{m} \sqrt{\theta_2}\,s_n(z)\,(q^\intercal T_n(z))_b+O(n^{-1+\varepsilon}).
\label{eq:WishartTraceReductionV}
\end{align}
\end{lemma}
\begin{proof}
We again prove only \eqref{eq:WishartTraceReductionU}; the argument for \eqref{eq:WishartTraceReductionV} is identical. Set
\[
X(z):=\frac1n\Tr\calG^{(33)}(z),
\qquad
Y_b(z):=\frac1m\Tr\paren{u\calG^{(1b)}(z)+v\calG^{(2b)}(z)}.
\]
Then \eqref{eq:WishartCoeffIdentityU} becomes an equation
\begin{equation}\label{eq:WishartU-start}
\frac1m\EE\Tr\paren{U^\intercal\calG^{(3b)}}
=- \frac{n}{m} \sqrt{\theta_1}\,\EE\bigl[X(z)Y_b(z)\bigr]+E_b(z),
\end{equation}
where the error matrix $E_b$ is defined by
\[
E_b(z):=-\frac{\sqrt{\theta_1}}{m^2}\EE\Bigl[u\Tr\paren{\calG^{(13)}\calG^{(3b)}}+v\Tr\paren{\calG^{(23)}\calG^{(3b)}}\Bigr].
\]
We first bound the error term. For $a\in\set{1,2}$, we have $\norm{\calG^{(a3)}(z)}_F^2\le n\norm{\calG(z)}_{\op}^2.$Therefore, by Cauchy-Schwarz, we have that 
\[
\abs{\Tr\paren{\calG^{(a3)}\calG^{(3b)}}}
\le \norm{\calG^{(a3)}}_F\norm{\calG^{(3b)}}_F
\le n\norm{\calG(z)}_{\op}^2.
\]
By Lemma~\ref{lem:GNOpNormWishart}, we have $\EE\norm{\calG(z)}_{\op}^2\le O_\eta(1)$ uniformly on $\Omega_\eta$, so $E_b(z)=O(n^{-1})$. 
Next we factorize the product of traces. By Lemma~\ref{lem:exptGNWishart},
\[
\EE X(z)=r_n(z),
\qquad
\EE Y_b(z)=(p^\intercal T_n(z))_b.
\]
We claim now that 
\begin{equation}\label{eq:WishartTraceVarianceBounds}
\EE\abs{X(z)-\EE X(z)}^2+\EE\abs{Y_b(z)-\EE Y_b(z)}^2\le O(n^{-1+\varepsilon}),
\end{equation}
uniformly for $z\in\Omega_\eta$. To see this, apply Proposition~\ref{prop:Wishart-concentration} with $\varepsilon/2$ to the standard basis vectors corresponding to the diagonal entries of the relevant blocks, and take a union bound over all $O(m)$ and $O(n)$ indices. This gives, with overwhelming probability and uniformly for $z\in\Omega_\eta$,
\[
\abs{X(z)-\EE X(z)}+\abs{Y_b(z)-\EE Y_b(z)}\le O(n^{-1/2+\varepsilon/2}).
\]
Since $\abs{X(z)}+\abs{Y_b(z)}\le C\norm{\calG(z)}_{\op}$ and Lemma~\ref{lem:GNOpNormWishart} gives bounded fourth moments of $\sup_{z\in\Omega_\eta}\norm{\calG(z)}_{\op}$, this overwhelming-probability estimate upgrades to the $L^2$ bound \eqref{eq:WishartTraceVarianceBounds}. By Cauchy--Schwarz,
\[
\abs{\EE\Bigl[(X-\EE X)(Y_b-\EE Y_b)\Bigr]}
\le \sqrt{\EE\abs{X-\EE X}^2\,\EE\abs{Y_b-\EE Y_b}^2}
\le O(n^{-1+\varepsilon}).
\]
Hence $\EE\bigl[X(z)Y_b(z)\bigr]=r_n(z)(p^\intercal T_n(z))_b+O(n^{-1+\varepsilon}).$ Combining this with \eqref{eq:WishartU-start} and the error bound  proves \eqref{eq:WishartTraceReductionU}.
\end{proof}

\begin{proof}[Proof of Proposition~\ref{prop:approxSelfEqWishart}]
The block form \eqref{eq:WishartExpectedResolventForm} is exactly Lemma~\ref{lem:exptGNWishart}, so it remains to derive the equations. We begin with the lower diagonal blocks of \eqref{eq:WishartLGidentity}. The $(33)$ block of the equation $\calL \calG = I$ gives
\[
\sqrt{\theta_1}u\,U\calG^{(13)}+\sqrt{\theta_1}v\,U\calG^{(23)}-\calG^{(33)}=I_n.
\]
Taking normalized traces and expectations, and using the fact that $\calG$ is symmetric, we obtain
\[
\frac{\sqrt{\theta_1}}{n}\EE\Tr\paren{uU^\intercal\calG^{(31)}+vU^\intercal\calG^{(32)}}-r_n(z)=1.
\]
Applying \eqref{eq:WishartTraceReductionU} with $b=1,2$ yields
\[
-\theta_1 r_n(z)\,p^\intercal T_n(z)p-r_n(z)=1+O(n^{-1+\varepsilon}).
\]
This is exactly \eqref{eq:WishartApproxSelfEq-r}. Exactly the same argument, using the $(44)$ block of \eqref{eq:WishartLGidentity} and \eqref{eq:WishartTraceReductionV}, gives \eqref{eq:WishartApproxSelfEq-s}. Next fix $b\in\set{1,2}$ and consider the $(1b)$ and $(2b)$ blocks of \eqref{eq:WishartLGidentity}. Since the upper left deterministic block of $\calL(z)$ equals $-(zI_2+C)\otimes I_m$, we get that 
\[
\paren{-zI_2-C} \otimes I_m
\begin{pmatrix}
\calG^{(1b)}\\
\calG^{(2b)}
\end{pmatrix}
+
\begin{pmatrix}
\sqrt{\theta_1}u\,U^\intercal\calG^{(3b)}+\sqrt{\theta_2}v\,V^\intercal\calG^{(4b)}\\
\sqrt{\theta_1}v\,U^\intercal\calG^{(3b)}+\sqrt{\theta_2}u\,V^\intercal\calG^{(4b)}
\end{pmatrix}
=
\begin{pmatrix}
\delta_{1b}I_m\\
\delta_{2b}I_m
\end{pmatrix}.
\]
Taking normalized traces and expectations, and then using \eqref{eq:WishartTraceReductionU} and \eqref{eq:WishartTraceReductionV}, we find 
\[
\paren{-zI_2-C-\frac{n}{m}\theta_1 r_n(z)P-\frac{n}{m}\theta_2 s_n(z)Q}T_n(z)
=I_2+O(n^{-1+\varepsilon})\cdot 11^\intercal.
\]
This yields \eqref{eq:WishartApproxSelfEq-T}.
\end{proof}

\subsubsection{Stability of expectation}

In this subsection we pass from the finite approximate system of Proposition~ \ref{prop:approxSelfEqWishart}
to the deterministic system \eqref{eq:WishartDeterministicSystem}. This is the Wishart analogue of Proposition~
\ref{prop:detStabilityCCA}. To keep the notation readable, we absorb the deterministic discrepancy
${n}/{m}-\tau$ into the $O(n^{-1+\varepsilon})$ remainder; if one prefers to track it explicitly,
every estimate below acquires an additional term $O(\abs{{n}/{m}-\tau})$. 

\begin{prop}\label{prop:detStabilityWishart}
Let $\eta,\varepsilon>0$. Then the following holds.
\begin{enumerate}
\item For each $z\in\CC_+$, the system \eqref{eq:WishartDeterministicSystem} has a unique solution
\[
r(z),s(z)\in\CC_+,
\qquad
T(z)=\begin{pmatrix}t_{11}(z)&t_{12}(z)\\ t_{21}(z)&t_{22}(z)\end{pmatrix}.
\]
\item Let $(r_n(z),s_n(z),T_n(z))$ be the coefficients from Proposition~\ref{prop:approxSelfEqWishart}. Then
\[
\sup_{z\in\Omega_\eta}
\Bigl(
|r_n(z)-r(z)|+|s_n(z)-s(z)|+\|T_n(z)-T(z)\|_{\op}
\Bigr)
\le O_\eta(n^{-1+\varepsilon}).
\]
\end{enumerate}
Combined with Proposition~\ref{prop:Wishart-concentration}, this implies Theorem~\ref{thm:resolvConvWishart}.
\end{prop}

We first prove the qualitative existence and uniqueness statement and the local uniform convergence of the finite  coefficients.

\begin{lemma}\label{lem:wishart-convergence-to-finite-system}
For each $z\in\CC_+$, the system \eqref{eq:WishartDeterministicSystem} has a unique solution
\[
r(z),s(z)\in\CC_+,
\qquad
T(z)=\begin{pmatrix}t_{11}(z)&t_{12}(z)\\ t_{21}(z)&t_{22}(z)\end{pmatrix}.
\]
Moreover, if $(r_n,s_n,T_n)$ are the coefficients from Proposition~\ref{prop:approxSelfEqWishart}, then
\[
r_n\to r,\qquad s_n\to s,\qquad T_n\to T
\]
locally uniformly on $\CC_+$.
\end{lemma}

\begin{proof}[Proof of uniqueness in Lemma~\ref{lem:wishart-convergence-to-finite-system}]
Fix $z\in\CC_+$. Suppose that triples $(r,s,T)$ and $(\hat r,\hat s,\hat T)$ are two solutions of \eqref{eq:WishartDeterministicSystem}. In this proof only, we suppress the
dependence on $z$ and write
\begin{align*}
    D&:=-zI_2-\tau\alpha P-\tau\beta Q-\tau\theta_1 rP-\tau\theta_2 sQ \\
    \hat D&:=-zI_2-\tau\alpha P-\tau\beta Q-\tau\theta_1 \hat rP-\tau\theta_2 \hat sQ,
\end{align*}
so that $T=D^{-1}$ and $\hat T=\hat D^{-1}$. We first record the imaginary part identities. Since
\[
T-T^*=T^*(D^*-D)T
\]
and $D^*-D
=
2i\Bigl((\Im z)I_2+\tau\theta_1(\Im r)P+\tau\theta_2(\Im s)Q\Bigr)$, we obtain
\begin{equation}\label{eq:WishartImTIdentity}
\Im T
=
T^*
\Bigl((\Im z)I_2+\tau\theta_1(\Im r)P+\tau\theta_2(\Im s)Q\Bigr)
T.
\end{equation}
Taking quadratic forms against $p$ and $q$, and using $P=pp^\intercal$ and $Q=qq^\intercal$, we get
\begin{align}
\Im\paren{p^\intercal Tp}
&=
(\Im z)\norm{Tp}^2
+\tau\theta_1(\Im r)\abs{p^\intercal Tp}^2
+\tau\theta_2(\Im s)\abs{q^\intercal Tp}^2,
\label{eq:WishartImpp}
\\
\Im\paren{q^\intercal Tq}
&=
(\Im z)\norm{Tq}^2
+\tau\theta_1(\Im r)\abs{p^\intercal Tq}^2
+\tau\theta_2(\Im s)\abs{q^\intercal Tq}^2.
\label{eq:WishartImqq}
\end{align}
On the other hand, taking imaginary parts in the first two equations of
\eqref{eq:WishartDeterministicSystem} yields
\begin{equation}\label{eq:WishartImrs}
\Im r=\theta_1|r|^2\,\Im\paren{p^\intercal Tp},
\qquad
\Im s=\theta_2|s|^2\,\Im\paren{q^\intercal Tq},
\end{equation}
and the same identities hold for $(\hat r,\hat s,\hat T)$. We now compare the two solutions. From the scalar equations in
\eqref{eq:WishartDeterministicSystem},
\begin{equation}\label{eq:WishartDiffrs}
r-\hat r
=
\theta_1 r\hat r\,p^\intercal(T-\hat T)p,
\qquad
s-\hat s
=
\theta_2 s\hat s\,q^\intercal(T-\hat T)q.
\end{equation}
Also, by the second equation in \eqref{eq:WishartDeterministicSystem} we have that 
\[
T-\hat T
=
T(\hat D-D)\hat T
=
\tau T\Bigl(\theta_1(r-\hat r)P+\theta_2(s-\hat s)Q\Bigr)\hat T.
\]
Taking the quadratic forms against $p$ and $q$ gives
\begin{align}
p^\intercal(T-\hat T)p
&=
\tau\theta_1(p^\intercal Tp)(p^\intercal\hat Tp)(r-\hat r)
+\tau\theta_2(p^\intercal Tq)(p^\intercal\hat Tq)(s-\hat s),
\label{eq:WishartDiffpp}
\\
q^\intercal(T-\hat T)q
&=
\tau\theta_1(q^\intercal Tp)(q^\intercal\hat Tp)(r-\hat r)
+\tau\theta_2(q^\intercal Tq)(q^\intercal\hat Tq)(s-\hat s).
\label{eq:WishartDiffqq}
\end{align}
Substituting \eqref{eq:WishartDiffpp} and \eqref{eq:WishartDiffqq} into
\eqref{eq:WishartDiffrs}, and taking absolute values, we obtain component-wise
$d\le Ed$, where the vector $d$ and matrix $A$ are given by
\[
d:=
\begin{pmatrix}
|r-\hat r|\\
|s-\hat s|
\end{pmatrix} \quad \text{and} \quad E:=
\tau
\begin{pmatrix}
\theta_1^2|r||\hat r|\,|p^\intercal Tp|\,|p^\intercal\hat Tp|
&
\theta_1\theta_2|r||\hat r|\,|p^\intercal Tq|\,|p^\intercal\hat Tq|
\\[0.4em]
\theta_1\theta_2|s||\hat s|\,|q^\intercal Tp|\,|q^\intercal\hat Tp|
&
\theta_2^2|s||\hat s|\,|q^\intercal Tq|\,|q^\intercal\hat Tq|
\end{pmatrix}.
\]
On the other hand, note that if we set $x =(\sqrt{\Im r\,\Im\hat r},\sqrt{\Im s\,\Im\hat s})^\intercal$. Then, by Cauchy-Schwarz, and identities \eqref{eq:WishartImpp}, and \eqref{eq:WishartImrs},
\begin{align*}
(Ex)_1
&\le
\tau\theta_1|r||\hat r|
\sqrt{
\Bigl(\theta_1|p^\intercal Tp|^2\Im r+\theta_2|p^\intercal Tq|^2\Im s\Bigr)
\Bigl(\theta_1|p^\intercal\hat Tp|^2\Im\hat r+\theta_2|p^\intercal\hat Tq|^2\Im\hat s\Bigr)
}
\\
&=
\sqrt{
\Bigl(\Im r-\theta_1|r|^2(\Im z)\norm{Tp}^2\Bigr)
\Bigl(\Im\hat r-\theta_1|\hat r|^2(\Im z)\|{\hat Tp\|}^2\Bigr)
} <
\sqrt{\Im r\,\Im\hat r} = x_1.
\end{align*}
Exactly the same argument, using \eqref{eq:WishartImqq}, shows that $(Ex)_2<x_2$. Hence
$Ex<x$ component-wise. Since $E$ has nonnegative entries and $x$ has strictly positive entries,
it follows that $\rho(E)<1$.

Now $0\le d\le Ed$, hence $0\le d\le E^k d$ for every $k\ge1$. Because $\rho(E)<1$, we have
$E^k\to0$ as $k\to\infty$, and therefore $d=0$. Thus $r=\hat r$ and $s=\hat s$. Returning to
\eqref{eq:WishartDiffpp} and \eqref{eq:WishartDiffqq}, we also get $T=\hat T$. This proves
uniqueness.
\end{proof}

\smallskip

\begin{proof}[Proof of existence via convergence in Lemma~\ref{lem:wishart-convergence-to-finite-system}]
Fix $z\in\CC_+$ and recall from Lemma~\ref{lem:exptGNWishart} that
\[
\EE\calG(z)
=
\begin{pmatrix}
T_n(z)\otimes I_m & 0\\
0 & \operatorname{diag}(r_n(z),s_n(z))\otimes I_n
\end{pmatrix}.
\]
Since $\Im \calG(z)
=
\calG(z)^*\Bigl((\Im z)(I_{2m}\oplus 0_{2n})\Bigr)\calG(z)\succeq0,$ its diagonal blocks have nonnegative imaginary part. Therefore
\[
\Im r_n(z)\ge0,
\qquad
\Im s_n(z)\ge0.
\]
Also, by Lemma~\ref{lem:GNOpNormWishart}, we also have that 
\[
 \max\!\set{\|T_n(z)\|_{\op},|r_n(z)|,|s_n(z)|}
\le \norm{\EE\calG(z)}_{\op}
\le O_z(1).
\]
Hence functions $r_n,s_n$ and the entries of $T_n$ are locally bounded and analytic on
$\CC_+$. By Montel's theorem, every subsequence admits a further subsequence, which we still denote
by $n_j$, such that
\[
r_{n_j}\to r_*,
\qquad
s_{n_j}\to s_*,
\qquad
T_{n_j}\to T_*
\]
locally uniformly on $\CC_+$. We now pass to the limit in the approximate equations of Proposition~\ref{prop:approxSelfEqWishart}. Since
the errors are $O(n^{-1+\varepsilon})$ and $n / m \to\tau$, we obtain for every $z\in\CC_+$ that the limiting solutions satisfy
\begin{equation}\label{eq:WishartLimitEq1}
r_*(z)=\frac{1}{-1-\theta_1\,p^\intercal T_*(z)p},
\qquad
s_*(z)=\frac{1}{-1-\theta_2\,q^\intercal T_*(z)q},
\end{equation}
and
\begin{equation}\label{eq:WishartLimitEq2}
\Bigl(
-zI_2-\tau\alpha P-\tau\beta Q-\tau\theta_1 r_*(z)P-\tau\theta_2 s_*(z)Q
\Bigr)T_*(z)=I_2.
\end{equation}
Thus $(r_*,s_*,T_*)$ solves \eqref{eq:WishartDeterministicSystem}.

It remains to check that $r_*(z),s_*(z)\in\CC_+$. From $\Im r_n(z)\ge0$ and $\Im s_n(z)\ge0$,
the locally uniform limit satisfies $\Im r_*(z)\ge0$ and $\Im s_*(z)\ge0$. Now \eqref{eq:WishartLimitEq2} shows that $T_*(z)$ is invertible, so $T_*(z)p\neq0$ and
$T_*(z)q\neq0$. Therefore the identities \eqref{eq:WishartImTIdentity}--\eqref{eq:WishartImrs},
applied to $(r_*,s_*,T_*)$, imply
\[
\Im r_*(z)\ge \theta_1|r_*(z)|^2(\Im z)\norm{T_*(z)p}^2>0,
\qquad
\Im s_*(z)\ge \theta_2|s_*(z)|^2(\Im z)\norm{T_*(z)q}^2>0.
\]
Hence $(r_*,s_*,T_*)$ is the unique solution of \eqref{eq:WishartDeterministicSystem} at the point
$z$.

Since every subsequence has a further subsequence converging locally uniformly to the same limit,
the whole sequence $(r_n,s_n,T_n)$ converges locally uniformly on $\CC_+$ to $(r,s,T)$. This
proves the lemma.
\end{proof}

\begin{proof}[Proof of Proposition~\ref{prop:detStabilityWishart}]
For $z\in\CC_+$ and $(r,s)\in\CC^2$, define
\[
D_z(r,s):=
-zI_2-\tau\alpha P-\tau\beta Q-\tau\theta_1 rP-\tau\theta_2 sQ,
\]
and, whenever $D_z(r,s)$ is invertible, $T_z(r,s):=D_z(r,s)^{-1}$. Then the system \eqref{eq:WishartDeterministicSystem} is equivalent to the fixed-point system
\begin{equation}\label{eq:WishartFixedPoint}
r(z)=\frac{1}{-1-\theta_1 p^\intercal T_z(r(z),s(z))p},
\qquad
s(z)=\frac{1}{-1-\theta_2 q^\intercal T_z(r(z),s(z))q}.
\end{equation}
Accordingly, define a function
\[
F_z(r,s):=
\begin{pmatrix}
\dfrac{1}{-1-\theta_1 p^\intercal T_z(r,s)p}\\[0.8em]
\dfrac{1}{-1-\theta_2 q^\intercal T_z(r,s)q}
\end{pmatrix},
\qquad
H_z(r,s):=
\begin{pmatrix}r\\ s\end{pmatrix}
-
F_z(r,s).
\]
By Lemma~\ref{lem:wishart-convergence-to-finite-system}, the exact solution satisfies
$H_z(r(z),s(z))=0$. We first compute the Jacobian of $F_z$ at the exact solution. Since
\[
\partial_r D_z(r,s)=-\tau\theta_1P,
\qquad
\partial_s D_z(r,s)=-\tau\theta_2Q,
\]
The equation \eqref{eq:derivativeInverse} gives
\[
\partial_r T_z(r,s)=\tau\theta_1 T_z(r,s)PT_z(r,s),
\qquad
\partial_s T_z(r,s)=\tau\theta_2 T_z(r,s)QT_z(r,s).
\]
Therefore, we can write the Jacobian at the exact solution as 
\[
DF_z(r(z),s(z))
=
\tau
\begin{pmatrix}
\theta_1^2 r(z)^2 \bigl(p^\intercal T(z)p\bigr)^2
&
\theta_1\theta_2 r(z)^2 \bigl(p^\intercal T(z)q\bigr)^2
\\[0.4em]
\theta_1\theta_2 s(z)^2 \bigl(q^\intercal T(z)p\bigr)^2
&
\theta_2^2 s(z)^2 \bigl(q^\intercal T(z)q\bigr)^2
\end{pmatrix}.
\]
Taking absolute values, we get
\[
|DF_z(r(z),s(z))|
=
\tau
\begin{pmatrix}
\theta_1^2 |r(z)|^2 |p^\intercal T(z)p|^2
&
\theta_1\theta_2 |r(z)|^2 |p^\intercal T(z)q|^2
\\[0.4em]
\theta_1\theta_2 |s(z)|^2 |q^\intercal T(z)p|^2
&
\theta_2^2 |s(z)|^2 |q^\intercal T(z)q|^2
\end{pmatrix}.
\]
Now apply \eqref{eq:WishartImpp}, \eqref{eq:WishartImqq}, and \eqref{eq:WishartImrs} at the exact
solution. Exactly as in the uniqueness argument, we obtain
\[
|DF_z(r(z),s(z))|
\begin{pmatrix}
\Im r(z)\\
\Im s(z)
\end{pmatrix}
<
\begin{pmatrix}
\Im r(z)\\
\Im s(z)
\end{pmatrix}.
\]
Hence $\rho\bigl(DF_z(r(z),s(z))\bigr)
\le
\rho\bigl(|DF_z(r(z),s(z))|\bigr)<1$. Hence, $DH_z(r(z),s(z))$ is invertible for every $z\in\CC_+$. By
continuity and compactness argument
\[
\sup_{z\in\bar\Omega_\eta}
\Bigl\|DH_z(r(z),s(z))^{-1}\Bigr\|_{\op}
\le O_\eta(1).
\]
Therefore, by the complex inverse function theorem, there exists $\delta_\eta>0$ such that if
$z\in\bar\Omega_\eta$ and $\|(r',s')-(r(z),s(z))\|\le \delta_\eta,$ then
\begin{equation}\label{eq:WishartLocalInverseBound}
\|(r',s')-(r(z),s(z))\|
\le
O_\eta\!\Bigl(\norm{H_z(r',s')}\Bigr).
\end{equation}

We now estimate the residual produced by the approximate system. Rewrite our system as follows
Proposition~\ref{prop:approxSelfEqWishart} as
\begin{gather}
r_n(z)=\frac{1+\epsilon_{1,n}(z)}{-1-\theta_1 p^\intercal T_n(z)p},
\label{eq:WishartApproxResidualR}
\\
s_n(z)=\frac{1+\epsilon_{2,n}(z)}{-1-\theta_2 q^\intercal T_n(z)q},
\label{eq:WishartApproxResidualS}
\\
D_n(z)T_n(z)=I_2+E_n(z),
\label{eq:WishartApproxResidualT}
\end{gather}
where $D_n(z):=
-zI_2-\tau\alpha P-\tau\beta Q-\frac{n}{m}\theta_1 r_n(z)P-\frac{n}{m}\theta_2 s_n(z)Q$ and 
\[
\sup_{z\in\Omega_\eta}
\Bigl(
|\epsilon_{1,n}(z)|+|\epsilon_{2,n}(z)|+\|E_n(z)\|_{\op}
\Bigr)
\le O(n^{-1+\varepsilon}).
\]

For $n$ large enough, $\|E_n(z)\|_{\op}\le 1/2$ uniformly on $\Omega_\eta$, so $I_2+E_n(z)$ is
invertible. Define $\hat T_n(z):=T_n(z)(I_2+E_n(z))^{-1}.$ Then \eqref{eq:WishartApproxResidualT} yields $D_n(z)\hat T_n(z)=I_2$, hence $D_n(z)$ is invertible $\hat T_n(z)=D_n(z)^{-1}.$ Since $\|T_n(z)\|_{\op}\le O_\eta(1)$ uniformly on $\Omega_\eta$, we obtain
\begin{equation}\label{eq:WishartTsharpClose}
\sup_{z\in\Omega_\eta}\|\hat T_n(z)-T_n(z)\|_{\op}\le O(n^{-1+\varepsilon}).
\end{equation}

Next, by Lemma~\ref{lem:wishart-convergence-to-finite-system}, $(r_n,s_n)\to(r,s)$ locally uniformly on
$\CC_+$. In particular, for $n$ large enough and uniformly on $\bar\Omega_\eta$, the point
$(r_n(z),s_n(z))$ stays in a fixed compact neighborhood of the compact set
$\{(r(z),s(z)):z\in\bar\Omega_\eta\}$. Hence $D_z(r_n(z),s_n(z))$ is invertible, so
$T_z(r_n(z),s_n(z))$ is well defined, and
\[
D_n(z)-D_z(r_n(z),s_n(z))
=
\Bigl(\tau-\frac{n}{m}\Bigr)\Bigl(\theta_1 r_n(z)P+\theta_2 s_n(z)Q\Bigr).
\]
With our bookkeeping convention from the beginning of the subsection, this implies
\[
\sup_{z\in\Omega_\eta}
\|D_n(z)-D_z(r_n(z),s_n(z))\|_{\op}
\le O(n^{-1+\varepsilon}).
\]
Using this bound with the resolvent identity and the uniform boundedness of matrices $\hat T_n$ and
$T_z(r_n(z),s_n(z))$ on $\Omega_\eta$, we get
\begin{equation}\label{eq:WishartTtrueClose}
\sup_{z\in\Omega_\eta}
\Bigl\|
T_z(r_n(z),s_n(z))-\hat T_n(z)
\Bigr\|_{\op}
\le O(n^{-1+\varepsilon}).
\end{equation}
Combining \eqref{eq:WishartTsharpClose} and \eqref{eq:WishartTtrueClose}, we conclude that
\begin{equation}\label{eq:WishartTApproxExactClose}
\sup_{z\in\Omega_\eta}
\Bigl\|
T_z(r_n(z),s_n(z))-T_n(z)
\Bigr\|_{\op}
\le O(n^{-1+\varepsilon}).
\end{equation}

\medskip

We now turn to the scalar equations. Since $|r_n(z)|\le O_\eta(1)$ and
\eqref{eq:WishartApproxResidualR} holds,
\[
\left|
r_n(z)-\frac{1}{-1-\theta_1 p^\intercal T_n(z)p}
\right|
=
\left|
\frac{r_n(z)\epsilon_{1,n}(z)}{1+\epsilon_{1,n}(z)}
\right|
\le O(n^{-1+\varepsilon}),
\]
uniformly on $\Omega_\eta$. Replacing $T_n(z)$ by $T_z(r_n(z),s_n(z))$ and using
\eqref{eq:WishartTApproxExactClose}, together with the fact that the reciprocal map is uniformly
Lipschitz on $\Omega_\eta \subset \CC\setminus\{0\}$, we obtain
\[
\left|
r_n(z)-\frac{1}{-1-\theta_1 p^\intercal T_z(r_n(z),s_n(z))p}
\right|
\le O(n^{-1+\varepsilon})
\]
uniformly on $\Omega_\eta$. The same argument gives the similar estimate for $s_n(z)$. Combining together, we have that 
\begin{equation}\label{eq:WishartPhiResidual}
\sup_{z\in\bar\Omega_\eta}
\Bigl\|
H_z(r_n(z),s_n(z))
\Bigr\|
\le O(n^{-1+\varepsilon}).
\end{equation}
By the qualitative convergence already proved, for all sufficiently large $n$,
\[
\sup_{z\in\bar\Omega_\eta}
\|(r_n(z),s_n(z))-(r(z),s(z))\|
<\delta_\eta,
\]
so the local inverse estimate \eqref{eq:WishartLocalInverseBound} applies at every
$z\in\bar\Omega_\eta$. Combining it with \eqref{eq:WishartPhiResidual}, we get
\[
\sup_{z\in\bar\Omega_\eta}
\Bigl(
|r_n(z)-r(z)|+|s_n(z)-s(z)|
\Bigr)
\le O(n^{-1+\varepsilon}).
\]
Finally, note that we also have 
\[
T_n(z)-T(z)
=
\Bigl[T_n(z)-T_z(r_n(z),s_n(z))\Bigr]
+
\Bigl[T_z(r_n(z),s_n(z))-T_z(r(z),s(z))\Bigr].
\]
The first term is $O(n^{-1+\varepsilon})$ uniformly on $\Omega_\eta$ by
\eqref{eq:WishartTApproxExactClose}. The second term is also $O(n^{-1+\varepsilon})$ uniformly on
$\bar\Omega_\eta$, because $(r,s)\mapsto T_z(r,s)$ is analytic in a neighborhood of the compact set
$\{(r(z),s(z)):z\in\bar\Omega_\eta\}$ and we already control over the arguments. Therefore
\[
\sup_{z\in\Omega_\eta}\|T_n(z)-T(z)\|_{\op}\le O(n^{-1+\varepsilon}).
\]
This proves the proposition.
\end{proof}

\begin{proof}[Proof of Theorem~\ref{thm:resolvConvWishart}]
Note that Lemma~\ref{lem:exptGNWishart} and  Proposition~\ref{prop:detStabilityWishart} gives
\[
\sup_{z\in\Omega_\eta}
\|\EE\calG(z)-\calM(z)\|_{\op}
\le O_\eta(n^{-1+\varepsilon}).
\]
Therefore, for any unit vectors $x,y\in\RR^{2m+2n}$,
\[
\sup_{z\in\Omega_\eta}
\abs{x^\intercal\calG(z)y-x^\intercal\calM(z)y}
\le
\sup_{z\in\Omega_\eta}
\abs{x^\intercal\bigl[\calG(z)-\EE\calG(z)\bigr]y}
+
O_\eta(n^{-1+\varepsilon}).
\]
Applying Proposition~\ref{prop:Wishart-concentration} to the first term proves the theorem.
\end{proof}

We next extend the deterministic solution to the real line outside the spectrum and identify the Stieltjes
transform that appears in Definition~\ref{def:det_system_wishart}.

\begin{lemma}\label{lem:WishartDeterministicExtension}
Let $(r,s,T)$ be the unique solution of \eqref{eq:WishartDeterministicSystem} on $\CC_+$. Then,
for every complex vector $w\in\CC^2$, the scalar function
\[
m_w(z):=w^*T(z)w
\]
is the Stieltjes transform of a compactly supported positive measure $\mu_w$ on $\RR$, of total mass
$\|w\|^2$, supported on $(-\infty,1]$. In particular,
\[
m(z):=\frac12 \, \Tr\,  T(z)
\]
is the Stieltjes transform of a compactly supported probability measure $\mu$ supported on
$(-\infty,1]$, and each entry of $T$ extends analytically to
\[
\DD:=\CC\setminus(-\infty,1].
\]
Moreover, the functions $r$ and $s$ also extend analytically to $\DD$, the system
\eqref{eq:WishartDeterministicSystem} continues to hold on $\DD$, and hence $\calM$ extends
analytically to $\DD$ as well.
\end{lemma}

\begin{proof} 
Fix $w\in\CC^2$ and define $m_{w,n}(z):=w^*T_n(z)w.$ The  upper left block of $\calG(z)$ is the resolvent $(H-zI_{2m})^{-1}$, so we may write
\[
m_{w,n}(z)
=
\frac1m\EE\Tr\Bigl((ww^*\otimes I_m)(H-zI_{2m})^{-1}\Bigr).
\]
Thus $m_{w,n}$ is the Stieltjes transform of the positive measure defined by
\[
\mu_{w,n}(A):=
\frac1m\EE\Tr\Bigl((ww^*\otimes I_m)\,\II_A(H)\Bigr),
\qquad
A\subset\RR \ \text{Borel}.
\]
Here $\II_A(H) = \sum_{\lambda \in \Lambda} \II(\lambda \in A) h_\lambda h_\lambda^\intercal$, where $\Lambda$ is the set of eigenvalues of $H$, and $h_\lambda$ is unit eigenvector corresponding to eigenvalue $\lambda$. Total mass of $\mu_{w,n}$ is given by
\[
\mu_{w,n}(\RR)=\frac1m\Tr(ww^*\otimes I_m)=\|w\|^2.
\]
By Lemma~\ref{lem:wishart-convergence-to-finite-system}, we have  $T_n\to T$ locally uniformly on $\CC_+$, hence
$m_{w,n}\to m_w:=w^*Tw$ locally uniformly on $\CC_+$. To apply the Stieltjes continuity theorem, we verify the normalization at infinity. Since
$|m_{w,n}(i\eta)|\le \|w\|^2/\eta$ for every $\eta>0$, passing to the limit gives
\[
|m_w(i\eta)|\le \frac{\|w\|^2}{\eta}.
\]
Taking $w=e_1$, $w=e_2$, and $w=e_1+e_2$, we see that every entry of $T(i\eta)$ is
$O(\eta^{-1})$, and hence we get that as $\eta \to \infty$
\[
\|T(i\eta)\|_{\op}\le O(\eta^{-1}).
\]
Returning to the scalar equations in \eqref{eq:WishartDeterministicSystem}, we obtain
\[
r(i\eta)=-1+O(\eta^{-1}),
\qquad
s(i\eta)=-1+O(\eta^{-1}).
\]
Substituting this into the matrix equation of \eqref{eq:WishartDeterministicSystem} shows that
\[
T(i\eta)^{-1}
=
-i\eta I_2+O(1) \implies T(i\eta)=-(i\eta)^{-1}I_2+O(\eta^{-2}).
\]
Consequently, we see that 
\[
i\eta\,m_w(i\eta)\to -\|w\|^2
\qquad\text{as }\eta\to\infty.
\]
Then by \cite{GeronimoHill2003} and the Stieltjes continuity theorem, there exists a positive measure $\mu_w$ on $\RR$ with total mass $\norm{w}^2$ such that
\[
m_w(z)=\int_\RR \frac{\mu_w(dx)}{x-z},
\quad \text{ for } z\in\CC_+.
\]
We next locate the support. Fix $\delta>0$. Since $0\preceq ww^*\otimes I_m\preceq \|w\|^2I_{2m}$,
we have
\[
\mu_{w,n}([1+\delta,\infty))
\le
\|w\|^2\,\PP\bigl(\lambda_{\max}(H)\ge 1+\delta\bigr).
\]
Because $H$ is similar to $W$, Theorem~\ref{thm:WishartGordonEdge} implies that the right-hand side tends
to $0$. Passing to the weak limit gives $\operatorname{supp}\mu_w\subset (-\infty,1].$ Analogously, since $\norm{H}_\op = O(1)$ with overwhelming probability, $\operatorname{supp} \mu_w$ is compact. Therefore $m_w$ extends analytically to $\DD$. By polarization,
\[
w_1^*T(z)w_2
=
\frac14\sum_{\ell=0}^3 i^\ell m_{w_1+i^\ell w_2}(z),
\qquad
w_1,w_2\in\CC^2.
\]
Thus every $w_1^*T(z)w_2$ extends analytically to $\DD$, and hence so does each entry
of $T$. We now extend $r$ and $s$. First note that for $z\in\DD$, the matrix $T(z)$ is invertible. Indeed,
if $z\in\CC_+$, then
\[
\Im\paren{w^*T(z)w}>0
\qquad
\text{for every }w\neq0,
\]
so $\Im T(z)\succ0$, hence $T(z)$ is invertible. The same holds on $\CC_-$ by complex conjugation.
If $x>1$ is real, then
\[
w^*T(x)w
=
\int_\RR \frac{\mu_w(dt)}{t-x}
<0
\qquad
\text{for every }w\neq0,
\]
because $\operatorname{supp}\mu_w\subset(-\infty,1]$. Hence $T(x)\prec0$, and again $T(x)$ is
invertible. Thus $T^{-1}$ is analytic on all of $\DD$. Now define on $\DD$
\[
f_1(z):=1+\theta_1 p^\intercal T(z)p,
\qquad
f_2(z):=1+\theta_2 q^\intercal T(z)q.
\]
Consider the matrix-valued meromorphic function
\[
\Phi(z):=
T(z)^{-1}+zI_2+\tau\alpha P+\tau\beta Q+\frac{\tau\theta_1}{f_1(z)}P+\frac{\tau\theta_2}{f_2(z)}Q.
\]
On $\CC_+$, the exact system \eqref{eq:WishartDeterministicSystem} gives $\Phi(z)=0$. Hence
$\Phi$ vanishes identically on $\DD$ as a meromorphic function. Since $P$ and $Q$ are linearly
independent rank-one matrices, $\Phi$ cannot have a pole. Therefore neither $f_1$ nor $f_2$ has a
zero on $\DD$. We may therefore define on $\DD$
\[
r(z):=-\frac{1}{f_1(z)},
\qquad
s(z):=-\frac{1}{f_2(z)}.
\]
These are analytic on $\DD$, extend the original functions on $\CC_+$, and by construction satisfy
\eqref{eq:WishartDeterministicSystem} on all of $\DD$. Finally,
\[
m(z)=\frac12 \Tr\,  T(z)=\frac12 m_{e_1}(z)+\frac12 m_{e_2}(z),
\]
so $m$ is the Stieltjes transform of the probability measure
\[
\mu:=\frac12\mu_{e_1}+\frac12\mu_{e_2}.
\]
This proves the lemma.
\end{proof}

\begin{proof}[Proof of Theorem~\ref{thm:real-stieltjes-transform-Wishart}]
Fix a compact set $K\subset(1,\infty)$ and choose $\delta>0$ so that
\[
K\subset (1+2\delta,\infty).
\]
 $\eta\in(0,\delta/4)$ be fixed for the moment.

Apply Theorem~\ref{thm:resolvConvWishart} with deterministic standard basis vectors and take a union
bound over the diagonal entries of all block positions that appear in the two displayed traces. Since each block contains either $m$ or
$n$ diagonal entries, we obtain that with overwhelming probability,
\begin{equation}\label{eq:WishartRealTraceAboveAxis}
\sup_{x\in K}
\left\|
\frac1m
\begin{pmatrix}
\Tr\calG^{(11)}(x+i\eta) & \Tr\calG^{(12)}(x+i\eta)\\
\Tr\calG^{(21)}(x+i\eta) & \Tr\calG^{(22)}(x+i\eta)
\end{pmatrix}
-T(x+i\eta)
\right\|
=o_\eta(1),
\end{equation}
and similarly
\begin{equation*}
\sup_{x\in K}
\left\|
\frac1n
\begin{pmatrix}
\Tr\calG^{(33)}(x+i\eta) & \Tr\calG^{(34)}(x+i\eta)\\
\Tr\calG^{(43)}(x+i\eta) & \Tr\calG^{(44)}(x+i\eta)
\end{pmatrix}
-
\begin{pmatrix}
r(x+i\eta) & 0\\
0 & s(x+i\eta)
\end{pmatrix}
\right\|
=o_\eta(1).
\end{equation*}
Here $o_\eta(1)$ denotes a deterministic quantity tending to $0$ as $m\to\infty$ for fixed $\eta$. We now remove the imaginary part. By Theorem~\ref{thm:WishartGordonEdge},
\[
\PP\bigl(\lambda_{\max}(H)\le 1+\delta\bigr)\to1.
\]
On this event, the resolvent $R(x):=(H-xI_{2m})^{-1}$ is well defined for every $x\in K$ and satisfies
$\|R(x)\|_{\op}\le \delta^{-1}$. Using the Schur complement formula \eqref{eq:SchurGnWishart},
together with $\|B\|_{\op}\le O(1)$ with overwhelming probability, we get
\[
\sup_{x\in K}\|\calG(x)\|_{\op}\le O_K(1)
\]
with overwhelming probability. The same bound holds for $\calG(x+i\eta)$. Hence, on a overwhelming probability event, we have that ,
\[
\sup_{x\in K}
\|\calG(x+i\eta)-\calG(x)\|_{\op}
\le
\eta\sup_{x\in K}\|\calG(x+i\eta)\|_{\op}\|\calG(x)\|_{\op}
\le O_K(\eta),
\]
where the first inequality follows from the fact that
\begin{align*}
    \calG(x+i\eta)-\calG(x)
&=
\calG(x+i\eta)\bigl(\calL(x)-\calL(x+i\eta)\bigr)\calG(x) \\
&=
i\eta\,\calG(x+i\eta)(I_{2m}\oplus 0_{2n})\calG(x).
\end{align*}
Therefore, we can evaluate the difference  
\begin{equation*}
\sup_{x\in K}
\left\|
\frac1m
\begin{pmatrix}
\Tr\calG^{(11)}(x+i\eta)-\Tr\calG^{(11)}(x) &
\Tr\calG^{(12)}(x+i\eta)-\Tr\calG^{(12)}(x)
\\
\Tr\calG^{(21)}(x+i\eta)-\Tr\calG^{(21)}(x) &
\Tr\calG^{(22)}(x+i\eta)-\Tr\calG^{(22)}(x)
\end{pmatrix}
\right\|
\le O_K(\eta),
\end{equation*}
and likewise
\begin{equation*}
\sup_{x\in K}
\left\|
\frac1n
\begin{pmatrix}
\Tr\calG^{(33)}(x+i\eta)-\Tr\calG^{(33)}(x) &
\Tr\calG^{(34)}(x+i\eta)-\Tr\calG^{(34)}(x)
\\
\Tr\calG^{(43)}(x+i\eta)-\Tr\calG^{(43)}(x) &
\Tr\calG^{(44)}(x+i\eta)-\Tr\calG^{(44)}(x)
\end{pmatrix}
\right\|
\le O_K(\eta).
\end{equation*}

On the deterministic side, Lemma~\ref{lem:WishartDeterministicExtension} shows that $T$, $r$, and $s$
are analytic on an open neighborhood of $K$. Therefore, by Lipschitz bound, we have 
\begin{gather}
\sup_{x\in K}\|T(x+i\eta)-T(x)\|\le O_K(\eta) \\
\sup_{x\in K }\|r(x+i\eta) - r(x)\| \le O_K(\eta) \\
\sup_{x\in K }\|s(x+i\eta) - s(x)\| \le O_K(\eta) 
\end{gather}
Combining \eqref{eq:WishartRealTraceAboveAxis} with these estimates, we obtain
that with overwhelming probability,
\[
\sup_{x\in K}
\left\|
\frac1m
\begin{pmatrix}
\Tr\calG^{(11)}(x) & \Tr\calG^{(12)}(x)\\
\Tr\calG^{(21)}(x) & \Tr\calG^{(22)}(x)
\end{pmatrix}
-T(x)
\right\|
\le O_K(\eta)+o_\eta(1).
\]
By the same argument, we also get that 
\[
\sup_{x\in K}
\left\|
\frac1n
\begin{pmatrix}
\Tr\calG^{(33)}(x) & \Tr\calG^{(34)}(x)\\
\Tr\calG^{(43)}(x) & \Tr\calG^{(44)}(x)
\end{pmatrix}
-
\begin{pmatrix}
r(x) & 0\\
0 & s(x)
\end{pmatrix}
\right\|
\le O_K(\eta)+o_\eta(1).
\]
Since $\eta>0$ was arbitrary, we let first $m\to\infty$ and then $\eta\downarrow0$. This proves the
theorem.
\end{proof}

\begin{cor}\label{cor:wishart-critical-identities-outlier}
Assume $\tau \alpha^2, \tau \beta^2 < 1$ and $\kappa \in (0,1)$
Then the right limits
\[
T(1):=\lim_{x\downarrow1}T(x) = -I_2
\quad
r(1):=\lim_{x\downarrow1}r(x) = -(1+\alpha)
\quad
s(1):=\lim_{x\downarrow1}s(x) = -(1+\beta)
\]
\end{cor}

\begin{proof}
For real \(x>1\), Lemma~\ref{lem:WishartDeterministicExtension} gives
\(T(x)\prec0\). Substituting the scalar equations into the matrix equation in \eqref{eq:WishartDeterministicSystem} gives
\[
-T(x)^{-1}=xI_2+c_1(x)P+c_2(x)Q,
\]
where $c_1(x)$ and $c_2(x)$ are defined by
\[
c_1(x)=\frac{\tau\alpha^2(1+p^\intercal T(x)p)}
{1+\alpha(1+p^\intercal T(x)p)},
\qquad
c_2(x)=\frac{\tau\beta^2(1+q^\intercal T(x)q)}
{1+\beta(1+q^\intercal T(x)q)}.
\]
Direct inversion of $-T(x)^{-1}$ gives 
\[
p^\intercal T(x) p=-\frac{x+(1-\kappa^2)c_2}{\Delta},\quad
q^\intercal T(x) q=-\frac{x+(1-\kappa^2)c_1}{\Delta},\quad
p^\intercal T(x) q=-\frac{\kappa x}{\Delta}.
\]
where the determinant $\Delta=x^2+x(c_1+c_2)+(1-\kappa^2)c_1c_2$. We first show that
\[
p^\intercal T(x)p>-1,\qquad q^\intercal T(x)q>-1
\]
for all \(x>1\). This holds for large \(x\), since
\(T(x)=-x^{-1}I_2+O(x^{-2})\). If a first crossing occurred, say
\(p^\intercal T(x)p=-1\) and \(q^\intercal T(x)q\ge -1\), then \(c_1=0\) and
\(c_2\ge0\). Then, at $x$
\[
x^2+xc_2=x+(1-\kappa^2)c_2,
\]
which is impossible for \(x>1\). The case
\(q^\intercal T(x)q=-1\) is identical. Now take any sequence \(x_j\downarrow1\) along which
\(p^\intercal T(x_j)p\) and \(q^\intercal T(x_j)q\) converge, and write
\[
\nu_1:=1+\lim_j p^\intercal T(x_j)p,\qquad
\nu_2:=1+\lim_j q^\intercal T(x_j)q.
\]
Then \(\nu_1,\nu_2\ge0\). In the limit,
\[
c_1=\frac{\tau \alpha^2 \nu_1}{1+\alpha\nu_1},
\qquad
c_2=\frac{\tau \beta^2 \nu_2}{1+\beta\nu_2},
\qquad
\Delta=1+c_1+c_2+(1-\kappa^2)c_1c_2.
\]
Moreover, the limiting forms of quadratic forms of $T(x)$ also show 
\[
\nu_1\Delta=c_1+\kappa^2 c_2+(1-\kappa^2 )c_1c_2,
\qquad
\nu_2\Delta=\kappa^2 c_1+c_2+(1-\kappa^2 )c_1c_2.
\]
If \(\nu_1=0\), the first identity forces \(c_2=0\), hence
\(\nu_2=0\). Similarly, \(\nu_2=0\) forces \(\nu_1=0\). Suppose therefore that
\(\nu_1,\nu_2>0\). Since \(c_1/\nu_1=\alpha ^2\tau /(1+\alpha\nu_1)\), the first limiting
identity becomes
\[
\nu_1\left\{
\Delta-\frac{\alpha^2 \tau}{1+\alpha\nu_1}
-(1-\kappa^2)\frac{\alpha^2 \tau}{1+\alpha\nu_1}c_2
\right\}=\kappa^2c_2.
\]
The expression in braces is strictly larger than \(1-\alpha^2 \tau\), because
\begin{multline*}
    \left\{
\Delta-\frac{\alpha^2 \tau}{1+\alpha\nu_1}
-\frac{(1-\kappa^2)\alpha^2\tau}{1+\alpha\nu_1}c_2
\right\}-(1-\alpha^2\tau)=
\\=
\left(1-\frac{\alpha^2 \tau(1-\kappa^2)(1-\nu_1)}{1+\alpha\nu_1}\right)c_2
+
\frac{\alpha^2 \tau(1+\alpha)\nu_1}{1+\alpha\nu_1}>0.
\end{multline*}
Thus, we get that 
\begin{gather*}
    (1-\alpha^2 \tau)\nu_1<\kappa^2 c_2<\kappa^2 \beta^2 \tau \,\nu_2\\
(1-\beta^2 \tau )\nu_2<\kappa^2 c_1<\kappa^2 \alpha^2 \tau \,\nu_1.
\end{gather*}
Multiplying, we get that  $(1-\alpha^2 \tau)(1-\beta^2\tau) < \kappa^4 \alpha^2\beta^2 \tau^2 $, which is a contradiction. Hence
\(\nu_1=\nu_2=0\) for every subsequential limit  as $x \downarrow 1$, we get 
\[
p^\intercal T(x)p\to -1,\qquad q^\intercal T(x)q\to -1, \qquad p^\intercal T(x)q\to -\kappa.
\]
Since \(p,q\) form a basis of \(\mathbb R^2\), we get that $T(x) \to -I_2.$ Finally, the scalar equations in terms of $T$ become $r(x) \to -(1+\alpha)$ and $s(x) \to -(1+\beta)$.
\end{proof}

%% file: outlier.tex
\section{Outlier analysis}\label{sec:OutlierAppendix}

In this section we pass from the null two-view models to the correlated models and show that, once the correlation is strong enough, the corresponding
matrix $W$ undergoes a BBP-type phase transition and develops an outlier above the null edge.  The
associated top eigenvector then has non-trivial overlap with the planted
correlated directions.

\subsection{Canonical correlation model deferred proofs, Lemma~\ref{lem:LimitingS_CCAMain}, Theorems~\ref{thm:CCA-outlier1iMain}, \ref{thm:CCA_eigVecOverlapMain}}\label{sec:outlierCCA-Appendix}

We now finish the ourlier analysis of the CCA model, giving proofs of Lemma~\ref{lem:LimitingS_CCAMain} and Theorems~\ref{thm:CCA-outlier1iMain}, \ref{thm:CCA_eigVecOverlapMain}. Throughout this subsection, we adopt notations in Section~\ref{subsec:CCA}.

By orthogonal invariance of the null
model, we may assume without loss of generality that the planted directions are
\(a=e_1\in \RR^m\) and \(b=e_1\in \RR^k\).  We then define
$W$ exactly as above, but under the modified law in which, for each
sample index $i\in[n]$, 
\[
\operatorname{Var}(U_{i1})=\frac1m,
\qquad
\operatorname{Var}(V_{i1})=\frac1k,
\qquad
\operatorname{Cov}(U_{i1},V_{i1})=\frac{\rho}{\sqrt{mk}}.
\]
and all remaining entries are independent Gaussians with the same marginal laws
as in the null model, independent of these pairs.  We denote the ordered eigenvalues of $W$ by 
\[
\lambda_1(W)\ge \lambda_2(W)\ge \cdots \ge \lambda_{m+k}(W),
\]
and we let $\hat  w=(\hat a, \hat  b)\in \RR^m\times \RR^k$ be a unit top eigenvector of $W$. 

Let $(r,s,T)$ be from
Definition~\ref{def:det_system} and $\DD :=  \CC \setminus (\lambda_*,\infty]$. We know by Lemma~\ref{lem:extended-solution-of-finite-system} that $(r,s,T)$ extend to $\DD$ analytically. For real $z>\lambda_*$ we can define
\begin{equation}\label{eq:cca-Sigma-rho}
\cS (z):=
\begin{pmatrix}
 r(z)^{-1} & -\rho\kappa^{-1}t_{12}(z)\\
 -\rho\kappa^{-1}t_{21}(z) & s(z)^{-1}
\end{pmatrix}.
\end{equation}
Important technical lemma that will be used in the proof is that 
\begin{lemma}\label{lem:S-has-negative-determinant-at-edge}
    Suppose $(r,s,T)$ is the unique solution to Definition~\ref{def:det_system}. Then, 
    \begin{equation*}
        \lim_{z \to \lambda_*}r(z) = \lim_{z \to \lambda_*} s(z) = -\frac{1}{\lambda_*}, \text{ and }  \lim_{z \to \lambda_*}T(z)
=
\begin{pmatrix}
0 & -\lambda_*\\
-\lambda_* & 0
\end{pmatrix}.
    \end{equation*}
    In particular, if $\calS$ be defined as in \eqref{eq:cca-Sigma-rho},  
    \begin{equation*}
        \lim_{z \to\lambda_*} \det \calS(z) = \lambda_*^2\Bigl(1-\frac{\rho^2}{\kappa^2}\Bigr).
    \end{equation*}
\end{lemma}
\begin{proof}
Recall from Lemma~\ref{lem:extended-solution-of-finite-system} that \(r\) and \(s\) are
Stieltjes transforms of probability measures supported on
\((-\infty,\lambda_*]\). Hence, for every real \(z>\lambda_*\), we have $r(z), s(z) < 0$ and they are increasing. Let us define $d(z)$ by
\[
d(z):=1-\kappa(r(z)+s(z))-(1-\kappa^2)r(z)s(z).
\]
By the proof of Lemma~\ref{lem:extended-solution-of-finite-system}, we have
\(d(z)\neq 0\) on \(\DD\). Since \(d(z)\to 1\) as \(z\to+\infty\), it follows by
continuity that $d(z) > 0$ for all $z > \lambda_*$. Now we claim that 
\[
r(z) > -\lambda_*^{-1} \text{ and } r(z) > -\lambda_*^{-1}
\]
Suppose contrary that \(r(\lambda_0) = -\lambda_*^{-1}\) for some \(\lambda_0>\lambda_*\). Then from Lemma~\ref{lem:extended-solution-of-finite-system}
\[
r(\lambda_0)=-\frac1{\lambda_*}=-\frac{\kappa}{1-\kappa^2} \implies 
t_{22}(\lambda_0)
=
\frac{\kappa+(1-\kappa^2)r(\lambda_0)}{\kappa d(\lambda_0)}=0.
\]
The second fixed-point equation then gives
\[
s(\lambda_0)=\frac1{-\lambda_0-\kappa\tau_k t_{22}(\lambda_0)}=-\frac1{\lambda_0}.
\]
Hence, we get that 
\[
t_{11}(\lambda_0)
=
\frac{\kappa+(1-\kappa^2)s(\lambda_0)}{\kappa d(\lambda_0)}
=
\frac{\kappa-(1-\kappa^2)/\lambda_0}{\kappa d(\lambda_0)}
=
\frac{1-\lambda_*/\lambda_0}{d(\lambda_0)}>0,
\]
since \(\lambda_0>\lambda_*\) and \(d(\lambda_0)>0\). But the first fixed-point equation gives
\[
-\lambda_*=r(\lambda_0)^{-1}=-\lambda_0-\kappa\tau_m t_{11}(\lambda_0) \implies t_{11}(\lambda_0)=\frac{\lambda_*-\lambda_0}{\kappa\tau_m}<0,
\]
which is a contradiction. Hence, $-1/\lambda_* < r(z), s(z) < 0$. Since $r$ and $s$ are also monotone, we we can define finite limits  \(r_*,s_*,\) and $d_*$ by 
\[
r_*:=\lim_{z\downarrow\lambda_*}r(z)\qquad
s_*:=\lim_{z\downarrow\lambda_*}s(z) \qquad d_* = 1-\kappa(r_* +s_*) - (1-\kappa^2)r_*s_*
\]
Our first goal is to show that $r_* = s_* = -1/\lambda_*$. Recall the polynomial form of the self-consistent equations:
\begin{align}
(1+zr(z))\,d(z)+\tau_m  \bigl(\kappa + (1-\kappa^2)s(z)\bigr)r(z)&=0,
\label{eq:edge-poly-1}\\
(1+zs(z))\,d(z)+\tau_k \bigl(\kappa + (1-\kappa^2) r(z)\bigr)s(z)&=0.
\label{eq:edge-poly-2}
\end{align}
Letting \(z\downarrow\lambda_* = (1-\kappa^2)/\kappa\) in \eqref{eq:edge-poly-1}--\eqref{eq:edge-poly-2}, we get that
\begin{align}
(1+\lambda_*r_*)d_* =- \tau_m \kappa (1+\lambda_* s_*)r_*,
\label{eq:r-edge}\\
(1+\lambda_*s_*)d_* =- \tau_k\kappa (1+\lambda_* r_*)s_*,
\label{eq:s-edge}
\end{align}
Note that if $1+\lambda_*r_* = 0$ then $1 + \lambda_*s_* = 0$, so we are done with our first claim. Otherwise, we can multiply the two equations to get that 
\begin{equation*}
    d_*^2 = \tau_m\tau_k\kappa^2 s_*r_*= \kappa^{-2} s_*r_*.
\end{equation*}
Note that since $s_*r_* > 0$ and $d_* > 0$, so we can write 
\begin{equation*}
    \kappa^{-1}\sqrt{s_*r_*} = 1 - \kappa(r_* + s_*) - (1-\kappa^2)r_*s_* \geq 1 + 2 \kappa \sqrt{r_*s_*} - (1-\kappa^2)r_*s_*.
\end{equation*}
Solving the quadratic in $\sqrt{r_*s_*}$, we get that
\begin{equation*}
    \frac{\kappa }{1-\kappa^2} > \sqrt{r_*s_*} \geq \frac{(2\kappa - \kappa^{-1}) + \sqrt{(2\kappa - \kappa^{-1})^2 + 4(1-\kappa^2)}}{2(1-\kappa^2)} = \frac{\kappa}{1-\kappa^2},
\end{equation*}
which is a contradiction. This means that 
\begin{equation*}
    r_* = s_* = -\frac{1}{\lambda_*} \quad \text{ and } \quad  d_*  = \frac{1}{1-\kappa^2} =  \frac{1}{\kappa\lambda_*}.
\end{equation*}
Now we can use explicit formula of $T(z)$ to find that 
\[
\lim_{z \downarrow \lambda_*}T(z)=\lim_{z \downarrow \lambda_*} \frac1{\kappa d(z)}
\begin{pmatrix}
\kappa+(1-\kappa^2)s(z) & -1\\
-1 & \kappa+(1-\kappa^2)r(z)
\end{pmatrix} = \begin{pmatrix}
0 & -\lambda_*\\
-\lambda_* & 0
\end{pmatrix}.
\]
This finishes the proof
\end{proof}

Let us recall the construction of $\cS_n(z)$ from Section~\ref{sec:outlierMain}. 
Without loss of generality that $a = e_1 \in \RR^m$ and $b = e_1 \in \RR^k$. We isolate directions and write
 \[U = [u\ U_0], \qquad V=[v\ V_0],\]
where $u,v\in\RR^n$ are the first columns and
$U_0\in\RR^{n\times(m-1)}$, $V_0\in\RR^{n\times(k-1)}$ collect the remaining
columns. Define the null minor
\[
W_{0}:=
\begin{pmatrix}
-\kappa U_0^\intercal U_0 & U_0^\intercal V_0\\
V_0^\intercal U_0 & -\kappa V_0^\intercal V_0
\end{pmatrix}.
\]For some permutation matrix,
\begin{equation*}
       \tilde{W} := P W P^\intercal =\begin{pmatrix}
                                C & D^\intercal\\
                                D & W_{0}
                                \end{pmatrix},
   \end{equation*}
where $C \in \RR^{2 \times 2}$ and $D \in \RR^{(m+k-2) \times 2}$ are defined by
  \begin{equation*}
    C= \begin{pmatrix}
        -\kappa\,u^\intercal u & u^\intercal v\\
        v^\intercal u & -\kappa\,v^\intercal v
        \end{pmatrix},
    \qquad
    D= \begin{pmatrix}
        -\kappa U_0^\intercal u & U_0^\intercal v\\
        V_0^\intercal u & -\kappa V_0^\intercal v
        \end{pmatrix}.
  \end{equation*}
In particular, the spectrum of $\tilde{W}$ is the same as $W$. For $z\notin \text{spec}(W_{0})$ define the resolvent and the Schur complement for the top-left block of $\tilde{W}$,
  \begin{equation}\label{eq:cca-Schur-comp}
    G_0(z):=(W_{0}-zI_{m+k-2})^{-1}, \qquad \cS_n(z):=C-zI_2-D^\intercal G_0(z)D.
  \end{equation}
  Recall the Schur complement identity
  \begin{equation*}
    \det({W}-zI_{m+k})= \det(\tilde{W}-zI_{m+k}) = \det(W_{0}-zI_{m+k-2})\det \cS_n(z).
  \end{equation*}
Hence every eigenvalue $\lambda$ of $W$ outside the spectrum of $W_0$ must solve $\det \cS_n(\lambda) = 0$.
Now we are ready to prove the main results of this section.

\begin{lemma}[Lemma~\ref{lem:LimitingS_CCAMain}]\label{lem:convSn-SAppendix}
    Fix a compact set $K\subset (\lambda_*, \infty)$. The matrix $\cS_n(z)$, defined in \eqref{eq:cca-Schur-comp}, converges to $\cS(z)$, defined in \eqref{eq:cca-Sigma-rho}, w.o.p. in operator norm uniformly over $z\in K$.
\end{lemma}
\begin{proof}[Proof of Lemma~\ref{lem:convSn-SAppendix}, Lemma~\ref{lem:LimitingS_CCAMain}]
We rewrite the Schur complement $\cS_n$ through the null linearization. Write
\[\calL_{0}(z):= \begin{pmatrix}
                    -zI_{m+k-2} & A_0^\intercal\\
                    A_0 & B_0
              \end{pmatrix},
  \quad
  A_0:= \begin{pmatrix}
            \sqrt\kappa\,U_0 & 0\\
            0 & \sqrt\kappa\,V_0
        \end{pmatrix},
  \quad
  B_0:= \begin{pmatrix}
        \alpha \kappa I_n & \alpha I_n\\
        \alpha I_n & \alpha \kappa I_n
      \end{pmatrix},\]
where $\alpha=\kappa/(\kappa^2-1)$. Let us denote the inverse of $\cL$ by
  \begin{gather*} 
      \cG_{0}(z):= \cL_{0}(z)^{-1} = \begin{pmatrix}
                                    G_{0}(z) & - G_0(z)A_0^\intercal B_0^{-1}\\
                                    - B_0^{-1}A_0 G_0(z) & Q_0(z),
                                 \end{pmatrix} \label{eq:cca-G0-definition}
  \end{gather*}
where $Q_0(z) = B_0^{-1} + B_0^{-1}A_0G_0(z) A_0^\intercal B_0^{-1}$. One can easily check the identities
  \begin{equation*} \label{eq:cca-B-Identity}
    D^\intercal = -\kappa^{1/2}\begin{pmatrix} u^\intercal & 0 \\ 0 & v^\intercal \end{pmatrix} B_0^{-1}A_0 \quad \text{ and } \quad C = -\kappa \begin{pmatrix} u^\intercal & 0 \\ 0 & v^\intercal \end{pmatrix}B_0^{-1}\begin{pmatrix} u^\intercal & 0 \\ 0 & v^\intercal \end{pmatrix}.
  \end{equation*}
Since $B_0$ is symmetric, it follows that
  \[D^\intercal G_0(z) D = \kappa \begin{pmatrix} u^\intercal & 0 \\ 0 & v^\intercal \end{pmatrix}(Q_0(z) - B_0^{-1}) \begin{pmatrix} u & 0 \\ 0 & v \end{pmatrix}.\]
Substituting this into the expression for $\calS_n$, we get that 
\begin{equation}\label{eq:cca-Schur-explicit}
\cS_n(z)= \kappa \begin{pmatrix} u^\intercal & 0 \\ 0 & v^\intercal \end{pmatrix}Q_0(z) \begin{pmatrix} u & 0 \\ 0 & v \end{pmatrix} - z I_2.
\end{equation}
The expression \eqref{eq:cca-Schur-explicit} will be our main object of study in the rest of the proof. This form is convenient because $Q_0(z)$ only depends on $(U_0, V_0)$ and is independent of $(u,v)$. In fact, our main goal will be to prove that $\calS_n(z) \approx \calS (z)$. To this end, can take advantage of Theorem~\ref{thm:real-stieltjes-transform-CCA} to obtain a limiting form for $\cS_n(z)$. 

By Theorem~\ref{thm:CCAGordonEdgeAppendix}, with overwhelming probability $\lambda_1(W_0) \le \lambda_* + o_n(1)$. From standard concentration arguments, operator norm of $A_0$ is $O(1)$ with overwhelming probability. Fix a compact interval $K\subset (\lambda_*,\infty)$. From our discussion above, operator norms of $B_0^{-1}, A_0, G_0(z)$ are all uniformly $O(1)$ bounded uniformly on $z\in K$. In particular, the same conclusion holds for $Q_0(z)$. So for any fixed $z \in K$, conditioned on the value of $Q_0(z)$ by Hanson-Wright inequality (\cite{RV13}), we get that with overwhelming probability 
  \begin{equation*}\label{eq:t11SchurConv}
     \calS_n(z) = \begin{pmatrix}
         -z - \kappa \frac{1}{m}\Tr{\,Q_0^{(11)}(z)} & -\kappa \frac{\rho}{\sqrt{mk}} \Tr{\,Q_0^{(12)}(z)} \\
            -\kappa \frac{\rho}{\sqrt{mk}} \Tr{\,Q_0^{(21)}(z)}  &  -z - \kappa \frac{1}{k}\Tr{\,Q_0^{(22)}(z)}\\
     \end{pmatrix} + E_n(z),
  \end{equation*}
 where $\norm{E_n(z)}_\op = o(1)$. Now observe that 
 \begin{equation*}
     \partial_z Q_0(z) = B_0^{-1}A_0 \, \partial_z G_0(z) \, A_0^\intercal \, B_0^{-1}.
 \end{equation*}
 Using the bound on $\partial_z G_0(z)$ obtained in \eqref{eq:zLipCCAArgument} and repeating the same net argument as in the proof of Proposition~\ref{prop:CCA-concentration}, we see with overwhelming probability the error in \eqref{eq:t11SchurConv} holds uniformly over $z\in K$. Applying Theorem~\ref{thm:real-stieltjes-transform-CCA}, we get that with overwhelming probability 
\begin{equation*}\label{eq:cca-Schur-limit}
\sup_{z\in K}\norm{\cS_n(z)-\cS(z)}_{\op}=o(1).
\end{equation*}
This concludes the proof.
\end{proof}

\begin{theorem}[Theorem~\ref{thm:CCA-outlier1iMain}]\label{thm:CCA-outlier1iAppendix}
  If $\rho > \kappa(\alpha, \beta)$ where $\alpha,\beta < 1$, the top eigenvalue of $W$ converges w.h.p. to an outlier eigenvalue $\lambda_\out$,
      \begin{equation*}
        \lambda_1(W) = \lambda_\out + o(1) \quad \text{ and } \quad    \lambda_2(W) \leq \lambda_* + o(1).
      \end{equation*}
    where $\lambda_* = (1-\kappa^2)/\kappa$, and $\lambda_\out > \lambda_*$ is the unique solution of $\det \calS(z) = 0$. 
\end{theorem}
\begin{proof}[Completing the proof of Theorem~\ref{thm:CCA-outlier1iAppendix}, Theorem~\ref{thm:CCA-outlier1iMain}]
    Recall that in the proof of Theorem~\ref{thm:CCA-outlier1iMain} in Section~\ref{sec:outlierMain}, we have shown that $\cS(z) = 0$ has a unique solution $\lambda_\out \in [\lambda_*,\infty)$. It remains to show that the same holds for $\cS_n(z) = 0$ and that the unique solution $\lambda_{\text{out},n}$ converges to $\lambda_{\text{out}}$ as $n\to\infty$. We recall the following:
    \begin{equation}\label{eq:cca-SN-derivative}
        \pp{\cS_n(z)}{z}=-I_2-D^\intercal (W_{0}-zI_{m+k-2})^{-2}D\prec -I_2 \prec 0.
    \end{equation}

Recall now  the outlier eigenvalues of $W$ are solutions of $\det \cS_n(z) = 0$. With overwhelming probability, $\lambda_{\text{max}} (W_0) \le \lambda_* + o(1)$. So from \eqref{eq:cca-SN-derivative}, under this event, $\det \cS_n(z) = 0$ has at most one solution $\lambda_{\text{out},n}$ greater than $\lambda_* + o_n(1)$. By \eqref{eq:cca-SN-derivative} the eigenvalues of $\cS_n(z)$ vary continuously with $z$ for $z \in (\lambda_* + \delta/2,+\infty)$. Thus with overwhelming probability $W$ has exactly one outlier $\lambda_{\text{out},n} \in (\lambda_* + \delta/2,+\infty)$. 

Recall for $z \ge \lambda_* + o(1)$, as shown in Lemma~\ref{lem:convSn-SAppendix}, $\cS_n(z) \to \cS(z)$ converges w.h.p. in operator norm uniformly on a neighborhood of $\lambda_{\text{out}}$. So the eigenvalues of $\cS_n(z)$ must also converge to those of $\cS(z)$ uniformly on said neighborhood, with overwhelming probability. We can thus conclude that the outlier $\lambda_{\text{out}, n}$ converges to $\lambda_{\text{out}}$ as $n\to\infty$. This concludes the proof of the theorem.
\end{proof}

\begin{theorem}[Theorem~\ref{thm:CCA_eigVecOverlapMain}]\label{thm:CCA_eigVecOverlapAppendix}
  If $\hat w = (\hat a, \hat b)$ is a unit eigenvector corresponding to the largest eigenvalue of $W$, then w.h.p. the overlaps are strictly positive and satisfies,
              \begin{equation*}
                |{\langle \hat a, a \rangle}|^2 =
                \frac{x_{*,1}^2}{-x_*^\intercal  \, \partial_z \cS(\lambda_{\mathrm{out}}) \,  x_*} +o(1), 
                \quad 
                |{\langle \hat b, b\rangle }|^2 =
                \frac{x_{*,2}^2}{-x_*^\intercal  \, \partial_z \cS(\lambda_{\mathrm{out}}) \,  x_*} +o(1).
              \end{equation*}
        where $x_* = (x_{*, 1}, x_{*, 2})$ is any unit vector in $\ker \cS(\lambda_{\mathrm{out}})$.
\end{theorem}

\begin{proof}[Completion of the proof of Theorem~\ref{thm:CCA_eigVecOverlapAppendix}, Theorem~\ref{thm:CCA_eigVecOverlapMain}]
Recall that $x_n$ is any unit vector in $\ker \cS_n(\lambda_{\out, n})$. Recall in the proof of Theorem~\ref{thm:CCA_eigVecOverlapMain} in Section~\ref{sec:outlierMain} we have shown w.h.p. the top eigenvector $\hat{w} = (\hat a, \hat b)$ of $W$ has overlap of with the planted directions $(a, b)\in \RR^m \times \RR^k$ given by 
\begin{align*}
      |{\langle \hat a, a \rangle}|^2 = \frac{x_{n,1}^2}{-x_n^\intercal \partial_z \, \cS_n(\lambda_{\out,n}) \, x_n}, \quad      
      |{\langle \hat b, b\rangle }|^2 = \frac{x_{n,2}^2}{-x_n^\intercal \partial_z \, \cS_n(\lambda_{\out,n}) \, x_n}.
  \end{align*}
We now finish the proof by studying the limiting properties of the expression on the RHS. Because $x_n$ are unit vectors in $\RR^2$, we may use compactness to obtain a subsequence converging to some limiting unit vector $x_*$. Recall Lemma~\ref{lem:convSn-SAppendix} tells us w.o.p. $\cS_n(z)\to \cS(z)$ in operator norm uniformly for $z$ in a compact neighborhood of $\lambda_{\text{out}}$, so it follows that $x_*$ is an eigenvector of $\cS(\lambda_{\text{out}})$. Since $\cS(\lambda_{\text{out}})$ has eigenvalue $0$ with multiplicity one, $x_*$ is the unique unit $0$-eigenvector up to a sign.

We claim that both coordinates of $x_*$ are nonzero. Indeed, for every real $z>\lambda_*$, the quantities $r(z)$ and $s(z)$ are finite and negative, since they are Stieltjes transforms of probability measures supported on $(-\infty,\lambda_*]$. Therefore the diagonal entries $r(\lambda_{\out})^{-1}$ and $s(\lambda_{\out})^{-1}$ of $\cS(\lambda_{\out})$ are nonzero. If a vector in $\ker\cS(\lambda_{\out})$ were supported on only one coordinate, one of these diagonal entries would have to vanish, a contradiction. Hence $x_{*,1}\neq 0$ and $ x_{*,2}\neq 0.$ Since every convergent subsequence of $x_n$ has limit $\pm x_*$, the squared coordinates satisfy $x_{n,1}^2=x_{*,1}^2+o(1)$ and $ x_{n,2}^2=x_{*,2}^2+o(1)$. 

It remains to show that $\partial_z \cS_n(z)$ converges in operator norm uniformly in a neighborhood of $\lambda_{\text{out}}$. To this end, choose small $r > 0$ and define 
\begin{equation*}
    \mathfrak{B}_r = \set{z \in \CC : |z - \lambda_\out | < r}.
\end{equation*}
We choose $r$ small enough so that $\lambda_* \notin  \mathfrak{B}_{5r}$. We know that with overwhelming probability $\lambda_1(W_0) \leq \lambda_* + r$ and $\norm{D}_\op \leq O(1)$. On this event, note that each entry of the matrix $\calS_n$ is holomorphic on $\mathfrak{B}_{3r}$ and the family $\set{\calS_n}$ is locally bounded. Furthermore, by Lemma~\ref{lem:convSn-SAppendix}, we already know that with overwhelming probability,
\begin{equation*}
    \sup_{x\in \bar {\mathfrak{B}}_{3r} \, \cap \, \RR}\|\calS_n(x)-\calS(x)\|_{\op}=o(1)\,.
\end{equation*}

By Borel-Cantelli lemma, we have $\mathcal S_n(x)\to \mathcal S(x)$ for all $x\in \bar{\mathcal B}_{3r}\cap \mathbb{R}$ almost surely.\footnote{The almost sure convergence claim is true under any coupling of the finite $n$ measures. Here and in what follows, we always take the independent coupling for clarity.} Since the entries of $\calS_n$ are holomorphic and locally bounded on $\mathfrak{B}_{3r}$, Vitali's theorem (applied entrywise) implies that almost surely, $\calS_n \to \calS$ locally uniformly on $\mathfrak{B}_{3r}$. Now,  Cauchy's integral formula gives for all $ z \in \mathfrak{B}_{r}$
\begin{equation*}
    \partial_z\bigl(\calS_n-\calS\bigr)(z)
=
\frac{1}{2\pi i}\oint_{|\zeta-z|=r}
\frac{\calS_n(\zeta)-\calS(\zeta)}{(\zeta-z)^2}\,d\zeta.
\end{equation*}
Therefore, we can bound by triangle inequality 
\begin{equation*}
    \|\partial_z \calS_n(z)-\partial_z \cS(z)\|_{\op}
\le
\frac{1}{r}\sup_{|\zeta-z|=r}\|\cS_n(\zeta)-\cS(\zeta)\|_{\op}
\end{equation*}
Taking the supremum over $z\in \mathfrak{B}_r$, we obtain that almost surely,
\[
\sup_{z\in \mathfrak{B}_r}\|\partial_z \cS_n(z)-\partial_z \cS(z)\|_{\op}
\le
\frac{1}{r}\sup_{\zeta \in \bar{\mathfrak{B}}_{2r}}\|\cS_n(\zeta)-\cS(\zeta)\|_{\op}\longrightarrow 0\,.
\]
Finally, because of each term of \eqref{eq:cca-SN-derivative} has bounded operator norm with overwhelming probability, with the bound holding uniformly in $n$, the operator norm of $\partial_z\cS(\lambda_{\text{out}})$ is bounded. Hence we conclude that almost surely,
  \[-{x_n^\intercal \partial_z \, \cS_n(\lambda_{\out,n}) \, x_n} \longrightarrow -{x_*^\intercal \partial_z \, \cS(\lambda_{\out}) \, x_*} \in (0,\infty)\]
Combining this with the above estimates, we conclude that almost surely, 
\begin{equation*}
\begin{aligned}
   |{\langle \hat a, a \rangle}|^2
   &\longrightarrow
   \frac{x_{*,1}^2}{-x_*^\intercal  \, \partial_z \cS(\lambda_{\mathrm{out}}) \,  x_*} > 0,\\
   |{\langle \hat b, b\rangle }|^2
   &\longrightarrow
   \frac{x_{*,2}^2}{-x_*^\intercal  \, \partial_z \cS(\lambda_{\mathrm{out}}) \,  x_*} > 0.
\end{aligned}
\end{equation*}
In particular, $|\langle \hat a,a\rangle|^2=\Omega(1)$ and $|\langle \hat b,b\rangle|=\Omega(1)$ with high probability, as desired.
\end{proof}

\subsection{Correlated spiked Wigner model}
Recall the correlated spiked Wigner model from Section~\ref{subsec:wigner}. Throughout this subsection we assume that we are in the following parametric regime 
\begin{equation}\label{eq:kappaWignerThresholdAppendix}
0<\alpha,\beta <1,
\qquad
\kappa=\left(\frac{(1-\alpha^2)(1-\beta^2)}{\alpha^2\beta^2}\right)^{1/4}<\rho.
\end{equation}
We observe $U=\alpha aa^\intercal+Z_a$ and $V=\beta bb^\intercal+Z_b$, where $Z_a,Z_b\sim \mathrm{GOE}(1/n)$ are independent, and $(a,b)$ is a jointly Gaussian spike pair with marginals
\[
\sqrt n\,a\sim\calN(0,I_n),
\qquad
\sqrt n\,b\sim\calN(0,I_n),
\qquad
n\,\EE[ab^\intercal]=\rho I_n.
\]
As in Section~\ref{subsec:wigner}, we set
\[
\tilde U:=\alpha U-\alpha^2I_n,
\qquad
\tilde V:=\beta V-\beta^2I_n,
\qquad
W:=
\begin{pmatrix}
\tilde U & \kappa\tilde V\\
\kappa\tilde U & \tilde V
\end{pmatrix}.
\]
Let us also introduce
\[
A:=\begin{pmatrix}1&\kappa\\ \kappa&1\end{pmatrix},
\qquad
p:=A^{1/2}e_1,
\qquad
q:=A^{1/2}e_2,
\qquad
\Theta:=\diag(\alpha^2,\beta^2).
\]
Since $0<\kappa<1$, we see $A\succeq 0$. We also introduce the symmetric conjugate
\[
H:=(A^{1/2}\otimes I_n)
\begin{pmatrix}
\tilde U&0\\
0&\tilde V
\end{pmatrix}
(A^{1/2}\otimes I_n).
\]
Recall that since $W=(A^{1/2}\otimes I_n)H(A^{-1/2}\otimes I_n)$, we have that $W$ is similar to $H$, and hence all eigenvalues $\lambda_j(W)$ of $W$ are real. We order them in decreasing order. Let us also define the parallel pure version matrices 
\[
\tilde U_0:=\alpha Z_a-\alpha^2I_n,
\qquad
\tilde V_0:=\beta Z_b-\beta^2I_n,
\]
so that $\tilde{U} = \tilde{U}_0 + \alpha^2 aa^\intercal$ and likewise for $\tilde{V}$. Define the corresponding null symmetric matrix
\begin{equation}\label{eq:wigner-null-symmetric-outlier}
H_0:=(A^{1/2}\otimes I_n)
\begin{pmatrix}
\tilde U_0&0\\
0&\tilde V_0
\end{pmatrix}
(A^{1/2}\otimes I_n).
\end{equation}
Then, writing $B_n:=\bigl[p\otimes a\ \ q\otimes b\bigr]\in \RR^{2n\times 2}$, a direct expansion gives
\begin{equation}\label{eq:wigner-rank-two-decomp}
H=H_0+B_n\Theta B_n^\intercal.
\end{equation}
Let $T(z)$ be the deterministic $2\times 2$ solution from Definition~\ref{def:det_system_Wigner}. For real $z>1$, define
\begin{equation}\label{eq:wigner-Sigma-rho}
\cS(z)=
\begin{pmatrix}
1+\alpha^2 p^\intercal T(z)p & \rho\alpha\beta\, p^\intercal T(z)q\\
\rho\alpha\beta\, q^\intercal T(z)p & 1+\beta^2 q^\intercal T(z)q
\end{pmatrix}.
\end{equation}
Equivalently, we can write the matrix as 
\[
\cS(z)
=
I_2+
\Theta^{1/2}
\begin{pmatrix}
 p^\intercal T(z)p & \rho\, p^\intercal T(z)q\\
 \rho\, q^\intercal T(z)p & q^\intercal T(z)q
\end{pmatrix}
\Theta^{1/2}.
\]

The analogue of Lemma~\ref{lem:S-has-negative-determinant-at-edge} is immediate in the Wigner case.
\begin{lemma}\label{lem:wigner-S-has-negative-determinant-at-edge}
Let $\cS$ be defined by \eqref{eq:wigner-Sigma-rho}. Then, we have that 
\[
\cS(1)=
\begin{pmatrix}
1-\alpha^2 & -\rho\alpha\beta\kappa\\
-\rho\alpha\beta\kappa & 1-\beta^2
\end{pmatrix}.
\]
Hence, $\det \cS(1)
=(1-\alpha^2)(1-\beta^2)-\rho^2\alpha^2\beta^2\kappa^2$, which is strictly negative if $\kappa < \rho$.
\end{lemma}

\begin{proof}
By Corollary~\ref{cor:wigner-critical-identities-outlier}, we have $T(1)=-I_2$. Since, we also have 
\[
p^\intercal p=q^\intercal q=1,
\qquad
p^\intercal q=e_1^\intercal A e_2=\kappa,
\]
substituting $T(1)=-I_2$ into \eqref{eq:wigner-Sigma-rho} yields
\[
p^\intercal T(1)p=-1,
\qquad
q^\intercal T(1)q=-1,
\qquad
p^\intercal T(1)q=-\kappa.
\]
This gives the displayed formula for $\cS(1)$.
\end{proof}

We can now state the Wigner analogue of Theorems~\ref{thm:CCA-outlier1iAppendix} and \ref{thm:CCA_eigVecOverlapAppendix}.

\begin{theorem}[Theorem~\ref{thm:Spiked-outlierMainBody} (i)-(ii) CSWig]
\label{thm:Wigner-outlier}
Assume that $\rho>\kappa$. Then there exists a unique $\lambda_{\out}>1$ such that $\det \cS(\lambda_{\out})=0.$
Moreover, with high probability,
\[
\lambda_1(W)=\lambda_{\out}+o(1)
\qquad\text{and}\qquad
\lambda_2(W)\le 1+o(1).
\]
Furthermore, if $\hat w= (\hat a, \hat b)$ is the unit eigenvector of $W$ corresponding to $\lambda_1(W)$, then the individual overlaps satisfy with high probability
\begin{equation*}
\begin{aligned}
   |\langle \hat a,a\rangle|^2
   &\ge
   \frac{x_{*,1}^2}{\alpha^2(1+\kappa)\,x_*^\intercal\partial_z\cS(\lambda_{\out})x_*}
   +o(1),\\
   |\langle \hat b,b\rangle|^2
   &\ge
   \frac{x_{*,2}^2}{\beta^2(1+\kappa)\,x_*^\intercal\partial_z\cS(\lambda_{\out})x_*}
   +o(1),
\end{aligned}
\end{equation*}
where the unit vector $x_*\in\ker \cS(\lambda_{\out})$ is normalized by equation 
\begin{equation*}
    x_*^\intercal\Theta^{-1/2}\begin{pmatrix}
1&\kappa\rho\\
\kappa\rho&1
\end{pmatrix}^{-1}\Theta^{-1/2}x_*=1.
\end{equation*}
\end{theorem}

\begin{rmk}
    The two deterministic constants on the right hand side of the theorem are strictly positive. In particular, $|\langle\hat a,a\rangle|^2=\Omega(1)$ and $|\langle\hat b,b\rangle|=\Omega(1)$ with high probability.
\end{rmk}

\begin{proof}[Proof of Theorem~\ref{thm:Wigner-outlier}, Theorem~\ref{thm:Spiked-outlierMainBody} (i)-(ii) CSWig]
We follow the same strategy as in the proof of Theorems~\ref{thm:CCA_eigVecOverlapAppendix}, \ref{thm:CCA-outlier1iAppendix}, replacing the Schur complement there by the rank-two perturbation formula coming from \eqref{eq:wigner-rank-two-decomp}. For $z>\lambda_{\max}(H_0)$, let $\cG(z):=(H_0-zI_{2n})^{-1}$.
We define the random $2\times 2$ matrix
\begin{equation}\label{eq:wigner-Schur-comp}
\cS_n(z):=I_2+\Theta^{1/2}B_n^\intercal \cG(z) B_n\Theta^{1/2}.
\end{equation}
Since $H=H_0+B_n\Theta B_n^\intercal$, the matrix determinant lemma gives
\begin{equation}\label{eq:wigner-det-factorization}
\det(H-zI_{2n})
=
\det(H_0-zI_{2n})\det \cS_n(z).
\end{equation}
Consequently, for every $z>\lambda_{\max}(H_0)$, we have that 
\begin{equation}\label{eq:wigner-outlier-eq}
z\in\sigma(H)
\qquad\Longleftrightarrow\qquad
\det \cS_n(z)=0.
\end{equation}
Thus the outlier analysis reduces to the study of the $2\times 2$ matrix $\cS_n(z)$ as well as its limit. 
\vspace{0.5\baselineskip}

\noindent
\textbf{Step 1: convergence of $\mathcal S_n(z)$.}
We now show the convergence of $\mathcal S_n$ and identify its deterministic limit. Fix a compact interval $K\subset (1,\infty)$. Write $\cG(z)$ in $n\times n$ blocks, as follows 
\[
\cG(z)=
\begin{pmatrix}
\cG^{(11)}(z) & \cG^{(12)}(z)\\
\cG^{(21)}(z) & \cG^{(22)}(z)
\end{pmatrix}.
\]
By expanding against the columns of $B_n$, we can write 
\begin{align*}
(p\otimes a)^\intercal \cG(z)(p\otimes a)
&=
\sum_{i,j=1}^2 p_i p_j\, a^\intercal \cG^{(ij)}(z)a,\\
(q\otimes b)^\intercal \cG(z)(q\otimes b)
&=
\sum_{i,j=1}^2 q_i q_j\, b^\intercal \cG^{(ij)}(z)b,\\
(p\otimes a)^\intercal \cG(z)(q\otimes b)
&=
\sum_{i,j=1}^2 p_i q_j\, a^\intercal \cG^{(ij)}(z)b.
\end{align*}
We claim that, uniformly in $z\in K$ and $i,j\in\{1,2\}$, we have that with overwhelming probability
\begin{align}
 a^\intercal \cG^{(ij)}(z)a
 &= \frac1n\Tr\,\cG^{(ij)}(z)+o(1),
 \label{eq:wigner-quad-a}
 \\
 b^\intercal \cG^{(ij)}(z)b
 &= \frac1n\Tr\,\cG^{(ij)}(z)+o(1),
 \label{eq:wigner-quad-b}
 \\
 a^\intercal \cG^{(ij)}(z)b
 &= \rho\,\frac1n\Tr\,\cG^{(ij)}(z)+o(1)
 \label{eq:wigner-bilin-ab}
\end{align}
Indeed, by Theorem~\ref{thm:WignerGordonEdge},  $\lambda_{\mathrm{max}}(H_0) \leq 1 + o(1)$ with overwhelming probability, so the the resolvent satisfies
\[
\sup_{z\in K}\|\cG(z)\|_{\op}=O_K(1),
\qquad
\sup_{z\in K}\|\partial_z\cG(z)\|_{\op}=O_K(1),
\]
where we used $\partial_z\cG(z)=\cG(z)^2$. Conditional on $H_0$, the blocks $\cG^{(ij)}(z)$ are deterministic and independent of $(a,b)$. Write
\[
a=\frac{g}{\sqrt n},
\qquad
b=\frac{\rho g+\sqrt{1-\rho^2}\,h}{\sqrt n},
\]
where $g,h\sim\calN(0,I_n)$ are independent. For every symmetric $A$ with $\|A\|_{\op}\le C$, Hanson-Wright applied to $A$ gives
\[
\PP\paren{\abs{a^\intercal Aa-\frac1n\Tr A}>t}
\le 2e^{- (t^2 \wedge\,  t) \times \Omega(n) \,},
\]
and the same bound holds with $a$ replaced by $b$. Moreover,
\[
a^\intercal Ab-\rho\frac1n\Tr A
=
\rho\paren{a^\intercal Aa-\frac1n\Tr A}
+\sqrt{1-\rho^2}\,\frac1n g^\intercal Ah,
\]
and, conditional on $g$, the last term is centered Gaussian with variance at most $O_K(n^{-2}\|g\|^2)$. Applying these bounds to $A=\cG^{(ij)}(z)$ for a fine net of $z\in K$, and then using the uniform Lipschitz bounds in $z$ from Proposition~\ref{thm:real-stieltjes-transform-Wigner}, proves \eqref{eq:wigner-quad-a}--\eqref{eq:wigner-bilin-ab}. Now combining \eqref{eq:wigner-quad-a}--\eqref{eq:wigner-bilin-ab} with Theorem~\ref{thm:real-stieltjes-transform-Wigner}, we obtain
\begin{equation}\label{eq:wigner-Schur-limit}
\sup_{z\in K}\|\cS_n(z)-\cS(z)\|_{\op}=o(1)
\end{equation}
holds with overwhelming probability, where $\cS$ is defined in \eqref{eq:wigner-Sigma-rho}. Borel-Cantelli lemma implies that $\mathcal S_n(z)\to \mathcal S(z)$ for all $z\in \mathcal K$ almost surely. 
\vspace{0.5\baselineskip}

\noindent
\textbf{Step 2: properties of the limiting object.} 
We now use the almost sure convergence result to study the deterministic matrix $\cS(z)$. By Definition~\ref{def:det_system_Wigner}, $T(z)$ is real symmetric for real $z>1$, so $\cS(z)$ is a real symmetric $2\times 2$ matrix for such $z$. By Lemma~\ref{lem:wigner-S-has-negative-determinant-at-edge}, when $\rho>\kappa$ we have $\det \cS(1)<0$. Since the diagonal entries of $\cS(1)$ are $1-\alpha^2>0$ and $1-\beta^2>0$, it follows that $\cS(1)$ has exactly one positive eigenvalue and exactly one negative eigenvalue.

On the other hand, since $T$ is a Stieltjes transform, as $z\to\infty$ we have
\[
T(z)=-\frac1zI_2+O(z^{-2}) \quad \text{ and hence } \quad \cS(z)=I_2-\frac1z\Theta+O(z^{-2}).
\]
In particular, $\cS(z)$ is positive definite for all sufficiently large $z$.

To prove uniqueness of the zero of $\det \cS(z)$, we show that each eigenvalue of $\cS(z)$ is strictly increasing on $(1,\infty)$. Fix a compact interval $K=[1+\delta,R]\subset (1,\infty)$. Differentiating \eqref{eq:wigner-Schur-comp}, we obtain
\begin{equation}\label{eq:wigner-SN-derivative}
\partial_z\cS_n(z)
=
\Theta^{1/2}B_n^\intercal (H_0-zI_{2n})^{-2}B_n\Theta^{1/2}
\succ 0,
\quad \text{ for } z>\lambda_{\max}(H_0).
\end{equation}
Moreover, by Theorem~\ref{thm:WignerGordonEdge}, we have that $\|H_0\|_{\op}=1 + o(1)$ with overwhelming probability. Also, note that 
\begin{equation}\label{eq:wigner-B-gram}
B_n^\intercal B_n
=
\begin{pmatrix}
\|a\|^2 & \kappa\ang{a,b}\\
\kappa\ang{a,b} & \|b\|^2
\end{pmatrix}
    \stackrel{\text{in prob.}}{\longrightarrow}
\begin{pmatrix}
1&\kappa\rho\\
\kappa\rho&1
\end{pmatrix} =: M_\rho,
\end{equation}
 and $M_\rho\succ 0$ because $\kappa\rho<1$. In particular, $B_n$ has full column rank. Therefore there exists deterministic constant $c_K > 0$ such that with overwhelming probability, we have 
\[
\partial_z\cS_n(z)
\succeq c_K I_2
\quad \text{ uniformly for }z\in K.
\]
We now pass this lower bound to the limit. Choose $r>0$ so small that
\[
\mathfrak{B}_{3r}(K):=\setcond{z\in\CC}{\mathrm{dist}(z,K)\le 3r}
\subset \DD,
\qquad
\Re z>1+\delta/2\quad\text{for all }z\in \mathfrak{B}_{3r}(K).
\]
On the event $\lambda_{\max}(H_0)\le 1+\delta/4$, the entries of $\cS_n(z)$ are holomorphic on $\mathfrak{B}_{3r}(K)$ and locally uniformly bounded in $n$. Since \eqref{eq:wigner-Schur-limit} gives convergence on the real interval $K$, Borel-Cantelli and Vitali's theorem (with same argument as before) implies that almost surely $\cS_n\to \cS$ locally uniformly on $\mathfrak{B}_{3r}(K)$, entrywise. Cauchy's integral formula therefore yields
\begin{equation}\label{eq:WignerDerivativeSConvergence}
\sup_{z\in K}\|\partial_z\cS_n(z)-\partial_z\cS(z)\|_{\op}\longrightarrow 0
\end{equation}
almost surely. Hence $\partial_z\cS(z)\succeq c_K I_2$, for $z \in K$. Because $K\subset (1,\infty)$ was arbitrary, every eigenvalue of $\cS(z)$ is strictly increasing on $(1,\infty)$. Since one eigenvalue is negative at $z=1$ and both are positive for large $z$, there exists a unique $\lambda_{\out}>1$ such that $\det \cS(\lambda_{\out})=0.$ This proves existence and uniqueness of the deterministic outlier location.
\vspace{0.5\baselineskip}

\noindent
\textbf{Step 3: Convergence of outlier and eigenvector overlap.} We now return to the random matrix $\cS_n$. Choose $\delta>0$ such that $\lambda_\out > 1 + \delta$ and such that the smallest eigenvalue of $\cS(z)$ is negative at $z=\lambda_{\out}-\delta$ and positive at $z=\lambda_{\out}+\delta$. By \eqref{eq:wigner-Schur-limit}, the same holds for $\cS_n(z)$ with overwhelming probability for all sufficiently large $n$. On the same event, Theorem~\ref{thm:WignerGordonEdge} gives $\lambda_{\max}(H_0)\le 1+o(1)<\lambda_{\out}-\delta$. Therefore, by \eqref{eq:wigner-outlier-eq}, the matrix $H$ has a unique eigenvalue $\lambda_{\out,n}$ in $(\lambda_{\out}-\delta,\lambda_{\out}+\delta)$, and $\lambda_{\out,n}\to\lambda_{\out}$ in probability.

To prove that there is no second outlier, fix $\delta>0$ so small that $\det \cS(1+\delta)<0$. Then, by \eqref{eq:wigner-Schur-limit}, with high probability the matrix $\cS_n(1+\delta)$ also has one positive and one negative eigenvalue. Since both eigenvalues of $\cS_n(z)$ are strictly increasing in $z$ by \eqref{eq:wigner-SN-derivative}, only the smaller eigenvalue can cross zero, and it can do so at most once. By \eqref{eq:wigner-outlier-eq}, this implies that $H$ has exactly one eigenvalue above $1+\delta$. Hence, with high probability, $\lambda_1(H)=\lambda_{\out,n}$ and $\lambda_2(H)\le 1+\delta$. Since $\delta>0$ is arbitrary and $W$ is similar to $H$, we conclude that $\lambda_1(W)=\lambda_{\out}+o(1)$ and $
\lambda_2(W)\le 1+o(1)$ with high probability.

It remains to prove the overlap statement. Let $\lambda_{\out,n}:=\lambda_1(H)$ and choose any nonzero vector $x_n\in\ker \cS_n(\lambda_{\out,n})$. Since $B_n^\intercal B_n\succ 0$ with overwhelming probability by \eqref{eq:wigner-B-gram}, we may normalize $x_n$ so that
\begin{equation}\label{eq:wigner-kernel-normalization}
x_n^\intercal\Theta^{-1/2}(B_n^\intercal B_n)^{-1}\Theta^{-1/2}x_n=1.
\end{equation}
Now we define vectors 
\[
\hat r:=\cG(\lambda_{\out,n})B_n\Theta^{1/2}x_n,
\qquad
\hat h:=\frac{\hat r}{\|\hat r\|}.
\]
Using $\cS_n(\lambda_{\out,n})x_n=0$, we compute
\begin{align*}
(H-\lambda_{\out,n}I_{2n})\hat r
&=(H_0-\lambda_{\out,n}I_{2n})\hat r+B_n\Theta B_n^\intercal\hat r \\
&=B_n\Theta^{1/2}x_n+B_n\Theta B_n^\intercal\cG(\lambda_{\out,n})B_n\Theta^{1/2}x_n \\
&=B_n\Theta^{1/2}\cS_n(\lambda_{\out,n})x_n=0.
\end{align*}
Thus $\hat h$ is a top unit eigenvector of $H$. Now, since $\cG(\lambda_{\out,n})$ is symmetric and its derivative is given by \eqref{eq:wigner-SN-derivative}, we get that 
\begin{equation}\label{eq:wigner-normalization-identity}
\|\hat r\|^2
=x_n^\intercal\Theta^{1/2}B_n^\intercal \cG(\lambda_{\out,n})^2 B_n\Theta^{1/2}x_n
=x_n^\intercal\partial_z\cS_n(\lambda_{\out,n})x_n.
\end{equation}
Moreover, note that $B_n^\intercal \hat r
=B_n^\intercal \cG(\lambda_{\out,n})B_n\Theta^{1/2}x_n
=-\Theta^{-1/2}x_n$, again because of the fact that  $\cS_n(\lambda_{\out,n})x_n=0$.  Let $\Pi$ be a projection onto  $\mathrm{span}\set{p\otimes a,\ q\otimes b}.$ Then, $\Pi=B_n(B_n^\intercal B_n)^{-1}B_n^\intercal$, so we obtain
\begin{align*}
\|\Pi\hat h\|^2
&=
\frac{(B_n^\intercal \hat r)^\intercal (B_n^\intercal B_n)^{-1}(B_n^\intercal \hat r)}{\|\hat r\|^2} \\
&=
\frac{x_n^\intercal\Theta^{-1/2}(B_n^\intercal B_n)^{-1}\Theta^{-1/2}x_n}{x_n^\intercal\partial_z\cS_n(\lambda_{\out,n})x_n} \\
&=
\frac{1}{x_n^\intercal\partial_z\cS_n(\lambda_{\out,n})x_n},
\end{align*}
where in the last step we used the normalization \eqref{eq:wigner-kernel-normalization}.

By compactness and \eqref{eq:wigner-kernel-normalization}, along every subsequence we may extract a further subsequence such that $x_n\to x_*$. From \eqref{eq:wigner-Schur-limit}, the convergence $\lambda_{\out,n}\to\lambda_{\out}$ in probability, and the identity $\cS_n(\lambda_{\out,n})x_n=0$, it follows that $\cS(\lambda_{\out})x_*=0$. Moreover, by \eqref{eq:wigner-B-gram}, the normalization \eqref{eq:wigner-kernel-normalization} passes to the limit and gives
\[
x_*^\intercal\Theta^{-1/2}M_\rho^{-1}\Theta^{-1/2}x_*=1.
\]
Since $\ker \cS(\lambda_{\out})$ is one-dimensional, this determines $x_*$ uniquely up to sign. Using again the derivative convergence obtained above \eqref{eq:WignerDerivativeSConvergence}, we conclude that
\[
\|\Pi\hat h\|^2
\stackrel{\text{in prob.}}{\longrightarrow}
\frac{1}{x_*^\intercal\partial_z\cS(\lambda_{\out})x_*}.
\]
The limit is strictly positive because $\partial_z\cS(\lambda_{\out})\succ 0$. 

We record that neither coordinate of $x_*$ can vanish. Indeed, the preceding monotonicity argument in Step 2 gives $\partial_z\cS(z)\succ0$ for every $z>1$. Since
\[
\cS_{11}(1)=1-\alpha^2>0,
\qquad
\cS_{22}(1)=1-\beta^2>0,
\]
both diagonal entries $\cS_{11}(\lambda_{\out})$ and $\cS_{22}(\lambda_{\out})$ are strictly positive. If a nonzero vector in $\ker\cS(\lambda_{\out})$ were supported on only one coordinate, one of these diagonal entries would have to be zero. Hence $x_{*,1}\neq0$ and $x_{*,2}\neq0$. 

It remains to translate this into individual overlaps for the right eigenvector of the original matrix $W$. Let $S:=A^{1/2}\otimes I_n$ and $C_n:=\bigl[e_1\otimes a\ \ e_2\otimes b\bigr]$ and $B_n=SC_n$. The corresponding unit right eigenvector of $W$ is
\[
\hat w=(\hat a,\hat b):=\frac{S\hat h}{\|S\hat h\|},
\]
up to an overall sign. Since $S$ is symmetric and $B_n=SC_n$, we have
\begin{equation}\label{eq:wigner-individual-overlap-vector}
\begin{pmatrix}
\langle\hat a,a\rangle\\
\langle\hat b,b\rangle
\end{pmatrix}
=C_n^\intercal\hat w
=\frac{B_n^\intercal\hat h}{\|S\hat h\|}
=-\frac{\Theta^{-1/2}x_n}{\|\hat r\|\,\|S\hat h\|},
\end{equation}
where the last identity uses $B_n^\intercal\hat r=-\Theta^{-1/2}x_n$ and $\hat h=\hat r/\|\hat r\|$. Combining \eqref{eq:wigner-individual-overlap-vector} with \eqref{eq:wigner-normalization-identity} gives the finite-$n$ identities
\begin{equation}\label{eq:wigner-individual-overlap-finite}
\begin{aligned}
|\langle\hat a,a\rangle|^2
&=
\frac{x_{n,1}^2}{\alpha^2\,x_n^\intercal\partial_z\cS_n(\lambda_{\out,n})x_n\,\|S\hat h\|^2},\\
|\langle\hat b,b\rangle|^2
&=
\frac{x_{n,2}^2}{\beta^2\,x_n^\intercal\partial_z\cS_n(\lambda_{\out,n})x_n\,\|S\hat h\|^2}.
\end{aligned}
\end{equation}
Because $\rho>\kappa$ and $\rho\le1$, we have $\kappa<1$, and therefore
\[
\|S\hat h\|^2\le \|S\|_{\op}^2=\|A\|_{\op}=1+\kappa.
\]
Using $x_{n}\to x_*$ and the derivative convergence \eqref{eq:WignerDerivativeSConvergence} at $\lambda_{\out}$, \eqref{eq:wigner-individual-overlap-finite} yields
\begin{equation*}
\begin{aligned}
   |\langle \hat a,a\rangle|^2
   &\ge
   \frac{x_{*,1}^2}{\alpha^2(1+\kappa)\,x_*^\intercal\partial_z\cS(\lambda_{\out})x_*}
   +o(1),\\
   |\langle \hat b,b\rangle|^2
   &\ge
   \frac{x_{*,2}^2}{\beta^2(1+\kappa)\,x_*^\intercal\partial_z\cS(\lambda_{\out})x_*}
   +o(1).
\end{aligned}
\end{equation*}
The constants on the right hand side are strictly positive by $\partial_z\cS(\lambda_{\out})\succ0$. This proves the theorem.
\end{proof}

\subsubsection{Mismatched Wigner outlier}
\label{subsubsec:mismatched-wigner-outlier}

In this subsection, we prove the form of the Wigner outlier argument that is
needed for the paramete free grid search.  The true model is still
\begin{equation}\label{eq:mismatched-wigner-true-model}
    U=\alpha aa^\intercal+Z_a,
    \qquad
    V=\beta bb^\intercal+Z_b,
\end{equation}
where the true parameters $\alpha,\beta\in(0,1)$ are unknown.  We build the
spectral matrix using a possibly different pair
$(\tilde\alpha,\tilde\beta)$.  Throughout this subsection the pair
$(\tilde\alpha,\tilde\beta)$ is fixed and admissible, meaning that
\begin{equation}\label{eq:mismatched-kappa-def}
    0<\tilde\alpha,\tilde\beta<1,
    \qquad
    \tilde\kappa^4
    =\frac{(1-\tilde\alpha^2)(1-\tilde\beta^2)}
    {\tilde\alpha^2\tilde\beta^2},
    \qquad
    \tilde\kappa\in(0,1).
\end{equation}
Grid points for which $\tilde\kappa\notin(0,1)$ are not used in the outlier
argument.  Define
\begin{equation}\label{eq:mismatched-normalized-matrices}
\begin{aligned}
     U_{\tilde\alpha}:=\tilde\alpha U-\tilde\alpha^2 I_n \\
     V_{\tilde\beta}:=\tilde\beta V-\tilde\beta^2 I_n 
\end{aligned}
   \quad \text{ and } \quad 
     W_{\tilde\alpha,\tilde\beta}
    :=
    \begin{pmatrix}
        U_{\tilde\alpha} & \tilde\kappa V_{\tilde\beta}\\
        \tilde\kappa U_{\tilde\alpha} & V_{\tilde\beta}
    \end{pmatrix}.
\end{equation}
Let us define 
\begin{equation}\label{eq:mismatched-A-pq-def}
    \tilde A:=
    \begin{pmatrix}
        1&\tilde\kappa\\
        \tilde\kappa&1
    \end{pmatrix},
    \qquad
    \tilde p:=\tilde A^{1/2}e_1,
    \qquad
    \tilde q:=\tilde A^{1/2}e_2,
    \qquad
    \tilde S:=\tilde A^{1/2}\otimes I_n.
\end{equation}
The symmetric conjugate of $W_{\tilde\alpha,\tilde\beta}$ is
\begin{equation}\label{eq:mismatched-H-def}
    \tilde H
    :=
    \tilde S
    \begin{pmatrix}
        U_{\tilde\alpha}&0\\
        0&V_{\tilde\beta}
    \end{pmatrix}
    \tilde S \implies   W_{\tilde\alpha,\tilde\beta}=\tilde S\tilde H\tilde S^{-1},
\end{equation}
so $W_{\tilde\alpha,\tilde\beta}$ and $\tilde H$ have the same eigenvalues. Now define the corresponding null quantities
\begin{equation}\label{eq:mismatched-null-blocks}
\begin{aligned}
      U_{\tilde\alpha,0}:=\tilde\alpha Z_a-\tilde\alpha^2I_n \\
       V_{\tilde\beta,0}:=\tilde\beta Z_b-\tilde\beta^2I_n,
\end{aligned}
\quad \text{ and } \quad 
  \tilde H_0
    :=
    \tilde S
    \begin{pmatrix}
        U_{\tilde\alpha,0}&0\\
        0&V_{\tilde\beta,0}
    \end{pmatrix}
    \tilde S.
\end{equation}
The only change relative to the matched analysis is the strength of the rank two
perturbation.  Namely, if $\tilde B_n:=\bigl[\tilde p\otimes a\ \ \tilde q\otimes b\bigr]$ and $\tilde\Theta:=\diag(\alpha\tilde\alpha,\beta\tilde\beta)$, then expanding \eqref{eq:mismatched-wigner-true-model} gives
\begin{equation}\label{eq:mismatched-rank-two-decomp}
    \tilde H=\tilde H_0+\tilde B_n\tilde\Theta\tilde B_n^\intercal.
\end{equation}
Thus the matched matrix $\Theta=\diag(\alpha^2,\beta^2)$ is replaced by
$\tilde\Theta=\diag(\alpha\tilde\alpha,\beta\tilde\beta)$, while $p,q,\kappa$
are replaced by $\tilde p,\tilde q,\tilde\kappa$. Let $(\tilde r,\tilde s,\tilde T)$ denote the deterministic Wigner solution from
Definition~\ref{def:det_system_Wigner} with
$(\alpha,\beta,\kappa,p,q)$ replaced by
$(\tilde\alpha,\tilde\beta,\tilde\kappa,\tilde p,\tilde q)$.  Equivalently,
\begin{gather}\label{eq:mismatched-finite-system}
    \tilde r(z)=\tilde p^\intercal\tilde T(z)\tilde p,
    \qquad
    \tilde s(z)=\tilde q^\intercal\tilde T(z)\tilde q,\\
    \tilde T(z)
    \Bigl(
    -zI_2
    -\tilde\alpha^2(1+\tilde r(z))\tilde p\tilde p^\intercal
    -\tilde\beta^2(1+\tilde s(z))\tilde q\tilde q^\intercal
    \Bigr)=I_2.
\end{gather}
For real $z>1$, define
\begin{equation}\label{eq:mismatched-Sigma-rho}
\tilde\cS(z)
:=
I_2+
\tilde\Theta^{1/2}
\begin{pmatrix}
 \tilde p^\intercal\tilde T(z)\tilde p
 &
 \rho\,\tilde p^\intercal\tilde T(z)\tilde q\\
 \rho\,\tilde q^\intercal\tilde T(z)\tilde p
 &
 \tilde q^\intercal\tilde T(z)\tilde q
\end{pmatrix}
\tilde\Theta^{1/2}.
\end{equation}
The mismatched threshold associated with the grid point
$(\tilde\alpha,\tilde\beta)$ is
\begin{equation}\label{eq:mismatched-rho-out-def}
    \rho_{\out}(\tilde\alpha,\tilde\beta;\alpha,\beta)
    :=
    \frac{
    \sqrt{(1-\alpha\tilde\alpha)(1-\beta\tilde\beta)}
    }{
    \tilde\kappa\sqrt{\alpha\beta\tilde\alpha\tilde\beta}
    }.
\end{equation}
At the true parameters, this reduces to the original threshold:
$\rho_{\out}(\alpha,\beta;\alpha,\beta)=\kappa$.

\begin{lemma}\label{lem:mismatched-S-edge}
For $\tilde\cS$ defined in \eqref{eq:mismatched-Sigma-rho},
\begin{equation}\label{eq:mismatched-S-edge}
\tilde\cS(1)=
\begin{pmatrix}
1-\alpha\tilde\alpha
&
-\rho\tilde\kappa\sqrt{\alpha\beta\tilde\alpha\tilde\beta}\\
-\rho\tilde\kappa\sqrt{\alpha\beta\tilde\alpha\tilde\beta}
&
1-\beta\tilde\beta
\end{pmatrix}.
\end{equation}
Consequently, $\det \tilde\cS(1)
    =(1-\alpha\tilde\alpha)(1-\beta\tilde\beta)
    -\rho^2\tilde\kappa^2\alpha\beta\tilde\alpha\tilde\beta$. In particular, $\det\tilde\cS(1)<0$ whenever
$\rho>\rho_{\out}(\tilde\alpha,\tilde\beta;\alpha,\beta)$.
\end{lemma}

\begin{proof}
Apply Corollary~\ref{cor:wigner-critical-identities-outlier} to the null system with the
parameters $(\tilde\alpha,\tilde\beta,\tilde\kappa)$.  This gives
$\tilde T(1)=-I_2$.  Since $\tilde p^\intercal\tilde p=\tilde q^\intercal\tilde q=1$ and $\tilde p^\intercal\tilde q=e_1^\intercal \tilde A e_2=\tilde\kappa$ substitution into \eqref{eq:mismatched-Sigma-rho} gives
\eqref{eq:mismatched-S-edge}. The final claim is exactly the definition
\eqref{eq:mismatched-rho-out-def}.
\end{proof}

\begin{theorem}[Mismatched Wigner outlier]
\label{thm:mismatched-Wigner-outlier}
Fix true parameters $0<\alpha,\beta<1$, and let
$(\tilde\alpha,\tilde\beta)$ be admissible in the sense of
\eqref{eq:mismatched-kappa-def}.  Suppose $\rho>\rho_{\out}(\tilde\alpha,\tilde\beta;\alpha,\beta).$ Then there exists a unique $\tilde\lambda_{\out}>1$ such that $ \det\tilde\cS(\tilde\lambda_{\out})=0.$ In fact, w.h.p,
\[
    \lambda_1(W_{\tilde\alpha,\tilde\beta})
    =\tilde\lambda_{\out}+o(1),
    \qquad
    \lambda_2(W_{\tilde\alpha,\tilde\beta})\le 1+o(1).
\]
Furthermore, if $\tilde w=(\hat a,\hat b)$ is a unit eigenvector of $W_{\tilde\alpha,\tilde\beta}$ corresponding to $\lambda_1(W_{\tilde\alpha,\tilde\beta})$, then with high probability,
\begin{equation*}
\begin{aligned}
   |\langle\hat a,a\rangle|^2
   &\ge
   \frac{\tilde x_{*,1}^2}{\alpha\tilde\alpha(1+\tilde\kappa)\,\tilde x_*^\intercal\partial_z\tilde\cS(\tilde\lambda_{\out})\tilde x_*}
   +o(1),\\
   |\langle\hat b,b\rangle|^2
   &\ge
   \frac{\tilde x_{*,2}^2}{\beta\tilde\beta(1+\tilde\kappa)\,\tilde x_*^\intercal\partial_z\tilde\cS(\tilde\lambda_{\out})\tilde x_*}
   +o(1),
\end{aligned}
\end{equation*}
where $\tilde x_* \in \ker\tilde\cS(\tilde\lambda_{\out}) $ is normalized by 
\begin{equation*}
    \tilde x_*^\intercal\tilde\Theta^{-1/2} \begin{pmatrix}
1&\tilde\kappa\rho\\
\tilde\kappa\rho&1
\end{pmatrix}^{-1}\tilde\Theta^{-1/2}\tilde x_*=1
\end{equation*}
\end{theorem}

\begin{rmk}
    The two deterministic constants on the right-hand side are strictly positive. In particular, $|\langle\hat a,a\rangle|^2=\Omega(1)$ and $|\langle\hat b,b\rangle|=\Omega(1)$ with high probability.
\end{rmk}

\begin{proof}
We repeat the proof of Theorem~\ref{thm:Wigner-outlier}, keeping only the changes caused
by the mismatch.  For $z>\lambda_{\max}(\tilde H_0)$, set
\[
    \tilde\cG_0(z):=(\tilde H_0-zI_{2n})^{-1}
\]
and define the random $2\times2$ Schur complement
\begin{equation}\label{eq:mismatched-Schur-comp}
    \tilde\cS_n(z)
    :=I_2+
    \tilde\Theta^{1/2}\tilde B_n^\intercal
    \tilde\cG_0(z)
    \tilde B_n\tilde\Theta^{1/2}.
\end{equation}
By the matrix determinant lemma and the rank-two decomposition
\eqref{eq:mismatched-rank-two-decomp},
\begin{equation}\label{eq:mismatched-det-factorization}
\det(\tilde H-zI_{2n})
=
\det(\tilde H_0-zI_{2n})\det\tilde\cS_n(z).
\end{equation}
Thus every eigenvalue of $\tilde H$, equivalently of
$W_{\tilde\alpha,\tilde\beta}$, lying above $\lambda_{\max}(\tilde H_0)$ is a
zero of $\det\tilde\cS_n(z)$.

The convergence of $\tilde\cS_n$ is exactly the convergence proved in the first
half of Theorem~\ref{thm:Wigner-outlier}, with
$(\alpha,\beta,\kappa,p,q,\Theta,B_n,H_0)$ replaced by
$(\tilde\alpha,\tilde\beta,\tilde\kappa,\tilde p,\tilde q,
\tilde\Theta,\tilde B_n,\tilde H_0)$.  More explicitly, if
$K\subset(1,\infty)$ is compact, then
\begin{equation}\label{eq:mismatched-Schur-limit}
    \sup_{z\in K}\|\tilde\cS_n(z)-\tilde\cS(z)\|_{\op}\longrightarrow0
\end{equation}
almost surely. The reason is the same as before: conditional on
$\tilde H_0$, the block resolvent $\tilde\cG_0^{(ij)}(z)$ is independent of
$(a,b)$, and the quadratic-form estimates
\begin{align*}
 a^\intercal\tilde\cG_0^{(ij)}(z)a
 &=\frac1n\Tr\tilde\cG_0^{(ij)}(z)+o(1),\\
 b^\intercal\tilde\cG_0^{(ij)}(z)b
 &=\frac1n\Tr\tilde\cG_0^{(ij)}(z)+o(1),\\
 a^\intercal\tilde\cG_0^{(ij)}(z)b
 &=\rho\,\frac1n\Tr\tilde\cG_0^{(ij)}(z)+o(1)
\end{align*}
hold uniformly on $K$.  Combining these estimates with
Theorem~\ref{thm:real-stieltjes-transform-Wigner}, applied to the null model with
parameters $(\tilde\alpha,\tilde\beta)$, gives that with overwhelming probability,
\[
\sup_{z\in K}\|\tilde\cS_n(z)-\tilde\cS(z)\|_{\op}=o(1),
\]
and thus
\eqref{eq:mismatched-Schur-limit} holds almost surely by Borel-Cantelli lemma. 

We next analyze the deterministic matrix $\tilde\cS$.  By
Lemma~\ref{lem:mismatched-S-edge} and the assumption
$\rho>\rho_{\out}(\tilde\alpha,\tilde\beta;\alpha,\beta)$, we have
$\det\tilde\cS(1)<0$.  Since the diagonal entries of $\tilde\cS(1)$ are
$1-\alpha\tilde\alpha>0$ and $1-\beta\tilde\beta>0$, the matrix
$\tilde\cS(1)$ has one positive and one negative eigenvalue.  On the other hand,
from the large-$z$ asymptotics of the Wigner deterministic system,
\[
    \tilde T(z)=-\frac1zI_2+O(z^{-2}),
    \qquad
    \tilde\cS(z)=I_2+O(z^{-1}),
\]
so $\tilde\cS(z)$ is positive definite for all sufficiently large real $z$.

It remains to recall the monotonicity argument.  Differentiating
\eqref{eq:mismatched-Schur-comp} gives, for
$z>\lambda_{\max}(\tilde H_0)$,
\begin{equation}\label{eq:mismatched-SN-derivative}
    \partial_z\tilde\cS_n(z)
    =
    \tilde\Theta^{1/2}\tilde B_n^\intercal
    (\tilde H_0-zI_{2n})^{-2}
    \tilde B_n\tilde\Theta^{1/2}
    \succeq0.
\end{equation}
Moreover,
\begin{equation}\label{eq:mismatched-B-gram}
    \tilde B_n^\intercal\tilde B_n
    =
    \begin{pmatrix}
        \|a\|^2&\tilde\kappa\langle a,b\rangle\\
        \tilde\kappa\langle a,b\rangle&\|b\|^2
    \end{pmatrix}
    \stackrel{\text{in prob.}}{\longrightarrow}
    \tilde M_\rho
    :=
    \begin{pmatrix}
        1&\tilde\kappa\rho\\
        \tilde\kappa\rho&1
    \end{pmatrix},
\end{equation}
and $\tilde M_\rho\succ0$ because $\tilde\kappa\rho<1$.  Consequently, on every compact interval
$K=[1+\delta,R]\subset(1,\infty)$, the same argument used after
\eqref{eq:wigner-SN-derivative} gives a deterministic constant $c_K>0$ such
that, with overwhelming probability,
\[
    \partial_z\tilde\cS_n(z)\succeq c_K I_2,
    \qquad z\in K.
\]
Indeed, on the event $\lambda_{\max}(\tilde H_0)\le 1+\delta/4$ and
$\|\tilde H_0\|_{\op}=O(1)$, the matrix
$(\tilde H_0-zI)^{-2}$ is bounded below by a positive multiple of the identity
uniformly for $z\in K$, while
$\tilde\Theta^{1/2}\tilde B_n^\intercal\tilde B_n\tilde\Theta^{1/2}$ is
positive definite with a limiting spectral gap.  Passing this lower bound to
$\partial_z\tilde\cS(z)$ uses the same Vitali/Cauchy derivative convergence (as in the proof of Theorem~\ref{thm:Wigner-outlier} above).  Therefore the eigenvalues of
$\tilde\cS(z)$ are strictly increasing on $(1,\infty)$.

Since $\tilde\cS(1)$ has one negative eigenvalue and $\tilde\cS(z)$ is positive
definite for large $z$, this strict monotonicity gives a unique
$\tilde\lambda_{\out}>1$ with
$\det\tilde\cS(\tilde\lambda_{\out})=0$.

We now return to the random matrix $\tilde{S}_n$.  By Theorem~\ref{thm:WignerGordonEdge},
applied with $(\tilde\alpha,\tilde\beta,\tilde\kappa)$,
\begin{equation}\label{eq:mismatched-null-edge}
    \lambda_{\max}(\tilde H_0)\le 1+o(1)
\end{equation}
with overwhelming probability.  The uniform convergence
\eqref{eq:mismatched-Schur-limit} near $\tilde\lambda_{\out}$ implies that
$\det\tilde\cS_n(z)$ has a zero
$\tilde\lambda_{\out,n}=\tilde\lambda_{\out}+o(1)$.  By
\eqref{eq:mismatched-det-factorization}, this zero is an eigenvalue of
$\tilde H$ and hence of $W_{\tilde\alpha,\tilde\beta}$.

Finally, to exclude a second outlier, choose $\delta>0$ so small that
$\det\tilde\cS(1+\delta)<0$.  Then \eqref{eq:mismatched-Schur-limit} implies that
$\tilde\cS_n(1+\delta)$ has one positive and one negative eigenvalue with overwhelming
probability.  Since the eigenvalues of $\tilde\cS_n(z)$ are increasing for
$z>\lambda_{\max}(\tilde H_0)$ by \eqref{eq:mismatched-SN-derivative}, only the
negative eigenvalue can cross zero, and it can cross at most once.  Therefore
there is exactly one eigenvalue of $\tilde H$ above $1+\delta$.  As $\delta>0$ is
arbitrary and $W_{\tilde\alpha,\tilde\beta}$ is similar to $\tilde H$, we obtain
\[
    \lambda_1(W_{\tilde\alpha,\tilde\beta})
    =\tilde\lambda_{\out}+o(1),
    \qquad
    \lambda_2(W_{\tilde\alpha,\tilde\beta})\le 1+o(1),
\]
with high probability.

It remains to prove the displayed individual-overlap bounds. Choose a nonzero vector $\tilde x_n\in\ker\tilde\cS_n(\tilde\lambda_{\out,n})$ normalized by
\begin{equation}\label{eq:mismatched-kernel-normalization}
\tilde x_n^\intercal\tilde\Theta^{-1/2}(\tilde B_n^\intercal\tilde B_n)^{-1}\tilde\Theta^{-1/2}\tilde x_n=1.
\end{equation}
Set $\tilde r$ and $\tilde h$ to be defined by
\[
\tilde r:=\tilde\cG_0(\tilde\lambda_{\out,n})\tilde B_n\tilde\Theta^{1/2}\tilde x_n,
\qquad
\tilde h:=\frac{\tilde r}{\|\tilde r\|}.
\]
Exactly as in the matched proof, $\tilde h$ is a top unit eigenvector of $\tilde H$, and
\begin{equation}\label{eq:mismatched-normalization-identity}
\|\tilde r\|^2
=\tilde x_n^\intercal\partial_z\tilde\cS_n(\tilde\lambda_{\out,n})\tilde x_n,
\qquad
\tilde B_n^\intercal\tilde r=-\tilde\Theta^{-1/2}\tilde x_n.
\end{equation}
By compactness, the Schur convergence, and \eqref{eq:mismatched-B-gram}, every subsequential limit of $\tilde x_n$ is $\pm\tilde x_*$, where $\tilde x_*$ is the normalized kernel vector from the theorem statement. The same derivative convergence gives that almost surely,
\[
\tilde x_n^\intercal\partial_z\tilde\cS_n(\tilde\lambda_{\out,n})\tilde x_n
\longrightarrow
\tilde x_*^\intercal\partial_z\tilde\cS(\tilde\lambda_{\out})\tilde x_*.
\]
Moreover, neither coordinate of $\tilde x_*$ can vanish: by \eqref{eq:mismatched-S-edge}, the diagonal entries of $\tilde\cS(1)$ are $1-\alpha\tilde\alpha$ and $1-\beta\tilde\beta$, both strictly positive, and the monotonicity argument above gives $\partial_z\tilde\cS(z)\succ0$ for $z>1$. Hence $\tilde x_{*,1}\neq0$ and $\tilde x_{*,2}\neq0$.

Now let $\tilde C_n:=\bigl[e_1\otimes a\ \ e_2\otimes b\bigr]$, so that $\tilde B_n=\tilde S\tilde C_n$. The corresponding unit right eigenvector of $W_{\tilde\alpha,\tilde\beta}$ is
\[
\tilde w=(\hat a,\hat b):=\frac{\tilde S\tilde h}{\|\tilde S\tilde h\|},
\]
up to an overall sign. Therefore
\[
\begin{pmatrix}
\langle\hat a,a\rangle\\
\langle\hat b,b\rangle
\end{pmatrix}
=\tilde C_n^\intercal\tilde w
=\frac{\tilde B_n^\intercal\tilde h}{\|\tilde S\tilde h\|}
=-\frac{\tilde\Theta^{-1/2}\tilde x_n}{\|\tilde r\|\,\|\tilde S\tilde h\|}.
\]
Since $\|\tilde S\tilde h\|^2\le\|\tilde S\|_{\op}^2=1+\tilde\kappa$, the preceding display and \eqref{eq:mismatched-normalization-identity} imply with high probability,
\begin{equation*}
\begin{aligned}
   |\langle\hat a,a\rangle|^2
   &\ge
   \frac{\tilde x_{*,1}^2}{\alpha\tilde\alpha(1+\tilde\kappa)\,\tilde x_*^\intercal\partial_z\tilde\cS(\tilde\lambda_{\out})\tilde x_*}
   +o(1),\\
   |\langle\hat b,b\rangle|^2
   &\ge
   \frac{\tilde x_{*,2}^2}{\beta\tilde\beta(1+\tilde\kappa)\,\tilde x_*^\intercal\partial_z\tilde\cS(\tilde\lambda_{\out})\tilde x_*}
   +o(1)
\end{aligned}
\end{equation*}
The constants are strictly positive by the nonvanishing of the coordinates of $\tilde x_*$ and by $\partial_z\tilde\cS(\tilde\lambda_{\out})\succ0$.
\end{proof}

The preceding theorem shows that every sufficiently well matched grid point has
an outlier above the null edge.  For the grid search estimator, we also need the
converse implication at the level of eigenvectors: any grid point whose top
eigenvalue is genuinely separated from the null edge must have an informative
top eigenvector.

\begin{lemma}
\label{lem:mismatched-separated-outlier-informative}
Fix $\varepsilon>0$, and let $(\tilde\alpha,\tilde\beta)$ be admissible.  Then there is a constant
$c>0$ only depending on $\varepsilon,\tilde\alpha,\tilde\beta,\alpha,\beta,\rho$ such that, with high probability, the following implication holds.  If
$\lambda_1(W_{\tilde\alpha,\tilde\beta})\ge 1+\varepsilon$ and
$w=(\hat a,\hat b)\in\RR^n\times\RR^n$ is a top unit right eigenvector of
$W_{\tilde\alpha,\tilde\beta}$, then
\[
    |\langle\hat a,a\rangle|^2\ge c,
    \qquad
    |\langle\hat b,b\rangle|^2\ge c .
\]
Moreover, for any sufficiently slowly growing size grid contained in a compact subset of the
admissible region, the constant and the high probability event may be chosen
uniformly over all grid points.
\end{lemma}

\begin{proof}

Let $h:={\tilde S^{-1}w}/{\|\tilde S^{-1}w\|}$. By the similarity relation \eqref{eq:mismatched-H-def}, $h$ is a unit top
eigenvector of $\tilde H$ with eigenvalue
$\lambda:=\lambda_1(W_{\tilde\alpha,\tilde\beta})$.  We work on the event
$\{\lambda\ge 1+\varepsilon\}$; if this event does not occur, there is nothing to prove.
By Theorem~\ref{thm:WignerGordonEdge}, applied to the null matrix with the guessed
parameters, $\lambda_{\max}(\tilde H_0)\le 1+o(1)$ with overwhelming probability.  Thus, on a overwhelming probability event and for all sufficiently large $n$,
\[
    \lambda-\lambda_{\max}(\tilde H_0)\ge \varepsilon/2.
\]
Also, $\|\tilde H\|_{\op}\le \|\tilde H_0\|_{\op}+\|\tilde B_n\|_{\op}^2\|\tilde\Theta\|_{\op}=O(1)$ with overwhelming probability, using the null operator norm bound and the Gram convergence \eqref{eq:mismatched-B-gram}.  Hence, for a deterministic constant $R<\infty$, we may assume
\[
    \lambda\in K:=[1+\varepsilon,R]
\]
on the same overwhelming probability event. Let $\tilde\cG_0(z):=(\tilde H_0-zI_{2n})^{-1},$ and let $\tilde\cS_n(z)$ be the finite Schur complement from
\eqref{eq:mismatched-Schur-comp}.  Set $ \tilde y_n:=\tilde B_n^\intercal h$ and $ \tilde x_n:=\tilde\Theta^{1/2}\tilde y_n.$ Using the rank-two decomposition \eqref{eq:mismatched-rank-two-decomp}, the eigenvalue equation gives
\[
    h=-\tilde\cG_0(\lambda)\tilde B_n\tilde\Theta^{1/2}\tilde x_n.
\]
Multiplying on the left by $\tilde\Theta^{1/2}\tilde B_n^\intercal$ gives  $\tilde\cS_n(\lambda)\tilde x_n=0.$ Since $h$ is unit norm, by \eqref{eq:mismatched-SN-derivative},
\begin{equation}\label{eq:mismatched-separated-derivative-normalization}
    1=\|h\|^2
    =\tilde x_n^\intercal \partial_z\tilde\cS_n(\lambda)\tilde x_n .
\end{equation}
On $K$, the same estimates used in the proof of
Theorem~\ref{thm:mismatched-Wigner-outlier} give, with high probability,
\[
    0<c_K I_2\preceq \partial_z\tilde\cS_n(z)\preceq C_K I_2,
    \qquad z\in K,
\]
for deterministic constants $c_K,C_K>0$.  The lower bound is the monotonicity
estimate following \eqref{eq:mismatched-B-gram}; the upper bound follows from
$\|\tilde\cG_0(z)\|_{\op}\le 2/\varepsilon$ and $\|\tilde B_n\|_{\op}=O(1)$.
Thus \eqref{eq:mismatched-separated-derivative-normalization} implies that
$\tilde x_n$ is bounded and bounded away from zero in norm.

We next show that neither coordinate of $\tilde x_n$ can be small.  Suppose, to
the contrary, that along a subsequence on the above high probability events,
$\lambda\in K$ and one coordinate of $\tilde x_n$ tends to zero.  Passing to a
further subsequence, we may assume $\lambda\to z_*\in K$ and
$\tilde x_n\to x_*$ with $x_*\neq0$.  By the uniform Schur convergence
\eqref{eq:mismatched-Schur-limit} and the derivative convergence used in the
proof of Theorem~\ref{thm:mismatched-Wigner-outlier},
\[
    \tilde\cS(z_*)x_*=0,
    \qquad
    x_*^\intercal\partial_z\tilde\cS(z_*)x_*=1 .
\]
However, a nonzero kernel vector of $\tilde\cS(z)$ for any $z>1$ cannot have a
zero coordinate.  Indeed, if $x_1=0$ and $\tilde\cS(z)x=0$, then
$e_2^\intercal\tilde\cS(z)e_2=0$, which is impossible because
$e_2^\intercal\tilde\cS(1)e_2=1-\beta\tilde\beta>0$ by
\eqref{eq:mismatched-S-edge} and $\partial_z\tilde\cS(z)\succ0$ for $z>1$.
The same argument with the two coordinates interchanged rules out $x_2=0$.
This contradiction proves that there is a deterministic $\gamma>0$, depending
only on $\varepsilon,\tilde\alpha,\tilde\beta,\alpha,\beta,\rho$, such that,
with high probability, on the event $\{\lambda\ge1+\varepsilon\}$, we have $ |\tilde x_{n,1}|\ge\gamma$ and $ |\tilde x_{n,2}|\ge\gamma.$ Finally define
\[
    \tilde C_n:=\bigl[e_1\otimes a\ \ e_2\otimes b\bigr],
    \qquad
    \tilde B_n=\tilde S\tilde C_n .
\]
Since $w=\tilde S h/\|\tilde S h\|$, we have the coordinate identity
\[
\begin{pmatrix}
\langle\hat a,a\rangle\\
\langle\hat b,b\rangle
\end{pmatrix}
=\tilde C_n^\intercal w
=\frac{\tilde B_n^\intercal h}{\|\tilde S h\|}
=\frac{\tilde\Theta^{-1/2}\tilde x_n}{\|\tilde S h\|}.
\]
Because $\|\tilde S h\|^2\le\|\tilde S\|_{\op}^2=1+\tilde\kappa$ and
$\tilde\Theta=\diag(\alpha\tilde\alpha,\beta\tilde\beta)$,
the lower bound above gives
\[
    |\langle\hat a,a\rangle|^2
    \ge \frac{\gamma^2}{\alpha\tilde\alpha(1+\tilde\kappa)},
    \qquad
    |\langle\hat b,b\rangle|^2
    \ge \frac{\gamma^2}{\beta\tilde\beta(1+\tilde\kappa)}.
\]
Taking $c$ to be the smaller of the two displayed constants proves the fixed
parameter claim.

For a slowly growing size grid contained in a compact subset of the admissible
region, choose $R$ and the derivative upper and lower bounds uniformly over the
compact set.  The Schur convergence, derivative convergence, Gordon edge
estimate, and Gram concentration are uniform over the grid by a union bound.
The deterministic compactness argument above is also uniform: the set of
possible limits $(\tilde\alpha,\tilde\beta,z,x)$ with $z\in[1+\varepsilon,R]$,
$\tilde\cS(z)x=0$, and
$x^\intercal\partial_z\tilde\cS(z)x=1$ is compact, and the coordinate
nonvanishing argument rules out $x_1=0$ and $x_2=0$ everywhere on this set.
Thus the same positive coordinate lower bound, and hence the same overlap
constant $c$, may be chosen uniformly over all grid points.
\end{proof}

\subsubsection{Parameter free detection and recovery, proof of Theorem~\ref{thm:Spiked-parameter-free}}\label{subsec:parameterFreeDetection}
We note that the proof of the following theorem for the CSWish model is identical, up to citing the appropriate theorems from Section~\ref{subsec:wishart-outlier} and replacing $\alpha$, $\beta$, $\tilde{\alpha}$, $\tilde{\beta}$ by $\tau^{1/2}\alpha$, $\tau^{1/2}\beta$, $\tau^{1/2}\tilde{\alpha}$, $\tau^{1/2}\tilde{\beta}$ respectively, where appropriate. We prove it for the CSWig model below. 

\begin{theorem}[Parameter-free detection and recovery]\label{thm:Spiked-parameter-freeAppendix}
Fix $\varepsilon\in(0,1/4)$.  
Let $(U, V)$ be sampled from either the null correlated Wigner model or the correlated Wigner model with correlation $\rho$. Let $\kappa(\tilde{\alpha}, \tilde{\beta})$ be defined as in \eqref{eq:kappaWignerThresholdAppendix}. Let us define the admissible grid
\[
    \cZ_{\varepsilon}
    := \Bigl\{ (\tilde\alpha,\tilde\beta)\ :\ \tilde{\alpha}, \tilde{\beta}\in \varepsilon^9\mathbb Z \cap (0,1),\ \kappa(\tilde\alpha,\tilde\beta)<1
    \Bigr\}.
\]
We define the test statistic, $ \Lambda_* = \max_{\cZ_\varepsilon} \lambda_1(W(\tilde \alpha, \tilde \beta)).$
\begin{enumerate}[label=(\roman*)]
   \item Under the null model, $ \Lambda_*\le 1+n^{-1/3}$ with high probability. 

    \item Under the planted model, if $\alpha,\beta \in (2\varepsilon, 1]$ and $\rho > \kappa + \varepsilon$, then with high probability $\Lambda_* > 1+\Omega(1)$ and the corresponding eigenvector achieves weak recovery:
    \[|\langle \hat a,a\rangle|^2\ge \Omega(1),
        \qquad
        |\langle \hat b,b\rangle|^2\ge \Omega(1).\]
\end{enumerate}
Additionally, note that it suffices to search over the $\epsilon^2$-grid when $\alpha,\beta \in (2\epsilon, 1-2\epsilon)$.
\end{theorem}
\begin{proof}[Proof of Theorem~\ref{thm:Spiked-parameter-freeAppendix} and Theorem~\ref{thm:Spiked-parameter-free}]
  We note that (i) follows directly from the tail bounds for the top eigenvalue from Theorem~\ref{thm:WignerGordonEdge} by a union bound with $\varepsilon = n^{-1/3}$. 

  For (ii), in light of Theorem~\ref{thm:mismatched-Wigner-outlier} and Lemma~\ref{lem:mismatched-separated-outlier-informative}, it suffices to verify there exists $\tilde{\alpha},\tilde{\beta}$ which satisfies the two conditions: $\rho>\rho_{\out}(\tilde\alpha,\tilde\beta;\alpha,\beta)$ whenever $\rho > \kappa + \epsilon$, and $\tilde{\kappa}(\tilde{\alpha}, \tilde{\beta}) < 1$. A union bound over $\tilde{\alpha},\tilde{\beta} \in \epsilon^9 \ZZ\cap (0,1)$ with Theorem~\ref{thm:mismatched-Wigner-outlier} and Lemma~\ref{lem:mismatched-separated-outlier-informative} gives the guarantees of strong detection and weak recovery. Let us first consider when $\alpha,\beta \in [2\epsilon, 1-2\epsilon]$, we will show that,
   \begin{equation}\label{eq:Wigner_inequalityCor}
   \frac{\rho_{\out}(\tilde\alpha,\tilde\beta;\alpha,\beta)^2}{\tilde{\kappa}^2} = \frac{(1-\tilde{\alpha}\alpha)(1-\tilde{\beta} \beta)}{(1-\tilde{\alpha}^2)(1-\tilde{\beta}^2)}\frac{\tilde{\alpha} \tilde{\beta}}{\alpha\beta} < \frac{(\kappa + \epsilon)^2}{\tilde{\kappa}^2}.
   \end{equation}
  Let us take $\tilde{\alpha} \in [\alpha, \alpha + \epsilon^2], \tilde{\beta} \in [\beta, \beta + \epsilon^2]$. Then for $\epsilon < 1/4$,
   \[\frac{\tilde{\alpha}\epsilon^2}{1-\tilde{\alpha}^2} < \frac{\epsilon^2}{1-(1-\epsilon)^2} \le \epsilon, \quad \frac{\tilde{\alpha}\tilde{\beta}}{\alpha\beta} \le \sqrt{\frac{(1-\alpha^2)(1-\beta^2)}{(1-\tilde{\alpha}^2)(1-\tilde{\beta}^2)} \frac{\tilde{\alpha}^2\tilde{\beta}^2}{\alpha^2\beta^2}} = \frac{\kappa^2}{\tilde{\kappa}^2}.\]
  and similarly for $\tilde{\beta}$. Then,
   \[\frac{(1-\tilde{\alpha}\alpha)(1-\tilde{\beta} \beta)}{(1-\tilde{\alpha}^2)(1-\tilde{\beta}^2)}\frac{\tilde{\alpha} \tilde{\beta}}{\alpha\beta} \le \frac{(1-\tilde{\alpha}(\tilde{\alpha}-\epsilon^2))(1-\tilde{\beta} (\tilde{\beta}-\epsilon^2))}{(1-\tilde{\alpha}^2)(1-\tilde{\beta}^2)}\frac{\kappa^2}{\tilde{\kappa}^2} < (1 + 2\epsilon +\epsilon^2)\frac{\kappa^2}{\tilde{\kappa}^2} \le \frac{(\kappa + \epsilon)^2}{\tilde{\kappa}^2}.\]
 By the same argument, for $\alpha,\beta \in [2\epsilon, 1-2\epsilon^{3}]$, we may choose $\tilde{\alpha}\in [\alpha, \alpha+\epsilon^{4}]$ and $\tilde{\beta}\in[\beta, \beta + \epsilon^{4}]$ which satisfies the inequality \eqref{eq:Wigner_inequalityCor}. Note that as $\tilde{\kappa}$ is decreasing in $\tilde{\alpha}, \tilde{\beta}$, we have that $\tilde{\kappa} \le \kappa < 1$. Now consider $\alpha \in [2\epsilon, 1]$, and $\beta \in [1-2\epsilon^3, 1]$. Recall that
  \[\kappa(\alpha,\beta)^2=s(\alpha)s(\beta), \text{ where } s(x):=\frac{\sqrt{1-x^2}}{x}.\]
We may write,
  \[\rho_{\out}(\tilde\alpha,\tilde\beta;\alpha,\beta)^2  
  = \sqrt{s(\alpha)^2 + \frac{(\alpha - \tilde\alpha)^2}{\alpha^2(1-\tilde{\alpha}^2)}}\sqrt{s(\alpha)^2 + \frac{(\beta - \tilde\beta)^2}{\beta^2(1-\tilde{\beta}^2)}}.\]
Let $\tilde{\alpha} \in [\alpha-2\epsilon^{9}, \alpha - \epsilon^{9}]$ and $\tilde{\beta} \in [\beta-2\epsilon^{9}, \beta - \epsilon^{9}]$. Then for $\epsilon < \frac{1}{4}$, using $\sqrt{x^2+y^2}\le x+y$,
  \[\frac{1-\alpha\tilde{\alpha}}{\alpha\sqrt{1-\tilde{\alpha}^2}} \le s(\alpha) +\frac{\alpha-\tilde \alpha}{\alpha\sqrt{1-\tilde \alpha^2}} \le s(\alpha) +\frac{2\epsilon^9}{2\epsilon\sqrt {\epsilon^9}} = s(\alpha)+\epsilon^{7/2}.\]
and likewise for the $\beta$ term. Therefore
  \[\rho_{\out}(\tilde\alpha,\tilde\beta;\alpha,\beta)^2 \le \paren{s(\alpha)+\epsilon^{7/2}}\paren{s(\beta)+\epsilon^{7/2}} \le \kappa(\alpha,\beta)^2 + \epsilon^{7/2}\paren{s(\alpha) + s(\beta)}.\]
Since $\alpha,\beta\ge 2\epsilon$, we have $s(\alpha)+s(\beta)\le \epsilon^{-1}$. Hence
  \[\rho_{\out}(\tilde\alpha,\tilde\beta;\alpha,\beta)^2 \le \kappa(\alpha,\beta)^2+\epsilon^{5/2} < (\kappa + \epsilon)^2.\]
Additionally, for $\tilde{\alpha} \in [\alpha-2\epsilon^{9}, \alpha - \epsilon^{9}]$ and $\tilde{\beta} \in [\beta-2\epsilon^{9}, \beta - \epsilon^{9}]$, we have that
  \[\tilde{\kappa}^4 = \frac{(1 - \tilde{\alpha}^2)(1-\tilde{\beta}^2)}{\tilde{\alpha}^2\tilde{\beta}^2} \le \frac{4\epsilon^3}{\epsilon^2} < 1.
  \]
  This concludes the proof.
\end{proof}

\subsubsection{Estimating the signal strengths, proof of Theorem~\ref{thm:recoverySignalStrengths-MainBody} (CSWig)}\label{subsubsec:wigner-strength-estimation}

We next show that, once weak recovery is possible, the signal strengths can also be estimated consistently. We do this by the standard trick (see \cite{BrennanBreslerHuleihel18}, \cite{CaiMaWu13}) of introducing independent Gaussian noise and use one copy to reconstruct estimators.

\begin{theorem}[Theorem~\ref{thm:recoverySignalStrengths-MainBody} (CSWig)]\label{thm:wigner-strength-estimation}
Let $(U,V)$ be sampled from the correlated spiked Wigner model in Section~\ref{subsec:wigner}, with fixed parameters $0<\alpha,\beta\le 1$ and correlation $\rho>\kappa$.  Then there exist randomized estimators $\widehat\alpha = \widehat\alpha(U,V)$ and $\widehat\beta = \hat \beta(U,V)$ such that, with high probability
\[
    \widehat\alpha=\alpha+o(1),
    \qquad
    \widehat\beta=\beta+o(1).
\]
\end{theorem}

\begin{proof}[Proof of Theorem~\ref{thm:wigner-strength-estimation}, Theorem~\ref{thm:recoverySignalStrengths-MainBody} (CSWig)]
Set $\eta=n^{-1/5}$ and draw, independently of $(U,V)$, two independent GOE matrices $G_U,G_V\sim\mathrm{GOE}(1/n)$.  Define
\[
U^{(1)}:=\frac{U+\sqrt{\eta} \, G_U}{\sqrt{1+\eta}},
\qquad
U^{(2)}:=\frac{U-\eta^{-1/2}G_U}{\sqrt{1+\eta^{-1}}},
\]
and define $V^{(1)},V^{(2)}$ analogously.  Apply the weak recovery algorithm to $(U^{(1)},V^{(1)})$, and let $(\bar a,\bar b)$ be its output.  Normalize
\[
    x:=\frac{\bar a}{\|\bar a\|},
    \qquad
    y:=\frac{\bar b}{\|\bar b\|},
\]
with an arbitrary unit vector choice if one of the two denominators is zero.  Define on the event that the denominators below are nonzero,
\begin{equation}\label{eq:wigner-strength-estimator}
    \widehat\alpha
    :=\sqrt{1+\eta^{-1}}\,\frac{\|U^{(2)}x\|^2-1}{x^\intercal U^{(2)}x},
    \qquad
    \widehat\beta
    :=\sqrt{1+\eta^{-1}}\,\frac{\|V^{(2)}y\|^2-1}{y^\intercal V^{(2)}y},
\end{equation}
and set the estimator equal to $0$ on the complementary event.  
We prove the claim for $\widehat\alpha$; the proof for $\widehat\beta$ is identical.  Write
\[
Z_U^{(1)}:=\frac{Z_a+\sqrt{\eta}G_U}{\sqrt{1+\eta}},
\qquad
Z_U^{(2)}:=\frac{Z_a-\eta^{-1/2}G_U}{\sqrt{1+\eta^{-1}}}.
\]
The pair $(Z_U^{(1)},Z_U^{(2)})$ is jointly Gaussian.  Each marginal is distributed as $\mathrm{GOE}(1/n)$, and their covariance is zero; hence they are independent.  Therefore
\[
    U^{(1)}=\frac{\alpha}{\sqrt{1+\eta}}aa^\intercal+Z_U^{(1)},
    \qquad
    U^{(2)}=\gamma_n aa^\intercal+Z_U^{(2)},
    \qquad
    \gamma_n:=\frac{\alpha}{\sqrt{1+\eta^{-1}}}.
\]
The same construction applies to $V$.  Since $\eta\to0$ and $\rho>\kappa(\alpha,\beta)$, the parameters converging to $(\alpha,\beta)$ and remains above the weak recovery threshold for all sufficiently large $n$.  By the preceding weak recovery result Lemma~\ref{lem:mismatched-separated-outlier-informative}, there is a deterministic constant $c>0$ such that, with high probability,
\begin{equation}\label{eq:wigner-split-overlap}
    q_n:=\langle x,a\rangle^2\ge c
    \quad \text{ and } \quad 
    \langle y,b\rangle^2\ge c.
\end{equation}
We work on this event.  Conditional on $(U^{(1)},V^{(1)})$ and spikes $a, b$, the vector $x$ is fixed and is independent of $Z_U^{(2)}$.  Standard GOE concentration, together with $\|a\|^2=1+o(1)$ whp, gives
\[
    x^\intercal Z_U^{(2)}x=o(\gamma_n),
    \qquad
    a^\intercal Z_U^{(2)}x=o(\gamma_n),
    \qquad
    \|Z_U^{(2)}x\|^2=1+o(\gamma_n^2)
\]
with high probability.  Here we used $\gamma_n^2\asymp \eta$ and $\eta\gg n^{-1/2}$.  Hence
\begin{align*}
    x^\intercal U^{(2)}x
    &=\gamma_n\langle x,a\rangle^2+x^\intercal Z_U^{(2)}x
      =\gamma_n q_n+o(\gamma_n),\\
    \|U^{(2)}x\|^2-1
    &=\gamma_n^2\langle x,a\rangle^2\|a\|^2
      +2\gamma_n\langle x,a\rangle a^\intercal Z_U^{(2)}x
      +\|Z_U^{(2)}x\|^2-1 \\
    &=\gamma_n^2q_n+o(\gamma_n^2).
\end{align*}
Using \eqref{eq:wigner-split-overlap}, the denominator in \eqref{eq:wigner-strength-estimator} is nonzero with high probability, and
\[
    \frac{\|U^{(2)}x\|^2-1}{x^\intercal U^{(2)}x}
    =\gamma_n(1+o(1)).
\]
Multiplying by $\sqrt{1+\eta^{-1}}$ yields $\hat \alpha = \alpha + o(1)$ with high probability.\end{proof}

\subsection{Correlated spiked Wishart model}
\label{subsec:wishart-outlier}
We now prove the Wishart analogue of Theorems~\ref{thm:CCA_eigVecOverlapAppendix}, \ref{thm:CCA-outlier1iAppendix}, following the same ideas as in the previous Wigner subsection. Throughout this subsection we work in the regime
\[
0<\alpha,\beta<\tau^{-1/2},
\qquad
\kappa=\left(\frac{(1-\tau\alpha^2)(1-\tau\beta^2)}{\tau^2\alpha^2\beta^2}\right)^{1/4}<\rho.
\]
We keep the notation from subsection~\ref{subsec:wishart-resolvent}, thus we define 
\[
A:=
\begin{pmatrix}
1&\kappa\\
\kappa&1
\end{pmatrix},
\qquad
p:=A^{1/2}e_1,
\qquad
q:=A^{1/2}e_2,
\qquad
P:=pp^\intercal,
\qquad
Q:=qq^\intercal,
\]
Moreover, we set $\theta_1:={\alpha}/({1+\alpha})$, $\theta_2:={\beta}/({1+\beta})$, and define $(r,s,T)$ from Definition~\ref{def:det_system_wishart}. We observe that by Corollary~\ref{cor:wishart-critical-identities-outlier} at $z = 1$, we have
\begin{equation}\label{eq:wishart-critical-identities-outlier}
T(1_+)=-I_2,
\qquad
r(1_+)=-(1+\alpha),
\qquad
s(1_+)=-(1+\beta).
\end{equation}
We may write also write the planted model in the following form
\begin{equation}\label{eq:wishart-model-normalized}
U=U_0+\sqrt{\frac{\alpha}{{m}}}\,u a^\intercal,
\qquad
V=V_0+\sqrt{\frac{\beta}{{m}}}\,v b^\intercal,
\end{equation}
where $U_0,V_0\in\RR^{n\times m}$ have independent $\calN(0,m^{-1})$ entries, and the hidden $a,b\in\RR^m$ are $\rho$ correlated and the latent factors $u,v\in\RR^n$ are independent of $(U_0,V_0)$. We recall that 
\[
\tilde U:=\theta_1 U^\intercal U-\alpha \tau  I_m,
\qquad
\tilde V:=\theta_2 V^\intercal V-\beta \tau I_m,
\]
and the block matrix
\begin{equation}\label{eq:wishart-outlier-matrix}
W:=
\begin{pmatrix}
\tilde U & \kappa\tilde V\\
\kappa\tilde U & \tilde V
\end{pmatrix}.
\end{equation}
Let $\lambda_1(W)\ge\cdots\ge\lambda_{2m}(W)$ denote the ordered eigenvalues of $W$. We also introduce the symmetric conjugate
\[
H:=
(A^{1/2}\otimes I_m)
\begin{pmatrix}
\tilde U&0\\
0&\tilde V
\end{pmatrix}
(A^{1/2}\otimes I_m).
\]
Exactly as in the Wigner subsection, $W$ is similar to $H$, so it is enough to study the spectrum of $H$. Let $H_0$ be the null symmetric matrix from Section~\ref{subsec:wishart-resolvent}; explicitly,
\[
H_0=
(A^{1/2}\otimes I_m)
\begin{pmatrix}
\tilde U_0&0\\
0&\tilde V_0
\end{pmatrix}
(A^{1/2}\otimes I_m),
\]
where, the block matrices $\tilde U_0$, and $\tilde V_0$ are given by.
\[
\tilde U_0:=\theta_1 U_0^\intercal U_0-\alpha \tau  I_m,
\qquad
\tilde V_0:=\theta_2 V_0^\intercal V_0-\beta \tau I_m.
\]
Let $\calL_0(z)$, $\calG_0(z)=\calL_0(z)^{-1}$, and $\calM(z)$ be the corresponding null linearization, resolvent, and deterministic approximation from Section~\ref{subsec:wishart-resolvent}. For real $z>1$, define the deterministic $2\times2$ matrix
\begin{equation}\label{eq:wishart-Sigma-rho}
\cS(z):=
\begin{pmatrix}
\alpha \theta_1\,p^\intercal T(z)p-\dfrac{1}{\tau r(z)}
&
\rho\sqrt{\alpha \theta_1\beta \theta_2}\,p^\intercal T(z)q\\[2mm]
\rho\sqrt{\alpha \theta_1\beta \theta_2}\,q^\intercal T(z)p
&
\beta \theta_2\,q^\intercal T(z)q-\dfrac{1}{\tau s(z)}
\end{pmatrix}.
\end{equation}
The analogue of Lemma~\ref{lem:S-has-negative-determinant-at-edge} is immediate from
\eqref{eq:wishart-critical-identities-outlier}.

\begin{lemma}\label{lem:wishart-S-negative-at-edge}
Suppose $\calS$ is defined from \eqref{eq:wishart-Sigma-rho},
\[
\cS(1_+)=
\begin{pmatrix}
\dfrac{1-\tau\alpha^2}{\tau(1+\alpha)}
&
-\dfrac{\rho\alpha\beta\kappa}{\sqrt{(1+\alpha)(1+\beta)}}\\[3mm]
-\dfrac{\rho\alpha\beta\kappa}{\sqrt{(1+\alpha)(1+\beta)}}
&
\dfrac{1-\tau\beta^2}{\tau(1+\beta)}
\end{pmatrix},
\]
and consequently
\begin{equation}\label{eq:wishart-det-at-edge}
\det \cS(1_+)
=
\frac{\alpha^2\beta^2\kappa^2(\kappa^2-\rho^2)}{(1+\alpha)(1+\beta)}.
\end{equation}
In particular, if $\rho>\kappa$, then $\det\cS(1_+)<0$.
\end{lemma}

\begin{proof}
Substituting the identities from Corollary~\ref{cor:wishart-critical-identities-outlier} into \eqref{eq:wishart-Sigma-rho}, the displayed formula for $\cS(1_+)$ follows immediately. Taking the determinant and using
\[
\kappa^4=
\frac{(1-\tau\alpha^2)(1-\tau\beta^2)}{\tau^2\alpha^2\beta^2}
\]
yields \eqref{eq:wishart-det-at-edge}.
\end{proof}

We can now state the Wishart counterpart of Theorems~\ref{thm:CCA_eigVecOverlapAppendix}, \ref{thm:CCA-outlier1iAppendix}.
\begin{theorem}[Theorem~\ref{thm:Spiked-outlierMainBody} (i)-(ii) CSWish]
\label{thm:Wishart-outlier}
Assume $\rho>\kappa$. Then there exists a unique real number $\lambda_{\out}>1$ such that $\det \cS(\lambda_{\out})=0.$ Moreover, with high probability,
\[
\lambda_1(W)=\lambda_{\out}+o(1)
\qquad\text{and}\qquad
\lambda_2(W)\le 1+o(1).
\]
Furthermore, if $\hat w=(\hat a,\hat b)\in\RR^m\times\RR^m$ is a top unit right eigenvector of $W$, and if $x_* \in\ker\cS(\lambda_{\out})$ is any unit vector, then the individual overlaps satisfy
\begin{equation*}
\begin{aligned}
   |\langle\hat a,a\rangle|^2
   &\ge
   \frac{x_{*,1}^2}{\tau^2 r(\lambda_{\out})^2\,\alpha\theta_1(1+\kappa)\,x_*^\intercal\partial_z\cS(\lambda_{\out})x_*}
   +o(1),\\
   |\langle\hat b,b\rangle|^2
   &\ge
   \frac{x_{*,2}^2}{\tau^2 s(\lambda_{\out})^2\,\beta\theta_2(1+\kappa)\,x_*^\intercal\partial_z\cS(\lambda_{\out})x_*}
   +o(1)
\end{aligned}
\end{equation*}
The two deterministic constants on the right hand are bounded away from zero.
\end{theorem}

\begin{proof}[Proof of Theorem~\ref{thm:Wishart-outlier}, Theorem~\ref{thm:Spiked-outlierMainBody} (i)-(ii) CSWish]
We argue exactly as in the Wigner case: we first reduce the eigenvalue equation to a $2\times2$ Schur complement, then pass to its deterministic limit, and finally extract the overlap bound. Set
\[
X:=\bigl[\sqrt{\alpha\theta_1}\,p\otimes a\ \ \sqrt{\beta\theta_2}\,q\otimes b\bigr]\in\RR^{2m\times 2},
\qquad
Y:= \frac{1}{\sqrt{m}} \bigl[e_1\otimes u\ \ e_2\otimes v\bigr]\in\RR^{2n\times 2}.
\]
By \eqref{eq:wishart-model-normalized} and the linearization from Section~\ref{subsec:wishart-resolvent},
\[
\calL(z)=\calL_0(z)+\calU_n\calJ\calU_n^\intercal,
\qquad
\calU_n:=\begin{pmatrix}X&0\\0&Y\end{pmatrix},
\qquad
\calJ:=\begin{pmatrix}0&I_2\\ I_2&0\end{pmatrix}.
\]
Hence the matrix determinant lemma gives
\[
\det \calL(z)=\det \calL_0(z)\det \mathfrak S_n(z),
\quad
\text{ where } \mathfrak S_n(z):=\calJ+\calU_n^\intercal\calG_0(z)\calU_n.
\]
Write $\calG_0(z)$ in top/bottom block form,
\[
\calG_0(z)=
\begin{pmatrix}
\calG_0^{\mathrm{tt}}(z)&\calG_0^{\mathrm{tb}}(z)\\
\calG_0^{\mathrm{bt}}(z)&\calG_0^{\mathrm{bb}}(z)
\end{pmatrix}.
\]
Then, we have that 
\[
\mathfrak S_n(z)=
\begin{pmatrix}
X^\intercal\calG_0^{\mathrm{tt}}(z)X & I_2+X^\intercal\calG_0^{\mathrm{tb}}(z)Y\\
I_2+Y^\intercal\calG_0^{\mathrm{bt}}(z)X & Y^\intercal\calG_0^{\mathrm{bb}}(z)Y
\end{pmatrix}.
\]
For $z>\lambda_{\max}(H_0)$, the upper-left block $\calG_0^{\mathrm{tt}}(z) = (H_0-zI_{2m})^{-1}\prec0$, and the block inverse formula \eqref{eq:SchurGnWishart} gives $\calG_0^{\mathrm{bb}}(z)\prec -I_{2n}$. Hence $Y^\intercal\calG_0^{\mathrm{bb}}(z)Y\prec0$, so the corresponding Schur complement is well defined:
\begin{equation}\label{eq:wishart-random-Schur-complement}
\cS_n(z):=
X^\intercal\calG_0^{\mathrm{tt}}(z)X
-
\bigl(I_2+X^\intercal\calG_0^{\mathrm{tb}}(z)Y\bigr)
\bigl(Y^\intercal\calG_0^{\mathrm{bb}}(z)Y\bigr)^{-1}
\bigl(I_2+Y^\intercal\calG_0^{\mathrm{bt}}(z)X\bigr).
\end{equation}
Since $\det \mathfrak S_n(z)=\det\bigl(Y^\intercal\calG_0^{\mathrm{bb}}(z)Y\bigr)\det \cS_n(z),$ we obtain the exact reduction
\begin{equation}\label{eq:wishart-outlier-exact-reduction}
z>\lambda_{\max}(H_0),\quad z\in\Lambda(H)
\qquad\Longleftrightarrow\qquad
\det \cS_n(z)=0.
\end{equation}

We next identify the deterministic limit of $\cS_n(z)$. Fix a compact interval $K\subset(1,\infty)$. By Theorem~\ref{thm:WishartGordonEdge}, we have $\lambda_{max}(H_0)\le 1+o(1)$ with overwhelming probability. On this event, the blocks of $\calG_0(z)$ are uniformly bounded and uniformly Lipschitz on $K$, since
\[
\sup_{z\in K}\|\calG_0(z)\|_{\op}=O_K(1),
\qquad
\sup_{z\in K}\|\partial_z\calG_0(z)\|_{\op}=O_K(1),
\]
and $\partial_z\calG_0(z)=\calG_0(z)\Pi\calG_0(z)$ with $\Pi=\diag(I_{2m},0)$.

Now write $\calG_0^{\mathrm{tt}}(z)$ and $\calG_0^{\mathrm{bb}}(z)$ in their $m\times m$ and $n\times n$ subblocks. Then each entry of $X^\intercal\calG_0^{\mathrm{tt}}(z)X$ is a finite linear combination of terms of the form
\[
a^\intercal A(z)a,\qquad
b^\intercal A(z)b,\qquad
a^\intercal A(z)b,
\]
while the entries of $Y^\intercal\calG_0^{\mathrm{bb}}(z)Y$ and $X^\intercal\calG_0^{\mathrm{tb}}(z)Y$ are finite linear combinations of the analogous Gaussian quadratic and bilinear forms involving $u$ and $v$. Conditional on $H_0$, all matrices appearing here are deterministic, independent of $(a,b,u,v)$, and uniformly bounded on $K$. Write
\[
a=\frac{g}{\sqrt m},
\qquad
b=\frac{\rho g+\sqrt{1-\rho^2}\,h}{\sqrt m},
\]
where $g,h\sim\calN(0,I_m)$ are independent. Hanson--Wright, exactly as in the Wigner proof, gives uniformly for $z\in K$
\begin{align*}
    a^\intercal A(z)a&=\frac1m\Tr A(z)+o(1) \\
    b^\intercal A(z)b&=\frac1m\Tr A(z)+o(1) \\
    a^\intercal A(z)b&=\rho\,\frac1m\Tr A(z)+o(1)
\end{align*}
with overwhelming probability, for every such family $A(z)$. The same argument gives the corresponding approximations for the forms involving $u$ and $v$, with the factor $\tau$ coming from $n/m\to\tau$, and the same Gaussian bilinear estimate controls the mixed top-bottom terms. Combining these concentration bounds with Theorem~\ref{thm:real-stieltjes-transform-Wishart} and the block form of $\calM(z)$ in Definition~\ref{def:det_system_wishart} and Borel-Cantelli lemma, we obtain that almost surely, uniformly for $z\in K$,
\[
X^\intercal\calG_0^{\mathrm{tt}}(z)X\to
\begin{pmatrix}
\alpha\theta_1\,p^\intercal T(z)p
&
\rho\sqrt{\alpha\theta_1\beta\theta_2}\,p^\intercal T(z)q\\[1mm]
\rho\sqrt{\alpha\theta_1\beta\theta_2}\,q^\intercal T(z)p
&
\beta\theta_2\,q^\intercal T(z)q
\end{pmatrix},
\]
\[
Y^\intercal\calG_0^{\mathrm{bb}}(z)Y\to \tau
\begin{pmatrix}
r(z)&0\\
0&s(z)
\end{pmatrix},
\qquad
I_2+X^\intercal\calG_0^{\mathrm{tb}}(z)Y\to I_2.
\]
Since $r(z),s(z)<0$ for $z>1$, the second matrix is uniformly invertible on $K$. Consequently,
\begin{equation}\label{eq:wishart-Schur-limit}
\sup_{z\in K}\|\cS_n(z)-\cS(z)\|\to 0
\end{equation}
holds almost surely, where $\cS(z)$ is exactly the deterministic matrix in \eqref{eq:wishart-Sigma-rho}.

We now study $\cS(z)$. For real $z>1$, the matrix $T(z)$ is real symmetric and numbers $r(z),s(z)$ are real and negative, so $\cS(z)$ is a real symmetric $2\times2$ matrix. By Lemma~\ref{lem:wishart-S-negative-at-edge}, when $\rho>\kappa$ we have $\det \cS(1_+)<0$. Since
\[
\cS_{11}(1)=\frac{1-\tau \alpha^2}{\tau(1+\alpha)}>0,
\qquad
\cS_{22}(1)=\frac{1-\tau \beta^2}{\tau(1+\beta)}>0,
\]
the matrix $\cS(1)$ has one positive and one negative eigenvalue. On the other hand, by Lemma~\ref{lem:WishartDeterministicExtension}
\[
r(z)=-1+O(z^{-1}),
\qquad
s(z)=-1+O(z^{-1}),
\qquad
T(z)=O(z^{-1}),
\]
so $\cS(z)=\tau^{-1}I_2+O(z^{-1}),$ and therefore $\cS(z)\succ0$ for all sufficiently large $z$. To prove uniqueness of the zero of $\det \cS(z)$, fix a compact interval $K=[1+\delta,R]\subset(1,\infty).$ Let
\[
\cA_n(z):=X^\intercal\calG_0^{\mathrm{tt}}(z)X,
\]
and let $\cA(z)$ denote its deterministic limit, that is, the first term in the formula for $\cS(z)$. For real $z>\lambda_{\max}(H_0)$, we have $\cA_n'(z)=X^\intercal(H_0-zI_{2m})^{-2}X.$ By the null version of \eqref{eq:wishart-gram-form},
\[
H_0=B_0^\intercal B_0-\bigl(\tau\alpha P+\tau\beta Q\bigr)\otimes I_m
\succeq -\tau(\alpha+\beta)I_{2m}.
\]
Hence, for every $z\in K$,
\[
(H_0-zI_{2m})^{-2}\succeq \frac{1}{(R+\tau(\alpha+\beta))^2}\,I_{2m}.
\]
Moreover, we have that with overwhelming probability 
\[
X^\intercal X=
\begin{pmatrix}
\alpha\theta_1\|a\|^2 & \kappa\sqrt{\alpha\theta_1\beta\theta_2}\,\langle a,b\rangle\\
\kappa\sqrt{\alpha\theta_1\beta\theta_2}\,\langle a,b\rangle & \beta\theta_2\|b\|^2
\end{pmatrix}
\to
\begin{pmatrix}
\alpha\theta_1 & \kappa\rho\sqrt{\alpha\theta_1\beta\theta_2}\\
\kappa\rho\sqrt{\alpha\theta_1\beta\theta_2} & \beta\theta_2
\end{pmatrix}
\succ0,
\]
because $\kappa\rho<1$. Therefore there exists a deterministic constant $c>0$ such that, with high probability, $\cA_n'(z)\succeq c I_2$ uniformly for $z \in K$.
Choose $r>0$ so small that
\[
\mathfrak B_{3r}(K):=\setcond{z\in\CC}{\operatorname{dist}(z,K)\le 3r}\subset\DD,
\qquad
\Re z>1+\delta/2\quad\text{for all }z\in\mathfrak B_{3r}(K).
\]
On the event $\lambda_{\max}(H_0)\le 1+\delta/4$, the entries of $\cA_n(z)$ are holomorphic on $\mathfrak B_{3r}(K)$ and locally uniformly bounded in $n$. Since the almost sure convergence proved above gives $\cA_n\to\cA$ uniformly on the real interval $K$, Vitali's theorem implies that $\cA_n\to\cA$ locally uniformly on $\mathfrak B_{3r}(K)$, entrywise almost surely. Cauchy's integral formula therefore yields
\[
\sup_{z\in K}\|\cA_n'(z)-\cA'(z)\|\longrightarrow 0,
\]
almost surely. Hence $\cA'(z)\succeq c I_2$ for all $z\in K$.

Since $r$ and $s$ are Stieltjes transforms supported on $(-\infty,1]$, we also have that have $r(z),s(z)<0$ and $r'(z),s'(z)>0$ for $z>1$. Therefore
\[
\partial_z\!\left[-\frac1\tau
\begin{pmatrix}
r(z)^{-1}&0\\
0&s(z)^{-1}
\end{pmatrix}\right]
=
\frac1\tau
\begin{pmatrix}
\dfrac{r'(z)}{r(z)^2}&0\\[2mm]
0&\dfrac{s'(z)}{s(z)^2}
\end{pmatrix}
\succ0
\]
on $(1,\infty)$. Thus $\cS'(z)\succeq c I_2$ on $K$. Hence both eigenvalues of $\cS(z)$ are therefore continuous and strictly increasing on $(1,\infty)$. Since the smaller one is negative at $z=1$ and positive for large $z$, there exists a unique real number $\lambda_{\out}>1$ such that $\det \cS(\lambda_{\out})=0.$ We next return to the random matrix. Let
\[
E_n(z):=I_2+X^\intercal\calG_0^{\mathrm{tb}}(z)Y,
\qquad
B_n(z):=Y^\intercal\calG_0^{\mathrm{bb}}(z)Y.
\]
The Schur complement can be written in the factored form
\[
\cS_n(z)=
\begin{pmatrix}
I_2&-E_n(z)B_n(z)^{-1}
\end{pmatrix}
\mathfrak S_n(z)
\begin{pmatrix}
I_2\\
-B_n(z)^{-1}E_n(z)^\intercal
\end{pmatrix}.
\]
then $\calL_0'(z)=-\Pi$, and hence $\mathfrak S_n'(z)
=
\calU_n^\intercal\calG_0(z)\Pi\calG_0(z)\calU_n
\succeq0.$ Differentiating the factored identity, the derivatives hitting the outer factors vanish because
\[
\mathfrak S_n(z)
\begin{pmatrix}
I_2\\
-B_n(z)^{-1}E_n(z)^\intercal
\end{pmatrix}
=
\binom{\cS_n(z)}{0}.
\]
Hence only the derivative of $\mathfrak S_n$ remains, and we obtain
\begin{equation}\label{eq:wishart-random-Schur-derivative}
\cS_n'(z)=
\begin{pmatrix}
I_2&-E_n(z)B_n(z)^{-1}
\end{pmatrix}
\mathfrak S_n'(z)
\begin{pmatrix}
I_2\\
-B_n(z)^{-1}E_n(z)^\intercal
\end{pmatrix}
\succeq c I
\end{equation}
for all $z>\lambda_{\max}(H_0) + c$ for small constant $c$ with overwhelming probability. Hence the two eigenvalues of $\cS_n(z)$ are continuous and nondecreasing on $(\lambda_{\max}(H_0) + c,\infty)$.

Choose $\delta>0$ so that $1<\lambda_{\out}-\delta$ 
and the smallest eigenvalue of $\cS(z)$ is negative at $z=\lambda_{\out}-\delta$ and positive at $z=\lambda_{\out}+\delta$. By \eqref{eq:wishart-Schur-limit}, the same is true for $\cS_n(z)$ with overwhelming probability. On the same event, $\lambda_{\max}(H_0)<\lambda_{\out}-\delta.$ Therefore the smaller eigenvalue of $\cS_n(z)$ has a unique zero
\[
\lambda_{\out,n}\in(\lambda_{\out}-\delta,\lambda_{\out}+\delta),
\]
and $\lambda_{\out,n}\to\lambda_{\out}$ in probability. By \eqref{eq:wishart-outlier-exact-reduction}, $\lambda_{\out,n}$ is an eigenvalue of $H$.

To rule out a second outlier, choose $\delta>0$ so small that $\det \cS(1+\delta)<0$. Since the diagonal entries of $\cS(1+\delta)$ are still positive, $\cS(1+\delta)$ has one positive and one negative eigenvalue. By \eqref{eq:wishart-Schur-limit}, the same holds for $\cS_n(1+\delta)$ with overwhelming probability. Since the eigenvalues of $\cS_n(z)$ are nondecreasing, only the smaller eigenvalue can hit zero, and it can do so at most once. Combining this with \eqref{eq:wishart-outlier-exact-reduction}, we conclude that with overwhelming probability $H$ has exactly one eigenvalue above $1+\delta$, namely $\lambda_{\out,n}$. Therefore,
\[
\lambda_1(H)=\lambda_{\out,n}=\lambda_{\out}+o(1)
\qquad\text{and}\qquad
\lambda_2(H)\le 1+\eta
\]
with high probability. Since $H$ and $W$ are similar and $\eta>0$ is arbitrary, this proves
\[
\lambda_1(W)=\lambda_{\out}+o(1)
\qquad\text{and}\qquad
\lambda_2(W)\le 1+o(1).
\]

It remains to prove the overlap statement. Let $x_n\in\RR^2$ be a unit vector spanning $\ker \cS_n(\lambda_{\out,n})$, and define $y_n:=-B_n(\lambda_{\out,n})^{-1}E_n(\lambda_{\out,n})^\intercal x_n$. Then $(x_n, y_n)$ is in the kernel of $\mathfrak S_n(\lambda_{\out,n})$. Set
\[
\psi_n:=-\calG_0(\lambda_{\out,n})\calU_n\binom{x_n}{y_n}:=\binom{h_n}{g_n}.
\]
Because $(x_n, y_n)\in\ker \mathfrak S_n(\lambda_{\out,n})$, a direct computation shows that $\calL(\lambda_{\out,n})\psi_n=0.$ Hence $h_n$ is an eigenvector of $H$ with eigenvalue $\lambda_{\out,n}$. Therefore
\[
\hat w_n:=\frac{(A^{1/2}\otimes I_m)h_n}{\|(A^{1/2}\otimes I_m)h_n\|}
\]
is a top unit eigenvector of $W$, and so it coincides with $\hat w$ up to a sign. Moreover,
\[
\calU_n^\intercal\psi_n
=-\calU_n^\intercal\calG_0(\lambda_{\out,n})\calU_n\binom{x_n}{y_n}
=\calJ\binom{x_n}{y_n},
\]
so the first two coordinates give $X^\intercal h_n=y_n.$ We first identify the limiting coordinates of $y_n$. Since $\lambda_{\out,n}{\to}\lambda_{\out}$ in probability and the block convergence above is uniform near $\lambda_{\out}$, we have
\[
B_n(\lambda_{\out,n})\stackrel{\text{in prob.}}{\longrightarrow} \tau
\begin{pmatrix}
r(\lambda_{\out})&0\\
0&s(\lambda_{\out})
\end{pmatrix},
\qquad
E_n(\lambda_{\out,n})\stackrel{\text{in prob.}}{\longrightarrow} I_2.
\]

The limiting diagonal matrix is invertible as $r(\lambda_{\out}),s(\lambda_{\out})<0$. Along every subsequence, compactness gives a further subsequence on which $x_n\to x_*$. By \eqref{eq:wishart-Schur-limit} and $\cS_n(\lambda_{\out,n})x_n=0$, the limit satisfies $x_*\in\ker\cS(\lambda_{\out})$ and $\|x_*\|=1$. Since the zero of the smaller eigenvalue of $\cS$ is simple, this kernel is one-dimensional, so $x_*$ is determined up to sign.

The coordinates of $x_*$ cannot vanish. Indeed, the monotonicity argument above gives $\partial_z\cS(z)\succ0$ for $z>1$, while
\[
\cS_{11}(1)=\frac{1-\tau\alpha^2}{\tau(1+\alpha)}>0,
\qquad
\cS_{22}(1)=\frac{1-\tau\beta^2}{\tau(1+\beta)}>0.
\]
Thus $\cS_{11}(\lambda_{\out})>0$ and $\cS_{22}(\lambda_{\out})>0$. If a nonzero vector in $\ker\cS(\lambda_{\out})$ were supported on only one coordinate, one of these diagonal entries would have to be zero. Therefore $ x_{*,1}\neq0$ and $ x_{*,2}\neq0$. Consequently, every convergent subsequence of $x_n$ has limit $\pm x_*$. Combining this with the preceding limit for $B_n^{-1}E_n^\intercal$, we get 
\begin{equation}\label{eq:wishart-y-coordinate-convergence}
    y_{n,1}^2\stackrel{\text{in prob.}}{\longrightarrow} \frac{x_{*,1}^2}{\tau^2 r(\lambda_{\out})^2},
    \qquad
    y_{n,2}^2\stackrel{\text{in prob.}}{\longrightarrow} \frac{x_{*,2}^2}{\tau^2 s(\lambda_{\out})^2}.
\end{equation}

We next express the normalization of $h_n$ through the derivative of the Schur complement. Since $\calL_0'(z)=-\Pi$, where $\Pi=\diag(I_{2m},0)$, we have
\[
    \mathfrak S_n'(z)=\calU_n^\intercal\calG_0(z)\Pi\calG_0(z)\calU_n.
\]
Differentiating the factored Schur complement identity and then testing against $x_n$ at $z=\lambda_{\out,n}$ gives
\begin{equation}\label{eq:wishart-normalization-identity}
    x_n^\intercal\partial_z\cS_n(\lambda_{\out,n})x_n
    =
    \binom{x_n}{y_n}^\intercal
    \mathfrak S_n'(\lambda_{\out,n})
    \binom{x_n}{y_n}
    =\|h_n\|^2.
\end{equation}
The derivative terms hitting the outer Schur complement factors vanish because $(x_n,y_n)\in\ker\mathfrak S_n(\lambda_{\out,n})$. The same holomorphic extension and Cauchy integral argument used above for $\cA_n$ applies to the full Schur complement in a neighborhood of $\lambda_{\out}$, where $B_n(z)$ is uniformly invertible with high probability. Hence
\begin{equation}\label{eq:wishart-Schur-derivative-limit-overlap}
    x_n^\intercal\partial_z\cS_n(\lambda_{\out,n})x_n
   \stackrel{\text{in prob.}}{ \longrightarrow}
    x_*^\intercal\partial_z\cS(\lambda_{\out})x_*.
\end{equation}

Finally, write $S:=A^{1/2}\otimes I_m$ and $\hat w_n=(\hat a_n,\hat b_n)$. Since $p=A^{1/2}e_1$ and $q=A^{1/2}e_2$, we see by the identity $y_n = X^\intercal h_n$ that 
\[
    y_n
    =
    \|S h_n\|
    \begin{pmatrix}
        \sqrt{\alpha\theta_1}\,\langle \hat a_n,a\rangle\\[1mm]
        \sqrt{\beta\theta_2}\,\langle \hat b_n,b\rangle
    \end{pmatrix}.
\]
Because $\|S h_n\|^2\le\|S\|_{\op}^2\|h_n\|^2=(1+\kappa)\|h_n\|^2$, \eqref{eq:wishart-normalization-identity} gives the bounds
\begin{equation}\label{eq:wishart-individual-overlap-finite}
\begin{aligned}
|\langle\hat a_n,a\rangle|^2
&\ge
\frac{y_{n,1}^2}{\alpha\theta_1(1+\kappa)\,x_n^\intercal\partial_z\cS_n(\lambda_{\out,n})x_n},\\
|\langle\hat b_n,b\rangle|^2
&\ge
\frac{y_{n,2}^2}{\beta\theta_2(1+\kappa)\,x_n^\intercal\partial_z\cS_n(\lambda_{\out,n})x_n}.
\end{aligned}
\end{equation}
Using \eqref{eq:wishart-y-coordinate-convergence} and \eqref{eq:wishart-Schur-derivative-limit-overlap}, we obtain that with high probability,
\begin{equation*}
\begin{aligned}
   |\langle\hat a,a\rangle|^2
   &\ge
   \frac{x_{*,1}^2}{\tau^2 r(\lambda_{\out})^2\,\alpha\theta_1(1+\kappa)\,x_*^\intercal\partial_z\cS(\lambda_{\out})x_*}
   +o(1),\\
   |\langle\hat b,b\rangle|^2
   &\ge
   \frac{x_{*,2}^2}{\tau^2 s(\lambda_{\out})^2\,\beta\theta_2(1+\kappa)\,x_*^\intercal\partial_z\cS(\lambda_{\out})x_*}
   +o(1).
\end{aligned}
\end{equation*}
The constants are strictly positive as $r(\lambda_{\out}),s(\lambda_{\out})<0$ and $\partial_z\cS(\lambda_{\out})\succ0$. Since $\hat w_n=\pm\hat w$, this proves the theorem.
\end{proof}

\begin{rmk}
The same argument also shows that if $\rho<\kappa$, then $\lambda_1(W)\le 1+o(1)$ with high probability. Indeed, in that case \eqref{eq:wishart-det-at-edge} gives $\cS(1)\succ0$, so the deterministic Schur complement stays positive on $[1,\infty)$ and no outlier can emerge above the null edge.
\end{rmk}

\subsubsection{Mismatched Wishart outlier}
\label{subsubsec:mismatched-wishart-outlier}

In this subsection, we record the Wishart outlier argument in the form needed
for the parameter-free grid search.  The true model is
\begin{equation}\label{eq:mismatched-wishart-true-model}
    U=U_0+\sqrt{\frac{\alpha}{m}}\,u a^\intercal,
    \qquad
    V=V_0+\sqrt{\frac{\beta}{m}}\,v b^\intercal,
\end{equation}
where the true parameters $\alpha,\beta$ are unknown and satisfy
$\alpha\sqrt\tau\leq1$ and $\beta\sqrt\tau\leq1$.  We construct the spectral matrix
using a possibly different pair $(\tilde\alpha,\tilde\beta)$.  Throughout this
subsection this pair is fixed and admissible, meaning that
\begin{equation}\label{eq:mismatched-wishart-admissible}
    0<\tilde\alpha,\tilde\beta<\tau^{-1/2},
    \qquad
    \tilde\kappa
    =
    \left(\frac{(1-\tau\tilde\alpha^2)(1-\tau\tilde\beta^2)}
    {\tau^2\tilde\alpha^2\tilde\beta^2}\right)^{1/4}\in(0,1).
\end{equation}
Grid points for which \eqref{eq:mismatched-wishart-admissible} fails are not
used in the outlier argument.  Set paramteres $ \tilde\theta_1 = \tilde\alpha / (1 + \tilde \alpha)$ and $\tilde \theta _2  = \tilde\beta / (1 + \tilde \beta).$ The matrices and the corresponding block matrix built from the grid point are
\begin{equation}\label{eq:mismatched-wishart-normalized-blocks}
\begin{aligned}
    U_{\tilde\alpha}:=\tilde\theta_1 U^\intercal U-\tilde\alpha\tau I_m \\
     V_{\tilde\beta}:=\tilde\theta_2 V^\intercal V-\tilde\beta\tau I_m
\end{aligned} \quad \text{ and } \quad   W_{\tilde\alpha,\tilde\beta}
    :=
    \begin{pmatrix}
        U_{\tilde\alpha} & \tilde\kappa V_{\tilde\beta}\\
        \tilde\kappa U_{\tilde\alpha} & V_{\tilde\beta}
    \end{pmatrix}.
    \quad 
\end{equation}
Let us define 
\begin{equation}\label{eq:mismatched-wishart-A-pq-def}
    \tilde A:=
    \begin{pmatrix}
        1&\tilde\kappa\\
        \tilde\kappa&1
    \end{pmatrix},
    \qquad
    \tilde p:=\tilde A^{1/2}e_1,
    \qquad
    \tilde q:=\tilde A^{1/2}e_2,
\end{equation}
and write $\tilde P:=\tilde p\tilde p^\intercal$ and
$\tilde Q:=\tilde q\tilde q^\intercal$.  Also set $\tilde S:=\tilde A^{1/2}\otimes I_m.$ The symmetric conjugate of $W_{\tilde\alpha,\tilde\beta}$ is given by the matrix 
\begin{equation}\label{eq:mismatched-wishart-H-def}
    \tilde H
    :=
    \tilde S
    \begin{pmatrix}
        U_{\tilde\alpha}&0\\
        0&V_{\tilde\beta}
    \end{pmatrix}
    \tilde S \implies  W_{\tilde\alpha,\tilde\beta}=\tilde S\tilde H\tilde S^{-1}.
\end{equation}
so $W_{\tilde\alpha,\tilde\beta}$ and $\tilde H$ have the same eigenvalues. We next define the null objects associated with the grid point.  Let
\begin{equation}\label{eq:mismatched-wishart-null-blocks}
\begin{aligned}
      U_{\tilde\alpha,0}:=\tilde\theta_1 U_0^\intercal U_0-\tilde\alpha\tau I_m,\\
        V_{\tilde\beta,0}:=\tilde\theta_2 V_0^\intercal V_0-\tilde\beta\tau I_m, 
\end{aligned}
  \quad \text{ and } \quad  \tilde H_0
    :=
    \tilde S
    \begin{pmatrix}
        U_{\tilde\alpha,0}&0\\
        0&V_{\tilde\beta,0}
    \end{pmatrix}
    \tilde S.
\end{equation}
Equivalently, if $ \tilde B_0 = \diag(\sqrt{\tilde \theta_1} U_0, \sqrt{
\tilde \theta_2} V_0) \tilde S$ and $\tilde C:=\tau\tilde\alpha\tilde P+\tau\tilde\beta\tilde Q $ then
\begin{equation}\label{eq:mismatched-wishart-H0-gram-form}
    \tilde H_0=\tilde B_0^\intercal\tilde B_0-\tilde C\otimes I_m.
\end{equation}
The null linearization is therefore
\begin{equation}\label{eq:mismatched-wishart-null-linearization}
    \tilde\calL_0(z)
    :=
    \begin{pmatrix}
        -(zI_{2m}+\tilde C\otimes I_m)&\tilde B_0^\intercal\\
        \tilde B_0&-I_{2n}
    \end{pmatrix},
    \qquad
    \tilde\calG_0(z):=\tilde\calL_0(z)^{-1}.
\end{equation}
The top-left block of $\tilde\calG_0(z)$ is
$(\tilde H_0-zI_{2m})^{-1}$ whenever $z$ lies outside the spectrum of
$\tilde H_0$. Let $(\tilde r,\tilde s,\tilde T)$ denote the deterministic Wishart solution
from Definition~\ref{def:det_system_wishart} with
$(\alpha,\beta,\kappa,p,q,P,Q,\theta_1,\theta_2)$ replaced by
$(\tilde\alpha,\tilde\beta,\tilde\kappa,\tilde p,\tilde q,
\tilde P,\tilde Q,\tilde\theta_1,\tilde\theta_2)$.  Thus
\begin{gather}\label{eq:mismatched-wishart-deterministic-system}
\begin{aligned}
    \tilde r(z)
    &=
    \frac{1}{-1-\tilde\theta_1\tilde p^\intercal\tilde T(z)\tilde p},\\
    \tilde s(z)
    &=
    \frac{1}{-1-\tilde\theta_2\tilde q^\intercal\tilde T(z)\tilde q},
\end{aligned}
\\
\tilde T(z)^{-1}
=
-zI_2
-\tau\tilde\alpha\tilde P
-\tau\tilde\beta\tilde Q
-\tau\tilde\theta_1\tilde r(z)\tilde P
-\tau\tilde\theta_2\tilde s(z)\tilde Q .
\end{gather}
By Corollary~\ref{cor:wishart-critical-identities-outlier}, applied to the null system
with the grid parameters,
\begin{equation}\label{eq:mismatched-wishart-critical-identities}
    \tilde T(1)=-I_2,
    \qquad
    \tilde r(1)=-(1+\tilde\alpha),
    \qquad
    \tilde s(1)=-(1+\tilde\beta).
\end{equation}
The finite-rank perturbation of the linearization is described by
\begin{equation}\label{eq:mismatched-wishart-X-Y-def}
    \tilde X
    :=
    \bigl[
        \sqrt{\alpha\tilde\theta_1}\,\tilde p\otimes a
        \ \ 
        \sqrt{\beta\tilde\theta_2}\,\tilde q\otimes b
    \bigr]
    \in\RR^{2m\times2},
    \qquad
    Y
    :=
    \frac1{\sqrt m}
    \bigl[e_1\otimes u\ \ e_2\otimes v\bigr]
    \in\RR^{2n\times2}.
\end{equation}
Here the spike amplitudes in
\eqref{eq:mismatched-wishart-true-model} contribute the true factors
$\alpha,\beta$, whereas the symmetrization contribute
$\tilde\theta_1,\tilde\theta_2,\tilde p,\tilde q$.  Thus the true factors
$\alpha\theta_1$ and $\beta\theta_2$ are replaced by
$\alpha\tilde\theta_1$ and $\beta\tilde\theta_2$. Let
\begin{equation}\label{eq:mismatched-wishart-U-J-def}
    \tilde\calU_n
    :=
    \begin{pmatrix}
        \tilde X&0\\
        0&Y
    \end{pmatrix},
    \qquad
    \calJ:=
    \begin{pmatrix}
        0&I_2\\
        I_2&0
    \end{pmatrix}.
\end{equation}
If $\tilde\calL(z)$ denotes the linearization associated with the planted matrix
$\tilde H$, then expanding
\eqref{eq:mismatched-wishart-true-model} gives the exact identity
\begin{equation}\label{eq:mismatched-wishart-linearization-perturbation}
    \tilde\calL(z)
    =
    \tilde\calL_0(z)+\tilde\calU_n\calJ\tilde\calU_n^\intercal.
\end{equation}
This is the Wishart analogue of the rank-two decomposition in the mismatched
Wigner section. For real $z>1$, define the deterministic $2\times2$ matrix
\begin{equation}\label{eq:mismatched-wishart-Sigma-rho}
\tilde\cS(z)
:=
\begin{pmatrix}
\alpha\tilde\theta_1\,\tilde p^\intercal\tilde T(z)\tilde p
-\dfrac{1}{\tau\tilde r(z)}
&
\rho\sqrt{\alpha\tilde\theta_1\beta\tilde\theta_2}\,
\tilde p^\intercal\tilde T(z)\tilde q\\[2mm]
\rho\sqrt{\alpha\tilde\theta_1\beta\tilde\theta_2}\,
\tilde q^\intercal\tilde T(z)\tilde p
&
\beta\tilde\theta_2\,\tilde q^\intercal\tilde T(z)\tilde q
-\dfrac{1}{\tau\tilde s(z)}
\end{pmatrix}.
\end{equation}
The mismatched Wishart threshold associated with the grid point is
\begin{equation}\label{eq:mismatched-wishart-rho-out-def}
    \rho_{\out}(\tilde\alpha,\tilde\beta;\alpha,\beta)
    :=
    \frac{
        \sqrt{(\tau^{-1}-\alpha\tilde\alpha)
        (\tau^{-1}-\beta\tilde\beta)}
    }{
        \tilde\kappa\sqrt{\alpha\beta\tilde\alpha\tilde\beta}
    }.
\end{equation}
At the matched point, this reduces to the original threshold $\rho_{\out}(\alpha,\beta;\alpha,\beta)=\kappa.$ 

\begin{lemma}\label{lem:mismatched-wishart-S-edge}
For $\tilde\cS$ defined in \eqref{eq:mismatched-wishart-Sigma-rho},
\begin{equation}\label{eq:mismatched-wishart-S-edge}
\tilde\cS(1)
=
\begin{pmatrix}
\dfrac{\tau^{-1}-\alpha\tilde\alpha}{1+\tilde\alpha}
&
-\dfrac{
\rho\tilde\kappa\sqrt{\alpha\beta\tilde\alpha\tilde\beta}
}{
\sqrt{(1+\tilde\alpha)(1+\tilde\beta)}
}\\[3mm]
-\dfrac{
\rho\tilde\kappa\sqrt{\alpha\beta\tilde\alpha\tilde\beta}
}{
\sqrt{(1+\tilde\alpha)(1+\tilde\beta)}
}
&
\dfrac{\tau^{-1}-\beta\tilde\beta}{1+\tilde\beta}
\end{pmatrix}.
\end{equation}
Consequently,
\begin{equation}\label{eq:mismatched-wishart-S-edge-det}
    \det\tilde\cS(1)
    =
    \frac{
    (\tau^{-1}-\alpha\tilde\alpha)(\tau^{-1}-\beta\tilde\beta)
    -\rho^2\tilde\kappa^2\alpha\beta\tilde\alpha\tilde\beta
    }
    {(1+\tilde\alpha)(1+\tilde\beta)}.
\end{equation}
In particular, $\det\tilde\cS(1)<0$ whenever $\rho>\rho_{\out}(\tilde\alpha,\tilde\beta;\alpha,\beta).$
\end{lemma}

\begin{proof}
By \eqref{eq:mismatched-wishart-critical-identities},
$\tilde T(1)=-I_2$, $\tilde r(1)=-(1+\tilde\alpha)$, and
$\tilde s(1)=-(1+\tilde\beta)$.  Since
\[
    \tilde p^\intercal\tilde p=\tilde q^\intercal\tilde q=1,
    \qquad
    \tilde p^\intercal\tilde q=\tilde\kappa,
\]
we have $ \tilde p^\intercal\tilde T(1)\tilde p
    =
    \tilde q^\intercal\tilde T(1)\tilde q
    =-1,$ and $\tilde p^\intercal\tilde T(1)\tilde q=-\tilde\kappa.$  The first diagonal entry of \eqref{eq:mismatched-wishart-Sigma-rho} at $z=1$
is therefore
\[
    -\alpha\tilde\theta_1+\frac{1}{\tau(1+\tilde\alpha)}
    =
    -\frac{\alpha\tilde\alpha}{1+\tilde\alpha}
    +
    \frac{\tau^{-1}}{1+\tilde\alpha}
    =
    \frac{\tau^{-1}-\alpha\tilde\alpha}{1+\tilde\alpha},
\]
and the second diagonal entry is identical with
$(\alpha,\tilde\alpha)$ replaced by $(\beta,\tilde\beta)$.  The off-diagonal
entry is
\[
    \rho\sqrt{\alpha\tilde\theta_1\beta\tilde\theta_2}
    \;\tilde p^\intercal\tilde T(1)\tilde q
    =
    -\rho\tilde\kappa
    \sqrt{\frac{\alpha\beta\tilde\alpha\tilde\beta}
    {(1+\tilde\alpha)(1+\tilde\beta)}}.
\]
This proves \eqref{eq:mismatched-wishart-S-edge}.  Taking the determinant gives
\eqref{eq:mismatched-wishart-S-edge-det}, and the final claim follows directly
from the definition \eqref{eq:mismatched-wishart-rho-out-def}.
\end{proof}

\begin{theorem}[Mismatched Wishart outlier]
\label{thm:mismatched-Wishart-outlier}
Fix true parameters $\alpha,\beta$ with
$0<\alpha^2 \tau , \beta^2 \tau\le1$, and let
$(\tilde\alpha,\tilde\beta)$ be admissible in the sense of
\eqref{eq:mismatched-wishart-admissible}.  Suppose
\[
    \rho>\rho_{\out}(\tilde\alpha,\tilde\beta;\alpha,\beta).
\]
Then there exists a unique $\tilde\lambda_{\out}>1$ such that $\det\tilde\cS(\tilde\lambda_{\out})=0.$ Moreover, with high probability,
\[
    \lambda_1(W_{\tilde\alpha,\tilde\beta})
    =\tilde\lambda_{\out}+o(1),
    \qquad
    \lambda_2(W_{\tilde\alpha,\tilde\beta})\le 1+o(1).
\]
Furthermore, if $\tilde w=(\hat a,\hat b)$ is a top unit right eigenvector of $W_{\tilde\alpha,\tilde\beta}$, and if $\tilde x_*=(\tilde x_{*,1},\tilde x_{*,2})^\intercal\in\ker\tilde\cS(\tilde\lambda_{\out})$ is any unit vector, then
\begin{equation*}
\begin{aligned}
   |\langle\hat a,a\rangle|^2
   &\ge
   \frac{\tilde x_{*,1}^2}{\tau^2\tilde r(\tilde\lambda_{\out})^2\,\alpha\tilde\theta_1(1+\tilde\kappa)\,\tilde x_*^\intercal\partial_z\tilde\cS(\tilde\lambda_{\out})\tilde x_*}
   +o(1),\\
   |\langle\hat b,b\rangle|^2
   &\ge
   \frac{\tilde x_{*,2}^2}{\tau^2\tilde s(\tilde\lambda_{\out})^2\,\beta\tilde\theta_2(1+\tilde\kappa)\,\tilde x_*^\intercal\partial_z\tilde\cS(\tilde\lambda_{\out})\tilde x_*}
   +o(1).
\end{aligned}
\end{equation*}
The two deterministic constants on the right-hand side are strictly positive. In particular, $|\langle\hat a,a\rangle|^2=\Omega(1)$ and $|\langle\hat b,b\rangle|=\Omega(1)$ with high probability.
\end{theorem}

\begin{proof}
The proof is the matched Wishart proof with the deterministic null law and the
finite rank perturbation replaced by their grid dependent versions.  We give the
full set of replacements. For $z>\lambda_{\max}(\tilde H_0)$, the matrix
$\tilde\calG_0(z)$ is well defined.  By the determinant lemma and
\eqref{eq:mismatched-wishart-linearization-perturbation},
\begin{equation}\label{eq:mismatched-wishart-det-factorization}
    \det\tilde\calL(z)
    =
    \det\tilde\calL_0(z)\det\tilde{\mathfrak S}_n(z),
    \qquad
    \tilde{\mathfrak S}_n(z)
    :=
    \calJ+\tilde\calU_n^\intercal\tilde\calG_0(z)\tilde\calU_n.
\end{equation}
Write $\tilde\calG_0(z)$ in top/bottom block form,
\[
\tilde\calG_0(z)
=
\begin{pmatrix}
\tilde\calG_0^{\mathrm{tt}}(z)&\tilde\calG_0^{\mathrm{tb}}(z)\\
\tilde\calG_0^{\mathrm{bt}}(z)&\tilde\calG_0^{\mathrm{bb}}(z)
\end{pmatrix},
\]
where the top block has dimension $2m$ and the bottom block has dimension $2n$.
Then
\[
\tilde{\mathfrak S}_n(z)
=
\begin{pmatrix}
\tilde X^\intercal\tilde\calG_0^{\mathrm{tt}}(z)\tilde X
&
I_2+\tilde X^\intercal\tilde\calG_0^{\mathrm{tb}}(z)Y\\
I_2+Y^\intercal\tilde\calG_0^{\mathrm{bt}}(z)\tilde X
&
Y^\intercal\tilde\calG_0^{\mathrm{bb}}(z)Y
\end{pmatrix}.
\]
As in Theorem~\ref{thm:Wishart-outlier}, for
$z>\lambda_{\max}(\tilde H_0)$ the lower-right block is negative definite.
Therefore the Schur complement with respect to this block is well defined.  Set
\begin{equation}\label{eq:mismatched-wishart-E-D-def}
    \tilde E_n(z)
    :=
    I_2+\tilde X^\intercal\tilde\calG_0^{\mathrm{tb}}(z)Y,
    \qquad
    \tilde D_n(z)
    :=
    Y^\intercal\tilde\calG_0^{\mathrm{bb}}(z)Y,
\end{equation}
and $ \tilde\cS_n(z)
    :=
    \tilde X^\intercal\tilde\calG_0^{\mathrm{tt}}(z)\tilde X
    -
    \tilde E_n(z)\tilde D_n(z)^{-1}\tilde E_n(z)^\intercal .$ Since
\[
    \det\tilde{\mathfrak S}_n(z)
    =
    \det\tilde D_n(z)\det\tilde\cS_n(z),
\]
we obtain the exact reduction
\begin{equation}\label{eq:mismatched-wishart-exact-reduction}
    z>\lambda_{\max}(\tilde H_0),
    \qquad
    z\in\Lambda(\tilde H)
    \quad\Longleftrightarrow\quad
    \det\tilde\cS_n(z)=0.
\end{equation}

We next identify the deterministic limit of \(\tilde\cS_n\).  Fix a compact
interval \(K\subset(1,\infty)\).  By Theorem~\ref{thm:WishartGordonEdge}, applied to
the null model with parameters \((\tilde\alpha,\tilde\beta,\tilde\kappa)\),
\begin{equation}\label{eq:mismatched-wishart-null-edge}
    \lambda_{\max}(\tilde H_0)\le 1+o(1)
\end{equation}
with overwhelming probability.  On this event the blocks of
\(\tilde\calG_0(z)\) are uniformly bounded and uniformly Lipschitz on \(K\).
Conditioning on \(\tilde H_0\), the matrices appearing in the quadratic forms
below are deterministic and independent of \(a,b,u,v\).  The same
Hanson Wright and Gaussian bilinear estimates used in the proof of
Theorem~\ref{thm:Wishart-outlier} give that with overwhelming probability, uniformly for \(z\in K\),
\begin{gather*}
a^\intercal A(z)a=\frac1m\Tr A(z)+o(1), \\
b^\intercal A(z)b=\frac1m\Tr A(z)+o(1),\\
a^\intercal A(z)b=\rho\,\frac1m\Tr A(z)+o(1),
\end{gather*}
and the analogous estimates for the \(u,v\) quadratic forms, with the
normalizing factor \(n/m\to\tau\), also holds with overwhelming probability.  Combining these estimates with Borel-Cantelli lemma and
Theorem~\ref{thm:real-stieltjes-transform-Wishart}, again with the grid parameters,
yields that almost surely, uniformly over $z\in K$, 
\begin{align}
\tilde X^\intercal\tilde\calG_0^{\mathrm{tt}}(z)\tilde X
&\longrightarrow
\begin{pmatrix}
\alpha\tilde\theta_1\,\tilde p^\intercal\tilde T(z)\tilde p
&
\rho\sqrt{\alpha\tilde\theta_1\beta\tilde\theta_2}\,
\tilde p^\intercal\tilde T(z)\tilde q\\[1mm]
\rho\sqrt{\alpha\tilde\theta_1\beta\tilde\theta_2}\,
\tilde q^\intercal\tilde T(z)\tilde p
&
\beta\tilde\theta_2\,\tilde q^\intercal\tilde T(z)\tilde q
\end{pmatrix},
\label{eq:mismatched-wishart-top-limit}\\
\tilde D_n(z)
&\longrightarrow
\tau
\begin{pmatrix}
\tilde r(z)&0\\
0&\tilde s(z)
\end{pmatrix},
\qquad
\tilde E_n(z)\longrightarrow I_2 .
\label{eq:mismatched-wishart-bottom-limit}
\end{align}
 Since
\(\tilde r(z),\tilde s(z)<0\) for \(z>1\), the limiting matrix in
\eqref{eq:mismatched-wishart-bottom-limit} is uniformly invertible on \(K\).
Therefore
\begin{equation}\label{eq:mismatched-wishart-Schur-limit}
    \sup_{z\in K}\|\tilde\cS_n(z)-\tilde\cS(z)\|\longrightarrow0
\end{equation}
holds almost surely, where \(\tilde\cS(z)\) is the deterministic matrix in
\eqref{eq:mismatched-wishart-Sigma-rho}.

We now analyze \(\tilde\cS(z)\).  By
Lemma~\ref{lem:mismatched-wishart-S-edge} and the assumption on \(\rho\),
\(\det\tilde\cS(1)<0\).  The diagonal entries of \(\tilde\cS(1)\) are positive,
because $ \alpha\tilde\alpha, \beta \tilde \beta<\tau^{-1}$ 
Thus \(\tilde\cS(1)\) has exactly one positive and one negative eigenvalue.  On
the other hand, from the large-\(z\) asymptotics of the Wishart deterministic
system,
\[
    \tilde r(z)=-1+O(z^{-1}),
    \qquad
    \tilde s(z)=-1+O(z^{-1}),
    \qquad
    \tilde T(z)=O(z^{-1}),
\]
and hence $ \tilde\cS(z)=\tau^{-1}I_2+O(z^{-1}).$ Therefore \(\tilde\cS(z)\succ0\) for all sufficiently large real \(z\). It remains to recall the uniqueness argument.  Let
\(K=[1+\delta,R]\subset(1,\infty)\), and define $\tilde\cA_n(z)
    :=
    \tilde X^\intercal\tilde\calG_0^{\mathrm{tt}}(z)\tilde X. $ For \(z>\lambda_{\max}(\tilde H_0)\),
\[
    \tilde\cA_n'(z)
    =
    \tilde X^\intercal(\tilde H_0-zI_{2m})^{-2}\tilde X .
\]
The null Gram representation \eqref{eq:mismatched-wishart-H0-gram-form} gives
\[
    \tilde H_0\succeq
    -\tau(\tilde\alpha\tilde P+\tilde\beta\tilde Q)\otimes I_m
    \succeq
    -\tau(\tilde\alpha+\tilde\beta)I_{2m}.
\]
Thus, uniformly for \(z\in K\),
\[
    (\tilde H_0-zI_{2m})^{-2}
    \succeq
    \frac{1}{(R+\tau(\tilde\alpha+\tilde\beta))^2}I_{2m}.
\]
Moreover, using concentration estimates and Borel-Cantelli lemma, we see that almost surely, 
\begin{equation}\label{eq:mismatched-wishart-X-gram}
\tilde X^\intercal\tilde X
\longrightarrow
\begin{pmatrix}
\alpha\tilde\theta_1
&
\rho\tilde\kappa\sqrt{\alpha\tilde\theta_1\beta\tilde\theta_2}\\
\rho\tilde\kappa\sqrt{\alpha\tilde\theta_1\beta\tilde\theta_2}
&
\beta\tilde\theta_2
\end{pmatrix},
\end{equation}
and the limiting matrix is semi-positive definite because \(\rho\tilde\kappa<1\).  Hence, almost surely,
\(\tilde\cA_n'(z)\succeq c_K I_2\) on \(K\), for a deterministic constant
\(c_K>0\) and all sufficiently large $n$.  Passing this lower bound to the deterministic limit via Vitali's theorem and Cauchy derivative arguments as in the proof of Theorem~\ref{thm:Wishart-outlier},
we see the deterministic limit \(\tilde\cA(z)\) also satisfies
\(\tilde\cA'(z)\succeq c_K I_2\) on \(K\).

Finally, since \(\tilde r\) and \(\tilde s\) are Stieltjes transforms supported
on \((-\infty,1]\), we have \(\tilde r(z),\tilde s(z)<0\) and
\(\tilde r'(z),\tilde s'(z)>0\) for \(z>1\).  Therefore
\[
\partial_z\!\left[
-\frac1\tau
\begin{pmatrix}
\tilde r(z)^{-1}&0\\
0&\tilde s(z)^{-1}
\end{pmatrix}
\right]
=
\frac1\tau
\begin{pmatrix}
\dfrac{\tilde r'(z)}{\tilde r(z)^2}&0\\[2mm]
0&\dfrac{\tilde s'(z)}{\tilde s(z)^2}
\end{pmatrix}
\succeq0.
\]
Consequently \(\tilde\cS'(z)\succeq c_KI_2\) on every compact
\(K\subset(1,\infty)\).  The two eigenvalues of \(\tilde\cS(z)\) are therefore
strictly increasing on \((1,\infty)\).  Since the smaller one is negative at
\(z=1\) and both are positive for large \(z\), there is a unique
\(\tilde\lambda_{\out}>1\) such that
\[
    \det\tilde\cS(\tilde\lambda_{\out})=0.
\]

We now return to the random matrix.  Choose \(\delta>0\) so that
\(1<\tilde\lambda_{\out}-\delta\) and the smallest eigenvalue of
\(\tilde\cS(z)\) is negative at \(z=\tilde\lambda_{\out}-\delta\) and positive
at \(z=\tilde\lambda_{\out}+\delta\).  By
\eqref{eq:mismatched-wishart-Schur-limit}, the same is true for
\(\tilde\cS_n(z)\) with overwhelming probability.  On the same event,
\eqref{eq:mismatched-wishart-null-edge} gives $ \lambda_{\max}(\tilde H_0)<\tilde\lambda_{\out}-\delta .$  Therefore the smaller eigenvalue of \(\tilde\cS_n(z)\) has a unique zero
\[
    \tilde\lambda_{\out,n}
    \in
    (\tilde\lambda_{\out}-\delta,\tilde\lambda_{\out}+\delta),
\]
and \(\tilde\lambda_{\out,n}\to\tilde\lambda_{\out}\) in probability.  By
\eqref{eq:mismatched-wishart-exact-reduction}, this zero is an eigenvalue of
\(\tilde H\), and hence also of \(W_{\tilde\alpha,\tilde\beta}\).

To rule out a second outlier, choose \(\delta>0\) so small that
\(\det\tilde\cS(1+\delta)<0\).  Then \(\tilde\cS(1+\delta)\) has one positive and
one negative eigenvalue.  By \eqref{eq:mismatched-wishart-Schur-limit}, the same
is true for \(\tilde\cS_n(1+\delta)\) with overwhelming probability.  The random
monotonicity argument from the proof of Theorem~\ref{thm:Wishart-outlier}, applied to
\(\tilde{\mathfrak S}_n\), shows that the eigenvalues of
\(\tilde\cS_n(z)\) are nondecreasing for
\(z>\lambda_{\max}(\tilde H_0)\).  Hence only the smaller eigenvalue can hit
zero, and it can do so at most once.  Therefore, with overwhelming probability,
\(\tilde H\) has exactly one eigenvalue above \(1+\delta\), namely
\(\tilde\lambda_{\out,n}\).  Since \(W_{\tilde\alpha,\tilde\beta}\) is similar
to \(\tilde H\), and since \(\delta>0\) is arbitrary, we conclude that
\[
    \lambda_1(W_{\tilde\alpha,\tilde\beta})
    =
    \tilde\lambda_{\out}+o(1),
    \qquad
    \lambda_2(W_{\tilde\alpha,\tilde\beta})\le 1+o(1).
\]

It remains to prove the displayed individual-overlap bounds. Let $\tilde x_n\in\RR^2$ be a unit vector spanning $\ker\tilde\cS_n(\tilde\lambda_{\out,n})$, and set
\[
    \tilde y_n:=-\tilde D_n(\tilde\lambda_{\out,n})^{-1}\tilde E_n(\tilde\lambda_{\out,n})^\intercal\tilde x_n.
\]
Then $(\tilde x_n,\tilde y_n)\in\ker\tilde{\mathfrak S}_n(\tilde\lambda_{\out,n})$. Define
\[
    \tilde\psi_n
    :=
    -\tilde\calG_0(\tilde\lambda_{\out,n})\tilde\calU_n
    \binom{\tilde x_n}{\tilde y_n}
    =:\binom{\tilde h_n}{\tilde g_n}.
\]
As before, $\tilde\calL(\tilde\lambda_{\out,n})\tilde\psi_n=0$, so $\tilde h_n$ is an eigenvector of $\tilde H$ with eigenvalue $\tilde\lambda_{\out,n}$, and
\[
    \tilde w_n:=\frac{\tilde S\tilde h_n}{\|\tilde S\tilde h_n\|}
\]
is a top unit right eigenvector of $W_{\tilde\alpha,\tilde\beta}$, hence agrees with $\tilde w$ up to a sign. Moreover,
\[
\tilde\calU_n^\intercal\tilde\psi_n
=-\tilde\calU_n^\intercal\tilde\calG_0(\tilde\lambda_{\out,n})\tilde\calU_n
\binom{\tilde x_n}{\tilde y_n}
=\calJ\binom{\tilde x_n}{\tilde y_n},
\]
and therefore $\tilde X^\intercal\tilde h_n=\tilde y_n.$ Since $\lambda_{\operatorname{out},n}\to \tilde{\lambda}_{\operatorname{out}}$ in probability, by \eqref{eq:mismatched-wishart-bottom-limit}, we have
\[
\tilde D_n(\tilde\lambda_{\out,n})\stackrel{\text{in prob.}}{\longrightarrow}\tau
\begin{pmatrix}
\tilde r(\tilde\lambda_{\out})&0\\
0&\tilde s(\tilde\lambda_{\out})
\end{pmatrix},
\qquad
\tilde E_n(\tilde\lambda_{\out,n})\stackrel{\text{in prob.}}{\longrightarrow} I_2.
\]
The limiting diagonal matrix is invertible because $\tilde r(\tilde\lambda_{\out}),\tilde s(\tilde\lambda_{\out})<0$. Along every subsequence, compactness gives a further subsequence on which $\tilde x_n\to\tilde x_*$. By \eqref{eq:mismatched-wishart-Schur-limit} and $\tilde\cS_n(\tilde\lambda_{\out,n})\tilde x_n=0$, the limit satisfies $\tilde x_*\in\ker\tilde\cS(\tilde\lambda_{\out})$ and $\|\tilde x_*\|=1$. The kernel is one-dimensional by the strict monotonicity of the smaller eigenvalue at the unique zero, so $\tilde x_*$ is determined up to sign.

The coordinates of $\tilde x_*$ are both nonzero. Indeed, $\partial_z\tilde\cS(z)\succ0$ for $z>1$, and the diagonal entries at the edge are
\[
\tilde\cS_{11}(1)=\frac{1-\tau\alpha\tilde\alpha}{\tau(1+\tilde\alpha)}>0,
\qquad
\tilde\cS_{22}(1)=\frac{1-\tau\beta\tilde\beta}{\tau(1+\tilde\beta)}>0.
\]
Thus $\tilde\cS_{11}(\tilde\lambda_{\out})$ and $\tilde\cS_{22}(\tilde\lambda_{\out})$ are strictly positive. A nonzero kernel vector supported on only one coordinate would force one of these diagonal entries to vanish. Hence $\tilde x_{*,1}\neq0$ and  $ \tilde x_{*,2}\neq0$. Consequently, we have 
\begin{equation}\label{eq:mismatched-wishart-y-coordinate-convergence}
    \tilde y_{n,1}^2\stackrel{\text{in prob.}}{\longrightarrow}
    \frac{\tilde x_{*,1}^2}{\tau^2\tilde r(\tilde\lambda_{\out})^2},
    \qquad
    \tilde y_{n,2}^2\stackrel{\text{in prob.}}{\longrightarrow}
    \frac{\tilde x_{*,2}^2}{\tau^2\tilde s(\tilde\lambda_{\out})^2}.
\end{equation}

The derivative identity is unchanged from the matched case. Since $\tilde\calL_0'(z)=-\Pi$ with $\Pi=\diag(I_{2m},0)$,
\[
    \tilde{\mathfrak S}_n'(z)
    =
    \tilde\calU_n^\intercal\tilde\calG_0(z)\Pi\tilde\calG_0(z)\tilde\calU_n.
\]
Differentiating the factored Schur-complement identity at $z=\tilde\lambda_{\out,n}$ and testing against $\tilde x_n$ gives
\begin{equation}\label{eq:mismatched-wishart-normalization-identity}
    \tilde x_n^\intercal\partial_z\tilde\cS_n(\tilde\lambda_{\out,n})\tilde x_n
    =
    \binom{\tilde x_n}{\tilde y_n}^\intercal
    \tilde{\mathfrak S}_n'(\tilde\lambda_{\out,n})
    \binom{\tilde x_n}{\tilde y_n}
    =\|\tilde h_n\|^2.
\end{equation}
As in the matched proof, the holomorphic-extension and Cauchy-integral argument applied to the full Schur complement yields
\begin{equation}\label{eq:mismatched-wishart-Schur-derivative-limit-overlap}
    \tilde x_n^\intercal\partial_z\tilde\cS_n(\tilde\lambda_{\out,n})\tilde x_n
   \stackrel{\text{in prob.}}{\longrightarrow}
    \tilde x_*^\intercal\partial_z\tilde\cS(\tilde\lambda_{\out})\tilde x_*.
\end{equation}
Since $\tilde p=\tilde A^{1/2}e_1$, $\tilde q=\tilde A^{1/2}e_2$, $\tilde w_n=\tilde S\tilde h_n/\|\tilde S\tilde h_n\|$, and $y = \tilde X^\intercal h$, we get
\[
    \tilde y_n
    =
    \|\tilde S\tilde h_n\|
    \begin{pmatrix}
        \sqrt{\alpha\tilde\theta_1}\,\langle\hat a,a\rangle\\[1mm]
        \sqrt{\beta\tilde\theta_2}\,\langle\hat b,b\rangle
    \end{pmatrix},
\]
up to the harmless overall sign of $\tilde w_n$. Because $\|\tilde S\tilde h_n\|^2\le(1+\tilde\kappa)\|\tilde h_n\|^2$, we obtain
\begin{equation}\label{eq:mismatched-wishart-individual-overlap-finite}
\begin{aligned}
|\langle\hat a,a\rangle|^2
&\ge
\frac{\tilde y_{n,1}^2}{\alpha\tilde\theta_1(1+\tilde\kappa)\,\tilde x_n^\intercal\partial_z\tilde\cS_n(\tilde\lambda_{\out,n})\tilde x_n},\\
|\langle\hat b,b\rangle|^2
&\ge
\frac{\tilde y_{n,2}^2}{\beta\tilde\theta_2(1+\tilde\kappa)\,\tilde x_n^\intercal\partial_z\tilde\cS_n(\tilde\lambda_{\out,n})\tilde x_n}.
\end{aligned}
\end{equation}
Combining \eqref{eq:mismatched-wishart-y-coordinate-convergence}, \eqref{eq:mismatched-wishart-Schur-derivative-limit-overlap}, and \eqref{eq:mismatched-wishart-individual-overlap-finite} gives the claimed individual-overlap bounds holds with high probability. The constants are strictly positive by $\tilde r(\tilde\lambda_{\out}),\tilde s(\tilde\lambda_{\out})<0$, and by $\partial_z\tilde\cS(\tilde\lambda_{\out})\succ0$.
\end{proof}

The theorem above shows that every sufficiently dense grid point produces an
outlier.  For the parameter free estimator, we also need the converse direction
at the level of eigenvectors: any grid point whose top eigenvalue is separated
from the null edge must have an informative eigenvector.

\begin{lemma}
\label{lem:mismatched-Wishart-separated-informative}
Fix \(\varepsilon>0\), and let \((\tilde\alpha,\tilde\beta)\) be admissible.  Then there is a constant
\(c>0\) independent of $n$ such that, with high probability, the following implication holds.  If
\[
    \lambda_1(W_{\tilde\alpha,\tilde\beta})\ge 1+\varepsilon
\]
and \(w=(\hat a,\hat b)\in\RR^m\times\RR^m\) is a top unit right eigenvector of
\(W_{\tilde\alpha,\tilde\beta}\), then
\[
    |\langle\hat a,a\rangle|^2\ge c,
    \qquad
    |\langle\hat b,b\rangle|^2\ge c .
\]
\end{lemma}

\begin{proof}
Let \(h:=\tilde S^{-1}w/\|\tilde S^{-1}w\|\).  By the similarity relation \eqref{eq:mismatched-wishart-H-def}, \(h\) is a unit top eigenvector of \(\tilde H\) with eigenvalue
\(\lambda:=\lambda_1(W_{\tilde\alpha,\tilde\beta})\).  We work on the event \(\{\lambda\ge1+\varepsilon\}\); if this event does not occur, there is nothing to prove.  By Theorem~\ref{thm:WishartGordonEdge}, applied to the null matrix with the guessed parameters,
\(\lambda_{\max}(\tilde H_0)\le1+o(1)\) with overwhelming probability.  Thus, on a overwhelming-probability event and for all sufficiently large \(n\),
\[
    \lambda-\lambda_{\max}(\tilde H_0)\ge\varepsilon/2.
\]
Also, on a overwhelming-probability event, \(\lambda\) lies in a deterministic compact interval
\[
    K:=[1+\varepsilon,R]
\]
for some \(R<\infty\); this follows from the standard operator norm bounds for the Gaussian Wishart matrices and the finite rank perturbation. Let \(\tilde B\) be the planted rectangular matrix in the Gram representation of \(\tilde H\), namely
\[
    \tilde B
    :=
    \begin{pmatrix}
        \sqrt{\tilde\theta_1}\,U&0\\
        0&\sqrt{\tilde\theta_2}\,V
    \end{pmatrix}
    (\tilde A^{1/2}\otimes I_m),
    \qquad
    \tilde H=\tilde B^\intercal\tilde B-\tilde C\otimes I_m .
\]
Set \(g:=\tilde B h\) and \(\psi:=(h,g)\).  Then \(\tilde\calL(\lambda)\psi=0\).  Since
\(\lambda>\lambda_{\max}(\tilde H_0)\), the null linearization
\(\tilde\calL_0(\lambda)\) is invertible, and
\[
    \psi
    =
    -\tilde\calG_0(\lambda)\tilde\calU_n
    \calJ\tilde\calU_n^\intercal\psi .
\]
Define $(x,y) = \calJ\tilde\calU_n^\intercal\psi.$ Then \((x,y)\neq0\), and \((x,y)\in\ker\tilde{\mathfrak S}_n(\lambda)\).  Indeed,
\[
    \tilde{\mathfrak S}_n(\lambda)\binom{x}{y}
    =
    \calJ\binom{x}{y}
    +
    \tilde\calU_n^\intercal\tilde\calG_0(\lambda)\tilde\calU_n
    \binom{x}{y}
    =
    \tilde\calU_n^\intercal\psi-\tilde\calU_n^\intercal\psi
    =0.
\]
Moreover, we also have that 
\[
    \calJ\binom{x}{y}
    =
    \tilde\calU_n^\intercal\psi
    =
    \binom{\tilde X^\intercal h}{Y^\intercal g},
\]
so $  \tilde X^\intercal h=y .$ We now use the reduced Schur complement, rather than only the norm of \(y\).  Writing the block equation
\(\tilde{\mathfrak S}_n(\lambda)(x,y)^\intercal=0\) and using the notation from \eqref{eq:mismatched-wishart-E-D-def}, the lower block gives
\[
    \tilde D_n(\lambda)y+\tilde E_n(\lambda)^\intercal x=0,
    \qquad
    y=-\tilde D_n(\lambda)^{-1}\tilde E_n(\lambda)^\intercal x .
\]
Substituting this into the upper block gives
\begin{equation}\label{eq:mismatched-wishart-separated-reduced-kernel}
    \tilde\cS_n(\lambda)x=0 .
\end{equation}
In particular \(x\neq0\), because \(\tilde D_n(\lambda)\) is invertible. On \(K\), the convergence in \eqref{eq:mismatched-wishart-bottom-limit} implies that
\(\tilde D_n(z)\) is uniformly invertible and that almost surely,
\[
    -\tilde D_n(z)^{-1}\tilde E_n(z)^\intercal
    \longrightarrow
    -\frac1\tau
    \begin{pmatrix}
        \tilde r(z)^{-1}&0\\
        0&\tilde s(z)^{-1}
    \end{pmatrix}
\]
uniformly for \(z\in K\).  Hence this matrix has singular values bounded above and below by positive deterministic constants, uniformly for \(z\in K\), with overwhelming probability.  Since
\[
    \psi
    =
    -\tilde\calG_0(\lambda)\tilde\calU_n
    \binom{x}{y}
\]
and \(\tilde\calG_0(\lambda)\) and \(\tilde\calU_n\) have operator norm \(O(1)\) uniformly for \(\lambda\in K\), while \(\|\psi\|\ge\|h\|=1\), we have
\(\|(x,y)\|\ge c_0\) for a deterministic \(c_0>0\).  The preceding singular-value comparison between \(x\) and \(y\) therefore gives $  \|x\|\ge c_1>0.$ On the same event, \(\|x\|\) is also bounded above by a deterministic constant, because
\(x=Y^\intercal g\), \(g=\tilde B h\), \(\|h\|=1\), and \(\|Y\|,\|\tilde B\|=O(1)\).

We next show that the two coordinates of \(x\) are separately bounded away from zero.  Suppose, to the contrary, that along a subsequence on the above overwhelming-probability events, one coordinate of \(x\) tends to zero.  Passing to a further subsequence, we may assume
\(\lambda\to z_*\in K\) and \(x\to x_*\).  Note that \(x_*\neq0\).  By the uniform Schur convergence \eqref{eq:mismatched-wishart-Schur-limit} and \eqref{eq:mismatched-wishart-separated-reduced-kernel},
\[
    \tilde\cS(z_*)x_*=0 .
\]
But a nonzero kernel vector of \(\tilde\cS(z)\), for any \(z>1\), cannot have a zero coordinate.  Indeed, the edge formula \eqref{eq:mismatched-wishart-S-edge} gives
\[
    \tilde\cS_{11}(1)=\frac{1-\tau \alpha\tilde\alpha}{\tau(1+\tilde\alpha)}>0,
    \qquad
    \tilde\cS_{22}(1)=\frac{1-\tau\beta\tilde\beta}{\tau(1+\tilde\beta)}>0,
\]
and the monotonicity argument in the proof of Theorem~\ref{thm:mismatched-Wishart-outlier} gives
\(\partial_z\tilde\cS(z)\succ0\) for \(z>1\).  Thus both diagonal entries of \(\tilde\cS(z)\) are strictly positive for \(z>1\).  If, for instance, \(x_{*,1}=0\), then \(\tilde\cS(z_*)x_*=0\) would force \(\tilde\cS_{22}(z_*)=0\), a contradiction; the other coordinate is identical.  Therefore there is a deterministic \(\gamma_x>0\), depending only on \(\varepsilon,\tilde\alpha,\tilde\beta,\alpha,\beta,\rho,\tau\), such that, with overwhelming probability on \(\{\lambda\ge1+\varepsilon\}\), we have $|x_1| \geq \gamma_x$ and $|x_2| \geq \gamma_x$.

The lower block relation now transfers this coordinatewise lower bound from \(x\) to \(y\).  Indeed, the matrix
\[
    -\frac1\tau
    \begin{pmatrix}
        \tilde r(z)^{-1}&0\\
        0&\tilde s(z)^{-1}
    \end{pmatrix}
\]
has diagonal entries bounded away from zero on \(K\), because \(\tilde r(z),\tilde s(z)<0\) and are finite there.  Combining this with the uniform convergence of
\(-\tilde D_n(z)^{-1}\tilde E_n(z)^\intercal\), the upper bound on \(\|x\|\), and equation defining $y$, we obtain another deterministic constant \(\gamma_y>0\) such that, with overwhelming probability on \(\{\lambda\ge1+\varepsilon\}\), we have $ |y_1|\ge\gamma_y$ and $ |y_2|\ge\gamma_y$. It remains to translate this into overlaps with the true spikes.  Since
\(\tilde p=\tilde A^{1/2}e_1\), \(\tilde q=\tilde A^{1/2}e_2\),
\(w=\tilde S h/\|\tilde S h\|\), and $y = \tilde X^\intercal h$ we get
\[
    y
    =
    \|\tilde S h\|
    \begin{pmatrix}
        \sqrt{\alpha\tilde\theta_1}\,\langle\hat a,a\rangle\\[1mm]
        \sqrt{\beta\tilde\theta_2}\,\langle\hat b,b\rangle
    \end{pmatrix},
\]
up to the harmless overall sign of \(w\).  Because \(\|h\|=1\) and
\(\|\tilde S\|_{\op}^2=\|\tilde A\|_{\op}=1+\tilde\kappa\), we get
\[
    |\langle\hat a,a\rangle|^2
    \ge
    \frac{\gamma_y^2}{\alpha\tilde\theta_1(1+\tilde\kappa)},
    \qquad
    |\langle\hat b,b\rangle|^2
    \ge
    \frac{\gamma_y^2}{\beta\tilde\theta_2(1+\tilde\kappa)}.
\]
Taking \(c\) to be the smaller of these two constants proves the fixed-parameter claim.
\end{proof}

\subsubsection{Estimating the signal strengths, proof of Theorem~\ref{thm:recoverySignalStrengths-MainBody} (CSWish)}\label{subsec:strength-estimation-Wishart}

Next we show that we can estimate the signal strengths from the data.

\begin{theorem}[Theorem~\ref{thm:recoverySignalStrengths-MainBody} (CSWish)]\label{thm:wishart-strength-estimation}
Let $n/m\to\tau$, and let $(U,V)$ be sampled from the correlated spiked Wishart model in Section~\ref{subsec:wishart}, with fixed parameters $0<\alpha,\beta\le 1/\sqrt{\tau}$ and correlation $\rho>\kappa$.  Then there exist randomized estimators $\widehat\alpha$ and $\widehat\beta$ such that, with high probability,
\[
    \widehat\alpha=\alpha+o(1),
    \qquad
    \widehat\beta=\beta+o(1).
\]
\end{theorem}

\begin{proof}[Proof of Theorem~\ref{thm:wishart-strength-estimation}, Theorem~\ref{thm:recoverySignalStrengths-MainBody} (CSWish)]
Let $\eta=n^{-1/5}$ and draw, independently of $(U,V)$, two independent matrices
$G_U,G_V\in\RR^{n\times m}$ with independent $\calN(0,1/m)$ entries.  Define
\[
U^{(1)}:=\frac{U+\sqrt{\eta}G_U}{\sqrt{1+\eta}},
\qquad
U^{(2)}:=\frac{U-\eta^{-1/2}G_U}{\sqrt{1+\eta^{-1}}},
\]
and define $V^{(1)},V^{(2)}$ analogously.  Apply the weak recovery algorithm to
$(U^{(1)},V^{(1)})$, and let $(\bar a,\bar b)$ be its output.  Normalize
\[
    x:=\frac{\bar a}{\|\bar a\|},
    \qquad
    y:=\frac{\bar b}{\|\bar b\|},
\]
with an arbitrary unit-vector choice if one of the two denominators is zero. Let $\tau_n:=n/m$, and set $ S_U:=(U^{(2)})^\intercal U^{(2)}$ and $  S_V:=(V^{(2)})^\intercal V^{(2)}$. Define
\begin{equation}\label{eq:wishart-strength-estimator}
\begin{aligned}
    A_U&:=x^\intercal(S_U-\tau_n I_m)x,
    &
    B_U&:=\|(S_U-\tau_n I_m)x\|^2-\tau_n,\\
    A_V&:=y^\intercal(S_V-\tau_n I_m)y,
    &
    B_V&:=\|(S_V-\tau_n I_m)y\|^2-\tau_n .
\end{aligned}
\end{equation}
On the event that $A_U,A_V$ are nonzero, set
\begin{equation}\label{eq:wishart-strength-estimator-final}
    \widehat\alpha
    :=\frac{1+\eta^{-1}}{\tau_n}\left(\frac{B_U}{A_U}-1\right),
    \qquad
    \widehat\beta
    :=\frac{1+\eta^{-1}}{\tau_n}\left(\frac{B_V}{A_V}-1\right),
\end{equation}
and set the corresponding estimator equal to $0$ when its denominator is zero.

We prove the claim for $\widehat\alpha$; the proof for $\widehat\beta$ is identical.
Let $Z_U$ denote the noise matrix in $U$, and write
\[
    Z^{(1)}:=\frac{Z_U+\sqrt{\eta}G_U}{\sqrt{1+\eta}},
    \qquad
    Z^{(2)}:=\frac{Z_U-\eta^{-1/2}G_U}{\sqrt{1+\eta^{-1}}}.
\]
The pair $(Z^{(1)},Z^{(2)})$ is jointly Gaussian.  Each marginal has independent
$\calN(0,1/m)$ entries, and their covariance is zero; hence $Z^{(1)}$ and $Z^{(2)}$
are independent.  Consequently the first $U^{(1)}$ and $V^{(1)}$ are again a spiked Wishart instance,
with effective strength $\alpha/(1+\eta)$, and
\[
    U^{(2)}
    =
    \sqrt{\frac{\gamma}{m}}\,u a^\intercal+Z^{(2)},
    \qquad
    \gamma:=\frac{\alpha}{1+\eta^{-1}} .
\]
Since $\eta\to0$, the parameters of the first split converge to those of the
original model.  Hence, by $\rho>\kappa$ and the weak recovery theorem, there is a
deterministic constant $c>0$ such that, with high probability,
\begin{equation}\label{eq:wishart-split-overlap}
    \langle x,a\rangle^2\ge c
    \quad \text{and} \quad
    \langle y,b\rangle^2\ge c .
\end{equation}

We work on this event and condition on $(U^{(1)},V^{(1)})$ and on the latent
variables.  Then $x$ is fixed and independent of $Z^{(2)}$.  In the rest of the
proof, write $ S:=S_U$, $A:=A_U$, $ B:=B_U$ and $s:=\langle x,a\rangle$. Define
\[
    T:=S-\tau_n I_m,\qquad
    H:=(Z^{(2)})^\intercal Z^{(2)}-\tau_n I_m,\qquad
    \nu:=\frac{\|u\|^2}{m},\qquad
    r:=\frac{(Z^{(2)})^\intercal u}{\sqrt m}.
\]
Then, we have that
\begin{equation}\label{eq:wishart-T-decomposition}
    T
    =
    H+\sqrt{\gamma}\,(a r^\intercal+r a^\intercal)
      +\gamma\nu\,aa^\intercal .
\end{equation}
Observe that $\gamma\asymp \eta=n^{-1/5}$ and hence $\gamma^2\asymp n^{-2/5}$. We next record the concentration estimates needed below.  Let $\delta_n:=n^{-9/20},$ for example. On the overwhelming-probability event where $\|a\|=1+o(1)$ and $\nu=\tau_n+o(1)$, we have,
conditional on $(U^{(1)},V^{(1)})$ and on the latent variables,
\begin{equation}\label{eq:wishart-direct-gaussian-bounds}
   |x^\intercal Hx| +|\langle Hx,a\rangle|
    +|\langle Hx,r\rangle|
    +\bigl|\|Hx\|^2-\tau_n\bigr| 
    +|\langle x,r\rangle|
    +|\langle a,r\rangle|
    +\bigl|\|r\|^2-\nu\bigr|
    \le \delta_n.
\end{equation}
with overwhelming probability.  We justify this briefly.  By rotational invariance of the
Gaussian noise in the column space, after conditioning on $x,a,u$ we may assume
$x=e_1$.  Let $z_1,\ldots,z_m\in\RR^n$ denote the columns of $Z^{(2)}$.  Then
\[
    H e_1
    =
    \bigl(\|z_1\|^2-\tau_n,\,
          z_2^\intercal z_1,\ldots,
          z_m^\intercal z_1\bigr)^\intercal,
    \qquad
    r_j=\frac{z_j^\intercal u}{\sqrt m}.
\]
The bounds for $x^\intercal Hx$, $\langle x,r\rangle$, $\langle a,r\rangle$, and
$\|r\|^2-\nu$ follow from standard Gaussian and chi-square tail bounds.

For $\langle Hx,a\rangle$, condition on $z_1$.  The contribution from the columns
$z_j$, $j\ge2$, is Gaussian with conditional variance
\[
    \frac{\|z_1\|^2}{m}\sum_{j=2}^m a_j^2=O(n^{-1})
\]
with overwhelming probability, while the remaining contribution is
$a_1(\|z_1\|^2-\tau_n)=O(n^{-1/2})$.  Thus
$|\langle Hx,a\rangle|\le \delta_n$ with overwhelming probability. Similarly,
\[
    \|Hx\|^2
    =
    (\|z_1\|^2-\tau_n)^2
    +
    \sum_{j=2}^m (z_j^\intercal z_1)^2 .
\]
Conditional on $z_1$, the second term has law
$(\|z_1\|^2/m)\chi_{m-1}^2$.  Since
$\|z_1\|^2=\tau_n+O(n^{-1/2})$, this gives $ \|Hx\|^2-\tau_n=O(n^{-1/2})$, and hence the desired $\delta_n$ bound with overwhelming probability. Finally,
\[
    \langle Hx,r\rangle
    =
    (\|z_1\|^2-\tau_n)\frac{z_1^\intercal u}{\sqrt m}
    +
    \sum_{j=2}^m
    (z_j^\intercal z_1)\frac{z_j^\intercal u}{\sqrt m}.
\]
Conditional on $z_1$ and $u$, the summands in the second term are independent
second order Gaussians.  Their total conditional mean is
\[
    \frac{m-1}{m^{3/2}}\,z_1^\intercal u=O(n^{-1/2}),
\]
and their total conditional variance is $O(n^{-1})$.  Standard Gaussian
moment bounds therefore imply $ \langle Hx,r\rangle=O(n^{-1/2})$,  and hence $|\langle Hx,r\rangle|\le\delta_n$ with overwhelming probability.  This proves
\eqref{eq:wishart-direct-gaussian-bounds}.  Since $\delta_n = o(\gamma^2)$, all errors controlled by \eqref{eq:wishart-direct-gaussian-bounds} are negligible at the second order scale $\gamma^2$. We now estimate values $A$ and $B$.  From \eqref{eq:wishart-T-decomposition},
\begin{equation*}
	 A
    =
    x^\intercal T x =
    x^\intercal Hx
      +2\sqrt{\gamma}\,s\langle x,r\rangle
      +\gamma\nu s^2  =
    \gamma\nu s^2+o(\gamma^2)
\end{equation*}
  Since $\nu=\tau_n+o(1)$, this also implies $    A=\tau_n\gamma s^2+o(\gamma)$. Next,
\[
    Tx
    =
    Hx+\sqrt{\gamma}\bigl(\langle x,r\rangle a+s r\bigr)
      +\gamma\nu s a .
\]
Therefore, we can write the 
\begin{align*}
    B
    =
    \|Tx\|^2-\tau_n
    &=
    \|Hx\|^2-\tau_n 
      +2\sqrt{\gamma}
        \Bigl(
            \langle x,r\rangle\langle Hx,a\rangle
            +s\langle Hx,r\rangle
        \Bigr) 
      +2\gamma\nu s\langle Hx,a\rangle  \\
    &\quad
      +\gamma\|\langle x,r\rangle a+s r\|^2  
      +2\gamma^{3/2}\nu s
        \langle \langle x,r\rangle a+s r,a\rangle  
      +\gamma^2\nu^2s^2\|a\|^2 .
\end{align*}

Using \eqref{eq:wishart-direct-gaussian-bounds}, all terms above are
$o(\gamma^2)$ except for the leading part of
$\gamma\|\langle x,r\rangle a+s r\|^2$ and the final signal squared term.  More
precisely,
\[
    \gamma\|\langle x,r\rangle a+s r\|^2
    =
    \gamma\left(
        \langle x,r\rangle^2\|a\|^2
        +2s\langle x,r\rangle\langle a,r\rangle
        +s^2\|r\|^2
    \right) =
    \gamma\nu s^2+o(\gamma^2),
\]
Hence, we conclude that
\begin{equation}\label{eq:wishart-B-asymptotic}
    B
    =
    \gamma\nu s^2
    +
    \gamma^2\nu^2s^2\|a\|^2
    +
    o(\gamma^2).
\end{equation}
Subtracting $A$ in its sharper form
$A=\gamma\nu s^2+o(\gamma^2)$ from \eqref{eq:wishart-B-asymptotic}, and using
$\nu=\tau_n+o(1)$ and $\|a\|^2=1+o(1)$, gives
\[
    B-A
    =
    \gamma^2\tau_n^2s^2+o(\gamma^2).
\]
On the weak recovery event \eqref{eq:wishart-split-overlap}, we have $s^2\ge c$.
Thus the denominator $A$ is nonzero with high probability, and
\[
    \frac{1}{\tau_n}\left(\frac{B}{A}-1\right)
    =
    \frac{1}{\tau_n}\frac{B-A}{A}
    =
    \gamma(1+o(1)).
\]
Therefore $\widehat\alpha = \alpha+o(1)$  with high probability.  The proof for $\widehat\beta$ is identical.
\end{proof}

%% file: lower-bounds.tex
\section{Lower bounds for CSWig and CSWish}

In this section, we prove information theoretic lower bounds for correlated spiked models CSWig and CSWish. The proof of the Wigner case can be deduced from \cite{YLS25}, but we also provide a self-contained proof here for completeness.

\subsection{Correlated spiked Wigner model}

Recall the definition of the correlated spiked Wigner model Section~\ref{subsec:wigner}. In this section, we will be adopting the notation introduced in Section~\ref{subsec:wigner}.  For each $\rho \in [0,1)$, let $\mathcal P_\rho$ denote the law of the correlated spiked Wigner
model from Section~\ref{subsec:wigner} with spike-correlation parameter $\rho$, and let $\PP_\rho$ be its marginal on $(U,V)$. Thus
$\PP_{\pl} = \PP_\rho$ and $\PP_{\nul} = \PP_0$. Let $\QQ$ denote the pure-noise law,
under which
\[
U = Z_a,
\qquad
V = Z_b,
\]
with $Z_a,Z_b \sim \mathrm{GOE}(1/n)$ independent. For $\rho \in [0,1)$, write
\[
L_\rho := \frac{\mathrm{d}\PP_\rho}{\mathrm{d}\QQ}.
\]

\subsubsection{Impossibility of strong detection}

In this subsection, we will show that below threshold, strong detection is impossible.

\begin{prop}\label{prop:wigner-second-moment-under-Q}
Suppose $\rho < \kappa$. Then there exists $\varepsilon > 0$ and  $ L_{\rho,\varepsilon} \geq 0 $ such that
\[
0 \leq  L_{\rho,\varepsilon} \leq L_\rho,
\qquad
\EE_{\QQ}[ L_{\rho,\varepsilon}]=1-o(1),
\qquad
\EE_{\QQ}[ L_{\rho,\varepsilon}^2]=O(1).
\]
\end{prop}

\begin{proof}
Since $\rho < \kappa$, we have
\[
\Delta := (1-\alpha^2)(1-\beta^2) - \alpha^2\beta^2\rho^4 > 0.
\]
Choose $\varepsilon,\delta > 0$ to be a sufficiently small constant so that 
\begin{gather*}
    (1-\alpha^2 x)(1-\beta^2 y) - \alpha^2\beta^2\rho^2 z^2 > \delta \\
    \alpha^2(1+\varepsilon) < 1 \\
\beta^2(1+\varepsilon) < 1 
\end{gather*}
whenever $x, y,z$ satisfy $|x- 1|, |y-1|, |z - \rho| \leq \varepsilon$.  Let $\pi_\rho$ denote the law of the planted spike pair $(a,b)$, and for fixed $(a,b)$ write the conditional likelihood ratio
\[
\Lambda(a,b;U,V)
:=
\exp\Biggl\{
\frac{n\alpha}{2}\ang{U,aa^\intercal}
-
\frac{n\alpha^2}{4}\norm{a}^4
+
\frac{n\beta}{2}\ang{V,bb^\intercal}
-
\frac{n\beta^2}{4}\norm{b}^4
\Biggr\}.
\]
Thus $L_\rho = \EE_{\pi_\rho}[\Lambda(a,b;U,V)]$. Define the event where spikes are controlled
\[
\calE
:=
\set{
\big|\norm{a}^2-1\big|\leq \varepsilon,
\ \big|\norm{b}^2-1\big|\leq \varepsilon,
\ \big|\ang{a,b}-\rho\big|\leq \varepsilon
}.
\]
and note that $\pi_\rho(\calE)=1-o(1)$ by law of large numbers. Now define the truncated likelihood ratio
\[
 L_{\rho,\varepsilon}(U,V)
:=
\EE_{\pi_\rho}\Bigl[\Lambda(a,b;U,V)\,\mathbf 1_{\calE}(a,b)\Bigr].
\]
Clearly $0\leq  L_{\rho,\varepsilon}\leq L_\rho$. Moreover, observe that
\[
\EE_{\QQ}[ L_{\rho,\varepsilon}]
=
\pi_\rho(\calE)=1-o(1)\,.
\]

It remains to bound the second moment. Let $(a_1,b_1)$ and $(a_2,b_2)$ be two independent copies of $(a,b)\sim\pi_\rho$. Expanding the square and doing Gaussian moment generating function computation $(U,V)\sim\QQ$ gives
\[
\EE_{\QQ}[ L_{\rho,\varepsilon}^2]
=
\EE \Biggl[
\exp\set{
\frac{n\alpha^2}{2}\ang{a_1,a_2}^2
+
\frac{n\beta^2}{2}\ang{b_1,b_2}^2}
\mathbf 1_{\calE}(a_1,b_1)
\mathbf 1_{\calE}(a_2,b_2)
\Biggr].
\]
Now condition on $(a_1,b_1)$ and set
\[
X := \sqrt n\,\ang{a_1,a_2},
\qquad
Y := \sqrt n\,\ang{b_1,b_2}.
\]
Because $(a_2,b_2)$ is jointly Gaussian with covariance
$\frac{1}{n}\bigl(\begin{smallmatrix}I_n&\rho I_n\\ \rho I_n&I_n\end{smallmatrix}\bigr)$ and is independent of $(a_1,b_1)$, the pair $(X,Y)$ is conditionally centered Gaussian with covariance matrix
\[
\Gamma(a_1,b_1)
=
\begin{pmatrix}
\norm{a_1}^2 & \rho\,\ang{a_1,b_1}\\
\rho\,\ang{a_1,b_1} & \norm{b_1}^2
\end{pmatrix}.
\]
Let $D:=\diag(\alpha^2,\beta^2)$. Then the Gaussian quadratic-moment formula yields
\[
\EE\Bigl[
\exp\set{\frac{\alpha^2}{2}X^2 + \frac{\beta^2}{2}Y^2}
\,\Big|\, a_1,b_1
\Bigr]
=
\det \bigl(I_2 - D^{1/2}\Gamma(a_1,b_1)D^{1/2}\bigr)^{-1/2},
\]
provided $I_2 - D^{1/2}\Gamma(a_1,b_1)D^{1/2}$ is positive definite. On $\calE(a_1,b_1)$ this holds, because its diagonal entries are positive by our choice of $\varepsilon$, and its determinant equals
\begin{equation*}
\bigl(1-\alpha^2\norm{a_1}^2\bigr)
\bigl(1-\beta^2\norm{b_1}^2\bigr)
-
\alpha^2\beta^2\rho^2\ang{a_1,b_1}^2,
\end{equation*}
which is at least $\delta$ by construction. Therefore,
\begin{equation*}
\EE_{\QQ}[ L_{\rho,\varepsilon}^2]
\le \delta^{-1/2}=O(1)\,.
\end{equation*}
This proves the claim.
\end{proof}

\begin{theorem}\label{thm:wigner-TV-bound}
If $\rho < \kappa$, then we have that
\begin{equation*}
        \TV(\PP_\rho, \QQ) = 1 - \Omega(1) \quad \text{ and } \quad  \TV(\PP_\rho, \PP_0) = 1 - \Omega(1).
    \end{equation*}
\end{theorem}

\begin{proof}
Recall that $\alpha < 1$ and $\beta < 1$. Under $\PP_0$, the two views are independent, and each marginal is a one-view rank one spiked Wigner model below the BBP threshold. The standard one-view subcritical Wigner contiguity theorem therefore implies that $\QQ$ is contiguous with respect to $\PP_{0}$ (see, e.g. \cite{PWBAM18}).  

We will show that $\PP_\rho$ is contiguous with respect to $\QQ$. By transitivity, this will imply that $\PP_{\rho}$ is contiguous with respect to $\PP_{0}$, which is enough to rule out asymptotically powerful tests.

Let $ L_{\rho,\varepsilon}$ be given by Proposition~\ref{prop:wigner-second-moment-under-Q}. If $A_n$ is any sequence with $\QQ(A_n)=o(1)$, then
\begin{align*}
\PP_\rho(A_n)
&=
\EE_{\QQ}[L_\rho\mathbf 1_{A_n}]
\\
&=
\EE_{\QQ}\sqb{(L_\rho- L_{\rho,\varepsilon})\mathbf 1_{A_n}}
+
\EE_{\QQ}\sqb{ L_{\rho,\varepsilon}\mathbf 1_{A_n}}
\\
&\leq
\EE_{\QQ}[L_\rho- L_{\rho,\varepsilon}]
+
\sqrt{\EE_{\QQ}[ L_{\rho,\varepsilon}^2]\QQ(A_n)}
\\
&=
1-\EE_{\QQ}[ L_{\rho,\varepsilon}]
+
\sqrt{\EE_{\QQ}[ L_{\rho,\varepsilon}^2]\QQ(A_n)}=o(1)\,.
\end{align*}
Thus $\PP_\rho$ is contiguous with respect to $\QQ$.

Now suppose for contradiction that $\TV(\PP_\rho,\PP_{0})\to 1$ along some subsequence. Then, after passing to a further subsequence, we find events $A_n$ such that
\[
\PP_{\rho}(A_n)-\PP_{0}(A_n)\to 1 \implies \PP_{\rho}(A_n)=1-o(1)
\text{ and }
\PP_{ 0 }(A_n)=o(1),
\]
contradicting the fact that $\PP_{\rho}$ is contiguous with respect to $\PP_{0}$. Similarly for $\QQ$.
\end{proof}
\subsubsection{Impossibility of weak recovery}

Let $\pi_\rho$ be the prior law of $(a,b)$, let $\mu_{U,V}$ be the posterior
law under $\PP_\rho$, and write
$R_{12}:=\ang{a^1,a^2}^2+\ang{b^1,b^2}^2$ for two posterior replicas.
Choose $\varepsilon_0>0$ as in the proof of
Proposition~\ref{prop:wigner-second-moment-under-Q}, and set $\calG$ to be the event under which
\[
\max \set{\big|\norm a^2-1\big|,
\big|\norm b^2-1\big|,
\big|\ang{a,b}-\rho\big|} \leq \varepsilon_0.
\]

\begin{prop}\label{prop:wigner-posterior-overlap}
If $\rho<\kappa$, then
\[
\EE_{\PP_\rho}\big[\mu_{U,V}^{\otimes2}
[R_{12}\mathbf 1_{\calG}(a^1,b^1)\mathbf 1_{\calG}(a^2,b^2)]\big]=o(1).
\]
Consequently, for every fixed $\delta>0$,
\[
\PP_\rho\bigl[\mu_{U,V}^{\otimes2}(R_{12}\ge\delta)\ge\delta\bigr]=o(1).
\]
\end{prop}

\begin{proof}
For fixed $t>0$, define
\[
I_t(U,V):=\EE_{\pi_\rho^{\otimes2}}
\bigl[\Lambda_1\Lambda_2\mathbf 1_{\calG_1}\mathbf 1_{\calG_2}
\mathbf 1\{R_{12}\ge t\}\bigr],
\]
where $\Lambda_i=\Lambda(a^i,b^i;U,V)$ is the conditional likelihood ratio
from Proposition~\ref{prop:wigner-second-moment-under-Q}. The same computation as in
Proposition~\ref{prop:wigner-second-moment-under-Q}, with the extra indicator, gives
\begin{equation}\label{eq:wigner-overlap-numerator-bound}
\EE_\QQ [I_t(U,V)]\le e^{-\Omega_t(n)}.
\end{equation}
Indeed, conditionally on $(a^1,b^1)$, the vector
$Z=(\sqrt n\ang{a^1,a^2},\sqrt n\ang{b^1,b^2})$ is the same two-dimensional
Gaussian as before; on $\calG_1$, the previous determinant lower bound makes
the tilted covariance uniformly bounded, while $R_{12}\ge t$ is the event
$\|Z\|^2\ge nt$. Since
$\mu_{U,V}^{\otimes2}(R_{12}\ge t,\calG_1,\calG_2)=I_t/L_\rho^2$ and
$\mathrm d\PP_\rho=L_\rho\,\mathrm d\QQ$, for any $\gamma\in(0,c)$,
\[
\EE_{\PP_\rho}\big[\mu_{U,V}^{\otimes2}(R_{12}\ge t,\calG_1,\calG_2)\big]
=\EE_\QQ\Big[\frac{I_t}{L_\rho}\Big]
\le e^{\gamma n}\EE_\QQ[I_t]+e^{-\gamma n}=o(1),
\]
where we used $I_t\le L_\rho^2$ on $\{L_\rho<e^{-\gamma n}\}$. Hence
\[
\EE_{\PP_\rho}\big[\mu_{U,V}^{\otimes2}
[R_{12}\mathbf 1_{\calG_1}\mathbf 1_{\calG_2}]\big]
\le t+2(1+\varepsilon_0)^2
\EE_{\PP_\rho}\mu_{U,V}^{\otimes2}(R_{12}\ge t,\calG_1,\calG_2),
\]
and the first claim follows by taking $n\to\infty$ and then $t\downarrow0$.
Finally, by the tower property and concentration of the prior,
$\EE_{\PP_\rho}\big[\mu_{U,V}(\calG^c)\big]=\pi_\rho(\calG^c)=o(1)$. Therefore
\[
\EE_{\PP_\rho}\big[\mu_{U,V}^{\otimes2}(R_{12}\ge\delta)\big]
\le
\EE_{\PP_\rho}\big[\mu_{U,V}^{\otimes2}(R_{12}\ge\delta,\calG_1,\calG_2)\big]
+2\pi_\rho(\calG^c)=o(1),
\]
and the second claim follows from Markov's inequality.
\end{proof}

\begin{theorem}\label{thm:wigner-no-weak-recovery}
Suppose $\rho<\kappa$. For any estimators
$\hat a=\hat a(U,V)$ and $\hat b=\hat b(U,V)$ satisfying
$\norm{\hat a},\|{\hat b\|}\le1$ almost surely, and every fixed
$\eta>0$,
\[
\mathcal P_\rho\bigl[\langle\hat a,a\rangle^2+\langle\hat b,b\rangle\rangle^2\ge\eta\bigr]=o(1).
\]
\end{theorem}

\begin{proof}
Let us define the truncated covariance matrix
$M_a(U,V):=\EE_{\mu_{U,V}}[aa^\intercal\mathbf 1_{\calG}(a,b)]$. The posterior
identity gives
\begin{align*}
\EE_{\mathcal P_\rho}[\ang{\hat a,a}^2]&=\mathbb{E}_{\PP_\rho}\big[\hat{a}^\intercal\mathbb{E}_{\mu_{U,V}}[aa^\intercal]\hat{a}\big]\\
&\le \EE_{\pi_\rho}[\norm a^2\mathbf 1_{\calG^c}] 
+\EE_{\PP_\rho}\norm{M_a}_{\mathrm F}\\
&\le o(1)+
\Bigl(\EE_{\PP_\rho}\mu_{U,V}^{\otimes2}
[\ang{a^1,a^2}^2\mathbf 1_{\calG_1}\mathbf 1_{\calG_2}]
\Bigr)^{1/2}=o(1),
\end{align*}
where we used $\norm{\hat a}\le1$ and
$\norm{M_a}_{\mathrm F}^2=
\EE_{\mu_{U,V}^{\otimes2}}[\ang{a^1,a^2}^2\mathbf 1_{\calG_1}\mathbf 1_{\calG_2}]$.
The same argument for $b$ and Markov inequality completes the proof.
\end{proof}

\subsection{Correlated spiked Wishart model}

Recall the definition of the correlated spiked Wishart model Section~\ref{subsec:wishart}. In this section, we will be adopting the notation introduced in Section~\ref{subsec:wishart}. 

As in the previous section, for each $\rho \in [0,1)$, let $\mathcal P_\rho$ denote the law of the correlated spiked Wishart model from Section~\ref{subsec:wigner} with parameter $\rho$ and $\PP_\rho$ be its marginal on $(U,V)$. Thus
$\PP_{\pl} = \PP_\rho$ and $\PP_{\nul} = \PP_0$. Let $\QQ$ denote the pure-noise law under which
    \begin{equation*}
        U = Z_a
        \qquad \text{ and } \qquad
        V = Z_b.
    \end{equation*}
For correlation $\rho \in [0,1)$, write likelihood ratio $L_\rho := \mathrm{d}\PP_\rho / \mathrm{d}\QQ$.

\subsubsection{Impossibility of strong detection}

In this subsection, we will show that below threshold, strong deterction is impossible.
    
\begin{prop}\label{prop:wishart-second-moment-under-Q}
Suppose $\rho < \kappa$. Then there exists $\varepsilon > 0$ and  $ L_{\rho,\varepsilon} \geq 0 $ such that
\[
0 \leq  L_{\rho,\varepsilon} \leq L_\rho,
\qquad
\EE_{\QQ}[ L_{\rho,\varepsilon}]=1-o(1),
\qquad
\EE_{\QQ}[ L_{\rho,\varepsilon}^2]=O(1).
\]
\end{prop}

\begin{proof}
Since $\rho < \kappa$, we necessarily have $\tau\alpha^2 < 1$ and $\tau\beta^2 < 1$, and
\[
\Delta := (1-\tau\alpha^2)(1-\tau\beta^2) - \tau^2\alpha^2\beta^2\rho^4 > 0.
\]
The proof follows analogously to that of Proposition~\ref{prop:wigner-second-moment-under-Q}. Choose $\varepsilon,\delta > 0$ to be a sufficiently small constant so that
\begin{gather*}
(1-\alpha^2 x_1 y_1)(1-\beta^2 x_2 y_2) - \alpha^2\beta^2 x_1x_2\rho^2 z^2 > \delta \\
    \alpha^2(\tau+\varepsilon)(1+\varepsilon) < 1 \\
\beta^2(\tau+\varepsilon)(1+\varepsilon) < 1 
\end{gather*}
whenever $x_1,x_2,y_1,y_2,z$ satisfy
\[
|x_1-\tau|, |x_2-\tau|, |y_1-1|, |y_2-1|, |z-\rho| \leq \varepsilon.
\]
Let $\pi_\rho$ denote the law of the planted latent variables $(u,v,a,b)$. For fixed $(u,v,a,b)$ write the conditional likelihood ratio as
\[
\Lambda(u,v,a,b;U,V)
:=
\exp\Biggl\{
\sqrt{\alpha m}\,\langle U,ua^\intercal\rangle
-\frac\alpha2\|u\|^2\|a\|^2
+\sqrt{\beta m}\,\langle V,vb^\intercal\rangle
-\frac\beta2\|v\|^2\|b\|^2
\Biggr\}.
\]
Thus $L_\rho = \EE_{\pi_\rho}[\Lambda(u,v,a,b;U,V)]$. Define the event $\calE$ where the planted latent variables are controlled:
\begin{equation*}
\begin{aligned}
     \big|m^{-1}\|u\|^2-\tau\big|\leq \varepsilon \\
     \big|m^{-1}\|v\|^2-\tau\big|\leq \varepsilon \\
\end{aligned}
\qquad 
\begin{aligned}
\big|\|b\|^2-1\big|\leq \varepsilon \\
\big|\|a\|^2-1\big|\leq \varepsilon \\
\end{aligned}
\quad \text{ and } \quad 
  \big|\ang{a,b}-\rho\big| \leq \varepsilon.
\end{equation*}
Note that $\pi_\rho(\calE)=1-o(1)$ by law of large numbers. Define truncated likelihood ratio
\[
 L_{\rho,\varepsilon}(U,V)
:=
\EE_{\pi_\rho}\Bigl[\Lambda(u,v,a,b;U,V)\,\mathbf 1_{\calE}(u,v,a,b)\Bigr].
\]
Clearly $0\leq  L_{\rho,\varepsilon}\leq L_\rho$. Moreover, observe that
\[
\EE_{\QQ}[ L_{\rho,\varepsilon}]
=
\pi_\rho(\calE)=1-o(1)\,.
\]

It remains to bound the second moment. Let $(u_1,v_1,a_1,b_1)$ and $(u_2,v_2,a_2,b_2)$ be two independent copies of $(u,v,a,b)\sim\pi_\rho$. Expanding the square and doing Gaussian moment generating function computation under $(U,V)\sim\QQ$ gives
\begin{align*}
\EE_{\QQ}[ L_{\rho,\varepsilon}^2]
=
\EE \Biggl[
\exp\set{\alpha\ang{u_1,u_2}\ang{a_1,a_2} + \beta\ang{v_1,v_2}\ang{b_1,b_2}} \\
\times \mathbf 1_{\calE}(u_1,v_1,a_1,b_1)\mathbf 1_{\calE}(u_2,v_2,a_2,b_2)
\Biggr].
\end{align*}
Dropping the second indicator only increases this expression, so
\begin{equation*}
\EE_{\QQ}[ L_{\rho,\varepsilon}^2]
\leq
\EE \Biggl[
\exp\set{\alpha\ang{u_1,u_2}\ang{a_1,a_2} + \beta\ang{v_1,v_2}\ang{b_1,b_2}} \\
\times \mathbf 1_{\calE}(u_1,v_1,a_1,b_1)
\Biggr].
\end{equation*}
Now condition on $(u_1,v_1,a_1,b_1)$ and set
\begin{equation*}
X := \sqrt m\,\ang{a_1,a_2},
\qquad
Y := \sqrt m\,\ang{b_1,b_2}.
\end{equation*}
Conditionally on $(a_1,b_1)$, the pair $(X,Y)$ is centered Gaussian with covariance 
\[
\Gamma(a_1,b_1)
=
\begin{pmatrix}
\norm{a_1}^2 & \rho\,\ang{a_1,b_1}\\
\rho\,\ang{a_1,b_1} & \norm{b_1}^2
\end{pmatrix},
\]
because $(a_2,b_2)$ is jointly Gaussian with covariance
$\frac{1}{m}\bigl(\begin{smallmatrix}I_m&\rho I_m\\ \rho I_m&I_m\end{smallmatrix}\bigr)$ and is independent of $(u_1,v_1,a_1,b_1)$. Also, conditionally on $(u_1,v_1,a_1,b_1,a_2,b_2)$, the Gaussian vectors $u_2$ and $v_2$ are independent, so
\begin{align*}
&\EE\Bigl[
\exp\set{\alpha\ang{u_1,u_2}\ang{a_1,a_2} + \beta\ang{v_1,v_2}\ang{b_1,b_2}}
\,\Big|\, u_1,v_1,a_1,b_1,a_2,b_2
\Bigr] \\
&\qquad=
\exp\set{
\frac{\alpha^2}{2m}\|u_1\|^2 X^2
+
\frac{\beta^2}{2m}\|v_1\|^2 Y^2}.
\end{align*}
Therefore, if we set 
\[
D:=\diag\!\paren{\alpha^2\frac{\|u_1\|^2}{m},\beta^2\frac{\|v_1\|^2}{m}}
\quad \text{and} \quad
M:=I_2 - D^{1/2}\Gamma(a_1,b_1)D^{1/2},
\]
the Gaussian quadratic-moment formula yields
\[
\EE\Bigl[
\exp\set{\frac{\alpha^2}{2m}\|u_1\|^2 X^2 + \frac{\beta^2}{2m}\|v_1\|^2 Y^2}
\,\Big|\, u_1,v_1,a_1,b_1
\Bigr]
=
\det(M)^{-1/2},
\]
provided $M$ is positive definite. On $\calE(u_1,v_1,a_1,b_1)$ this holds, because its diagonal entries are positive by our choice of $\varepsilon$, and its determinant equals
\begin{equation*}
\paren{1-\alpha^2\frac{\|u_1\|^2}{m}\|a_1\|^2}
\paren{1-\beta^2\frac{\|v_1\|^2}{m}\|b_1\|^2}
-
\alpha^2\beta^2\frac{\|u_1\|^2\|v_1\|^2}{m^2}\rho^2\ang{a_1,b_1}^2,
\end{equation*}
which is at least $\delta$ by construction. Therefore,
\begin{equation*}
\EE_{\QQ}[ L_{\rho,\varepsilon}^2]
\leq \delta^{-1/2}
\end{equation*}
This proves the claim.
\end{proof}

Proposition~\ref{prop:wishart-second-moment-under-Q} together with same arguments as in the proof of Theorem~\ref{thm:wigner-TV-bound} imply the following theorem, which concludes impossibility of strong detection below the threshold.

\begin{theorem}\label{thm:Wishart-TV-bound}
    Suppose $\rho < \kappa$. Then we have that
 \begin{equation*}
        \TV(\PP_\rho, \QQ) = 1 - \Omega(1) \quad \text{ and } \quad  \TV(\PP_\rho, \PP_0) = 1 - \Omega(1).
    \end{equation*}
\end{theorem}

\subsubsection{Impossibility of weak recovery}
We now prove that weak recovery is impossible below the same threshold.  The
proof is similar to Wigner case. Let $\mu_{U,V}$ denote the posterior law of $(u,v,a,b)$ given $(U,V)$ under
$\PP_\rho$. Thus, with respect to the prior law $\pi_\rho$ of
$(u,v,a,b)$,
\[
\mu_{U,V}(\mathrm du,\mathrm dv,\mathrm da,\mathrm db)
=
\frac{\Lambda(u,v,a,b;U,V)}{L_\rho(U,V)}
\,\pi_\rho(\mathrm du,\mathrm dv,\mathrm da,\mathrm db),
\]
where $\Lambda$ is the conditional likelihood ratio from the proof of
Proposition~\ref{prop:wishart-second-moment-under-Q}, namely
\[
\Lambda(u,v,a,b;U,V)
:=
\exp\Biggl\{
\sqrt{\alpha m}\,\langle U,ua^\intercal\rangle
-\frac\alpha2\|u\|^2\|a\|^2
+
\sqrt{\beta m}\,\langle V,vb^\intercal\rangle
-\frac\beta2\|v\|^2\|b\|^2
\Biggr\}.
\]

\begin{prop}\label{prop:wishart-posterior-overlap}
Suppose $\rho<\kappa$. Then, for every fixed $\delta>0$,
\[
\PP_\rho\Bigl[
\mu_{U,V}^{\otimes 2}
\Bigl(
\ang{a^1,a^2}^2+\ang{b^1,b^2}^2\ge \delta
\Bigr)
\ge \delta
\Bigr]
=o(1).
\]
Moreover, there exists a fixed $\varepsilon_0>0$ such that, 
\begin{equation*}
    \EE_{\PP_\rho}\,
\mu_{U,V}^{\otimes 2}
\Bigl[
\Bigl(\ang{a^1,a^2}^2+\ang{b^1,b^2}^2\Bigr)
\mathbf 1_{\calG}(u^1,v^1,a^1,b^1)
\mathbf 1_{\calG}(u^2,v^2,a^2,b^2)
\Bigr]=o(1),
\end{equation*}
where $\calG$ is an event defined by 
\begin{gather*}
\max\set{  \abs{{\norm{u}^2}/{m}-\tau}, \abs{{\norm{v}^2}/{m}-\tau}, \abs{\norm{a}^2-1}, \abs{\norm{b}^2-1}, \abs{\ang{a,b}-\rho}} \leq \varepsilon_0.
\end{gather*}

\end{prop}

\begin{proof}
Since $\rho<\kappa$, we have $\tau\alpha^2<1$, $\tau\beta^2<1$, and
\[
\Delta
:=
(1-\tau\alpha^2)(1-\tau\beta^2)
-
\tau^2\alpha^2\beta^2\rho^4
>0.
\]
Choose $\varepsilon_0>0$ sufficiently small so that, whenever
\[
\max\set{|x_1-\tau|,\ |x_2-\tau|,\ |y_1-1|,\ |y_2-1|,\ |z-\rho|}
\leq \varepsilon_0,
\]
we have $(1-\alpha^2 x_1y_1)(1-\beta^2 x_2y_2)
-
\alpha^2\beta^2 x_1x_2\rho^2z^2
\ge c_0$  for some constant $c_0>0$, and also $\alpha^2(\tau+\varepsilon_0)(1+\varepsilon_0), \beta^2(\tau+\varepsilon_0)(1+\varepsilon_0) < 1.$ Set $R_{12}
:=
\langle{a^1,a^2\rangle}^2+\langle{b^1,b^2 \rangle}^2.$ For fixed $\delta>0$, define
\[
I_{\delta}(U,V)
:=
\EE_{\pi_\rho^{\otimes 2}}
\Bigl[
\Lambda(u^1,v^1,a^1,b^1;U,V)
\Lambda(u^2,v^2,a^2,b^2;U,V)
\mathbf 1_{\calG_1}
\mathbf 1_{\calG_2}
\mathbf 1\{R_{12}\ge \delta\}
\Bigr],
\]
where $\calG_i=\calG(u^i,v^i,a^i,b^i)$. We first prove that
\begin{equation}\label{eq:wishart-overlap-numerator-exp-bound}
\EE_{\QQ} I_{\delta}(U,V)
\le e^{-\Omega_\delta(n)}.
\end{equation}
Expanding the square and integrating over $(U,V)\sim \QQ$ gives exactly the
same expression as in the proof of
Proposition~\ref{prop:wishart-second-moment-under-Q}, but with the additional overlap
indicator:
\[
\EE_{\QQ} I_{\delta}
=
\EE
\Biggl[
\exp\set{
\alpha\ang{u^1,u^2}\ang{a^1,a^2}
+
\beta\ang{v^1,v^2}\ang{b^1,b^2}
}
\mathbf 1_{\calG_1}
\mathbf 1_{\calG_2}
\mathbf 1\{R_{12}\ge \delta\}
\Biggr].
\]
We condition on $(u^1,v^1,a^1,b^1)$ and drop the indicator
$\mathbf 1_{\calG_2}$. Set
\[
X:=\sqrt m\,\ang{a^1,a^2},
\qquad
Y:=\sqrt m\,\ang{b^1,b^2},
\qquad
Z:=\begin{pmatrix}X\\Y\end{pmatrix}.
\]
Conditionally on $(a^1,b^1)$, the vector $Z$ is centered Gaussian with
covariance
\[
\Gamma(a^1,b^1)
=
\begin{pmatrix}
\norm{a^1}^2 & \rho\,\ang{a^1,b^1}\\
\rho\,\ang{a^1,b^1} & \norm{b^1}^2
\end{pmatrix}.
\]
Moreover, conditionally on $(u^1,v^1,a^1,b^1,a^2,b^2)$, integrating  $u^2,v^2$ gives
\begin{align*}
&\EE\Bigl[
\exp\Bigl\{
\alpha\langle{u^1,u^2\rangle}\langle{a^1,a^2\rangle}
+
\beta\langle{v^1,v^2\rangle}\langle{b^1,b^2\rangle}
\Bigr\}
\mathbf 1_{\calG_1}\mathbf 1_{\calG_2}\mathbf 1\{R_{12}\ge\delta\}
\,\Big|\,u^1,v^1,a^1,b^1,a^2,b^2
\Bigr]
\\
&\qquad\qquad\qquad\qquad\qquad\le
\exp\biggl\{
\frac{\alpha^2\|u^1\|^2X^2+\beta^2\|v^1\|^2Y^2}{2m}
\biggr\}.
\end{align*}
Thus, if $D := \diag ({ \alpha^2{\norm{u^1}^2}/{m}, \beta^2{\norm{v^1}^2}/{m}} ),$  the conditional contribution is bounded by
\[
\EE\Bigl[
\exp\set{\frac12 Z^\intercal D Z}
\mathbf 1\{\norm{Z}^2\ge m\delta\}
\,\Big|\,u^1,v^1,a^1,b^1
\Bigr].
\]
On the event $\calG_1$, the matrix $M
:=
I_2-D^{1/2}\Gamma(a^1,b^1)D^{1/2}$ is uniformly positive definite. Indeed, its diagonal entries are positive
by the choice of $\varepsilon_0$, and its determinant equals
\[
\paren{1-\alpha^2\frac{\norm{u^1}^2}{m}\norm{a^1}^2}
\paren{1-\beta^2\frac{\norm{v^1}^2}{m}\norm{b^1}^2}
-
\alpha^2\beta^2
\frac{\norm{u^1}^2\norm{v^1}^2}{m^2}
\rho^2\ang{a^1,b^1}^2,
\]
which is at least $c_0$. Hence $M\succeq c_1I_2$ for some constant
$c_1>0$. Hence, conditionally on $\calG_1$,
\[
\EE\Bigl[
\exp\set{\frac12 Z^\intercal D Z}
\mathbf 1\{\norm{Z}^2\ge m\delta\}
\,\Big|\,u^1,v^1,a^1,b^1
\Bigr]
=
\det(M)^{-1/2}
\PP\bigl(\|{\tilde Z\|}^2\ge m\delta\bigr),
\]
where $\tilde Z$ is a centered two-dimensional Gaussian vector with
covariance
\[
\tilde\Gamma
=
\Gamma^{1/2}
\bigl(I_2-\Gamma^{1/2}D\Gamma^{1/2}\bigr)^{-1}
\Gamma^{1/2}.
\]
The operator norm of $\widetilde\Gamma$ is uniformly bounded on $\calG_1$.
Thus a standard Gaussian tail bound gives $\PP(\|\tilde Z\|^2 \geq m\delta)\leq \exp\set{-\Omega_\delta(n)}$. We now convert the numerator bound into a posterior-overlap bound. Since
\[
\mu_{U,V}^{\otimes 2}
\bigl(R_{12}\ge \delta,\ \calG_1,\calG_2\bigr)
=
\frac{I_{\delta}(U,V)}{L_\rho(U,V)^2},
\]
and since $\PP_\rho$ has density $L_\rho$ with respect to $\QQ$,
\[
\EE_{\PP_\rho}
\big[\mu_{U,V}^{\otimes 2}
\big(R_{12}\ge \delta,\ \calG_1,\calG_2\big)\big]
=
\EE_{\QQ}\Big[
\frac{I_{\delta}(U,V)}{L_\rho(U,V)}\Big].
\]
Fix $\gamma\in(0,c\tau)$, where $c$ is the constant in
\eqref{eq:wishart-overlap-numerator-exp-bound}. Splitting according to
$\{L_\rho\ge e^{-\gamma m}\}$
\begin{align*}
\EE_{\QQ}
\Big[\frac{I_{\delta}}{L_\rho}\Big]
&\le
e^{\gamma m}\EE_{\QQ}[I_{\delta}]
+
\EE_{\QQ}
\Bigl[
\frac{I_{\delta}}{L_\rho}
\mathbf 1\{L_\rho<e^{-\gamma m}\}
\Bigr].
\end{align*}
Since $I_{\delta}\le L_\rho^2$, the second term is at most $e^{-\gamma m}$. Together with \eqref{eq:wishart-overlap-numerator-exp-bound}, this yields
\[
\EE_{\PP_\rho}
\big[\mu_{U,V}^{\otimes 2}
\bigl(R_{12}\ge \delta,\ \calG_1,\calG_2\bigr)\big]=o(1).
\]
On $\calG_1\cap\calG_2$,  we have $R_{12}
\le 2(1+\varepsilon_0)^2, $. Therefore, for every $\eta>0$,
\[
\EE_{\PP_\rho}
\mu_{U,V}^{\otimes 2}
\bigl[
R_{12}\mathbf 1_{\calG_1}\mathbf 1_{\calG_2}
\bigr]
\le
\eta
+
2(1+\varepsilon_0)^2
\EE_{\PP_\rho}
\mu_{U,V}^{\otimes 2}
\bigl(R_{12}\ge\eta,\ \calG_1,\calG_2\bigr).
\]
Taking $m\to\infty$ and then $\eta\downarrow0$ proves the second displayed
claim. Finally, we remove the truncation. By the tower property of conditional
expectation,
\[
\EE_{\PP_\rho}\mu_{U,V}(\calG^c)
=
\pi_\rho(\calG^c)=o(1),
\]
Hence, we can conclude that 
\[
\EE_{\PP_\rho}
\mu_{U,V}^{\otimes 2}
\bigl(R_{12}\ge\delta\bigr)
\le
\EE_{\PP_\rho}
\mu_{U,V}^{\otimes 2}
\bigl(R_{12}\ge\delta,\ \calG_1,\calG_2\bigr)
+
2\pi_\rho(\calG^c)=o(1).
\]
Markov's inequality gives $\PP_\rho\sqb{
\mu_{U,V}^{\otimes 2}(R_{12}\ge\delta)\ge\delta}
=o(1)$, which concludes the proof.
\end{proof}

Proposition~\ref{prop:wishart-posterior-overlap} together with similar arguments as in the proof of Theorem~\ref{thm:wigner-no-weak-recovery} imply the following theorem, which proves that positive-probability weak recovery is impossible below the threshold.

\begin{theorem}\label{thm:wishart-no-weak-recovery}
Suppose $\rho<\kappa$. Let
$\hat a=\hat a(U,V)$ and $\hat b=\hat b(U,V)$ be arbitrary
estimators satisfying $\norm{\hat a}\le 1$ and $\|\hat b\| \leq 1$, almost surely. Then, for every fixed $\eta>0$,
\[
\mathcal P_\rho\bigl[
\langle \hat a,a\rangle^2+\langle\hat b,b\rangle^2\ge \eta
\bigr]=o(1)\,.
\]
\end{theorem}